\documentclass[11pt,a4paper,leqno]{article}
\usepackage{a4wide}
\usepackage{latexsym}
\usepackage{amsmath} 
\usepackage{multicol} 
\usepackage{amssymb}

\usepackage{color}

\newtheorem{defin}{Definition}
\newtheorem{lemma}{Lemma}
\newtheorem{prop}{Proposition}
\newtheorem{theo}{Theorem}
\newtheorem{corol}{Corollary}
\newenvironment{proof}{\medskip\par\noindent{\bf Proof}}{\hfill $\Box$
\medskip\par}

\def\C{\mathbb{C}}
\def\N{\mathbb{N}}
\def\R{\mathbb{R}}
\def\Q{\mathbb{Q}}

\begin{document}
\title{Gevrey multiscale expansions of singular solutions of PDEs with cubic nonlinearity}

\maketitle
\thispagestyle{empty}

\begin{multicols}{2} 
\textbf{A. Lastra}\footnote{Grupo ASYNACS (Ref. CCEE2011/R34). Partially supported by the project MTM2016-77642-C2-1-P of Ministerio de Econom\'ia,
Industria y Competitividad, Spain}\\
{\small{Dpto. de F\'isica y Matem\'aticas. \\
Universidad de Alcal\'a. Ap. Correos 20,\\ 
E-28871 Alcal\'a de Henares, Madrid (Spain)\\
{\tt alberto.lastra@uah.es}}}

\textbf{S. Malek}\footnote{Partially supported by the project MTM2016-77642-C2-1-P of Ministerio de Econom\'ia, Industria y Competitividad, Spain}\\
{\small{University of Lille 1, Laboratoire Paul Painlev\'e,\\
59655 Villeneuve d'Ascq cedex, France\\
{\tt Stephane.Malek@math.univ-lille1.fr }}}

\end{multicols}

{ \small \begin{center}
{\bf Abstract}
\end{center}

We study a singularly perturbed PDE with cubic nonlinearity depending on a complex perturbation parameter $\epsilon$. This is the continuation of the precedent work~\cite{ma} by the first author. We construct two families of sectorial meromorphic solutions obtained as a small perturbation in $\epsilon$ of two branches of an algebraic slow curve of the equation in time scale. We show that the nonsingular part of the solutions of each family shares a common formal power series in $\epsilon$ as Gevrey asymptotic expansion which might be different one to each other, in general.

\medskip


}
\bigskip 

\section{Introduction}\label{seccion1}

The main aim of this work is to study a family of singularly perturbed PDEs of the form

\begin{equation}
Q(\partial_{z})( P_{1}(t,\epsilon)u(t,z,\epsilon) + P_{2}(t,\epsilon)u^{2}(t,z,\epsilon)+P_{3}(t,\epsilon)u^3(t,z,\epsilon) )\\
= f(t,z,\epsilon) + P_{4}(t,\epsilon,\partial_{t},\partial_{z})u(t,z,\epsilon),
\label{e1}
\end{equation}
where $Q,P_j$ are polynomials with complex coefficients, for all $j=1,2,3,4$, and $f$ is an analytic function with respect to $(t,\epsilon)$ in a vicinity of the origin, and holomorphic with respect to $z$ on an horizontal strip $H_{\beta}=\{z\in\C:|\hbox{Im}(z)|<\beta\}\subseteq\C$, for some $\beta>0$.

Here, $\epsilon$ is considered as a small complex perturbation parameter. The study of singularly perturbed ordinary and partial differential equations has been recently developed by several authors. We can cite \cite{bamo,cms,camo2} as works in which the study of ODEs in which irregular singular operators appear. In~\cite{yayo}, the authors study singularly perturbed semilinear systems of equations involving fuchsian singularities in several variables. This study is now being generalized by the authors concerning both irregular and fuchsian operators~\cite{yosh}.

Recently, Carrillo and Mozo-Fern\'andez~\cite{camo} have studied integrable systems of PDEs involving irregular singularities in two variables obtained as coupled singularly perturbed problems. In~\cite{camo2}, the authors study families of linear PDEs in which the action of the sum of two singularly perturbed operators appear.

 This work follows a series of previous advances by the authors in which fixed point techniques are used to solve such problems, such as~\cite{lama,lama2,ma3,ma}. 


It provides a natural continuation of the study made by the second author in~\cite{ma}. In that work, the author considered a quadratic nonlinearity, which corresponds to our equation in the case of $P_3\equiv 0$. The main goal was to construct actual holomorphic solutions and study their asymptotic properties with respect to the complex perturbation parameter $\epsilon$. More precisely, the author has constructed a family of analytic solutions $(y_{p}(t,z,\epsilon))_{0\le p\le \varsigma-1}$ defined on a product of a finite sector with vertex at the origin, an horizontal substrip $H_{\beta'}\subset H_{\beta}$ and $\mathcal{E}_p$; where $(\mathcal{E}_p)_{0\le p\le \varsigma-1}$ is a finite set of bounded sectors which cover a pointed neighborhood of the origin. We notice that such solutions are singular with respect to $\epsilon$ and $t$ at the origin. Indeed, each solution can be split into the sum of two terms: a singular part and a bounded analytic function which admits an asymptotic expansion with respect to $\epsilon$ in $\mathcal{E}_{p}$. This asymptotic expansion turns out to be of Gevrey type. Each solution has a multiple-scale expansion in the sense of~\cite{best}, Chapter 11, which has the form
\begin{equation}\label{eaux1}
y_p(t,z,\epsilon)\sim\epsilon^\beta\left(Y_0(\epsilon^\alpha t)+\sum_{n\ge1}Y_{n}(\epsilon^\alpha t)\epsilon^{n}\right),
\end{equation}
for some $\alpha>0$, $\beta\in\Q$. Here, $Y_0$ is the unique nonvanishing rational solution of a second order algebraic equation.

The main aim of the present work is to construct sets of actual solution of (\ref{e1}), and investigate their asymptotic behavior at $\epsilon=0$, as much like as in the precedent work~\cite{ma}, in this more general framework.


As in the previous work, we construct families of solutions admitting  a multiple-scale expansion of the form
$$\epsilon^\beta\left(U_0(\epsilon^\alpha t)+\sum_{n\ge1}U_{n}(\epsilon^\alpha t)\epsilon^{n}\right),$$
comparable to those described in (\ref{eaux1}), where $\alpha>0$, $\beta\in\Q$ satisfy some restrictions described in the paper, and where $U_0$ now satisfies the algebraic equation
$$A(T) U(T)^2+B(T)U(T)+C(T)=0,$$
with $A(T)=P_1(T,0)$, $B(T)=P_2(T,0)$ and $C(T)=P_3(T,0)$.
$U_0$ is an algebraic function admitting two different branches, $U_{01}$ and $U_{02}$. This gives rise to two families of singular solutions.

On the one hand, one family, associated to $U_{01}$ is given by
$$u_{1}^{\mathfrak{d}_p}(t,z,\epsilon)=\epsilon^\beta(U_{01}(\epsilon^\alpha t)+(\epsilon^\alpha t)^{\gamma_1}v_1^{\mathfrak{d}_p}(t,z,\epsilon)),$$
is an analytic solution of the problem (\ref{e1}) defined in $\mathcal{T}_1\times H_{\beta'}\times\mathcal{E}_{p},$ for every $0\le p\le \varsigma_1-1$. Here, $\mathcal{T}_1$ stands for a finite sector with vertex at the origin and $H_{\beta'}$ is an horizontal strip in the complex plane and $(\mathcal{E}_p)_{0\le p\le \varsigma_1-1}$ is a good covering (see Definition~\ref{defingoodcovering}) of $\C^\star$.

On the other hand, a second family related to $U_{02}$ is given by
$$u_{2}^{\tilde{\mathfrak{d}}_p}(t,z,\epsilon)=\epsilon^\beta(U_{02}(\epsilon^\alpha t)+(\epsilon^\alpha t)^{\gamma_2}v_2^{\tilde{\mathfrak{d}}_p}(t,z,\epsilon)),$$
is an analytic solution of the problem (\ref{e1}) defined in $\mathcal{T}_2\times H_{\beta'}\times\tilde{\mathcal{E}}_{p},$ for every $0\le p\le \varsigma_2-1$, where $\mathcal{T}_2$ is a finite sector with vertex at the origin and $(\tilde{\mathcal{E}}_p)_{0\le p\le \varsigma_2-1}$ is a good covering of $\C^\star$.

The crucial and surprising point is that the nonsingular part of each family of solutions admits a Gevrey asymptotic expansion with respect to $\epsilon$, which are distinct, in general.

More precisely, for every $0\le p\le \varsigma_1-1$, one has that $v_{1}^{\mathfrak{d}_p}(t,z,\epsilon)$ admits the formal power series $\hat{v}_1(t,z,\epsilon)$ as its Gevrey asymptotic expansion of order
 $(\Delta_D+\beta-\alpha k_{0,1})^{-1}\delta_D$, with respect to $\epsilon$ on $\mathcal{E}_{p}$, uniformly in $\mathcal{T}_1\times H_{\beta'}$. Also, 
one has that $v_{2}^{{\tilde{\mathfrak{d}}}_p}(t,z,\epsilon)$ admits $\hat{v}_2(t,z,\epsilon)$ as its Gevrey asymptotic expansion of order $(\Delta_D+\beta-\alpha (2k_{0,2}-k_{0,3}))^{-1}\delta_D$ with respect to $\epsilon$ on $\tilde{\mathcal{E}}_{p}$, uniformly in $\mathcal{T}_2\times H_{\beta'}$.


Gevrey orders come from the highest order term of the operator $P_4$ which is an irregular operator of the shape $\epsilon^{\Delta_D}t^{d_{D}}\partial_t^{\delta_D}R_D(\partial_z)$, and the lowest powers with respect to $t$ in $P_1,P_2,P_3$.


This work falls into the recent trend of research on singular solutions of nonlinear partial differential equations.
In the framework of linear PDEs, the case of so-called Fuchsian or regular singularity in one complex variable
is a well understood subject until the fundamental works of M. Baouendi and C. Goulaouic
\cite{baouendigoulaouic}, H. Tahara \cite{tahara1} and T. Mandai \cite{mandai} who extended the
classical Frobenius method working for ODEs in order to provide the structure of all analytic, singular
with polynomial growth and logarithmic solutions near the isolated singularity. In the nonlinear context, the
results are however more partial. Nevertheless, we can quote some deep and recent results regarding
this topic. Namely, we can refer to the work by T. Kobayashi \cite{kobayashi} (inspired by the seminal contribution by
J. Weiss, M. Tabor and G. Carnevale on the celebrated Painlev\'e property for PDEs, \cite{weisstaborcarnevale})
who constructed solutions having the form of a convergent Puiseux expansions
$t^{\sigma} \sum_{k \geq 0} u_{k}(x)t^{k/p}$ for some $\sigma \in \mathbb{Q}$, $p \geq 1$ integer, for some
PDEs with non singular coefficients and polynomial nonlinearity. The situation of general analytic
nonlinearity has been performed later on by H. Tahara in \cite{tahara2}. This study has been further extended
by H. Tahara and H. Yamane in \cite{taharayamane} when resonances appear for which solutions with logarithmic
terms can be built up. In the case with singular coefficients, first order PDEs with Fuchsian singularity
known as Briot-Bouquet type equations
(as defined in the monography by R. G\'erard and H. Tahara \cite{geta}) have been extensively studied.
Namely, the general structure of bounded singular solutions with polynomial growth and logarithmic terms near
the Fuchsian singularity has been exhibited first under non resonant constraints by R. G\'erard and H. Tahara
\cite{gerardtahara2} and by H. Yamazawa in the general case, see \cite{yamazawa}.

Our main result in this paper provides in particular an example of analytic unbounded singular solutions with polynomial growth in the framework of nonlinear higher order PDEs with irregular singularity and
singular coefficients. Notice that very few works exist in this direction among the literature.

The paper is organized as follows.

In Section~\ref{seccion2}, we recall the definition and main properties under certain operators of certain Banach spaces of exponential decay and growth in different variables. Section~\ref{seccion3} is devoted to the review of analytic and formal $m_k$-Borel transformation, which is a slightly modified version of the classical ones, and which have already been used in previous works by the authors. We also describe the link between them via Gevrey asymptotic expansions. We finally consider Fourier inverse transform acting on functions with exponential decay.
 
In Section~\ref{seccion4}, we make successive transformations on the main problem (\ref{e1}) to finally arrive at two auxiliary problems in Section 4.1, studied in detail in Section 4.2 and 4.4. In Sections 4.3 and 4.5, we study the analytic solution of each of the singularly perturbed problems which have arisen from the main problem under study. This is made by means of a fixed point argument in the Banach space of functions described in Section~\ref{seccion2}.

Section~\ref{seccion5} studies the singular analytic solutions of the main problem in two different good coverings (see Theorem~\ref{teo2461}), and provides upper bounds on solutions with non empty intersection of the corresponding elements in the good covering, with respect to the perturbation parameter. In Section 6, we recall Ramis-Sibuya theorem which allows us to conclude with the second main result in the present work, Theorem~\ref{teo3045}, in which we guarantee the existence of two formal power series which asymptotically approximate some analytic functions quite related to the analytic solutions of the main problem. The work concludes with an example in which the theory developed is applied.

The following sections consist of the proofs of some results which have been left at the end for a more comprehensive lecture of the work.

\section{Banach spaces of exponential growth and decay}\label{seccion2}

The Banach spaces defined in this section are adequate modifications of those appearing in~\cite{lama1,lama2}. They incorporate both, exponential decay with respect to $m$ variable which is linked to Fourier transform, and exponential growth in $\tau$ variable, which is associated to different levels in which Borel-Laplace summation is held. This behavior is also connected to the action of the perturbation parameter $\epsilon$, as it can be observed in the following definitions.

We denote $D(0,\rho)$ the open disc centered at 0, with positive radius $\rho$, and $\bar{D}(0,\rho)$ stands for its closure. Let $S_d$ be an open unbounded sector with bisecting direction $d\in\R$ and vertex at the origin, and let $\mathcal{E}$ be an open sector with vertex at the origin, and finite radius $r_{\mathcal{E}}>0$.

\begin{defin}\label{def92}
Let $\beta>0$, $\mu>1$ be real numbers. We denote $E_{(\beta,\mu)}$ the vector space of functions $h:\R\to\C$ satisfying
$$\left\|h(m)\right\|_{(\beta,\mu)}=\sup_{m\in\R}(1+|m|)^{\mu}\exp(\beta|m|)|h(m)|<\infty.$$
The pair $(E_{(\beta,\mu)},\left\|\cdot\right\|_{(\beta,\mu)})$ turns out to be a Banach space.
\end{defin}

In view of Proposition 5 in~\cite{lama1}, it is straight to check the following result.

\begin{prop}\label{prop10}
The Banach space $(E_{(\beta,\mu)},\left\|\cdot\right\|_{(\beta,\mu)})$ is a Banach algebra when endowed with the convolution product
$$(f\star g)(m)=\int_{-\infty}^{\infty}f(m-m_1)g(m_1)dm_1.$$
More precisely, there exists $C_1>0$, depending on $\mu$, such that
$$\left\|(f\star g)(m)\right\|_{(\beta,\mu)}\le C_1\left\|f(m)\right\|_{(\beta,\mu)}\left\|g(m)\right\|_{(\beta,\mu)},$$
for every $f,g\in E_{(\beta,\mu)}$.
\end{prop}

\begin{defin}\label{def108}
Let $\nu,\rho>0$ and $\beta>0$, $\mu>1$ be real numbers. Let $\kappa\ge 1$ and $\chi,\alpha\ge0$ be integers. Let $\epsilon\in\mathcal{E}$. We denote $F^{d}_{(\nu,\beta,\mu,\chi,\alpha,\kappa,\epsilon)}$ the vector space of continuous functions $(\tau,m)\mapsto h(\tau,m)$ on $(\bar{D}(0,\rho)\cup S_d)\times\R$, holomorphic with respect to $\tau$ on $D(0,\rho)\cup S_d$ and such that
\begin{multline}
\left\|h(\tau,m)\right\|_{(\nu,\beta,\mu,\chi,\alpha,\kappa,\epsilon)}\\
=\sup_{\tau\in\bar{D}(0,\rho)\cup S_{d},m\in\R}(1+|m|)^{\mu}\exp\left(\beta|m|\right)\frac{1+\left|\frac{\tau}{\epsilon^{\chi+\alpha}}\right|^{2\kappa}}{\left|\frac{\tau}{\epsilon^{\chi+\alpha}}\right|}\exp\left(-\nu\left|\frac{\tau}{\epsilon^{\chi+\alpha}}\right|^{\kappa}\right)|h(\tau,m)|<\infty.
\end{multline}
The pair $(F^{d}_{(\nu,\beta,\mu,\chi,\alpha,\kappa,\epsilon)},\left\|\cdot\right\|_{(\nu,\beta,\mu,\chi,\alpha,\kappa,\epsilon)})$ is a Banach space.
\end{defin}

The next results describe inner transformations in the spaces introduced.  Through the whole section, we preserve the notations in Definition~\ref{def92} and Definition~\ref{def108}.

\begin{lemma}\label{lema1}
Let $\gamma_1\ge0,\gamma_2\ge1$ be integer numbers. Let $\tilde{R}(X)\in\C[X]$ such that $\tilde{R}(im)\neq 0$ for all $m\in\R$. Let $\tilde{B}(m)\in E_{(\beta,\mu)}$and let $a_{\gamma_1,\kappa}(\tau,m)$ be a continuous function defined on $(\bar{D}(0,\rho)\cup S_d)\times\R$, and holomorphic with respect to $\tau$ on $D(0,\rho)\cup S_d$, satisfying
$$|a_{\gamma_1,\kappa}(\tau,m)|\le\frac{1}{(1+|\tau|^{\kappa})^{\gamma_1}|\tilde{R}(im)|}\quad,\tau\in \bar{D}(0,\rho)\cup S_d,\quad m\in\R.$$
Then, the function $\epsilon^{-\chi\gamma_2}\tau^{\gamma_2}\tilde{B}(m)a_{\gamma_1,\kappa}(\tau,m)\in F^{d}_{(\nu,\beta,\mu,\chi,\alpha,\kappa,\epsilon)}$, and it holds that
\begin{equation}\label{e123}
\left\|\epsilon^{-\chi\gamma_2}\tau^{\gamma_2}\tilde{B}(m)a_{\gamma_1,\kappa}(\tau,m)\right\|_{(\nu,\beta,\mu,\chi,\alpha,\kappa,\epsilon)}\le C_2\frac{\left\|\tilde{B}(m)\right\|_{(\beta,\mu)}}{\inf_{m\in\R}|\tilde{R}(im)|} |\epsilon|^{\gamma_2\alpha},\qquad \epsilon\in\mathcal{E},
\end{equation}
for some $C_2>0$.
\end{lemma}
\begin{proof}
The definition of the norm in the space $F^{d}_{(\nu,\beta,\mu,\chi,\alpha,\kappa,\epsilon)}$ allows us to write
\begin{align*}
&\left\|\epsilon^{-\chi\gamma_2}\tau^{\gamma_2}\tilde{B}(m)a_{\gamma_1,\kappa}(\tau,m)\right\|_{(\nu,\beta,\mu,\chi,\alpha,\kappa,\epsilon)}\\
&\le \frac{\left\|\tilde{B}(m)\right\|_{(\beta,\mu)}}{\inf_{m\in\R}|\tilde{R}(im)|} |\epsilon|^{\gamma_2\alpha}\sup_{x\ge0}(1+x^{2\kappa})x^{\gamma_2-1}\exp(-\nu x^{\kappa}),
\end{align*}
which yileds to the result.
\end{proof}

A similar result to the following one can be found in Proposition 2,~\cite{lama1}. However, more accurate bounds are needed in the sequel, which will be provided by estimates on Mittag-Leffler function as those appearing in Proposition 1 and Proposition 5 in~\cite{lama2}.

\begin{prop}\label{prop1}
Let $\gamma_1,\gamma_2,\gamma_3\in\R$, with $\gamma_1\ge0$. Let $\tilde{R}(X),\tilde{R}_D(X)\in\C[X]$ with $\deg(\tilde{R})\le \deg(\tilde{R}_{D})$ and such that $\tilde{R}_D(im)\neq 0$ for all $m\in\R$. Let $a_{\gamma_1,\kappa}(\tau,m)$be a continuous function defined on $(\bar{D}(0,\rho)\cup S_d)\times\R$, and holomorphic with respect to $\tau$ on $D(0,\rho)\cup S_d$, satisfying
$$|a_{\gamma_1,\kappa}(\tau,m)|\le\frac{1}{(1+|\tau|^{\kappa})^{\gamma_1}|\tilde{R}_{D}(im)|}\quad,\tau\in \bar{D}(0,\rho)\cup S_d,\quad m\in\R.$$ 

We also assume that
\begin{equation}\label{e143}
\frac{1}{\kappa}+\gamma_3+1\ge0,\quad\gamma_2+\gamma_3+2\ge0,\quad\gamma_2>-1.
\end{equation}

We consider two cases:
\begin{enumerate}
\item[1)] If $\gamma_3\le- 1$, then, there exists $C_3>0$ (depending on $\nu,\kappa,\gamma_2,\gamma_3,\tilde{R}(X),\tilde{R}_D(X)$) such that
\begin{multline}
\left\|\epsilon^{-\gamma_0}a_{\gamma_1,\kappa}(\tau,m)\tilde{R}(im)\tau^{\kappa}\int_{0}^{\tau^\kappa}(\tau^\kappa-s)^{\gamma_2}s^{\gamma_3}f(s^{1/\kappa},m)ds\right\|_{(\nu,\beta,\mu,\chi,\alpha,\kappa,\epsilon)}\\
\le C_3|\epsilon|^{(\chi+\alpha)\kappa(\gamma_2+\gamma_3+2)-\gamma_0}\left\|f(\tau,m)\right\|_{(\nu,\beta,\mu,\chi,\alpha,\kappa,\epsilon)},
\end{multline}
for every $f(\tau,m)\in F^d_{(\nu,\beta,\mu,\chi,\alpha,\kappa,\epsilon)}$.
\item[2)] If $\gamma_3>- 1$ and $\gamma_1\ge 1+\gamma_3$, then, there exists $C'_3>0$ (depending on $\nu,\kappa,\gamma_1,\gamma_2,\gamma_3,\tilde{R}(X),\tilde{R}_D(X)$) such that
\begin{multline}
\left\|\epsilon^{-\gamma_0}a_{\gamma_1,\kappa}(\tau,m)\tilde{R}(im)\tau^{\kappa}\int_{0}^{\tau^\kappa}(\tau^\kappa-s)^{\gamma_2}s^{\gamma_3}f(s^{1/\kappa},m)ds\right\|_{(\nu,\beta,\mu,\chi,\alpha,\kappa,\epsilon)}\\
\le C'_3|\epsilon|^{(\chi+\alpha)\kappa(\gamma_2+\gamma_3+2)-\gamma_0-(\chi+\alpha)\kappa\gamma_1}\left\|f(\tau,m)\right\|_{(\nu,\beta,\mu,\chi,\alpha,\kappa,\epsilon)},
\end{multline}
for every $f(\tau,m)\in F^d_{(\nu,\beta,\mu,\chi,\alpha,\kappa,\epsilon)}$.
\end{enumerate}
\end{prop}

Some norm estimates concerning bilinear convolution operators acting on the Banach space above are needed.

\begin{prop}\label{prop3}
There exists $C_4>0$, depending on $\mu$ and $\kappa$, such that
\begin{multline}
\left\|\tau^{\kappa-1}\int_{0}^{\tau^{\kappa}}\int_{-\infty}^{\infty}f((\tau^\kappa-s')^{1/\kappa},m-m_1)g((s')^{1/\kappa},m_1)\frac{1}{(\tau^{\kappa}-s')s'}ds'dm_1\right\|_{(\nu,\beta,\mu,\chi,\alpha,\kappa,\epsilon)}\\
\le\frac{C_4}{|\epsilon|^{\chi+\alpha}} \left\|f(\tau,m)\right\|_{(\nu,\beta,\mu,\chi,\alpha,\kappa,\epsilon)} \left\|g(\tau,m)\right\|_{(\nu,\beta,\mu,\chi,\alpha,\kappa,\epsilon)},
\end{multline}
for every $f(\tau),g(\tau)\in F^d_{(\nu,\beta,\mu,\chi,\alpha,\kappa,\epsilon)}$.
\end{prop}

\begin{corol}\label{coro3}
There exists $C_4>0$, depending on $\mu$ and $\kappa$, such that
\begin{multline}
\left\|\tau^{k}\int_{0}^{\tau^{k}}\int_{-\infty}^{\infty}f((\tau^k-s)^{1/k},m-m_1)g(s^{1/k},m_1)\frac{dm_1 ds}{(\tau^k-s)s}\right\|_{(\nu,\beta,\mu,\chi,\alpha,\kappa,\epsilon)}\\
\le C_4\left\|f(\tau,m)\right\|_{(\nu,\beta,\mu,\chi,\alpha,\kappa,\epsilon)} \left\|g(\tau,m)\right\|_{(\nu,\beta,\mu,\chi,\alpha,\kappa,\epsilon)},
\end{multline}
for every $f(\tau,m),g(\tau,m)\in F^d_{(\nu,\beta,\mu,\chi,\alpha,\kappa,\epsilon)}$.
\end{corol}
\begin{proof}
The proof of Proposition~\ref{prop3} can be followed step by step.
\end{proof}

\section{Review on analytic and formal transformations}\label{seccion3}

This section provides a brief review on the concept of the $k-$Borel summability method of formal power series, under slightly modified transformations, which provide adecquate conditions when applied to the operators appearing in the problem under study, considered in previous works such as \cite{lama} and \cite{lama1} when studying Cauchy problems under the presence of a small perturbation parameter, and also in~\cite{ma}. The classical procedure is described in detail in~\cite{ba}, Section 3.2.

We also define and state some properties associated to Fourier inverse transform acting on functions with exponential decay.

\begin{defin}\label{defi398}
Let $k\ge1$ be an integer. Let $(m_{k}(n))_{n\ge1}$ be the sequence
$$m_{k}(n)=\Gamma\left(\frac{n}{k}\right)=\int_{0}^{\infty}t^{\frac{n}{k}-1}e^{-t}dt,\qquad n\ge1.$$
Let $(\mathbb{E},\left\|\cdot\right\|_{\mathbb{E}})$ be a complex Banach space. We say a formal power series
$$\hat{X}(T)=\sum_{n=1}^{\infty}a_{n}T^{n}\in T\mathbb{E}[[T]]$$ is $m_{k}-$summable with respect to $T$
in the direction $d\in[0,2\pi)$ if the following assertions hold:
\begin{enumerate}
\item There exists $\rho>0$ such that the $m_{k}-$Borel transform of $\hat{X}$, $\mathcal{B}_{m_{k}}(\hat{X})$, is absolutely convergent for $|\tau|<\rho$, where
$$\mathcal{B}_{m_{k}}(\hat{X})(\tau)=\sum_{n=1}^{\infty}\frac{a_{n}}{\Gamma\left(\frac{n}{k}\right)}\tau^{n}\in\tau\mathbb{E}[[\tau]].$$
\item The series $\mathcal{B}_{m_{k}}(\hat{X})$ can be analytically continued in a sector
$S=\{\tau \in \mathbb{C}^{\star}:|d-\arg(\tau)|<\delta\}$ for some $\delta>0$. In addition to this,
the extension is of exponential growth at most $k$ in $S$, meaning that there exist $C,K>0$ such that
$$\left\|\mathcal{B}_{m_{k}}(\hat{X})(\tau)\right\|_{\mathbb{E}}\le Ce^{K|\tau|^{k}},\quad \tau \in S.$$
\end{enumerate}
Under these assumptions, the vector valued Laplace transform of $\mathcal{B}_{m_{k}}(\hat{X})$ along
direction $d$ is defined by
$$\mathcal{L}_{m_{k}}^{d}\left(\mathcal{B}_{m_{k}}(\hat{X})\right)(T)=
k\int_{L_{\gamma}}\mathcal{B}_{m_{k}}(\hat{X})(u)e^{-(u/T)^k}\frac{du}{u},$$
where $L_{\gamma}$ is the path parametrized by $u\in[0,\infty)\mapsto ue^{i\gamma}$, for some
appropriate direction $\gamma$ depending on $T$, such that $L_{\gamma}\subseteq S$ and
$\cos(k(\gamma-\arg(T)))\ge\Delta>0$ for some $\Delta>0$.

The function $\mathcal{L}_{m_{k}}^{d}(\mathcal{B}_{m_{k}}(\hat{X}))$ is well defined and turns out to be a
holomorphic and bounded function in any sector of the form
$S_{d,\theta,R^{1/k}}=\{T\in\mathbb{C}^{\star}:|T|<R^{1/k},|d-\arg(T)|<\theta/2\}$, for
some $\frac{\pi}{k}<\theta<\frac{\pi}{k}+2\delta$ and $0<R<\Delta/K$. This function is known
as the $m_{k}-$sum of the formal power series $\hat{X}(T)$ in the direction $d$.
\end{defin}

The following are some elementary properties concerning the $m_{k}-$sums of formal power series which will be
crucial in our procedure.

1) The function $\mathcal{L}_{m_{k}}^{d}(\mathcal{B}_{m_{k}}(\hat{X}))(T)$ admits $\hat{X}(T)$ as its
Gevrey asymptotic expansion of order $1/k$ with respect to $T$ in $S_{d,\theta,R^{1/k}}$. More precisely,
for every $\frac{\pi}{k}<\theta_1<\theta$, there exist $C,M>0$ such that
$$\left\|\mathcal{L}^{d}_{m_{k}}(\mathcal{B}_{m_{k}}(\hat{X}))(T)-
\sum_{p=1}^{n-1}a_{p}T^{p}\right\|_{\mathbb{E}}\le CM^{n}\Gamma(1+\frac{n}{k})|T|^{n},$$
for every $n\ge2$ and $T\in S_{d,\theta_{1},R^{1/k}}$. Watson's lemma (see Proposition 11 p.75 in \cite{ba2})
allows us to affirm that $\mathcal{L}^{d}_{m_{k}}(\mathcal{B}_{m_{k}}(\hat{X}))(T)$ is unique provided that
the opening $\theta_1$ is larger than $\frac{\pi}{k}$.

2) Whenever $\mathbb{E}$ is a Banach algebra, the set of holomorphic functions having Gevrey
asymptotic expansion of order $1/k$ on a sector with values in $\mathbb{E}$ turns out
to be a differential algebra (see Theorem 18, 19 and 20 in \cite{ba2}). This, and the
uniqueness provided by Watson's lemma allow us to obtain some properties on $m_{k}-$summable
formal power series in direction $d$.

By $\star$ we denote the product in the Banach algebra and also the Cauchy product of formal power
series with coefficients in $\mathbb{E}$. Let $\hat{X}_{1}$, $\hat{X}_{2}\in T\mathbb{E}[[T]]$ be
$m_{k}-$summable formal power series in direction $d$. Let $q_1\ge q_2\ge1$ be integers.
Then $ \hat{X}_{1}+\hat{X}_{2}$, $\hat{X}_{1}\star \hat{X}_{2}$ and $T^{q_1}\partial_{T}^{q_2}\hat{X}_{1}$,
which are elements of $T\mathbb{E}[[T]]$, are $m_{k}-$summable in direction $d$. Moreover, one has
$$\mathcal{L}_{m_{k}}^{d}(\mathcal{B}_{m_{k}}(\hat{X}_{1}))(T)+
\mathcal{L}_{m_{k}}^{d}(\mathcal{B}_{m_{k}}(\hat{X}_{2}))(T)=
\mathcal{L}_{m_{k}}^{d}(\mathcal{B}_{m_{k}}(\hat{X}_{1}+\hat{X}_{2}))(T),$$
$$\mathcal{L}_{m_{k}}^{d}(\mathcal{B}_{m_{k}}(\hat{X}_{1}))(T)\star
\mathcal{L}_{m_{k}}^{d}(\mathcal{B}_{m_{k}}(\hat{X}_{2}))(T)=
\mathcal{L}_{m_{k}}^{d}(\mathcal{B}_{m_{k}}(\hat{X}_{1}\star\hat{X}_{2}))(T),$$
$$T^{q_1}\partial_{T}^{q_2}\mathcal{L}^{d}_{m_{k}}(\mathcal{B}_{m_{k}}(\hat{X}_{1}))(T)
=\mathcal{L}_{m_{k}}^{d}(\mathcal{B}_{m_{k}}(T^{q_1}\partial_{T}^{q_2}\hat{X}_{1}))(T),$$
for every $T\in S_{d,\theta_1,R^{1/k}}$.

The next proposition is written without proof, which can be found in \cite{lama1}, Proposition 6.

\begin{prop}\label{prop462}
Let $\hat{f}(t)=\sum_{n \geq 1}f_nt^n$ and $\hat{g}(t) = \sum_{n \geq 1} g_{n} t^{n}$ that belong to
$\mathbb{E}[[t]]$, where $(\mathbb{E},\left\|\cdot\right\|_{\mathbb{E}})$ is a Banach algebra. Let
$k,m\ge1$ be integers. The following formal identities hold.
$$\mathcal{B}_{m_{k}}(t^{k+1}\partial_{t}\hat{f}(t))(\tau)=k\tau^{k}\mathcal{B}_{m_{k}}(\hat{f}(t))(\tau),$$
$$\mathcal{B}_{m_{k}}(t^{m}\hat{f}(t))(\tau)=
\frac{\tau^{k}}{\Gamma\left(\frac{m}{k}\right)}\int_{0}^{\tau^{k}}
(\tau^{k}-s)^{\frac{m}{k}-1}\mathcal{B}_{m_{k}}(\hat{f}(t))(s^{1/k})\frac{ds}{s}$$
and
$$
\mathcal{B}_{m_k}( \hat{f}(t) \star \hat{g}(t) )(\tau) = \tau^{k}\int_{0}^{\tau^{k}}
\mathcal{B}_{m_k}(\hat{f}(t))((\tau^{k}-s)^{1/k}) \star \mathcal{B}_{m_k}(\hat{g}(t))(s^{1/k})
\frac{1}{(\tau^{k}-s)s} ds.
$$
\end{prop}

The proof of the next result can be found in Proposition 7 in~\cite{lama1}, and it is concern with the properties of the inverse Fourier
transform acting on continuous functions with exponential decay on $\mathbb{R}$.

\begin{prop}\label{prop479}
1) Let $f : \mathbb{R} \rightarrow \mathbb{R}$ be a continuous function with a constant $C>0$ such that
$|f(m)| \leq C \exp(-\beta |m|)$ for all $m \in \mathbb{R}$, for some $\beta > 0$. The inverse Fourier
transform of $f$ is defined by the integral representation
$$ \mathcal{F}^{-1}(f)(x) = \frac{1}{ (2\pi)^{1/2} } \int_{-\infty}^{+\infty} f(m) \exp( ixm ) dm $$
for all $x \in \mathbb{R}$. It turns out that the function $\mathcal{F}^{-1}(f)$ extends to an analytic function
on the horizontal strip
$$H_{\beta} = \{ z \in \mathbb{C} / |\mathrm{Im}(z)| < \beta \}. \label{strip_H_beta}$$
Let $\phi(m) = im f(m)$. Then, we have 
$$\partial_{z} \mathcal{F}^{-1}(f)(z) = \mathcal{F}^{-1}(\phi)(z),\quad z \in H_{\beta}.$$

2) Let $f,g \in E_{(\beta,\mu)}$ and let $\psi(m) = \frac{1}{(2\pi)^{1/2}} f \star g(m)$, the convolution product of $f$ and $g$, for all $m \in \mathbb{R}$.
From Proposition~\ref{prop10}, we know that $\psi \in E_{(\beta,\mu)}$. Moreover, one has
$$\mathcal{F}^{-1}(f)(z)\mathcal{F}^{-1}(g)(z) = \mathcal{F}^{-1}(\psi)(z),\quad z \in H_{\beta}.$$
\end{prop}

We adopt some additional notation which makes the technical reading more easy to handle. Let $k\in\N$. For every $f(\tau,m)\in\mathbb{E}_{(\beta,\mu)}[[\tau]]$, and all $g(\tau)\in\C[[\tau]]$, we write
$$g(\tau)\star_{k} f(\tau,m):=\tau^{k}\int_0^{\tau^{k}}g((\tau^{k}-s)^{1/k})f(s^{1/k},m)\frac{ds}{(\tau^k-s)s}.$$
For every $f(\tau),g(\tau)\in\C[[\tau]]$, we write
$$g(\tau)\star_{k} f(\tau):=\tau^{k}\int_0^{\tau^{k}}g((\tau^{k}-s)^{1/k})f(s^{1/k})\frac{ds}{(\tau^k-s)s}.$$
Finally, for every $f(\tau,m),g(\tau,m)\in\mathbb{E}_{(\beta,\mu)}[[\tau]]$, we write
$$g(\tau,m)\star_{k}^{E} f(\tau,m):=\tau^{k}\int_0^{\tau^{k}}\int_{-\infty}^{\infty}g((\tau^{k}-s)^{1/k},m-m_1)f(s^{1/k},m_1)\frac{dm_1 ds}{(\tau^k-s)s}.$$

\section{Main and related auxiliary problems}\label{seccion4}

Let $M_1,M_2,M_3\ge0$, $D\ge2$ be integer numbers. For every $\lambda=1,2,3$ and all $\ell\in\{0,1,\ldots,M_{\lambda}\}$ we take non negative integers $k_{\ell,\lambda},m_{\ell,\lambda}$ and complex numbers $a_{\ell,\lambda}$, with $a_{0,\lambda}\neq0$. We assume that $k_{\ell,\lambda}< k_{\ell+1,\lambda}$ for every $0\le \ell\le M_{\lambda}-1$. Let $\Delta_{\ell},d_\ell,\delta_\ell$ be non negative integers for $\ell\in\{1,\ldots, D\}$ such that $1\le \delta_\ell<\delta_{\ell+1}$ for $\ell\in\{1,\ldots,D-1\}$, and assume that $\kappa_1,\kappa_2$ are fixed positive integers which are determined in the sequel.

We also assume the two following conditions hold:
\begin{equation}\label{e92}
k_{\ell,2}+k_{0,2}>k_{0,3}+k_{0,1}>2k_{0,2},\qquad \ell\ge1,
\end{equation}
and
\begin{equation}\label{e93}
k_{0,1}\le d_{\ell}-\delta_{\ell},\qquad 1\le \ell\le D.
\end{equation}
More precisely, we assume that 
\begin{equation}\label{e94}
k_{0,1}= d_\ell-\delta_{\ell}-\delta_{\ell}\kappa_1-d_{\ell,0}, \qquad 1\le \ell\le D,
\end{equation}
where $d_{\ell,0}\ge1$ for $1\le \ell\le D-1$ and $d_{D,0}=0$.

Observe that, under (\ref{e92}) and (\ref{e93}), we have $2k_{0,2}-k_{0,3}<d_\ell-\delta_\ell$ for all $1\le \ell\le D$. We also consider elements satisfying 
\begin{equation}\label{e94b}
2k_{0,2}-k_{0,3}=d_\ell-\delta_\ell-\delta_\ell\kappa_2-\tilde{d}_{\ell,0},\quad 1\le \ell\le D,
\end{equation}
where $\tilde{d}_{\ell,0}\ge 1$ for $1\le \ell\le D-1$ and $\tilde{d}_{D,0}=0$.

Observe that, in view of conditions (\ref{e94}) and (\ref{e94b}), we get that
$$\delta_D>0,\quad d_{\ell,0}-\tilde{d}_{\ell,0}<\delta_{\ell}(\kappa_2-\kappa_1),\quad \kappa_2>\kappa_1,$$
for every $1\le\ell\le D-1$.

Let $Q(X), R_{\ell}(X)\in\C[X]$ for every $1\le \ell\le D$ which satisfy there exists a common positive integer $\upsilon$ such that 
\begin{equation}\label{e91}
Q(X)=X^{\upsilon}\tilde{Q}(X),\qquad R_{\ell}(X)=X^{\upsilon}\tilde{R}_{\ell}(X),
\end{equation}
and such that
\begin{equation}
\mathrm{deg}(\tilde{Q}) = \mathrm{deg}(\tilde{R}_{D}) \geq \mathrm{deg}(\tilde{R}_{\ell}) \ \ , 
\tilde{Q}(im) \neq 0 \ \ , \ \ \tilde{R}_{D}(im) \neq 0 \label{e540}
\end{equation}
for all $m \in \mathbb{R}$, all $1 \leq \ell \leq D-1$.
We consider the main problem under study:
\begin{multline}\label{e85}
Q(\partial_z)\left(p_1(t,\epsilon)u(t,z,\epsilon)+p_2(t,\epsilon)u^2(t,z,\epsilon)+p_3(t,\epsilon)u^3(t,z,\epsilon)\right)\\
=\sum_{j=0}^{Q}b_{j}(z)\epsilon^{n_j}t^{b_j}+\sum_{\ell=1}^{D}\epsilon^{\Delta_{\ell}}t^{d_{\ell}}\partial_{t}^{\delta_{\ell}}R_{\ell}(\partial_{z})u(t,z,\epsilon),
\end{multline}
where $p_\lambda(t,\epsilon)\in\C[t,\epsilon]$, for $\lambda=1,2,3$. More precisely, for every $\lambda=1,2,3$, we write
$$p_{\lambda}(t,\epsilon)=\sum_{\ell=0}^{M_\lambda}a_{\ell,\lambda}\epsilon^{m_{\ell,\lambda}}t^{k_{\ell,\lambda}},$$
for some non negative integers $M_{\lambda},m_{\ell,\lambda},k_{\ell,\lambda}$ and some $a_{\ell,\lambda}\in\C$. We assume that $a_{0,\lambda}\neq0$.

The coefficients $b_{j}$ are constructed as follows. For every $0\le j\le Q$, we consider functions $m\mapsto \tilde{B}_j(m)$ in the space $E_{(\beta,\mu)}$ for some $\mu>1$ and $\beta>0$. We write $B_j(m)=(im)^{\upsilon}\tilde{B}_j(m)$, where $\upsilon$ is determined in (\ref{e91}), and put
\begin{equation}\label{e95}
b_j(z)=\mathcal{F}^{-1}(m\mapsto B_j(m))(z).
\end{equation}
Observe that the construction of $b_j$ and the properties of inverse Fourier transform described in Proposition~\ref{prop479}, one has $b_j(z)=\partial^{\upsilon}_z\tilde{b}_j(z)$, where 
$$\tilde{b}_j(z)=\mathcal{F}^{-1}(m\mapsto \tilde{B}_j(m))(z).$$

We search for the solutions of (\ref{e85}) of the form 
\begin{equation}\label{e118}
u(t,z,\epsilon)=\epsilon^{\beta}U(\epsilon^{\alpha}t,z,\epsilon)
\end{equation}
for some $\alpha,\beta\in\Q$ with $\alpha>0$. We write the initial problem (\ref{e85}) in terms of $U(T,z,\epsilon)$ to get 
\begin{multline}
Q( \partial_z)\left(\left(\sum_{\ell=0}^{M_1}a_{\ell,1}\epsilon^{m_{\ell,1}+\beta-\alpha k_{\ell,1}}T^{k_{\ell,1}}\right)U(T,z,\epsilon)+\left(\sum_{\ell=0}^{M_2}a_{\ell,2}\epsilon^{m_{\ell,2}+2\beta-\alpha k_{\ell,2}}T^{k_{\ell,2}}\right)U^2(T,z,\epsilon)\right.\label{e124}\\
\left.+\left(\sum_{\ell=0}^{M_3}a_{\ell,3}\epsilon^{m_{\ell,3}+3\beta-\alpha k_{\ell,3}}T^{k_{\ell,3}}\right)U^3(T,z,\epsilon)\right)\\
=\sum_{j=0}^{Q}b_{j}(z)\epsilon^{n_j-\alpha b_j}T^{b_j}+\sum_{\ell=1}^{D}\epsilon^{\Delta_{\ell}+\alpha(\delta_\ell-d_\ell)+\beta}T^{d_{\ell}}R_{\ell}(\partial_{z})\partial_{T}^{\delta_{\ell}}U(T,z,\epsilon).
\end{multline}

\subsection{Construction of two distiguished solutions}

In this subsection, we determine two distinguished solutions of (\ref{e124}), $U_{01}$ and $U_{02}$, from which two different families of solutions are provided.

We assume $\alpha,\beta$ in (\ref{e118}) can be chosen so that
\begin{equation}\label{e134}
\Delta_{\ell}+\alpha(\delta_{\ell}-d_{\ell})+\beta>0,\qquad n_{j}-\alpha b_j>0,
\end{equation}
for every $1\le \ell\le D$ and $0\le j\le Q$. Moreover, we assume that for every $\lambda=1,2,3$ there exists $0\le s_{\lambda}\le M_{\lambda}-1$ such that
\begin{equation}\label{e138}
m_{\ell,\lambda}+\lambda\beta-\alpha k_{\ell,\lambda}\left\{ \begin{array}{lcc}
             =0 &   if  & 0\le \ell\le s_{\lambda} \\
             \\ >0 &  if  & s_{\lambda}+1\le \ell\le M_{\lambda}.
             \end{array}\right.
\end{equation}

The motivation for this last assumption is related to the nature of the roots of the polynomial $p_{\lambda}(t,\epsilon)$. In order to illustrate this, let us consider $M_{\lambda}=1$ and $1\le k_{0,\lambda}<k_{1,\lambda}$. Then, $p_{\lambda}$ admits $t=0$ as a root of order $k_{0,\lambda}$ when considered as a polynomial in $t$ variable. The modulus of the other nonzero $k_{1,\lambda}-k_{0,\lambda}$ roots of $p_{\lambda}$ equals 
$$(|a_{0,\lambda}|/|a_{1,\lambda}|)^{1/(k_{1,\lambda}-k_{0,\lambda})}|\epsilon|^{\frac{m_{0,\lambda}-m_{1,\lambda}}{k_{1,\lambda}-k_{0,\lambda}}}.$$
Assumption (\ref{e138}) entails that the only root of $p_{\lambda}$ with respect to $t$ which remains bounded for all $\epsilon$ closed to zero is $t=0$.

We assume a solution of (\ref{e124}), $U(T,z,\epsilon)$, can be written as a power series with respect to $\epsilon$ in the form
\begin{equation}\label{e150}
U(T,z,\epsilon)=U_{0}(T)+\sum_{n\ge1}U_{n}(T,z)\epsilon^n,
\end{equation}
where $U_0(T)\neq0$ is chosen among the nonzero solutions of 
\begin{equation}\label{e154}
\left(\sum_{\ell=0}^{s_1}a_{\ell,1}T^{k_{\ell,1}}\right)U_0(T)+\left(\sum_{\ell=0}^{s_2}a_{\ell,2}T^{k_{\ell,2}}\right)(U_0(T))^2+\left(\sum_{\ell=0}^{s_3}a_{\ell,3}T^{k_{\ell,3}}\right)(U_0(T))^3=0.
\end{equation}
Such a function is known as a slow curve following the terminology in~\cite{cdrss}.

Under the hypotheses (\ref{e540}) and (\ref{e95}), we observe by factoring out the operator $\partial_{z}^{\upsilon}$ from (\ref{e124}), that $U(T,z,\epsilon)$ solves the related PDE
\begin{multline}
 \tilde{Q}( \partial_z)\left(\left(\sum_{\ell=0}^{M_1}a_{\ell,1}\epsilon^{m_{\ell,1}+\beta-\alpha k_{\ell,1}}T^{k_{\ell,1}}\right)U(T,z,\epsilon)+\left(\sum_{\ell=0}^{M_2}a_{\ell,2}\epsilon^{m_{\ell,2}+2\beta-\alpha k_{\ell,2}}T^{k_{\ell,2}}\right)U^2(T,z,\epsilon)\right.\\
\left.+\left(\sum_{\ell=0}^{M_3}a_{\ell,3}\epsilon^{m_{\ell,3}+3\beta-\alpha k_{\ell,3}}T^{k_{\ell,3}}\right)U^3(T,z,\epsilon)\right)\\
=\sum_{j=0}^{Q}\tilde{b}_{j}(z)\epsilon^{n_j-\alpha b_j}T^{b_j}+F(T,z,\epsilon)+\sum_{\ell=1}^{D}\epsilon^{\Delta_{\ell}+\alpha(\delta_\ell-d_\ell)+\beta}T^{d_{\ell}}\tilde{R}_{\ell}(\partial_{z})\partial_{T}^{\delta_{\ell}}U(T,z,\epsilon).
\label{e614}
\end{multline}
where the forcing term $F(T,z,\epsilon)$ is a polynomial in $z$ of degree less than $\upsilon-1$.

According to the assumptions (\ref{e134}), (\ref{e138}) and using the fact
that $\tilde{Q}(0) \neq 0$, by taking $\epsilon=0$ into equation (\ref{e614}) we see that the constraint (\ref{e154}) is equivalent to
the fact that $F(T,z,0) \equiv 0$.
 The precise shape of the term $F(T,z,\epsilon)$ will be given
in Section~\ref{seccion5}, see (\ref{e2629}) and (\ref{e235bd}).

The nonzero solutions $U_0(T)$ of (\ref{e150}) satisfy the following equation
\begin{equation}\label{e162}
A(T)(U_0(T))^2+B(T)U_0(T)+C(T)=0,
\end{equation}
where 
$$A(T)=\sum_{\ell=0}^{s_3}a_{\ell,3}T^{k_{\ell,3}},\quad B(T)=\sum_{\ell=0}^{s_2}a_{\ell,2}T^{k_{\ell,2}},\quad C(T)=\sum_{\ell=0}^{s_1}a_{\ell,1}T^{k_{\ell,1}}.$$

Let $\Delta(T)=B^2-4AC$. Equation (\ref{e162}) has two nonzero solutions, namely
\begin{equation}\label{e168}
U_{01}(T)=\frac{-B+\sqrt{\Delta}}{2A},\qquad U_{02}(T)=\frac{-B-\sqrt{\Delta}}{2A}.
\end{equation}

We have
$$A(T)=a_{03}T^{k_{0,3}}(1+\tilde{A}(T)),\quad \tilde{A}\in\C[T], \tilde{A}(0)=0,$$
$$B(T)=a_{02}T^{k_{0,2}}(1+\tilde{B}(T)),\quad \tilde{B}\in\C[T], \tilde{B}(0)=0,$$
$$C(T)=a_{01}T^{k_{0,1}}(1+\tilde{C}(T)),\quad \tilde{C}\in\C[T], \tilde{C}(0)=0,$$
which yields
$$\Delta=a_{02}^2T^{2k_{0,2}}(1+\tilde{B}(T))^2-4a_{0,3}a_{0,1}T^{k_{0,3}+k_{0,1}}(1+\tilde{A}(T))(1+\tilde{C}(T)).$$
Regarding (\ref{e92}), we can write
$$\Delta=a_{02}^2T^{2k_{0,2}}\left((1+\tilde{B}(T))^2-\frac{4a_{0,3}a_{0,1}}{a_{0,2}^2}T^{k_{0,3}+k_{0,1}-2k_{0,2}}(1+\tilde{A}(T))(1+\tilde{C}(T))\right).$$
Again, by (\ref{e92}), we guarantee the existence of $\tilde{B}_2(T)\in\C[T]$ with $\tilde{B}_2(0)=0$ such that
$$\tilde{B}(T)=T^{k_{0,3}+k_{0,1}-2k_{0,2}}\tilde{B}_2(T).$$
This yields 
\begin{align*}
\sqrt{\Delta}&=a_{02}T^{k_{0,2}}\left(1+2T^{k_{0,3}+k_{0,1}-2k_{0,2}}\tilde{B}_2(T)+(\tilde{B}_2(T))^2T^{2(k_{0,3}+k_{0,1}-2k_{0,2})}\right.\\
&\left.-\frac{4a_{0,3}a_{0,1}}{a_{0,2}^2}T^{k_{0,3}+k_{0,1}-2k_{0,2}}(1+\tilde{A}(T))(1+\tilde{C}(T))\right)^{1/2}\\
&=a_{02}T^{k_{0,2}}\left(1-\frac{4a_{0,3}a_{0,1}}{a_{0,2}^2}T^{k_{0,3}+k_{0,1}-2k_{0,2}}(1+D(T))\right)^{1/2},\\
&=a_{02}T^{k_{0,2}}\left(1-\frac{2a_{0,3}a_{0,1}}{a_{0,2}^2}T^{k_{0,3}+k_{0,1}-2k_{0,2}}+T^{k_{0,3}+k_{0,1}-2k_{0,2}}E(T)\right),
\end{align*}
with $D(T),E(T)\in\C\{T\}$ with $D(0)=E(0)=0$. Taking this into account, we get the two solutions of (\ref{e162}) have the form of (\ref{e168}), and are given by
\begin{align}
U_{01}(T)=&\frac{-a_{0,2}T^{k_{0,2}}(1+\tilde{B}(T))+a_{0,2}T^{k_{0,2}}\left(1-\frac{2a_{0,3}a_{0,1}}{a_{0,2}^2}T^{k_{0,3}+k_{0,1}-2k_{0,2}}+T^{k_{0,3}+k_{0,1}-2k_{0,2}}E(T)\right)}{2a_{0,3}T^{k_{0,3}}(1+\tilde{A}(T))}\nonumber\\
&=\frac{-a_{0,2}T^{k_{0,1}-k_{0,2}}\tilde{B}_2(T)-\frac{2a_{0,3}a_{0,1}}{a_{0,2}}T^{k_{0,1}-k_{0,2}}+T^{k_{0,1}-k_{0,2}}E(T)}{2a_{0,3}(1+\tilde{A}(T))}\nonumber\\
&=-\frac{a_{0,1}}{a_{0,2}}T^{k_{0,1}-k_{0,2}}(1+\mathcal{J}_1(T))\label{e199},
\end{align}
with $\mathcal{J}_1(T)\in\C\{T\}$ and $\mathcal{J}_1(0)=0$, and 
\begin{align}
U_{02}(T)=&\frac{-a_{0,2}T^{k_{0,2}}(1+\tilde{B}(T))-a_{0,2}T^{k_{0,2}}\left(1-\frac{2a_{0,3}a_{0,1}}{a_{0,2}^2}T^{k_{0,3}+k_{0,1}-2k_{0,2}}+T^{k_{0,3}+k_{0,1}-2k_{0,2}}E(T)\right)}{2a_{0,3}T^{k_{0,3}}(1+\tilde{A}(T))}\nonumber\\
&=\frac{-a_{0,2}T^{k_{0,2}}(2+E_2(T))}{2a_{0,3}T^{k_{0,3}}(1+\tilde{A}(T))}=-\frac{a_{0,2}}{a_{0,3}}T^{k_{0,2}-k_{0,3}}(1+\mathcal{J}_2(T))\label{e204},
\end{align}
with $E_2(T)\in\C\{T\}$, $E_2(0)=0$, and $\mathcal{J}_2(T)\in\C\{T\}$ with $\mathcal{J}_2(0)=0$.

The behavior of $U_{01}(T)$ and $U_{02}(T)$ near the origin motivates the  choice as candidates for solutions of (\ref{e124}) described in the two following subsections.

\subsection{First perturbed auxiliary problem}

The form of $U_{01}(T)$ in (\ref{e199}) motivates a first concrete form of a solution of (\ref{e124}):
\begin{equation}\label{e216}
U_{1}(T,z,\epsilon)=-\frac{a_{0,1}}{a_{0,2}}T^{k_{0,1}-k_{0,2}}(1+\mathcal{J}_1(T))+T^{\gamma_1}V_1(T,z,\epsilon),
\end{equation}
for some $\gamma_1\in\Q$. We assume this choice is made accordingly to the following conditions:
\begin{equation}\label{e218}
\gamma_1\ge k_{0,1}-k_{0,2}, 
\end{equation}
and 
\begin{equation}\label{e222}
\gamma_1\le  b_j-k_{0,1},\qquad j=0,\ldots,Q.
\end{equation}

Observe that, in view of (\ref{e92}) we derive
\begin{equation}\label{e686}
\gamma_1\ge k_{0,2}-k_{0,3},\qquad\hbox{ and }\qquad 2\gamma_1\ge k_{0,1}-k_{0,3}.
\end{equation}

In order to search for such a solution, we plug the previous expression into (\ref{e124}). In view of Assumption (\ref{e91}), only those terms depending on $z$ appear on the resulting equation when $Q(\partial_z)$ or $R_{\ell}(\partial_z)$ are applied. This yields
\begin{multline}
Q(\partial_z)\left[\left(\sum_{\ell=0}^{s_1}a_{\ell,1}T^{k_{\ell,1}}+\sum_{\ell=s_1+1}^{M_1}a_{\ell,1}\epsilon^{m_{\ell,1}+\beta-\alpha k_{\ell,1}}T^{k_{\ell,1}}\right)T^{\gamma_1}V_1(T,z,\epsilon)  \right.\\
+\left(\sum_{\ell=0}^{s_2}a_{\ell,2}T^{k_{\ell,2}}+\sum_{\ell=s_2+1}^{M_2}a_{\ell,2}\epsilon^{m_{\ell,2}+2\beta-\alpha k_{\ell,2}}T^{k_{\ell,2}}\right)\\
\times\left(T^{2\gamma_1}V_1^2(T,z,\epsilon)-\frac{2a_{0,1}}{a_{0,2}}T^{k_{0,1}-k_{0,2}+\gamma_1}(1+\mathcal{J}_1(T))V_1(T,z,\epsilon)\right)\\
+\left(\sum_{\ell=0}^{s_3}a_{\ell,3}T^{k_{\ell,3}}+\sum_{\ell=s_3+1}^{M_3}a_{\ell,3}\epsilon^{m_{\ell,3}+3\beta-\alpha k_{\ell,3}}T^{k_{\ell,3}}\right)\left\{3\left(\frac{a_{01}}{a_{02}}T^{k_{0,1}-k_{0,2}}(1+\mathcal{J}_1(T))\right)^2T^{\gamma_1}V_1(T,z,\epsilon)\right.\\
\left.\left.-3\left(\frac{a_{01}}{a_{02}}T^{k_{0,1}-k_{0,2}}(1+\mathcal{J}_1(T))\right)T^{2\gamma_1}V_1^2(T,z,\epsilon)+T^{3\gamma_1}V_1^3(T,z,\epsilon) \right\}\right] \\
=\sum_{j=0}^{Q}b_{j}(z)\epsilon^{n_j-\alpha b_j}T^{b_j}\\
+\sum_{\ell=1}^{D}\epsilon^{\Delta_{\ell}+\alpha(\delta_{\ell}-d_{\ell})+\beta}T^{d_{\ell}}R_{\ell}(\partial_z)\left(\sum_{q_1+q_2=\delta_\ell}\frac{\delta_\ell!}{q_1!q_2!}\prod_{d=0}^{q_1-1}(\gamma_1-d)T^{\gamma_1-q_1}\partial_{T}^{q_2}V_1(T,z,\epsilon)\right),\label{e234}
\end{multline}
where we have used the notation $\prod_{d=0}^{-1}(\gamma_1-d)=1$.

In view of conditions (\ref{e92}), (\ref{e218}), (\ref{e222}) and the monotony of the sequence $(k_{\ell,\lambda})_{\ell\ge0}$ for all $\lambda=1,2,3$ one can divide equation (\ref{e234}) by $T^{k_{0,1}+\gamma_1}$ and preserve the analiticity near the origin with respect to $T$ in the terms involved in the equation. In addition to that, the coefficient of $Q(\partial_z)V_1(T,z,\epsilon)$ turns out to be invertible at $T=0$ since $a_{0,1}\neq0$. The resulting problem, whose terms have been arranged by increasing powers of $V_1(T,z,\epsilon)$, reads as follows:

\begin{multline}
Q(\partial_z)V_1(T,z,\epsilon)\left[-a_{0,1}+\sum_{\ell=1}^{s_1}a_{\ell,1}T^{k_{\ell,1}-k_{0,1}}+\sum_{\ell=s_1+1}^{M_1}a_{\ell,1}\epsilon^{m_{\ell,1}+\beta-\alpha k_{\ell,1}}T^{k_{\ell,1}-k_{0,1}}\right.\\
+\left(\sum_{\ell=1}^{s_2}a_{\ell,2}T^{k_{\ell,2}-k_{0,2}}+\sum_{\ell=s_2+1}^{M_2}a_{\ell,2}\epsilon^{m_{\ell,2}+2\beta-\alpha k_{\ell,2}}T^{k_{\ell,2}-k_{0,2}}\right)\left(-\frac{2a_{0,1}}{a_{0,2}}(1+\mathcal{J}_1(T))\right)\\
\left.+\left(\sum_{\ell=0}^{s_3}a_{\ell,3}T^{k_{\ell,3}-k_{0,1}}+\sum_{\ell=s_3+1}^{M_3}a_{\ell,3}\epsilon^{m_{\ell,3}+3\beta-\alpha k_{\ell,3}}T^{k_{\ell,3}-k_{0,1}}\right)3\left(\frac{a_{01}}{a_{02}}T^{k_{0,1}-k_{0,2}}(1+\mathcal{J}_1(T))\right)^2\right]\\
+Q(\partial_{z})V_1^2(T,z,\epsilon)\left[\left(\sum_{\ell=0}^{s_2}a_{\ell,2}T^{k_{\ell,2}}+\sum_{\ell=s_2+1}^{M_2}a_{\ell,2}\epsilon^{m_{\ell,2}+2\beta-\alpha k_{\ell,2}}T^{k_{\ell,2}}\right)T^{\gamma_1-k_{0,1}}\right.\\
\left.+\left(\sum_{\ell=0}^{s_3}a_{\ell,3}T^{k_{\ell,3}}+\sum_{\ell=s_3+1}^{M_3}a_{\ell,3}\epsilon^{m_{\ell,3}+3\beta-\alpha k_{\ell,3}}T^{k_{\ell,3}}\right)\left(\frac{-3a_{01}}{a_{02}}T^{k_{0,1}-k_{0,2}}(1+\mathcal{J}_1(T))\right)T^{\gamma_1-k_{0,1}}\right]\\
+Q(\partial_{z})V_1^3(T,z,\epsilon)\left[\left(\sum_{\ell=0}^{s_3}a_{\ell,3}T^{k_{\ell,3}}+\sum_{\ell=s_3+1}^{M_3}a_{\ell,3}\epsilon^{m_{\ell,3}+3\beta-\alpha k_{\ell,3}}T^{k_{\ell,3}}\right)T^{2\gamma_1-k_{0,1}}\right]\\
=\sum_{j=0}^{Q}b_{j}(z)\epsilon^{n_j-\alpha b_j}T^{b_j-k_{0,1}-\gamma_1}\\
+\sum_{\ell=1}^{D}\epsilon^{\Delta_{\ell}+\alpha(\delta_{\ell}-d_{\ell})+\beta}\left(\sum_{q_1+q_2=\delta_\ell}\frac{\delta_\ell!}{q_1!q_2!}\prod_{d=0}^{q_1-1}(\gamma_1-d)T^{d_{\ell}-q_1-k_{0,1}}R_{\ell}(\partial_z)\partial_{T}^{q_2}V_1(T,z,\epsilon)\right).\label{e235}
\end{multline}

At this point, we specify the form of $U_1(T,z,\epsilon)$ in (\ref{e216}), with 
\begin{equation}\label{e256}
V_1(T,z,\epsilon):=\mathbb{V}_1(\epsilon^{\chi_1}T,z,\epsilon), \hbox{ with } \chi_1:=\frac{\Delta_D+\alpha(\delta_{D}-d_D)+\beta}{d_D-k_{0,1}-\delta_D}.
\end{equation}

Equation (\ref{e235}) reads as follows:

\begin{multline}
Q(\partial_z)\mathbb{V}_1(\mathbb{T},z,\epsilon)\left[-a_{0,1}+\sum_{\ell=1}^{s_1}a_{\ell,1}\epsilon^{-\chi_1(k_{\ell,1}-k_{0,1})}\mathbb{T}^{k_{\ell,1}-k_{0,1}}+\sum_{\ell=s_1+1}^{M_1}a_{\ell,1}\epsilon^{m_{\ell,1}+\beta-\alpha k_{\ell,1}-\chi_1(k_{\ell,1}-k_{0,1})}\mathbb{T}^{k_{\ell,1}-k_{0,1}}\right.\\
\qquad+\left(\sum_{\ell=1}^{s_2}a_{\ell,2}\epsilon^{-\chi_1(k_{\ell,2}-k_{0,2})}\mathbb{T}^{k_{\ell,2}-k_{0,2}}+\sum_{\ell=s_2+1}^{M_2}a_{\ell,2}\epsilon^{m_{\ell,2}+2\beta-\alpha k_{\ell,2}-\chi_1(k_{\ell,2}-k_{0,2})}\mathbb{T}^{k_{\ell,2}-k_{0,2}}\right)\left(\frac{-2a_{0,1}}{a_{0,2}}\right)\\
+\left(\sum_{\ell=0}^{s_2}a_{\ell,2}\epsilon^{-\chi_1(k_{\ell,2}-k_{0,2})}\mathbb{T}^{k_{\ell,2}-k_{0,2}}+\sum_{\ell=s_2+1}^{M_2}a_{\ell,2}\epsilon^{m_{\ell,2}+2\beta-\alpha k_{\ell,2}-\chi_1(k_{\ell,2}-k_{0,2})}\mathbb{T}^{k_{\ell,2}-k_{0,2}}\right)\left(\frac{-2a_{0,1}}{a_{0,2}}\mathcal{J}_1(\epsilon^{-\chi_1}\mathbb{T})\right)\\
+\left(\sum_{\ell=0}^{s_3}a_{\ell,3}\epsilon^{-\chi_1(k_{\ell,3}-k_{0,1})}\mathbb{T}^{k_{\ell,3}-k_{0,1}}+\sum_{\ell=s_3+1}^{M_3}a_{\ell,3}\epsilon^{m_{\ell,3}+3\beta-\alpha k_{\ell,3}-\chi_1(k_{\ell,3}-k_{0,1})}\mathbb{T}^{k_{\ell,3}-k_{0,1}}\right)\\
\left.\hfill\times3\left(\frac{a_{01}}{a_{02}}\epsilon^{-\chi_1(k_{0,1}-k_{0,2})}\mathbb{T}^{k_{0,1}-k_{0,2}}(1+\mathcal{J}_1(\epsilon^{-\chi_1}\mathbb{T}))\right)^2\right]\\
+Q(\partial_{z})\mathbb{V}_1^2(\mathbb{T},z,\epsilon)\left[\left(\sum_{\ell=0}^{s_2}a_{\ell,2}\epsilon^{-\chi_1k_{\ell,2}}\mathbb{T}^{k_{\ell,2}}+\sum_{\ell=s_2+1}^{M_2}a_{\ell,2}\epsilon^{m_{\ell,2}+2\beta-\alpha k_{\ell,2}-\chi_1k_{\ell,2}}\mathbb{T}^{k_{\ell,2}}\right)\epsilon^{-\chi_1(\gamma_1-k_{0,1})}\mathbb{T}^{\gamma_1-k_{0,1}}\right.\\
+\left(\sum_{\ell=0}^{s_3}a_{\ell,3}\epsilon^{-\chi_1k_{\ell,3}}\mathbb{T}^{k_{\ell,3}}+\sum_{\ell=s_3+1}^{M_3}a_{\ell,3}\epsilon^{m_{\ell,3}+3\beta-\alpha k_{\ell,3}-\chi_1k_{\ell,3}}\mathbb{T}^{k_{\ell,3}}\right)\hfill\\
\left.\hfill\times\left(\frac{-3a_{01}}{a_{02}}\epsilon^{-\chi_1(k_{0,1}-k_{0,2})}\mathbb{T}^{k_{0,1}-k_{0,2}}(1+\mathcal{J}_1(\epsilon^{-\chi_1}\mathbb{T}))\right)\epsilon^{-\chi_1(\gamma_1-k_{0,1})}\mathbb{T}^{\gamma_1-k_{0,1}}\right]\\
+Q(\partial_{z})\mathbb{V}_1^3(\mathbb{T},z,\epsilon)\left[\left(\sum_{\ell=0}^{s_3}a_{\ell,3}\epsilon^{-\chi_1k_{\ell,3}}\mathbb{T}^{k_{\ell,3}}+\sum_{\ell=s_3+1}^{M_3}a_{\ell,3}\epsilon^{m_{\ell,3}+3\beta-\alpha k_{\ell,3}-\chi_1k_{\ell,3}}\mathbb{T}^{k_{\ell,3}}\right)\epsilon^{-\chi_1(2\gamma_1-k_{0,1})}\mathbb{T}^{2\gamma_1-k_{0,1}}\right]\\
=\sum_{j=0}^{Q}b_{j}(z)\epsilon^{n_j-\alpha b_j-\chi_1(b_j-k_{0,1}-\gamma_1)}\mathbb{T}^{b_j-k_{0,1}-\gamma_1}\\
+\sum_{\ell=1}^{D-1}\epsilon^{\Delta_{\ell}+\alpha(\delta_{\ell}-d_{\ell})+\beta}\left(\sum_{q_1+q_2=\delta_\ell}\frac{\delta_\ell!}{q_1!q_2!}\prod_{d=0}^{q_1-1}(\gamma_1-d)\epsilon^{-\chi_1(d_{\ell}-q_1-k_{0,1}-q_2)}\mathbb{T}^{d_{\ell}-q_1-k_{0,1}}R_{\ell}(\partial_z)\partial_{\mathbb{T}}^{q_2}\mathbb{V}_1(\mathbb{T},z,\epsilon)\right)\\
+\left(\sum_{q_1+q_2=\delta_D}\frac{\delta_D!}{q_1!q_2!}\prod_{d=0}^{q_1-1}(\gamma_1-d)\mathbb{T}^{d_{D}-q_1-k_{0,1}}R_{D}(\partial_z)\partial_{\mathbb{T}}^{q_2}\mathbb{V}_1(\mathbb{T},z,\epsilon)\right).\label{e236}
\end{multline}

Observe that the choice in (\ref{e256}) makes the term with index $\ell=D$ on the right handside of (\ref{e236}) do not depend on $\epsilon$. We have split this term for the sake of clarity of the subsequent argument, and the prominent role played on it.

Let $1\le \ell\le D$. It is worth pointing out that for every nonnegative integers $q_1,q_2$ such that $q_1+q_2=\delta_{\ell}$, and in view of (\ref{e94}), it holds that
\begin{equation}\label{e287}
d_{\ell}-k_{0,1}-q_1=(\kappa_1+1)q_2+d_{\ell,q_1,q_2},
\end{equation}
with $d_{\ell,q_1,q_2}\ge1$ for $1\le \ell\le D-1$ or $\ell=D$ and $q_2<\delta_D$; and $d_{D,0,\delta_{D}}=0$.

Regarding (\ref{e94}) and (\ref{e287}), one can apply Formula (8.7) in~\cite{taya} p. 3630 which yields
\begin{multline}
\mathbb{T}^{d_D-k_{0,1}}\partial_{\mathbb{T}}^{\delta_{D}}\mathbb{V}_1(\mathbb{T},z,\epsilon)=\left((\mathbb{T}^{\kappa_1+1}\partial_{\mathbb{T}})^{\delta_{D}}+\sum_{1\le p\le \delta_{D}-1}A_{\delta_{D},p}\mathbb{T}^{\kappa_1(\delta_{D}-p)}(\mathbb{T}^{\kappa_1+1}\partial_{\mathbb{T}})^p\right)\mathbb{V}_1(\mathbb{T},z,\epsilon),\\
\mathbb{T}^{d_\ell-k_{0,1}-(\delta_{\ell}-1)}\partial_{\mathbb{T}}\mathbb{V}_1(\mathbb{T},z,\epsilon)=\mathbb{T}^{d_{\ell,\delta_\ell-1,1}}(\mathbb{T}^{\kappa_1+1}\partial_{\mathbb{T}})\mathbb{V}_1(\mathbb{T},z,\epsilon),\\
\mathbb{T}^{d_\ell-k_{0,1}-q_1}\partial_{\mathbb{T}}^{q_2}\mathbb{V}_1(\mathbb{T},z,\epsilon)=\mathbb{T}^{d_{\ell,q_1,q_2}}\left((\mathbb{T}^{\kappa_1+1}\partial_{\mathbb{T}})^{q_2}+\sum_{1\le p\le q_2-1}A_{q_2,p}\mathbb{T}^{\kappa_1(q_2-p)}(\mathbb{T}^{\kappa_1+1}\partial_{\mathbb{T}})^p\right)\mathbb{V}_1(\mathbb{T},z,\epsilon),\label{e301}
\end{multline}
for every $1\le \ell\le D-1$, and all integers $q_1\ge0$ and $q_2\ge 2$ with $q_1+q_2=\delta_\ell$. Here, $A_{\delta_{D},p}$ for $1\le p\le \delta_D-1$ and $A_{q_2,p}$ for $1\le p\le q_2-1$ stand for real constants.

The previous identities allow us to obtain positive results in the Borel plane due to the properties held by Borel transform with respect to the terms involved in those identities. For that purpose, we assume that $\mathbb{V}_1(\mathbb{T},z,\epsilon)$ has a formal power expansion of the form
\begin{equation}\label{e306}
\mathbb{V}_1(\mathbb{T},z,\epsilon)=\sum_{n\ge1}\mathbb{V}_{n,1}(z,\epsilon)\mathbb{T}^n,
\end{equation}
where its coefficients are defined as the inverse Fourier transform of certain appropriate functions in $E_{(\beta,\mu)}$, depending holomorphically on $\epsilon$ on some punctured disc $D(0,\epsilon_0)\setminus\{0\}$, for some $\epsilon_0>0$.
$$\mathbb{V}_{n,1}(z,\epsilon)=\mathcal{F}^{-1}(m\mapsto \omega_{n,1}(m,\epsilon))(z).$$
Our main aim is to search for such coefficients, and we proceed following a fixed point argument in appropriate Banach spaces. We consider the formal $m_{\kappa_1}-$Borel transform with respect to $\mathbb{T}$ and the Fourier transform with respect to $z$ of $\mathbb{V}_1(\mathbb{T},z,\epsilon)$, and we denote it by
$$\omega_{1}(\tau,m,\epsilon)=\sum_{n\ge 1}\frac{\omega_{n,1}(m,\epsilon)}{\Gamma\left(\frac{n}{\kappa_1}\right)}\tau^n.$$

By plugging $w_{1}(\tau,m,\epsilon)$ into (\ref{e236}) and taking into account (\ref{e301}) and the hypotheses made on the differential operators in (\ref{e91}), we arrive at the following auxiliary problem
\begin{equation}\label{e315}
L_{1,\kappa_1}(\omega_{1}(\tau,m,\epsilon))+L_{2,\kappa_1}(\omega_{1}(\tau,m,\epsilon))+L_{3,\kappa_1}(\omega_{1}(\tau,m,\epsilon))=R_{1,\kappa_1}(\omega_{1}(\tau,m,\epsilon)),
\end{equation}  
with $\omega_{1}(0,m,\epsilon)\equiv 0$. We have taken into account the properties and the notation described in Section~\ref{seccion3}, for a more compact writing. We write $\mathcal{B}_{\kappa_1}J_1(\tau,\epsilon)$ for the $m_{\kappa_1}-$Borel transform of $\mathcal{J}_1(\epsilon^{-\chi_1}\mathbb{T})$ with respect to $\mathbb{T}$. The operators in (\ref{e315}) are given by
\begin{multline}
L_{1,\kappa_1}(\omega_1)=\tilde{Q}(im)\left[-a_{0,1}\omega_1+\sum_{\ell=1}^{s_1}a_{\ell,1}\epsilon^{-\chi_1(k_{\ell,1}-k_{0,1})}\frac{\tau^{k_{\ell,1}-k_{0,1}}}{\Gamma\left(\frac{k_{\ell,1}-k_{0,1}}{\kappa_1}\right)}\star_{\kappa_1}\omega_1\right.\\
\hfill+\sum_{\ell=s_1+1}^{M_1}a_{\ell,1}\epsilon^{m_{\ell,1}+\beta-\alpha k_{\ell,1}-\chi_1(k_{\ell,1}-k_{0,1})}\frac{\tau^{k_{\ell,1}-k_{0,1}}}{\Gamma\left(\frac{k_{\ell,1}-k_{0,1}}{\kappa_1}\right)}\star_{\kappa_1} \omega_1\\
+\sum_{\ell=1}^{s_2}\frac{-2a_{\ell,2}a_{0,1}}{a_{0,2}}\epsilon^{-\chi_1(k_{\ell,2}-k_{0,2})}\frac{\tau^{k_{\ell,2}-k_{0,2}}}{\Gamma\left(\frac{k_{\ell,2}-k_{0,2}}{\kappa_1}\right)}\star_{\kappa_1}\omega_1\hfill\\
\hfill+\sum_{\ell=s_2+1}^{M_2}\frac{-2a_{\ell,2}a_{0,1}}{a_{0,2}}\epsilon^{m_{\ell,2}+2\beta-\alpha k_{\ell,2}-\chi_1(k_{\ell,2}-k_{0,2})}\frac{\tau^{k_{\ell,2}-k_{0,2}}}{\Gamma\left(\frac{k_{\ell,2}-k_{0,2}}{\kappa_1}\right)}\star_{\kappa_1}\omega_1\\
+\sum_{\ell=0}^{s_2}\frac{-2a_{\ell,2}a_{0,1}}{a_{0,2}}\epsilon^{-\chi_1(k_{\ell,2}-k_{0,2})}\frac{\tau^{k_{\ell,2}-k_{0,2}}}{\Gamma\left(\frac{k_{\ell,2}-k_{0,2}}{\kappa_1}\right)}\star_{\kappa_1}\mathcal{B}_{\kappa_1}J_1(\tau,\epsilon)\star_{\kappa_1}\omega_1\hfill\\
\hfill+\sum_{\ell=s_2+1}^{M_2}\frac{-2a_{\ell,2}a_{0,1}}{a_{0,2}}\epsilon^{m_{\ell,2}+2\beta-\alpha k_{\ell,2}-\chi_1(k_{\ell,2}-k_{0,2})}\frac{\tau^{k_{\ell,2}-k_{0,2}}}{\Gamma\left(\frac{k_{\ell,2}-k_{0,2}}{\kappa_1}\right)}\star_{\kappa_1}\mathcal{B}_{\kappa_1}J_1(\tau,\epsilon)\star_{\kappa_1}\omega_1\\
+\sum_{\ell=0}^{s_3}\frac{3a_{0,1}^2a_{\ell,3}}{a_{0,2}^2}\epsilon^{-\chi_1(k_{\ell,3}+k_{0,1}-2k_{0,2})}\frac{\tau^{k_{\ell,3}+k_{0,1}-2k_{0,2}}}{\Gamma\left(\frac{k_{\ell,3}+k_{0,1}-2k_{0,2}}{\kappa_1}\right)}\star_{\kappa_1}\omega_1\hfill\\
\hfill+\sum_{\ell=0}^{s_3}\frac{3a_{0,1}^2a_{\ell,3}}{a_{0,2}^2}\epsilon^{-\chi_1(k_{\ell,3}+k_{0,1}-2k_{0,2})}\frac{\tau^{k_{\ell,3}+k_{0,1}-2k_{0,2}}}{\Gamma\left(\frac{k_{\ell,3}+k_{0,1}-2k_{0,2}}{\kappa_1}\right)}\star_{\kappa_1}2\mathcal{B}_{\kappa_1}J_1(\tau,\epsilon)\star_{\kappa_1}\omega_1\\
+\sum_{\ell=0}^{s_3}\frac{3a_{0,1}^2a_{\ell,3}}{a_{0,2}^2}\epsilon^{-\chi_1(k_{\ell,3}+k_{0,1}-2k_{0,2})}\frac{\tau^{k_{\ell,3}+k_{0,1}-2k_{0,2}}}{\Gamma\left(\frac{k_{\ell,3}+k_{0,1}-2k_{0,2}}{\kappa_1}\right)}\star_{\kappa_1}\mathcal{B}_{\kappa_1}J_1(\tau,\epsilon)\star_{\kappa_1}\mathcal{B}_{\kappa_1}J_1(\tau,\epsilon)\star_{\kappa_1}\omega_1\\
+\sum_{\ell=s_3+1}^{M_3}\frac{3a_{0,1}^2a_{\ell,3}}{a_{0,2}^2}\epsilon^{m_{\ell,3}+3\beta-\alpha k_{\ell,3}-\chi_1(k_{\ell,3}+k_{0,1}-2k_{0,2})}\frac{\tau^{k_{\ell,3}+k_{0,1}-2k_{0,2}}}{\Gamma\left(\frac{k_{\ell,3}+k_{0,1}-2k_{0,2}}{\kappa_1}\right)}\star_{\kappa_1}\omega_1\hfill\\
+\sum_{\ell=s_3+1}^{M_3}\frac{3a_{0,1}^2a_{\ell,3}}{a_{0,2}^2}\epsilon^{m_{\ell,3}+3\beta-\alpha k_{\ell,3}-\chi_1(k_{\ell,3}+k_{0,1}-2k_{0,2})}\frac{\tau^{k_{\ell,3}+k_{0,1}-2k_{0,2}}}{\Gamma\left(\frac{k_{\ell,3}+k_{0,1}-2k_{0,2}}{\kappa_1}\right)}\star_{\kappa_1}2\mathcal{B}_{\kappa_1}J_1(\tau,\epsilon)\star_{\kappa_1}\omega_1\\
\left.+\sum_{\ell=s_3+1}^{M_3}\frac{3a_{0,1}^2a_{\ell,3}}{a_{0,2}^2}\epsilon^{m_{\ell,3}+3\beta-\alpha k_{\ell,3}-\chi_1(k_{\ell,3}+k_{0,1}-2k_{0,2})}\frac{\tau^{k_{\ell,3}+k_{0,1}-2k_{0,2}}}{\Gamma\left(\frac{k_{\ell,3}+k_{0,1}-2k_{0,2}}{\kappa_1}\right)}\star_{\kappa_1}\mathcal{B}_{\kappa_1}J_1(\tau,\epsilon)\star_{\kappa_1}\mathcal{B}_{\kappa_1}J_1(\tau,\epsilon)\star_{\kappa_1}\omega_1\right],
\end{multline}

\begin{multline}
L_{2,\kappa_1}(\omega_1)=\tilde{Q}(im)\left[\sum_{\ell=0}^{s_2}a_{\ell,2}\epsilon^{-\chi_1(k_{\ell,2}+\gamma_1-k_{0,1})}\frac{\tau^{k_{\ell,2}+\gamma_1-k_{0,1}}}{\Gamma\left(\frac{k_{\ell,2}+\gamma_1-k_{0,1}}{\kappa_1}\right)}\star_{\kappa_1} \omega_1\star_{\kappa_1}^{E}\omega_1\right.\\
+\sum_{\ell=s_2+1}^{M_2}a_{\ell,2}\epsilon^{m_{\ell,2}+2\beta-\alpha k_{\ell,2}-\chi_1(k_{\ell,2}+\gamma_1-k_{0,1})}\frac{\tau^{k_{\ell,2}+\gamma_1-k_{0,1}}}{\Gamma\left(\frac{k_{\ell,2}+\gamma_1-k_{0,1}}{\kappa_1}\right)}\star_{\kappa_1} \omega_1\star_{\kappa_1}^{E}\omega_1\\
+\sum_{\ell=0}^{s_3}\frac{-3a_{\ell,3}a_{01}}{a_{02}}\epsilon^{-\chi_1(k_{\ell,3}+\gamma_1-k_{0,2})}\frac{\tau^{k_{\ell,3}+\gamma_1-k_{0,2}}}{\Gamma\left(\frac{k_{\ell,3}+\gamma_1-k_{0,2}}{\kappa_1}\right)}\star_{\kappa_1} \omega_1\star_{\kappa_1}^{E}\omega_1\hfill\\
\hfill+\sum_{\ell=s_3+1}^{M_3}\frac{-3a_{\ell,3}a_{01}}{a_{02}}\epsilon^{m_{\ell,3}+3\beta-\alpha k_{\ell,3}-\chi_1(k_{\ell,3}+\gamma_1-k_{0,2})}\frac{\tau^{k_{\ell,3}+\gamma_1-k_{0,2}}}{\Gamma\left(\frac{k_{\ell,3}+\gamma_1-k_{0,2}}{\kappa_1}\right)}\star_{\kappa_1} \omega_1\star_{\kappa_1}^{E}\omega_1\\
+\sum_{\ell=0}^{s_3}\frac{-3a_{\ell,3}a_{01}}{a_{02}}\epsilon^{-\chi_1(k_{\ell,3}+\gamma_1-k_{0,2})}\frac{\tau^{k_{\ell,3}+\gamma_1-k_{0,2}}}{\Gamma\left(\frac{k_{\ell,3}+\gamma_1-k_{0,2}}{\kappa_1}\right)}\star_{\kappa_1}\mathcal{B}_{\kappa_1}J_1(\tau,\epsilon)\star_{\kappa_1} \omega_1\star_{\kappa_1}^{E}\omega_1\hfill\\
\left.+\sum_{\ell=s_3+1}^{M_3}\frac{-3a_{\ell,3}a_{01}}{a_{02}}\epsilon^{m_{\ell,3}+3\beta-\alpha k_{\ell,3}-\chi_1(k_{\ell,3}+\gamma_1-k_{0,2})}\frac{\tau^{k_{\ell,3}+\gamma_1-k_{0,2}}}{\Gamma\left(\frac{k_{\ell,3}+\gamma_1-k_{0,2}}{\kappa_1}\right)}\star_{\kappa_1}\mathcal{B}_{\kappa_1}J_1(\tau,\epsilon)\star_{\kappa_1} \omega_1\star_{\kappa_1}^{E}\omega_1\right],
\end{multline}

\begin{multline}
L_{3,\kappa_1}(\omega_1)=\tilde{Q}(im)\left[\sum_{\ell=0}^{s_3}a_{\ell,3}\epsilon^{-\chi_1(k_{\ell,3}+2\gamma_1-k_{0,1})}\frac{\tau^{k_{\ell,3}+2\gamma_1-k_{0,1}}}{\Gamma\left(\frac{k_{\ell,3}+2\gamma_1-k_{0,1}}{\kappa_1}\right)}\star_{\kappa_1}\omega_1\star_{\kappa_1}^{E} \omega_1\star_{\kappa_1}^{E}\omega_1\right.\\
\left.+\sum_{\ell=s_3+1}^{M_3}a_{\ell,3}\epsilon^{m_{\ell,3}+3\beta-\alpha k_{\ell,3}-\chi_1(k_{\ell,3}+2\gamma_1-k_{0,1})}\frac{\tau^{k_{\ell,3}+2\gamma_1-k_{0,1}}}{\Gamma\left(\frac{k_{\ell,3}+2\gamma_1-k_{0,1}}{\kappa_1}\right)}\star_{\kappa_1}\omega_1\star_{\kappa_1}^{E} \omega_1\star_{\kappa_1}^{E}\omega_1\right].
\end{multline}

For the righthand side of the equation, we make use of (\ref{e301}) and the properties of $m_{\kappa_1}-$Borel transformation. We have

\begin{multline}
R_{1,\kappa_1}(\omega_1)=\sum_{j=0}^{Q}\tilde{B}_{j}(im)\epsilon^{n_j-\alpha b_j-\chi_1(b_j-k_{0,1}-\gamma_1)}\frac{\tau^{b_j-k_{0,1}-\gamma_1}}{\Gamma\left(\frac{b_j-k_{0,1}-\gamma_1}{\kappa_1}\right)}\\
+\sum_{\ell=1}^{D-1}\epsilon^{\Delta_{\ell}+\alpha(\delta_{\ell}-d_{\ell})+\beta}\sum_{q_1+q_2=\delta_\ell}\frac{\delta_\ell!}{q_1!q_2!}\prod_{d=0}^{q_1-1}(\gamma_1-d)\epsilon^{-\chi_1(d_{\ell}-q_1-k_{0,1}-q_2)}\tilde{R}_{\ell}(im)\\
\times\left(\frac{\tau^{d_{\ell,q_1,q_2}}}{\Gamma\left(\frac{d_{\ell,q_1,q_2}}{\kappa_1}\right)}\star_{\kappa_1}\left((\kappa_1\tau^{\kappa_1})^{q_2}\omega_1\right)+\sum_{1\le p\le q_2-1}A_{q_2,p}\frac{\tau^{d_{\ell,q_1,q_2}+\kappa_1(q_2-p)}}{\Gamma\left(\frac{d_{\ell,q_1,q_2}+\kappa_1(q_2-p)}{\kappa_1}\right)}\star_{\kappa_1}\left((\kappa_1\tau^{\kappa_1})^{p}\omega_1\right)\right)\\
+\sum_{q_1+q_2=\delta_D,q_1\ge1}\frac{\delta_D!}{q_1!q_2!}\prod_{d=0}^{q_1-1}(\gamma_1-d)\tilde{R}_{D}(im)\\
\times\left(\frac{\tau^{d_{D,q_1,q_2}}}{\Gamma\left(\frac{d_{D,q_1,q_2}}{\kappa_1}\right)}\star_{\kappa_1}\left((\kappa_1\tau^{\kappa_1})^{q_2}\omega_1\right)+\sum_{1\le p\le q_2-1}A_{q_2,p}\frac{\tau^{d_{D,q_1,q_2}+\kappa_1(q_2-p)}}{\Gamma\left(\frac{d_{D,q_1,q_2}+\kappa_1(q_2-p)}{\kappa_1}\right)}\star_{\kappa_1}\left((\kappa_1\tau^{\kappa_1})^{p}\omega_1\right)\right)\\
+\tilde{R}_{D}(im)\left((\kappa_1\tau^{\kappa_1})^{\delta_{D}}\omega_1+\sum_{1\le p\le \delta_D-1}A_{\delta_D,p}\frac{\tau^{\kappa_1(\delta_D-p)}}{\Gamma\left(\frac{\kappa_1(\delta_D-p)}{\kappa_1}\right)}\star_{\kappa_1}\left((\kappa_1\tau^{\kappa_1})^{p}\omega_1\right)\right).
\end{multline}

\subsection{Analytic solution of the first perturbed auxiliary problem}\label{subseccion1}

The main purpose of this section is to state the existence of a unique solution of (\ref{e315}) within an appropriate Banach space of functions. The geometry of the problem is analogous to that stated in~\cite{lama1} which demands some restrictions on the domains and the functions involved in the problem. More precisely, we assume there exists an unbounded sector
$$S_{\tilde{Q},\tilde{R}_D}=\{z\in\C^{\star}:|z|\ge r_{\tilde{Q},\tilde{R}_D},|\arg(z)-d_{\tilde{Q},\tilde{R}_D}|\le \nu_{\tilde{Q},\tilde{R}_D}\},$$
for some bisecting direction $d_{\tilde{Q},\tilde{R}_D}\in\R$, opening $\nu_{\tilde{Q},\tilde{R}_D}>0$, and $r_{\tilde{Q},\tilde{R}_D}>0$ such that 
\begin{equation}\label{e366}
\frac{\tilde{Q}(im)}{\tilde{R}_D(im)}\in S_{\tilde{Q},\tilde{R}_D},\qquad m\in\R.
\end{equation}

For every $m\in\R$, the roots of the polynomial $\tilde{P}_m=-\tilde{Q}(im)a_{0,1}-\tilde{R}_D(im)\kappa_1^{\delta_{D}}\tau^{\delta_{D}\kappa_1}$ are given by
\begin{equation}\label{e371}
q_{\ell}(m)=\left(\frac{|a_{0,1}\tilde{Q}(im)|}{|\tilde{R}_{D}(im)|\kappa_1^{\delta_{D}}}\right)^{\frac{1}{\delta_{D}\kappa_1}}\exp\left(\sqrt{-1}(\arg(\frac{-a_{0,1}\tilde{Q}(im)}{\tilde{R}_{D}(im)\kappa_1^{\delta_{D}}}))\frac{1}{\delta_{D}\kappa_1}+\frac{2\pi\ell}{\delta_{D}\kappa_1}\right),
\end{equation}
for $0\le \ell\le \delta_{D}\kappa_1-1$. Let $S_d$ be an unbounded sector of bisecting direction $d\in\R$ and vertex at the origin, and $\rho>0$ such that the three next conditions are satisfied:

1) There exists $M_1>0$ such that
\begin{equation}\label{e377}
|\tau-q_{\ell}(m)|\ge M_1(1+|\tau|),
\end{equation}
for every $0\le \ell\le \delta_{D}\kappa_1$, $m\in\R$ and $\tau\in S_{d}\cup \bar{D}(0,\rho)$. This is possible due to (\ref{e366}), for some adecquate choice of $r_{\tilde{Q},\tilde{R}_D}$ and $\rho>$. By choosing small enough $\nu_{\tilde{Q},\tilde{R}_D}>0$ the set 
$$\{\frac{q_{\ell}(m)}{\tau}:\tau\in S_d,m\in\R,0\le \ell\le \delta_{D}\kappa_1-1\}$$
is such that it has positive distance to 1.

2) There exists $M_2>0$ such that
\begin{equation}\label{e378}
|\tau-q_{\ell_0}(m)|\ge M_2|q_{\ell_0}|,
\end{equation}
for some $0\le \ell_0\le \delta_{D}\kappa_1-1$, all $m\in\R$ and all $\tau\in S_{d}\cup\bar{D}(0,\rho)$. This fact is immediate in view of 1).

By construction of the roots (\ref{e371}), and by (\ref{e377}) and (\ref{e378}), we get a constant $C_{\tilde{P}}>0$ such that
\begin{multline}
|\tilde{P}_{m}(\tau)| \geq M_{1}^{\delta_{D}\kappa_1-1}M_{2}|\tilde{R}_{D}(im)\kappa_1^{\delta_D}
|(\frac{|a_{0,1}\tilde{Q}(im)|}{|\tilde{R}_{D}(im)|\kappa_1^{\delta_{D}}})^{\frac{1}{\delta_{D}\kappa_1}}
(1+|\tau|)^{\delta_{D}\kappa_1-1}\\
\geq M_{1}^{\delta_{D}\kappa_1-1}M_{2}
\frac{\kappa_1^{\delta_D}|a_{0,1}|^{\frac{1}{\delta_{D}\kappa_1}}}{(\kappa_1^{\delta_{D}})^{\frac{1}{\delta_{D}\kappa_1}}}
(r_{\tilde{Q},\tilde{R}_{D}})^{\frac{1}{\delta_{D}\kappa_1}} |\tilde{R}_{D}(im)| \\
\times (\min_{x \geq 0}
\frac{(1+x)^{\delta_{D}\kappa_1-1}}{(1+x^{\kappa_1})^{\delta_{D} - \frac{1}{\kappa_1}}})
(1 + |\tau|^{\kappa_1})^{\delta_{D} - \frac{1}{\kappa_1}}\\
= C_{\tilde{P}} (r_{\tilde{Q},\tilde{R}_{D}})^{\frac{1}{\delta_{D}\kappa_1}} |\tilde{R}_{D}(im)|
(1+|\tau|^{\kappa_1})^{\delta_{D} - \frac{1}{\kappa_1}} \label{e949}
\end{multline}
for all $\tau \in S_{d} \cup \bar{D}(0,\rho)$, all $m \in \mathbb{R}$.

In the next proposition, we provide sufficient conditions under which the main convolution equation (\ref{e315}) admits solutions $\omega_{\kappa_1}^{d}(\tau,m,\epsilon)$ in the Banach space $F_{(\nu,\beta,\mu,\chi_1,\alpha,\kappa_1,\epsilon)}^{d}$ described in Section~\ref{seccion2}.

\begin{lemma}\label{lema533}
One has
$$\mathcal{B}_{\kappa_1}J_1(\tau,\epsilon)\star_{\kappa_1}\mathcal{B}_{\kappa_1}J_1(\tau,\epsilon)=\mathcal{B}_{\kappa_1}\tilde{J}_1(\tau,\epsilon),$$
where $\mathcal{B}_{\kappa_1}\tilde{J}_1(\tau,\epsilon)$ stands for the $m_{\kappa_1}-$Borel transform of the formal power series $\hat{J}(\mathbb{T})=J_1(\mathbb{T})\cdot J_1(\mathbb{T})$, evaluated at $\epsilon^{-\chi_1}\mathbb{T}$, i.e.
\begin{equation}\label{e537}
\mathcal{B}_{\kappa_1}J_1(\tau,\epsilon)\star_{\kappa_1}\mathcal{B}_{\kappa_1}J_1(\tau,\epsilon)=\sum_{j\ge1}\epsilon^{-\chi_1j}\left(\sum_{j_1+j_2=j}J_{j_1}J_{j_2}\right)\frac{\tau^j}{\Gamma\left(\frac{j}{\kappa_1}\right)},
\end{equation}
with $(J_j)_{j\ge0}$ stands for the sequence of coefficients of the series $J_1$. We denote
$$\tilde{J}_j:=\sum_{j_1+j_2=j}J_{j_1}J_{j_2},\quad j\ge1.$$

\end{lemma}
\begin{proof}
From the definition of $\star_{\kappa_1}$, and usual properties of Gamma function, we get
$$\mathcal{B}_{\kappa_1}J_1(\tau,\epsilon)\star_{\kappa_1}\mathcal{B}_{\kappa_1}J_1(\tau,\epsilon)$$
\begin{align*}
&\tau^{\kappa_1}\int_0^{\tau^{\kappa_1}}\left(\sum_{j\ge1}J_j\epsilon^{-\chi_1 j}\frac{(\tau^{\kappa_1}-s)^{j}}{\Gamma\left(\frac{j}{\kappa_1}\right)}\right)\left(\sum_{j\ge1}J_j\epsilon^{-\chi_1 j}\frac{(\tau^{\kappa_1}-s)^{j}}{\Gamma\left(\frac{j}{\kappa_1}\right)}\right)\frac{1}{(\tau^{\kappa_1}-s)s}ds\\
&=\tau^{\kappa_1}\int_0^{\tau^{\kappa_1}}\sum_{j\ge1}\epsilon^{-\chi_1 j}\sum_{j_1+j_2=j}\frac{J_{j_1}J_{j_2}}{\Gamma\left(\frac{j_1}{\kappa_1}\right)\Gamma\left(\frac{j_2}{\kappa_1}\right)}\int_0^{\tau^{\kappa_1}}(\tau^{\kappa_1}-s)^{\frac{j_1}{\kappa_1}-1}s^{\frac{j_2}{\kappa_1}-1}ds\\
&=\tau^{\kappa_1}\int_0^{\tau^{\kappa_1}}\sum_{j\ge1}\epsilon^{-\chi_1 j}\sum_{j_1+j_2=j}\frac{J_{j_1}J_{j_2}}{\Gamma\left(\frac{j_1}{\kappa_1}\right)\Gamma\left(\frac{j_2}{\kappa_1}\right)}\int_0^{1}(\tau^{\kappa_1}-\tau^{\kappa_1}t)^{\frac{j_1}{\kappa_1}-1}(\tau^{\kappa_1}t)^{\frac{j_2}{\kappa_1}-1}\tau^{\kappa_1}dt\\
&=\tau^{\kappa_1}\sum_{j\ge1}\epsilon^{-\chi_1 j}\sum_{j_1+j_2=j}\frac{J_{j_1}J_{j_2}}{\Gamma\left(\frac{j_1}{\kappa_1}\right)\Gamma\left(\frac{j_2}{\kappa_1}\right)}\tau^{j-\kappa_1}\frac{\Gamma\left(\frac{j_1}{\kappa_1}\right)\Gamma\left(\frac{j_2}{\kappa_1}\right)}{\Gamma\left(\frac{j}{\kappa_1}\right)},
\end{align*}
which coincides with (\ref{e537}).
\end{proof}

\begin{lemma}\label{lema887}
Under the assumption (\ref{e218}), one has
\begin{multline}
(\chi_1 + \alpha)( k_{\ell,3}+\gamma_1-k_{0,2}-\kappa_1\delta_D+1) - \chi_1(k_{\ell,3}+\gamma_1-k_{0,2})\\
\le (\chi_1 + \alpha)( k_{\ell,3}+2\gamma_1-k_{0,1}-\kappa_1\delta_D+1) - \chi_1(k_{\ell,3}+2\gamma_1-k_{0,1})
\end{multline}
 for every $0\le \ell\le M_3$.
\end{lemma}

The proof of the next result is left to Section~\ref{seclem4}.

\begin{lemma}\label{lema568}
Let the following conditions hold:

\begin{multline}
\delta_{D} \geq \frac{2}{\kappa_1} \ \ , \ \ \gamma_1\ge k_{0,1}-k_{0,2} \ \ , \ \ b_{j} - k_{0,1} - \gamma_1 \geq 1,\\
(\chi_1 + \alpha)( k_{\ell_2,2}+\gamma_1-k_{0,1}-\kappa_1\delta_D+1) - \chi_1(k_{\ell_2,2}+\gamma_1-k_{0,1}) \geq 0\\
(\chi_1 + \alpha)( k_{\ell_3,3}+\gamma_1-k_{0,2}-\kappa_1\delta_D+1) - \chi_1(k_{\ell_3,3}+\gamma_1-k_{0,2}) \geq 0
\label{e887}
\end{multline}
for all $0 \leq \ell_2 \leq M_2$, $0 \leq \ell_3 \leq M_3$, $0 \leq j \leq Q$,
\begin{multline}
\delta_{D} \geq \frac{1}{\kappa_1} + \delta_{\ell},\\
\Delta_{\ell} + \alpha(\delta_{\ell} - d_{\ell}) + \beta + 
(\chi_1 + \alpha)\kappa_1(\frac{d_{\ell,q_{1},q_{2}}}{\kappa_1} + q_{2} - \delta_{D} + \frac{1}{\kappa_1})
-\chi_1(d_{\ell} - k_{0,1} - \delta_{\ell}) \geq 0
\label{e896}
\end{multline}
for all $q_{1} \geq 0, q_{2} \geq 1$ such that $q_{1}+q_{2}=\delta_{\ell}$, for $1 \leq \ell \leq D-1$ and
\begin{equation}
\Delta_{D} + \alpha( \delta_{D} - d_{D})+ \beta + 
(\chi_1 + \alpha)\kappa_1(\frac{d_{D,q_{1},q_{2}}}{\kappa_1} + q_{2} - \delta_{D} + \frac{1}{\kappa_1})
-\chi_1(d_{D} - k_{0,1} - \delta_{D}) \geq 0
\label{e903}
\end{equation}
for all $q_{1} \geq 1, q_{2} \geq 1$ such that $q_{1}+q_{2}=\delta_{D}$.

Then, there exist large enough $r_{\tilde{Q},\tilde{R}_{D}}>0$ and small enough $\epsilon_0>0$, $\varpi>0$ such that for every $\epsilon\in D(0,\epsilon_0)\setminus\{0\}$, the map $\mathcal{H}_\epsilon$
satisfies that $\mathcal{H}_{\epsilon}(\bar{B}(0,\varpi))\subseteq\bar{B}(0,\varpi)$, where $\bar{B}(0,\varpi)$ is the closed disc of radius $\varpi>0$ centered at 0, in $F^{d}_{(\nu,\beta,\mu,\chi_1,\alpha,\kappa_1,\epsilon)}$, for every $\epsilon\in D(0,\epsilon_0)\setminus\{0\}$. Moreover, it holds that
\begin{equation}\label{e442}
\left\|\mathcal{H}_{\epsilon}(\omega_1)-\mathcal{H}_{\epsilon}(\omega_2)\right\|_{(\nu,\beta,\mu,\chi_1,\alpha,\kappa_1,\epsilon)}\le\frac{1}{2} \left\|\omega_1-\omega_2\right\|_{(\nu,\beta,\mu,\chi_1,\alpha,\kappa_1,\epsilon)},
\end{equation}
for every $\omega_1,\omega_2\in \bar{B}(0,\varpi)$, and every $\epsilon\in D(0,\epsilon_0)\setminus\{0\}$.

Here, $$\mathcal{H}_{\epsilon}=\mathcal{H}_{\epsilon}^1+\mathcal{H}_{\epsilon}^2+\mathcal{H}_{\epsilon}^3+\mathcal{H}_{\epsilon}^4,$$
where
\begin{multline}
\mathcal{H}_{\epsilon}^1(\omega_1(\tau,m)):=\sum_{j=0}^{Q}\tilde{B}_{j}(im)\epsilon^{n_j-\alpha b_j-\chi_1(b_j-k_{0,1}-\gamma_1)}\frac{\tau^{b_j-k_{0,1}-\gamma_1}}{\tilde{P}_m(\tau)\Gamma\left(\frac{b_j-k_{0,1}-\gamma_1}{\kappa_1}\right)}\\
-\tilde{Q}(im)\left\{\sum_{\ell=1}^{s_1}\frac{a_{\ell,1}\epsilon^{-\chi_1(k_{\ell,1}-k_{0,1})}}{\tilde{P}_{m}(\tau)}\left[\frac{\tau^{k_{\ell,1}-k_{0,1}}}{\Gamma\left(\frac{k_{\ell,1}-k_{0,1}}{\kappa_1}\right)}\star_{\kappa_1}\omega_1\right]\right.\\
\hfill+\sum_{\ell=s_1+1}^{M_1}\frac{a_{\ell,1}\epsilon^{m_{\ell,1}+\beta-\alpha k_{\ell,1}-\chi_1(k_{\ell,1}-k_{0,1})}}{\tilde{P}_{m}(\tau)}\left[\frac{\tau^{k_{\ell,1}-k_{0,1}}}{\Gamma\left(\frac{k_{\ell,1}-k_{0,1}}{\kappa_1}\right)}\star_{\kappa_1} \omega_1\right]\\
+\sum_{\ell=1}^{s_2}\frac{-2a_{\ell,2}a_{0,1}}{a_{0,2}}\frac{\epsilon^{-\chi_1(k_{\ell,2}-k_{0,2})}}{\tilde{P}_{m}(\tau)}\left[\frac{\tau^{k_{\ell,2}-k_{0,2}}}{\Gamma\left(\frac{k_{\ell,2}-k_{0,2}}{\kappa_1}\right)}\star_{\kappa_1}\omega_1\right]\hfill\\
\hfill+\sum_{\ell=s_2+1}^{M_2}\frac{-2a_{\ell,2}a_{0,1}}{a_{0,2}}\frac{\epsilon^{m_{\ell,2}+2\beta-\alpha k_{\ell,2}-\chi_1(k_{\ell,2}-k_{0,2})}}{\tilde{P}_{m}(\tau)}\left[\frac{\tau^{k_{\ell,2}-k_{0,2}}}{\Gamma\left(\frac{k_{\ell,2}-k_{0,2}}{\kappa_1}\right)}\star_{\kappa_1}\omega_1\right]\\
+\sum_{\ell=0}^{s_2}\frac{-2a_{\ell,2}a_{0,1}}{a_{0,2}}\frac{\epsilon^{-\chi_1(k_{\ell,2}-k_{0,2})}}{\tilde{P}_{m}(\tau)}\left[\frac{\tau^{k_{\ell,2}-k_{0,2}}}{\Gamma\left(\frac{k_{\ell,2}-k_{0,2}}{\kappa_1}\right)}\star_{\kappa_1}\mathcal{B}_{\kappa_1}J_1(\tau,\epsilon)\star_{\kappa_1}\omega_1\right]\hfill\\
\hfill+\sum_{\ell=s_2+1}^{M_2}\frac{-2a_{\ell,2}a_{0,1}}{a_{0,2}}\frac{\epsilon^{m_{\ell,2}+2\beta-\alpha k_{\ell,2}-\chi_1(k_{\ell,2}-k_{0,2})}}{\tilde{P}_{m}(\tau)}\left[\frac{\tau^{k_{\ell,2}-k_{0,2}}}{\Gamma\left(\frac{k_{\ell,2}-k_{0,2}}{\kappa_1}\right)}\star_{\kappa_1}\mathcal{B}_{\kappa_1}J_1(\tau,\epsilon)\star_{\kappa_1}\omega_1\right]\\
+\sum_{\ell=0}^{s_3}\frac{3a_{0,1}^2a_{\ell,3}}{a_{0,2}^2}\frac{\epsilon^{-\chi_1(k_{\ell,3}+k_{0,1}-2k_{0,2})}}{\tilde{P}_{m}(\tau)}\left[\frac{\tau^{k_{\ell,3}+k_{0,1}-2k_{0,2}}}{\Gamma\left(\frac{k_{\ell,3}+k_{0,1}-2k_{0,2}}{\kappa_1}\right)}\star_{\kappa_1}\omega_1\right]\hfill\\
\hfill+\sum_{\ell=0}^{s_3}\frac{3a_{0,1}^2a_{\ell,3}}{a_{0,2}^2}\frac{\epsilon^{-\chi_1(k_{\ell,3}+k_{0,1}-2k_{0,2})}}{\tilde{P}_{m}(\tau)}\left[\frac{\tau^{k_{\ell,3}+k_{0,1}-2k_{0,2}}}{\Gamma\left(\frac{k_{\ell,3}+k_{0,1}-2k_{0,2}}{\kappa_1}\right)}\star_{\kappa_1}2\mathcal{B}_{\kappa_1}J_1(\tau,\epsilon)\star_{\kappa_1}\omega_1\right]\\
+\sum_{\ell=0}^{s_3}\frac{3a_{0,1}^2a_{\ell,3}}{a_{0,2}^2}\frac{\epsilon^{-\chi_1(k_{\ell,3}+k_{0,1}-2k_{0,2})}}{\tilde{P}_{m}(\tau)}\left[\frac{\tau^{k_{\ell,3}+k_{0,1}-2k_{0,2}}}{\Gamma\left(\frac{k_{\ell,3}+k_{0,1}-2k_{0,2}}{\kappa_1}\right)}\star_{\kappa_1}\mathcal{B}_{\kappa_1}J_1(\tau,\epsilon)\star_{\kappa_1}\mathcal{B}_{\kappa_1}J_1(\tau,\epsilon)\star_{\kappa_1}\omega_1\right]\\
+\sum_{\ell=s_3+1}^{M_3}\frac{3a_{0,1}^2a_{\ell,3}}{a_{0,2}^2}\frac{\epsilon^{m_{\ell,3}+3\beta-\alpha k_{\ell,3}-\chi_1(k_{\ell,3}+k_{0,1}-2k_{0,2})}}{\tilde{P}_{m}(\tau)}\left[\frac{\tau^{k_{\ell,3}+k_{0,1}-2k_{0,2}}}{\Gamma\left(\frac{k_{\ell,3}+k_{0,1}-2k_{0,2}}{\kappa_1}\right)}\star_{\kappa_1}\omega_1\right]\hfill\\
+\sum_{\ell=s_3+1}^{M_3}\frac{3a_{0,1}^2a_{\ell,3}}{a_{0,2}^2}\frac{\epsilon^{m_{\ell,3}+3\beta-\alpha k_{\ell,3}-\chi_1(k_{\ell,3}+k_{0,1}-2k_{0,2})}}{\tilde{P}_{m}(\tau)}\left[\frac{\tau^{k_{\ell,3}+k_{0,1}-2k_{0,2}}}{\Gamma\left(\frac{k_{\ell,3}+k_{0,1}-2k_{0,2}}{\kappa_1}\right)}\star_{\kappa_1}2\mathcal{B}_{\kappa_1}J_1(\tau,\epsilon)\star_{\kappa_1}\omega_1\right]\\
+\sum_{\ell=s_3+1}^{M_3}\frac{3a_{0,1}^2a_{\ell,3}}{a_{0,2}^2}\frac{\epsilon^{m_{\ell,3}+3\beta-\alpha k_{\ell,3}-\chi_1(k_{\ell,3}+k_{0,1}-2k_{0,2})}}{\tilde{P}_{m}(\tau)}\hfill\\
\hfill\left.\times\left[\frac{\tau^{k_{\ell,3}+k_{0,1}-2k_{0,2}}}{\Gamma\left(\frac{k_{\ell,3}+k_{0,1}-2k_{0,2}}{\kappa_1}\right)}\star_{\kappa_1}\mathcal{B}_{\kappa_1}J_1(\tau,\epsilon)\star_{\kappa_1}\mathcal{B}_{\kappa_1}J_1(\tau,\epsilon)\star_{\kappa_1}\omega_1\right]\right\},
\end{multline}

\begin{multline}
\mathcal{H}_{\epsilon}^2(\omega_1(\tau,m)):=\sum_{\ell=1}^{D-1}\epsilon^{\Delta_{\ell}+\alpha(\delta_{\ell}-d_{\ell})+\beta}\sum_{q_1+q_2=\delta_\ell}\frac{\delta_\ell!}{q_1!q_2!}\prod_{d=0}^{q_1-1}(\gamma_1-d)\epsilon^{-\chi_1(d_{\ell}-q_1-k_{0,1}-q_2)}\tilde{R}_{\ell}(im)\frac{1}{\tilde{P}_{m}(\tau)}\\
\times\left[\frac{\tau^{d_{\ell,q_1,q_2}}}{\Gamma\left(\frac{d_{\ell,q_1,q_2}}{\kappa_1}\right)}\star_{\kappa_1}\left((\kappa_1\tau^{\kappa_1})^{q_2}\omega_1\right)+\sum_{1\le p\le q_2-1}\frac{A_{q_2,p}}{\tilde{P}_{m}(\tau)}\left(\frac{\tau^{d_{\ell,q_1,q_2}+\kappa_1(q_2-p)}}{\Gamma\left(\frac{d_{\ell,q_1,q_2}+\kappa_1(q_2-p)}{\kappa_1}\right)}\star_{\kappa_1}\left((\kappa_1\tau^{\kappa_1})^{p}\omega_1\right)\right)\right]\\
-\tilde{Q}(im)\left\{\sum_{\ell=0}^{s_2}a_{\ell,2}\frac{\epsilon^{-\chi_1(k_{\ell,2}+\gamma_1-k_{0,1})}}{\tilde{P}_m(\tau)}\left[\frac{\tau^{k_{\ell,2}+\gamma_1-k_{0,1}}}{\Gamma\left(\frac{k_{\ell,2}+\gamma_1-k_{0,1}}{\kappa_1}\right)}\star_{\kappa_1} \omega_1\star_{\kappa_1}^{E}\omega_1\right]\right.\\
+\sum_{\ell=s_2+1}^{M_2}a_{\ell,2}\frac{\epsilon^{m_{\ell,2}+2\beta-\alpha k_{\ell,2}-\chi_1(k_{\ell,2}+\gamma_1-k_{0,1})}}{\tilde{P}_m(\tau)}\left[\frac{\tau^{k_{\ell,2}+\gamma_1-k_{0,1}}}{\Gamma\left(\frac{k_{\ell,2}+\gamma_1-k_{0,1}}{\kappa_1}\right)}\star_{\kappa_1} \omega_1\star_{\kappa_1}^{E}\omega_1\right]\\
+\sum_{\ell=0}^{s_3}\frac{-3a_{\ell,3}a_{01}}{a_{02}}\frac{\epsilon^{-\chi_1(k_{\ell,3}+\gamma_1-k_{0,2})}}{\tilde{P}_m(\tau)}\left[\frac{\tau^{k_{\ell,3}+\gamma_1-k_{0,2}}}{\Gamma\left(\frac{k_{\ell,3}+\gamma_1-k_{0,2}}{\kappa_1}\right)}\star_{\kappa_1} \omega_1\star_{\kappa_1}^{E}\omega_1\right]\hfill\\
\hfill+\sum_{\ell=s_3+1}^{M_3}\frac{-3a_{\ell,3}a_{01}}{a_{02}}\frac{\epsilon^{m_{\ell,3}+3\beta-\alpha k_{\ell,3}-\chi_1(k_{\ell,3}+\gamma_1-k_{0,2})}}{\tilde{P}_m(\tau)}\left[\frac{\tau^{k_{\ell,3}+\gamma_1-k_{0,2}}}{\Gamma\left(\frac{k_{\ell,3}+\gamma_1-k_{0,2}}{\kappa_1}\right)}\star_{\kappa_1} \omega_1\star_{\kappa_1}^{E}\omega_1\right]\\
+\sum_{\ell=0}^{s_3}\frac{-3a_{\ell,3}a_{01}}{a_{02}}\frac{\epsilon^{-\chi_1(k_{\ell,3}+\gamma_1-k_{0,2})}}{\tilde{P}_m(\tau)}\left[\frac{\tau^{k_{\ell,3}+\gamma_1-k_{0,2}}}{\Gamma\left(\frac{k_{\ell,3}+\gamma_1-k_{0,2}}{\kappa_1}\right)}\star_{\kappa_1}\mathcal{B}_{\kappa_1}J_1(\tau,\epsilon)\star_{\kappa_1} \omega_1\star_{\kappa_1}^{E}\omega_1\right]\hfill\\
\left.+\sum_{\ell=s_3+1}^{M_3}\frac{-3a_{\ell,3}a_{01}}{a_{02}}\frac{\epsilon^{m_{\ell,3}+3\beta-\alpha k_{\ell,3}-\chi_1(k_{\ell,3}+\gamma_1-k_{0,2})}}{\tilde{P}_m(\tau)}\left[\frac{\tau^{k_{\ell,3}+\gamma_1-k_{0,2}}}{\Gamma\left(\frac{k_{\ell,3}+\gamma_1-k_{0,2}}{\kappa_1}\right)}\star_{\kappa_1}\mathcal{B}_{\kappa_1}J_1(\tau,\epsilon)\star_{\kappa_1} \omega_1\star_{\kappa_1}^{E}\omega_1\right]\right\},
\end{multline}

\begin{multline}
\mathcal{H}_{\epsilon}^3(\omega_1(\tau,m)):=
\frac{1}{\tilde{P}_m(\tau)}\sum_{q_1+q_2=\delta_D,q_1\ge1}\frac{\delta_D!}{q_1!q_2!}\prod_{d=0}^{q_1-1}(\gamma_1-d)\tilde{R}_{D}(im)\\
\times\left[\frac{\tau^{d_{D,q_1,q_2}}}{\Gamma\left(\frac{d_{D,q_1,q_2}}{\kappa_1}\right)}\star_{\kappa_1}\left((\kappa_1\tau^{\kappa_1})^{q_2}\omega_1\right)+\sum_{1\le p\le q_2-1}A_{q_2,p}\frac{\tau^{d_{D,q_1,q_2}+\kappa_1(q_2-p)}}{\Gamma\left(\frac{d_{D,q_1,q_2}+\kappa_1(q_2-p)}{\kappa_1}\right)}\star_{\kappa_1}\left((\kappa_1\tau^{\kappa_1})^{p}\omega_1\right)\right]\\
+\frac{\tilde{R}_{D}(im)}{\tilde{P}_m(\tau)}\left[\sum_{1\le p\le \delta_D-1}A_{\delta_D,p}\frac{\tau^{\kappa_1(\delta_D-p)}}{\Gamma\left(\frac{\kappa_1(\delta_D-p)}{\kappa_1}\right)}\star_{\kappa_1}\left((\kappa_1\tau^{\kappa_1})^{p}\omega_1\right)\right],
\end{multline}

and

\begin{multline}
\mathcal{H}_{\epsilon}^4(\omega_1(\tau,m)):=-\tilde{Q}(im)\left\{\sum_{\ell=0}^{s_3}a_{\ell,3}\frac{\epsilon^{-\chi_1(k_{\ell,3}+2\gamma_1-k_{0,1})}}{\tilde{P}_m(\tau)}\left[\frac{\tau^{k_{\ell,3}+2\gamma_1-k_{0,1}}}{\Gamma\left(\frac{k_{\ell,3}+2\gamma_1-k_{0,1}}{\kappa_1}\right)}\star_{\kappa_1}\omega_1\star_{\kappa_1}^{E} \omega_1\star_{\kappa_1}^{E}\omega_1\right]\right.\\
\left.+\sum_{\ell=s_3+1}^{M_3}a_{\ell,3}\frac{\epsilon^{m_{\ell,3}+3\beta-\alpha k_{\ell,3}-\chi_1(k_{\ell,3}+2\gamma_1-k_{0,1})}}{\tilde{P}_m(\tau)}\left[\frac{\tau^{k_{\ell,3}+2\gamma_1-k_{0,1}}}{\Gamma\left(\frac{k_{\ell,3}+2\gamma_1-k_{0,1}}{\kappa_1}\right)}\star_{\kappa_1}\omega_1\star_{\kappa_1}^{E} \omega_1\star_{\kappa_1}^{E}\omega_1\right]\right\}.
\end{multline}

\end{lemma}

\begin{prop}\label{prop1518}
Under the assumptions (\ref{e887}), (\ref{e896}), (\ref{e903}), there exist $r_{\tilde{Q},\tilde{R}_D}>0$, $\epsilon_0>0$ and $\varpi>0$ such that the problem (\ref{e315}) admits a unique solution $\omega_{\kappa_1}^{d}(\tau,m,\epsilon)$ belonging to the Banach space $F^d_{(\nu,\beta,\mu,\chi_1,\alpha,\kappa_1,\epsilon)}$, with 
$$  \left\|\omega_{\kappa_1}^{d}(\tau,m,\epsilon)\right\|_{(\nu,\beta,\mu,\chi_1,\alpha,\kappa_1,\epsilon)}\le\varpi,$$
for every $\epsilon\in D(0,\epsilon_0)\setminus\{0\}$, where $d\in\R$ is such that (\ref{e377}) and (\ref{e378}) are satisfied.
\end{prop}
\begin{proof}
Let $r_{\tilde{Q},\tilde{R}_D}>0$, $\epsilon_0>0$ and $\varpi>0$ be as in the proof of Lemma~\ref{lema568}. That result allows us to apply a fixed point argument on $\mathcal{H}_\epsilon$ for every $\epsilon\in D(0,\epsilon_0)\setminus\{0\}$ and obtain a unique element $\omega_{\kappa_1}^{d}(\tau,m,\epsilon)\in F^d_{(\nu,\beta,\mu,\chi_1,\alpha,\kappa_1,\epsilon)}$ with norm upper estimated by $\varpi$, which satisfies that 
$$\mathcal{H}_\epsilon(\omega_{\kappa_1}^{d}(\tau,m,\epsilon))=\omega_{\kappa_1}^{d}(\tau,m,\epsilon).$$  
This function also depends holomorphically on $\epsilon\in D(0,\epsilon_0)\setminus\{0\}$.

Observe that the terms in the equation (\ref{e315}) can be rearranged to write it in the form
$$w_1(\tau,m,\epsilon)=\mathcal{H}_{\epsilon}(w_1(\tau,m,\epsilon)),$$
by leaving $\tilde{Q}(im)a_{0,1}$ and  $\tilde{R}_{D}(im)(\kappa_1\tau^{\kappa_1})^{\delta_{D}}$ on one side and dividing the resulting equation by the polynomial $\tilde{P}_{m}(\tau)=-\tilde{Q}(im)a_{0,1}-\tilde{R}_{D}(im)(\kappa_1\tau^{\kappa_1})^{\delta_{D}}$.

Therefore, $\omega_{\kappa_1}^{d}(\tau,m,\epsilon)$ turns out to be a solution of (\ref{e315}), with initial data $\omega_{\kappa_1}^{d}(0,m,\epsilon)\equiv 0$.
\end{proof}

\subsection{Second perturbed auxiliary problem}

The form of $U_{02}(T)$ in (\ref{e204}) motivates a second particular form of a solution of (\ref{e124}):
\begin{equation}\label{e216b}
U_{2}(T,z,\epsilon)=-\frac{a_{0,2}}{a_{0,3}}T^{k_{0,2}-k_{0,3}}(1+\mathcal{J}_2(T))+T^{\gamma_2}V_2(T,z,\epsilon),
\end{equation}
for some $\gamma_2\in\Q$. We assume this choice is made accordingly to the following conditions:

\begin{equation}\label{e218b}
\gamma_2\ge k_{0,2}-k_{0,3}, 
\end{equation}
and 
\begin{equation}\label{e222b}
\gamma_2\le  b_j-2k_{0,2}+k_{0,3},\qquad j=0,\ldots,Q.
\end{equation}

We proceed as in Subsection~\ref{subseccion1} and plug (\ref{e216b}) into (\ref{e124}). We get

\begin{multline}
Q(\partial_z)\left[\left(\sum_{\ell=0}^{s_1}a_{\ell,1}T^{k_{\ell,1}}+\sum_{\ell=s_1+1}^{M_1}a_{\ell,1}\epsilon^{m_{\ell,1}+\beta-\alpha k_{\ell,1}}T^{k_{\ell,1}}\right)T^{\gamma_2}V_2(T,z,\epsilon)  \right.\\
+\left(\sum_{\ell=0}^{s_2}a_{\ell,2}T^{k_{\ell,2}}+\sum_{\ell=s_2+1}^{M_2}a_{\ell,2}\epsilon^{m_{\ell,2}+2\beta-\alpha k_{\ell,2}}T^{k_{\ell,2}}\right)\\
\times\left(T^{2\gamma_2}V_2^2(T,z,\epsilon)-\frac{2a_{0,2}}{a_{0,3}}T^{k_{0,2}-k_{0,3}+\gamma_2}(1+\mathcal{J}_2(T))V_2(T,z,\epsilon)\right)\\
+\left(\sum_{\ell=0}^{s_3}a_{\ell,3}T^{k_{\ell,3}}+\sum_{\ell=s_3+1}^{M_3}a_{\ell,3}\epsilon^{m_{\ell,3}+3\beta-\alpha k_{\ell,3}}T^{k_{\ell,3}}\right)\left\{3\left(\frac{a_{02}}{a_{03}}T^{k_{0,2}-k_{0,3}}(1+\mathcal{J}_2(T))\right)^2T^{\gamma_2}V_2(T,z,\epsilon)\right.\\
\left.\left.-3\left(\frac{a_{02}}{a_{03}}T^{k_{0,2}-k_{0,3}}(1+\mathcal{J}_2(T))\right)T^{2\gamma_2}V_2^2(T,z,\epsilon)+T^{3\gamma_2}V_2^3(T,z,\epsilon) \right\}\right] \\
=\sum_{j=0}^{Q}b_{j}(z)\epsilon^{n_j-\alpha b_j}T^{b_j}\\
+\sum_{\ell=1}^{D}\epsilon^{\Delta_{\ell}+\alpha(\delta_{\ell}-d_{\ell})+\beta}T^{d_{\ell}}R_{\ell}(\partial_z)\left(\sum_{q_1+q_2=\delta_\ell}\frac{\delta_\ell!}{q_1!q_2!}\prod_{d=0}^{q_1-1}(\gamma_2-d)T^{\gamma_2-q_1}\partial_{T}^{q_2}V_2(T,z,\epsilon)\right).\label{e234b}
\end{multline}

Observe that, in view of (\ref{e92}) and (\ref{e94}) we have
\begin{equation}\label{e1150}
2k_{0,2}-k_{0,3}<k_{0,1}=d_\ell-\delta_\ell-\delta_\ell\kappa_1-d_{\ell,0}\le d_\ell-\delta_\ell.
\end{equation}

Conditions (\ref{e92}), (\ref{e218b}), (\ref{e222b}), (\ref{e1150}), and the fact that $(k_{\ell,2})_{\ell\ge0}$ and $(k_{\ell,3})_{\ell\ge0}$  are increasing sequences, allow us to divide equation (\ref{e234b}) by $T^{2k_{0,2}-k_{0,3}+\gamma_2}$, preserving analyticity of the coefficients involved. Invertibility of the coefficient of $Q(\partial_z)V_2(T,z,\epsilon)$ at $T=0$ is guaranteed due to $a_{0,3}\neq0$. The resulting problem can be rewritten in this form:

\begin{multline}
Q(\partial_z)V_2(T,z,\epsilon)\left[\frac{a_{0,2}^2}{a_{0,3}}+\sum_{\ell=0}^{s_1}a_{\ell,1}T^{k_{\ell,1}-2k_{0,2}+k_{0,3}}+\sum_{\ell=s_1+1}^{M_1}a_{\ell,1}\epsilon^{m_{\ell,1}+\beta-\alpha k_{\ell,1}}T^{k_{\ell,1}-2k_{0,2}+k_{0,3}}\right.\\
+\sum_{\ell=1}^{s_2}\frac{-2a_{0,2}a_{\ell,2}}{a_{0,3}}T^{k_{\ell,2}-k_{0,2}}+\sum_{\ell=s_2+1}^{M_2}\frac{-2a_{0,2}a_{\ell,2}}{a_{0,3}}\epsilon^{m_{\ell,2}+2\beta-\alpha k_{\ell,2}}T^{k_{\ell,2}-k_{0,2}}\hfill\\
+\sum_{\ell=0}^{s_2}\frac{-2a_{0,2}a_{\ell,2}}{a_{0,3}}T^{k_{\ell,2}-k_{0,2}}\mathcal{J}_2(T)+\sum_{\ell=s_2+1}^{M_2}\frac{-2a_{0,2}a_{\ell,2}}{a_{0,3}}\epsilon^{m_{\ell,2}+2\beta-\alpha k_{\ell,2}}T^{k_{\ell,2}-k_{0,2}}\mathcal{J}_2(T)\\
\hfill+\sum_{\ell=1}^{s_3}\frac{3a_{0,2}^2a_{\ell,3}}{a_{0,3}^2}T^{k_{\ell,3}-k_{0,3}}+\sum_{\ell=s_3+1}^{M_3}\frac{3a_{0,2}^2a_{\ell,3}}{a_{0,3}^2}\epsilon^{m_{\ell,3}+3\beta-\alpha k_{\ell,3}}T^{k_{\ell,3}-k_{0,3}}\\
\left.+\left(\sum_{\ell=0}^{s_3}\frac{3a_{0,2}^2a_{\ell,3}}{a_{0,3}^2}T^{k_{\ell,3}-k_{0,3}}+\sum_{\ell=s_3+1}^{M_3}\frac{3a_{0,2}^2a_{\ell,3}}{a_{0,3}^2}\epsilon^{m_{\ell,3}+3\beta-\alpha k_{\ell,3}}T^{k_{\ell,3}-k_{0,3}}\right)\left(2\mathcal{J}_2(T)+\mathcal{J}_2^2(T)\right)\right]\\
+Q(\partial_{z})V_2^2(T,z,\epsilon)\left[\left(\sum_{\ell=0}^{s_2}a_{\ell,2}T^{k_{\ell,2}-k_{0,2}+k_{0,3}}+\sum_{\ell=s_2+1}^{M_2}a_{\ell,2}\epsilon^{m_{\ell,2}+2\beta-\alpha k_{\ell,2}}T^{k_{\ell,2}-k_{0,2}+k_{0,3}}\right)T^{\gamma_2-k_{0,2}}\right.\\
\hfill\left.+\left(\sum_{\ell=0}^{s_3}a_{\ell,3}T^{k_{\ell,3}}+\sum_{\ell=s_3+1}^{M_3}a_{\ell,3}\epsilon^{m_{\ell,3}+3\beta-\alpha k_{\ell,3}}T^{k_{\ell,3}}\right)\left(1+\mathcal{J}_2(T)\right)\right]\\
+Q(\partial_{z})V_2^3(T,z,\epsilon)\left[\left(\sum_{\ell=0}^{s_3}a_{\ell,3}T^{k_{\ell,3}}+\sum_{\ell=s_3+1}^{M_3}a_{\ell,3}\epsilon^{m_{\ell,3}+3\beta-\alpha k_{\ell,3}}T^{k_{\ell,3}}\right)T^{2\gamma_2-2k_{0,2}+k_{0,3}}\right]\\
=\sum_{j=0}^{Q}b_{j}(z)\epsilon^{n_j-\alpha b_j}T^{b_j-2k_{0,2}+k_{0,3}-\gamma_2}\\
+\sum_{\ell=1}^{D}\epsilon^{\Delta_{\ell}+\alpha(\delta_{\ell}-d_{\ell})+\beta}\left(\sum_{q_1+q_2=\delta_\ell}\frac{\delta_\ell!}{q_1!q_2!}\prod_{d=0}^{q_1-1}(\gamma_2-d)T^{d_{\ell}-q_1-2k_{0,2}+k_{0,3}}R_{\ell}(\partial_z)\partial_{T}^{q_2}V_2(T,z,\epsilon)\right).\label{e235b}
\end{multline}

We specify the form of $U_2(T,z,\epsilon)$ in (\ref{e216b}), where
\begin{equation}\label{e256b}
V_2(T,z,\epsilon):=\mathbb{V}_2(\epsilon^{\chi_1}T,z,\epsilon), \hbox{ with } \chi_2:=\frac{\Delta_D+\alpha(\delta_{D}-d_D)+\beta}{d_D-2k_{0,2}+k_{0,3}-\delta_D}.
\end{equation}

Equation (\ref{e235b}) reads as follows:

\begin{multline}
Q(\partial_z)\mathbb{V}_2(\mathbb{T},z,\epsilon)\left[\frac{a_{0,2}^2}{a_{0,3}}+\sum_{\ell=0}^{s_1}a_{\ell,1}\epsilon^{-\chi_2(k_{\ell,1}-2k_{0,2}+k_{0,3})}\mathbb{T}^{k_{\ell,1}-2k_{0,2}+k_{0,3}}\right.\\
\hfill+\sum_{\ell=s_1+1}^{M_1}a_{\ell,1}\epsilon^{m_{\ell,1}+\beta-\alpha k_{\ell,1}-\chi_2(k_{\ell,1}-2k_{0,2}+k_{0,3})}\mathbb{T}^{k_{\ell,1}-2k_{0,2}+k_{0,3}}\\
+\sum_{\ell=1}^{s_2}\frac{-2a_{0,2}a_{\ell,2}}{a_{0,3}}\epsilon^{-\chi_2(k_{\ell,2}-k_{0,2})}\mathbb{T}^{k_{\ell,2}-k_{0,2}}+\sum_{\ell=s_2+1}^{M_2}\frac{-2a_{0,2}a_{\ell,2}}{a_{0,3}}\epsilon^{m_{\ell,2}+2\beta-\alpha k_{\ell,2}-\chi_2(k_{\ell,2}-k_{0,2})}\mathbb{T}^{k_{\ell,2}-k_{0,2}}\\
+\sum_{\ell=0}^{s_2}\frac{-2a_{0,2}a_{\ell,2}}{a_{0,3}}\epsilon^{-\chi_2(k_{\ell,2}-k_{0,2})}\mathbb{T}^{k_{\ell,2}-k_{0,2}}\mathcal{J}_2(\epsilon^{-\chi_2}\mathbb{T})\hfill\\
\hfill+\sum_{\ell=s_2+1}^{M_2}\frac{-2a_{0,2}a_{\ell,2}}{a_{0,3}}\epsilon^{m_{\ell,2}+2\beta-\alpha k_{\ell,2}-\chi_2(k_{\ell,2}-k_{0,2})}T^{k_{\ell,2}-k_{0,2}}\mathcal{J}_2(\epsilon^{-\chi_2}\mathbb{T})\\
+\sum_{\ell=1}^{s_3}\frac{3a_{0,2}^2a_{\ell,3}}{a_{0,3}^2}\epsilon^{-\chi_2(k_{\ell,3}-k_{0,3})}\mathbb{T}^{k_{\ell,3}-k_{0,3}}+\sum_{\ell=s_3+1}^{M_3}\frac{3a_{0,2}^2a_{\ell,3}}{a_{0,3}^2}\epsilon^{m_{\ell,3}+3\beta-\alpha k_{\ell,3}-\chi_2(k_{\ell,3}-k_{0,3})}\mathbb{T}^{k_{\ell,3}-k_{0,3}}\\
+\sum_{\ell=0}^{s_3}\frac{3a_{0,2}^2a_{\ell,3}}{a_{0,3}^2}\epsilon^{-\chi_2(k_{\ell,3}-k_{0,3})}\mathbb{T}^{k_{\ell,3}-k_{0,3}}\left(2\mathcal{J}_2(\epsilon^{-\chi_2}\mathbb{T})+\mathcal{J}_2^2(\epsilon^{-\chi_2}\mathbb{T})\right)\hfill\\
\hfill\left.+\sum_{\ell=s_3+1}^{M_3}\frac{3a_{0,2}^2a_{\ell,3}}{a_{0,3}^2}\epsilon^{m_{\ell,3}+3\beta-\alpha k_{\ell,3}-\chi_2(k_{\ell,3}-k_{0,3})}\mathbb{T}^{k_{\ell,3}-k_{0,3}}\left(2\mathcal{J}_2(\epsilon^{-\chi_2}\mathbb{T})+\mathcal{J}_2^2(\epsilon^{-\chi_2}\mathbb{T})\right)\right]\\
+Q(\partial_{z})\mathbb{V}_2^2(\mathbb{T},z,\epsilon)\left[\sum_{\ell=0}^{s_2}a_{\ell,2}\epsilon^{-\chi_2(k_{\ell,2}-2k_{0,2}+k_{0,3}+\gamma_2)}\mathbb{T}^{k_{\ell,2}-2k_{0,2}+k_{0,3}+\gamma_2}\right. \hfill\\
\hfill+\sum_{\ell=s_2+1}^{M_2}a_{\ell,2}\epsilon^{m_{\ell,2}+2\beta-\alpha k_{\ell,2}-\chi_2(k_{\ell,2}-2k_{0,2}+k_{0,3}+\gamma_2)}\mathbb{T}^{k_{\ell,2}-2k_{0,2}+k_{0,3}+\gamma_2}\\
+\left.\left(\sum_{\ell=0}^{s_3}a_{\ell,3}\epsilon^{-\chi_2k_{\ell,3}}\mathbb{T}^{k_{\ell,3}}+\sum_{\ell=s_3+1}^{M_3}a_{\ell,3}\epsilon^{m_{\ell,3}+3\beta-\alpha k_{\ell,3}-\chi_2k_{\ell,3}}\mathbb{T}^{k_{\ell,3}}\right)\left(1+\mathcal{J}_2(\epsilon^{-\chi_2}\mathbb{T})\right)\right]\\
+Q(\partial_{z})\mathbb{V}_2^3(\mathbb{T},z,\epsilon)\left[\sum_{\ell=0}^{s_3}a_{\ell,3}\epsilon^{-\chi_2(k_{\ell,3}+2\gamma_2-2k_{0,2}+k_{0,3})}\mathbb{T}^{k_{\ell,3}+2\gamma_2-2k_{0,2}+k_{0,3}}\right.\hfill\\
\hfill\left.+\sum_{\ell=s_3+1}^{M_3}a_{\ell,3}\epsilon^{m_{\ell,3}+3\beta-\alpha k_{\ell,3}-\chi_2(k_{\ell,3}+2\gamma_2-2k_{0,2}+k_{0,3})}\mathbb{T}^{k_{\ell,3}+2\gamma_2-2k_{0,2}+k_{0,3}}\right]\\
=\sum_{j=0}^{Q}b_{j}(z)\epsilon^{n_j-\alpha b_j-\chi_2(b_j-2k_{0,2}+k_{0,3}-\gamma_2)}\mathbb{T}^{b_j-2k_{0,2}+k_{0,3}-\gamma_2}+\sum_{\ell=1}^{D-1}\epsilon^{\Delta_{\ell}+\alpha(\delta_{\ell}-d_{\ell})+\beta}\hfill\\
\times\left(\sum_{q_1+q_2=\delta_\ell}\frac{\delta_\ell!}{q_1!q_2!}\prod_{d=0}^{q_1-1}(\gamma_2-d)\epsilon^{-\chi_2(d_{\ell}-q_1-2k_{0,2}+k_{0,3}-q_2)}\mathbb{T}^{d_{\ell}-q_1-2k_{0,2}+k_{0,3}}R_{\ell}(\partial_z)\partial_{\mathbb{T}}^{q_2}\mathbb{V}_2(\mathbb{T},z,\epsilon)\right)\\
+\sum_{q_1+q_2=\delta_D}\frac{\delta_D!}{q_1!q_2!}\prod_{d=0}^{q_1-1}(\gamma_2-d)\epsilon^{-\chi_2(d_D-q_1-2k_{0,2}+k_{0,3}-q_2)}\mathbb{T}^{d_{D}-q_1-2k_{0,2}+k_{0,3}}R_{D}(\partial_z)\partial_{\mathbb{T}}^{q_2}\mathbb{V}_2(\mathbb{T},z,\epsilon).
\end{multline}

Let $1\le \ell\le D$. It is worth pointing out that for every nonnegative integers $q_1,q_2$ such that $q_1+q_2=\delta_{\ell}$, and in view of (\ref{e94b}), it holds that
\begin{equation}\label{e287b}
d_{\ell}-(2k_{0,2}-k_{0,3})-q_1=(\kappa_2+1)q_2+\tilde{d}_{\ell,q_1,q_2},
\end{equation}
with $\tilde{d}_{\ell,q_1,q_2}\ge1$ for $1\le \ell\le D-1$ or $\ell=D$ and $q_2<\delta_D$; and $\tilde{d}_{D,0,\delta_{D}}=0$; Indeed, we have 
$$k_{0,1}-(2k_{0,2}-k_{0,3})=(\kappa_1-\kappa_2)q_2+\tilde{d}_{\ell,q_1,q_2}-d_{\ell,q_1,q_2}.$$

Observe that, in view of (\ref{e287}), we have that
\begin{equation}\label{e1222}
(\kappa_2-\kappa_1)q_2>d_{\ell,q_1,q_2}-\tilde{d}_{\ell,q_1,q_2},
\end{equation}
for $1\le \ell\le D-1$ or $\ell=D$ and $q_2<\delta_D$. 
Observe that, indeed it holds
\begin{equation}\label{star1222}
k_{0,1}-(2k_{0,2}-k_{0,3})=q_2(\kappa_2-\kappa_1)+\tilde{d}_{\ell,q_1,q_2}-d_{\ell,q_1,q_2}.
\end{equation}

It also holds that
\begin{equation}\label{e1226}
k_{0,1}-(2k_{0,2}-k_{0,3})=\delta_{D}(\kappa_2-\kappa_1).
\end{equation}

In view of (\ref{e94b}) and (\ref{e287b}), we get 

\begin{multline}
\mathbb{T}^{d_D-2k_{0,2}+k_{0,3}}\partial_{\mathbb{T}}^{\delta_{D}}\mathbb{V}_2(\mathbb{T},z,\epsilon)=\left((\mathbb{T}^{\kappa_2+1}\partial_{\mathbb{T}})^{\delta_{D}}+\sum_{1\le p\le \delta_{D}-1}\tilde{A}_{\delta_{D},p}\mathbb{T}^{\kappa_2(\delta_{D}-p)}(\mathbb{T}^{\kappa_2+1}\partial_{\mathbb{T}})^p\right)\mathbb{V}_2(\mathbb{T},z,\epsilon),\\
\mathbb{T}^{d_\ell-2k_{0,2}+k_{0,3}-(\delta_{\ell}-1)}\partial_{\mathbb{T}}\mathbb{V}_2(\mathbb{T},z,\epsilon)=\mathbb{T}^{\tilde{d}_{\ell,\delta_\ell-1,1}}(\mathbb{T}^{\kappa_2+1}\partial_{\mathbb{T}})\mathbb{V}_2(\mathbb{T},z,\epsilon),\\
\mathbb{T}^{d_\ell-2k_{0,2}+k_{0,3}-q_1}\partial_{\mathbb{T}}^{q_2}\mathbb{V}_2(\mathbb{T},z,\epsilon)=\mathbb{T}^{\tilde{d}_{\ell,q_1,q_2}}\left((\mathbb{T}^{\kappa_2+1}\partial_{\mathbb{T}})^{q_2}+\sum_{1\le p\le q_2-1}\tilde{A}_{q_2,p}\mathbb{T}^{\kappa_2(q_2-p)}(\mathbb{T}^{\kappa_2+1}\partial_{\mathbb{T}})^p\right)\mathbb{V}_2(\mathbb{T},z,\epsilon),\label{e301b}
\end{multline}
for every $1\le \ell\le D-1$, and all integers $q_1\ge0$ and $q_2\ge 2$ with $q_1+q_2=\delta_\ell$. Here, $\tilde{A}_{\delta_{D},p}$ for $1\le p\le \delta_D-1$ and $\tilde{A}_{q_2,p}$ for $1\le p\le q_2-1$ stand for real constants.

We make an analogous assumption as in the first problem, namely, we assume that $\mathbb{V}_2(\mathbb{T},z,\epsilon)$ has a formal power expansion of the form
\begin{equation}\label{e306b}
\mathbb{V}_2(\mathbb{T},z,\epsilon)=\sum_{n\ge1}\mathbb{V}_{n,2}(z,\epsilon)\mathbb{T}^n,
\end{equation}
where its coefficients are defined as the inverse Fourier transform of certain appropriate functions in $E_{(\beta,\mu)}$, depending holomorphically on $\epsilon$ on some punctured disc $D(0,\epsilon_0)\setminus\{0\}$, for some $\epsilon_0>0$:
$$\mathbb{V}_{n,2}(z,\epsilon)=\mathcal{F}^{-1}(m\mapsto \omega_{n,2}(m,\epsilon))(z).$$
We consider the formal $m_{\kappa_2}-$Borel transform with respect to $\mathbb{T}$ and the Fourier transform with respect to $z$ of $\mathbb{V}_2(\mathbb{T},z,\epsilon)$, and we denote it by
$$\omega_{2}(\tau,m,\epsilon)=\sum_{n\ge 1}\frac{\omega_{n,2}(m,\epsilon)}{\Gamma\left(\frac{n}{\kappa_2}\right)}\tau^n.$$

By plugging $w_{2}(\tau,m,\epsilon)$ into (\ref{e236b}) and taking into account (\ref{e301b}) and the hypotheses made on the differential operators in (\ref{e91}), we arrive at the following auxiliary problem
\begin{equation}\label{e315b}
L_{1,\kappa_2}(\omega_{2}(\tau,m,\epsilon))+L_{2,\kappa_2}(\omega_{2}(\tau,m,\epsilon))+L_{3,\kappa_2}(\omega_{2}(\tau,m,\epsilon))=R_{1,\kappa_2}(\omega_{2}(\tau,m,\epsilon)),
\end{equation}  
with $\omega_{2}(0,m,\epsilon)\equiv 0$. We have taken into account the properties and notations described in Section~\ref{seccion3} for a more compact writting. We put $\mathcal{B}_{\kappa_2}J_2(\tau,\epsilon)$ for the $m_{\kappa_2}-$Borel transform of $\mathcal{J}_2(\epsilon^{-\chi_2}\mathbb{T})$ with respect to $\mathbb{T}$. The operators in (\ref{e315b}) are defined by

\begin{multline}\label{e236b}
L_{1,\kappa_2}(\omega_2)=\tilde{Q}(im)\left[\frac{a_{0,2}^2}{a_{0,3}}\omega_2+\sum_{\ell=0}^{s_1}a_{\ell,1}\epsilon^{-\chi_2(k_{\ell,1}-2k_{0,2}+k_{0,3})}\frac{\tau^{k_{\ell,1}-2k_{0,2}+k_{0,3}}}{\Gamma\left(\frac{k_{\ell,1}-2k_{0,2}+k_{0,3}}{\kappa_2}\right)}\star_{\kappa_2}\omega_2\right.\\
\hfill+\sum_{\ell=s_1+1}^{M_1}a_{\ell,1}\epsilon^{m_{\ell,1}+\beta-\alpha k_{\ell,1}-\chi_2(k_{\ell,1}-2k_{0,2}+k_{0,3})}\frac{\tau^{k_{\ell,1}-2k_{0,2}+k_{0,3}}}{\Gamma\left(\frac{k_{\ell,1}-2k_{0,2}+k_{0,3}}{\kappa_2}\right)}\star_{\kappa_2}\omega_2\\
+\sum_{\ell=1}^{s_2}\frac{-2a_{0,2}a_{\ell,2}}{a_{0,3}}\epsilon^{-\chi_2(k_{\ell,2}-k_{0,2})}\frac{\tau^{k_{\ell,2}-k_{0,2}}}{\Gamma\left(\frac{k_{\ell,2}-k_{0,2}}{\kappa_2}\right)}\star_{\kappa_2}\omega_2\hfill\\
\hfill+\sum_{\ell=s_2+1}^{M_2}\frac{-2a_{0,2}a_{\ell,2}}{a_{0,3}}\epsilon^{m_{\ell,2}+2\beta-\alpha k_{\ell,2}-\chi_2(k_{\ell,2}-k_{0,2})} \frac{\tau^{k_{\ell,2}-k_{0,2}}}{\Gamma\left(\frac{k_{\ell,2}-k_{0,2}}{\kappa_2}\right)}\star_{\kappa_2}\omega_2\\
+\sum_{\ell=0}^{s_2}\frac{-2a_{0,2}a_{\ell,2}}{a_{0,3}}\epsilon^{-\chi_2(k_{\ell,2}-k_{0,2})}\frac{\tau^{k_{\ell,2}-k_{0,2}}}{\Gamma\left(\frac{k_{\ell,2}-k_{0,2}}{\kappa_2}\right)}\star_{\kappa_2}\mathcal{B}_{\kappa_2}J_2(\tau,\epsilon)\star_{\kappa_2}\omega_2\hfill\\
\hfill+\sum_{\ell=s_2+1}^{M_2}\frac{-2a_{0,2}a_{\ell,2}}{a_{0,3}}\epsilon^{m_{\ell,2}+2\beta-\alpha k_{\ell,2}-\chi_2(k_{\ell,2}-k_{0,2})}\frac{\tau^{k_{\ell,2}-k_{0,2}}}{\Gamma\left(\frac{k_{\ell,2}-k_{0,2}}{\kappa_2}\right)}\star_{\kappa_2}\mathcal{B}_{\kappa_2}J_2(\tau,\epsilon)\star_{\kappa_2}\omega_2\\
+\sum_{\ell=1}^{s_3}\frac{3a_{0,2}^2a_{\ell,3}}{a_{0,3}^2}\epsilon^{-\chi_2(k_{\ell,3}-k_{0,3})}\frac{\tau^{k_{\ell,3}-k_{0,3}}}{\Gamma\left(\frac{k_{\ell,3}-k_{0,3}}{\kappa_2}\right)}\star_{\kappa_2}\omega_2\hfill\\
\hfill+\sum_{\ell=s_3+1}^{M_3}\frac{3a_{0,2}^2a_{\ell,3}}{a_{0,3}^2}\epsilon^{m_{\ell,3}+3\beta-\alpha k_{\ell,3}-\chi_2(k_{\ell,3}-k_{0,3})}\frac{\tau^{k_{\ell,3}-k_{0,3}}}{\Gamma\left(\frac{k_{\ell,3}-k_{0,3}}{\kappa_2}\right)}\star_{\kappa_2}\omega_2\\
+\sum_{\ell=0}^{s_3}\frac{3a_{0,2}^2a_{\ell,3}}{a_{0,3}^2}\epsilon^{-\chi_2(k_{\ell,3}-k_{0,3})}\frac{\tau^{k_{\ell,3}-k_{0,3}}}{\Gamma\left(\frac{k_{\ell,3}-k_{0,3}}{\kappa_2}\right)}\star_{\kappa_2}2\mathcal{B}_{\kappa_2}J_2(\tau,\epsilon)\star_{\kappa_2}\omega_2\hfill\\    
\hfill+\sum_{\ell=s_3+1}^{M_3}\frac{3a_{0,2}^2a_{\ell,3}}{a_{0,3}^2}\epsilon^{-\chi_2(k_{\ell,3}-k_{0,3})}\frac{\tau^{k_{\ell,3}-k_{0,3}}}{\Gamma\left(\frac{k_{\ell,3}-k_{0,3}}{\kappa_2}\right)}\star_{\kappa_2}2\mathcal{B}_{\kappa_2}J_2(\tau,\epsilon)\star_{\kappa_2}\omega_2\\
+\sum_{\ell=0}^{s_3}\frac{3a_{0,2}^2a_{\ell,3}}{a_{0,3}^2}\epsilon^{-\chi_2(k_{\ell,3}-k_{0,3})}\frac{\tau^{k_{\ell,3}-k_{0,3}}}{\Gamma\left(\frac{k_{\ell,3}-k_{0,3}}{\kappa_2}\right)}\star_{\kappa_2}\mathcal{B}_{\kappa_2}J_2(\tau,\epsilon)\star_{\kappa_2}\mathcal{B}_{\kappa_2}J_2(\tau,\epsilon)\star_{\kappa_2}\omega_2\hfill\\    
\hfill+\sum_{\ell=s_3+1}^{M_3}\frac{3a_{0,2}^2a_{\ell,3}}{a_{0,3}^2}\epsilon^{-\chi_2(k_{\ell,3}-k_{0,3})}\frac{\tau^{k_{\ell,3}-k_{0,3}}}{\Gamma\left(\frac{k_{\ell,3}-k_{0,3}}{\kappa_2}\right)}\star_{\kappa_2}\mathcal{B}_{\kappa_2}J_2(\tau,\epsilon)\star_{\kappa_2}\mathcal{B}_{\kappa_2}J_2(\tau,\epsilon)\star_{\kappa_2}\omega_2,
\end{multline}

\begin{multline}
L_{2,\kappa_2}(\omega_2)=\tilde{Q}(im)\left[\sum_{\ell=0}^{s_2}a_{\ell,2}\epsilon^{-\chi_2(k_{\ell,2}-2k_{0,2}+k_{0,3}+\gamma_2)}\frac{\tau^{k_{\ell,2}-2k_{0,2}+k_{0,3}+\gamma_2}}{\Gamma\left(\frac{k_{\ell,2}-2k_{0,2}+k_{0,3}+\gamma_2}{\kappa_2}\right)}\star_{\kappa_2}\omega_2\star_{\kappa_2}^{E}\omega_2\right. \hfill\\
\hfill+\sum_{\ell=s_2+1}^{M_2}a_{\ell,2}\epsilon^{m_{\ell,2}+2\beta-\alpha k_{\ell,2}-\chi_2(k_{\ell,2}-2k_{0,2}+k_{0,3}+\gamma_2)}\frac{\tau^{k_{\ell,2}-2k_{0,2}+k_{0,3}+\gamma_2}}{\Gamma\left(\frac{k_{\ell,2}-2k_{0,2}+k_{0,3}+\gamma_2
}{\kappa_2}\right)}\star_{\kappa_2}\omega_2\star_{\kappa_2}^{E}\omega_2\\
+\sum_{\ell=0}^{s_3}a_{\ell,3}\epsilon^{-\chi_2k_{\ell,3}}
\frac{\tau^{k_{\ell,3}}}{\Gamma\left(\frac{k_{\ell,3}}{\kappa_2}\right)}\star_{\kappa_2}\omega_2\star_{\kappa_2}^{E}\omega_2\hfill\\
\hfill+\sum_{\ell=s_3+1}^{M_3}a_{\ell,3}\epsilon^{-\chi_2k_{\ell,3}}
\frac{\tau^{k_{\ell,3}}}{\Gamma\left(\frac{k_{\ell,3}}{\kappa_2}\right)}\star_{\kappa_2}\omega_2\star_{\kappa_2}^{E}\omega_2\hfill\\
+\sum_{\ell=0}^{s_3}a_{\ell,3}\epsilon^{m_{\ell,3}+3\beta-\alpha k_{\ell,3}-\chi_2k_{\ell,3}}
\frac{\tau^{k_{\ell,3}}}{\Gamma\left(\frac{k_{\ell,3}}{\kappa_2}\right)}\star_{\kappa_2}\mathcal{B}_{\kappa_2}J_2(\tau,\epsilon)\star_{\kappa_2}\star_{\kappa_2}\omega_2\star_{\kappa_2}^{E}\omega_2\hfill\\
\hfill\left.+\sum_{\ell=s_3+1}^{M_3}a_{\ell,3}\epsilon^{m_{\ell,3}+3\beta-\alpha k_{\ell,3}-\chi_2k_{\ell,3}}
\frac{\tau^{k_{\ell,3}}}{\Gamma\left(\frac{k_{\ell,3}}{\kappa_2}\right)}\star_{\kappa_2}\mathcal{B}_{\kappa_2}J_2(\tau,\epsilon)\star_{\kappa_2}\star_{\kappa_2}\omega_2\star_{\kappa_2}^{E}\omega_2\right],\hfill\\
\end{multline}

\begin{multline}
L_{3,\kappa_2}(\omega_2)=\tilde{Q}(im)\left[\sum_{\ell=0}^{s_3}a_{\ell,3}\epsilon^{-\chi_2(k_{\ell,3}+2\gamma_2-2k_{0,2}+k_{0,3})}
\frac{\tau^{k_{\ell,3}+2\gamma_2-2k_{0,2}+k_{0,3}}}{\Gamma\left(\frac{k_{\ell,3}+2\gamma_2-2k_{0,2}+k_{0,3}}{\kappa_2}\right)}\star_{\kappa_2}\omega_2\star_{\kappa_2}^{E}\omega_2\star_{\kappa_2}^{E}\omega_2\right.\hfill\\
\hfill\left.+\sum_{\ell=s_3+1}^{M_3}a_{\ell,3}\epsilon^{-\chi_2(k_{\ell,3}+2\gamma_2-2k_{0,2}+k_{0,3})}
\frac{\tau^{k_{\ell,3}+2\gamma_2-2k_{0,2}+k_{0,3}}}{\Gamma\left(\frac{k_{\ell,3}+2\gamma_2-2k_{0,2}+k_{0,3}}{\kappa_2}\right)}\star_{\kappa_2}\omega_2\star_{\kappa_2}^{E}\omega_2\star_{\kappa_2}^{E}\omega_2\right],
\end{multline}

\begin{multline}
R_{1,\kappa_2}(\omega_2)=\sum_{j=0}^{Q}\tilde{B}_{j}(im)\epsilon^{n_j-\alpha b_j-\chi_2(b_j-k_{0,2}+k_{0,3}-\gamma_2)}\frac{\tau^{b_j-k_{0,2}+k_{0,3}-\gamma_2}}{\Gamma\left(\frac{b_j-k_{0,2}+k_{0,3}-\gamma_2}{\kappa_2}\right)}\\
+\sum_{\ell=1}^{D-1}\epsilon^{\Delta_{\ell}+\alpha(\delta_{\ell}-d_{\ell})+\beta}\sum_{q_1+q_2=\delta_\ell}\frac{\delta_\ell!}{q_1!q_2!}\prod_{d=0}^{q_1-1}(\gamma_2-d)\epsilon^{-\chi_2(d_{\ell}-q_1-k_{0,2}+k_{0,3}-q_2)}\tilde{R}_{\ell}(im)\\
\times\left(\frac{\tau^{\tilde{d}_{\ell,q_1,q_2}}}{\Gamma\left(\frac{\tilde{d}_{\ell,q_1,q_2}}{\kappa_2}\right)}\star_{\kappa_2}\left((\kappa_2\tau^{\kappa_2})^{q_2}\omega_2\right)+\sum_{1\le p\le q_2-1}\tilde{A}_{q_2,p}\frac{\tau^{\tilde{d}_{\ell,q_1,q_2}+\kappa_2(q_2-p)}}{\Gamma\left(\frac{\tilde{d}_{\ell,q_1,q_2}+\kappa_2(q_2-p)}{\kappa_2}\right)}\star_{\kappa_2}\left((\kappa_2\tau^{\kappa_2})^{p}\omega_2\right)\right)\\
+\sum_{q_1+q_2=\delta_D,q_1\ge1}\frac{\delta_D!}{q_1!q_2!}\prod_{d=0}^{q_1-1}(\gamma_2-d)\tilde{R}_{D}(im)\\
\times\left(\frac{\tau^{\tilde{d}_{D,q_1,q_2}}}{\Gamma\left(\frac{\tilde{d}_{D,q_1,q_2}}{\kappa_2}\right)}\star_{\kappa_2}\left((\kappa_2\tau^{\kappa_2})^{q_2}\omega_2\right)+\sum_{1\le p\le q_2-1}\tilde{A}_{q_2,p}\frac{\tau^{\tilde{d}_{D,q_1,q_2}+\kappa_2(q_2-p)}}{\Gamma\left(\frac{\tilde{d}_{D,q_1,q_2}+\kappa_2(q_2-p)}{\kappa_2}\right)}\star_{\kappa_2}\left((\kappa_2\tau^{\kappa_2})^{p}\omega_2\right)\right)\\
+\tilde{R}_{D}(im)\left((\kappa_2\tau^{\kappa_2})^{\delta_{D}}\omega_2+\sum_{1\le p\le \delta_D-1}\tilde{A}_{\delta_D,p}\frac{\tau^{\kappa_2(\delta_D-p)}}{\Gamma\left(\frac{\kappa_2(\delta_D-p)}{\kappa_2}\right)}\star_{\kappa_2}\left((\kappa_2\tau^{\kappa_2})^{p}\omega_2\right)\right).
\end{multline}

\subsection{Analytic solution of the second perturbed auxiliary problem}

This section states the geometry of the second problem, in the same way as in Section~\ref{subseccion1}. We omit the details on this construction.

We assume there exists an unbounded sector
$$\tilde{S}_{\tilde{Q},\tilde{R}_D}=\{z\in\C^{\star}:|z|\ge r_{\tilde{Q},\tilde{R}_D},|\arg(z)-\tilde{d}_{\tilde{Q},\tilde{R}_D}|\le \nu_{\tilde{Q},\tilde{R}_D}\},$$
for some bisecting direction $\tilde{d}_{\tilde{Q},\tilde{R}_D}\in\R$, opening $\nu_{\tilde{Q},\tilde{R}_D}>0$, and $r_{\tilde{Q},\tilde{R}_D}>0$ such that 
\begin{equation}\label{e366b}
\frac{\tilde{Q}(im)}{\tilde{R}_D(im)}\in \tilde{S}_{\tilde{Q},\tilde{R}_D},\qquad m\in\R.
\end{equation}

For every $m\in\R$, the roots of the polynomial $\tilde{P}_{2,m}=-\frac{a_{0,2}^2}{a_{0,3}}\tilde{Q}(im)-\tilde{R}_D(im)\kappa_2^{\delta_{D}}\tau^{\delta_{D}\kappa_2}$ are given by
\begin{equation}\label{e371b}
\tilde{q}_{\ell}(m)=\left(\frac{|a_{0,2}^2\tilde{Q}(im)|}{|a_{0,3}\tilde{R}_{D}(im)|\kappa_2^{\delta_{D}}}\right)^{\frac{1}{\delta_{D}\kappa_2}}\exp\left(\sqrt{-1}(\arg(\frac{-a_{0,2}^2\tilde{Q}(im)}{a_{0,3}\tilde{R}_{D}(im)\kappa_2^{\delta_{D}}}))\frac{1}{\delta_{D}\kappa_2}+\frac{2\pi\ell}{\delta_{D}\kappa_2}\right),
\end{equation}
for $0\le \ell\le \delta_{D}\kappa_2-1$. Let $S_{\tilde{d}}$ be an unbounded sector of bisecting direction $\tilde{d}\in\R$ and vertex at the origin, and $\rho>0$ such that the three next conditions are satisfied:

1) There exists $M_1>0$ such that
\begin{equation}\label{e377b}
|\tau-\tilde{q}_{\ell}(m)|\ge M_1(1+|\tau|),
\end{equation}
for every $0\le \ell\le \delta_{D}\kappa_2$, $m\in\R$ and $\tau\in S_{\tilde{d}}\cup \bar{D}(0,\rho)$. 

2) There exists $M_2>0$ such that
\begin{equation}\label{e378b}
|\tau-\tilde{q}_{\ell_0}(m)|\ge M_2|\tilde{q}_{\ell_0}|,
\end{equation}
for some $0\le \ell_0\le \delta_{D}\kappa_2-1$, all $m\in\R$ and all $\tau\in S_{\tilde{d}}\cup\bar{D}(0,\rho)$.

Following the steps in (\ref{e949}), we get a constant $C_{\tilde{P}}>0$ such that
\begin{equation}
|\tilde{P}_{2,m}(\tau)|\ge C_{\tilde{P}} (r_{\tilde{Q},\tilde{R}_{D}})^{\frac{1}{\delta_{D}\kappa_2}} |\tilde{R}_{D}(im)|
(1+|\tau|^{\kappa_2})^{\delta_{D} - \frac{1}{\kappa_2}} \label{e949b}
\end{equation}
for all $\tau \in S_{\tilde{d}} \cup \bar{D}(0,\rho)$, all $m \in \mathbb{R}$.

The proof of the next result is analogous to that of Lemma~\ref{lema533}, so we omit it.

\begin{lemma}\label{lema533b}
One has
$$\mathcal{B}_{\kappa_2}J_2(\tau,\epsilon)\star_{\kappa_2}\mathcal{B}_{\kappa_2}J_2(\tau,\epsilon)=\mathcal{B}_{\kappa_2}\tilde{J}_2(\tau,\epsilon),$$
where $\mathcal{B}_{\kappa_2}\tilde{J}_2(\tau,\epsilon)$ stands for the $m_{\kappa_2}-$Borel transform of the formal power series 
$$\hat{J}_2(\mathbb{T})=(J_2(\mathbb{T})\cdot J_2(\mathbb{T})),$$
evaluated at $\epsilon^{-\chi_2}\mathbb{T}$. Its coefficients are notated by $\tilde{J}_{2j}$.
\end{lemma}

The following result reduces the number of global restrictions on the parameters involved in the problem, relating those appearing in the first problem, with those naturally arising from the second one.

\begin{lemma}
Under assumptions (\ref{e94}) and (\ref{e287}), the following statement holds: Let $1\le \ell\le D-1$, $q_1,q_2\in\N$ such that $q_1+q_2=\delta_\ell$, and $q_2\ge1$. Then, it holds that
\begin{multline}\label{e1749}
 \Delta_{\ell} + \alpha(\delta_{\ell} - d_{\ell}) + \beta + 
(\chi_1 + \alpha)\kappa_1(\frac{d_{\ell,q_{1},q_{2}}}{\kappa_1} + q_{2} - \delta_{D} + \frac{1}{\kappa_1})
-\chi_1(d_{\ell} - k_{0,1} - \delta_{\ell})\\
\ge\Delta_{\ell} + \alpha(\delta_{\ell} - d_{\ell}) + \beta + 
(\chi_2 + \alpha)\kappa_2(\frac{\tilde{d}_{\ell,q_{1},q_{2}}}{\kappa_2} + q_{2} - \delta_{D} + \frac{1}{\kappa_2})
-\chi_2(d_{\ell} - k_{0,2}+k_{0,3} - \delta_{\ell}).
\end{multline}

Under assumption (\ref{e887}), and $\kappa_2<\kappa_1$, one has
$$\delta_D\ge \frac{2}{\kappa_2},\qquad \delta_D\ge \frac{1}{\kappa_2}+\delta_\ell,$$
for every $1\le \ell\le D-1$.

\end{lemma}
\begin{proof}
We apply (\ref{star1222}) and (\ref{e1226}) reduce the inequality (\ref{e1749}) to
$$ 
(\chi_1 + \alpha)(d_{\ell,q_{1},q_{2}} + (q_{2} - \delta_{D})\kappa_1 +1 )
-\chi_1(d_{\ell} - k_{0,1} - \delta_{\ell})\ge
(\chi_2 + \alpha)(\tilde{d}_{\ell,q_{1},q_{2}} + (q_{2} - \delta_{D})\kappa_2 + 1)-\chi_2(d_{\ell} - k_{0,2}+k_{0,3} - \delta_{\ell}).$$
After an arrangement of the terms, and the application of (\ref{e287}) and (\ref{e94}) we derive that the previous inequality holds if the following does:
$$\chi_1(1-\delta_{D}\kappa_1)\ge \chi_2(-\delta_{D}\kappa_1+1-k_{0,1}+k_{0,2}-k_{0,3}).$$
Finally, the definition of $\chi_1$ and $\chi_2$ in (\ref{e256}) and (\ref{e256b}) resp., and again the application of (\ref{e94}) leads to the equivalent inequality
$$\delta_{D}(\kappa_2-\kappa_1)+k_{0,2}\delta_D\kappa_1\ge0,$$
which is satisfied.

The second statement is direct from the hypotheses made.
\end{proof}

The proof of the next result is left to Section~\ref{seccionaux2}.

\begin{lemma}\label{lema568b}
Let the following conditions hold:

\begin{multline}
\delta_{D} \geq \frac{2}{\kappa_2} \ \ , \ \ \gamma_2\ge k_{0,2}-k_{0,3} \ \ , \ \ b_{j} - 2k_{0,2}+k_{0,3} - \gamma_2 \geq 1,\\
(\chi_2 + \alpha)( k_{\ell_2,2}+\gamma_2-2k_{0,2}+k_{0,3}-\kappa_2\delta_D+1) - \chi_2(k_{\ell_2,2}+\gamma_2-2k_{0,2}+k_{0,3}) \geq 0\\
(\chi_2 + \alpha)( k_{\ell_3,3}-\kappa_2\delta_D+1) - \chi_2k_{\ell_3,3} \geq 0
\label{e887b}
\end{multline}
for all $0 \leq \ell_2 \leq M_2$, $0 \leq \ell_3 \leq M_3$, $0 \leq j \leq Q$,
\begin{multline}
\delta_{D} \geq \frac{1}{\kappa_2} + \delta_{\ell},\\
\Delta_{\ell} + \alpha(\delta_{\ell} - d_{\ell}) + \beta + 
(\chi_2 + \alpha)\kappa_2(\frac{\tilde{d}_{\ell,q_{1},q_{2}}}{\kappa_2} + q_{2} - \delta_{D} + \frac{1}{\kappa_2})
-\chi_2(d_{\ell} - k_{0,2} +k_{0,3}- \delta_{\ell}) \geq 0,
\label{e896b}
\end{multline}
for every $q_1\ge0,q_2\ge1$ such that $q_1+q_2=\delta_\ell$, for $1\le \ell\le D-1$.

Then, there exist large enough $r_{\tilde{Q},\tilde{R}_{D}}>0$ and small enough $\varpi>0$ such that for every $\epsilon\in D(0,\epsilon_0)\setminus\{0\}$, the map $\mathcal{\epsilon}$ defined by
$$\tilde{\mathcal{H}}_{\epsilon}=\tilde{\mathcal{H}}_{\epsilon}^1+\tilde{\mathcal{H}}_{\epsilon}^2+\tilde{\mathcal{H}}_{\epsilon}^3+\tilde{\mathcal{H}}_{\epsilon}^4,$$

satisfies that $\tilde{\mathcal{H}}_{\epsilon}(\bar{B}(0,\varpi))\subseteq\bar{B}(0,\varpi)$, where $\bar{B}(0,\varpi)$ is the closed disc of radius $\varpi>0$ centered at 0, in $F^{\tilde{d}}_{(\nu,\beta,\mu,\chi_2,\alpha,\kappa_2,\epsilon)}$, for every $\epsilon\in D(0,\epsilon_0)\setminus\{0\}$. Moreover, it holds that
\begin{equation}\label{e442b}
\left\|\tilde{\mathcal{H}}_{\epsilon}(\omega_1)-\tilde{\mathcal{H}}_{\epsilon}(\omega_2)\right\|_{(\nu,\beta,\mu,\chi_2,\alpha,\kappa_2,\epsilon)}\le\frac{1}{2} \left\|\omega_1-\omega_2\right\|_{(\nu,\beta,\mu,\chi_2,\alpha,\kappa_2,\epsilon)},
\end{equation}
for every $\omega_1,\omega_2\in \bar{B}(0,\varpi)$, and every $\epsilon\in D(0,\epsilon_0)\setminus\{0\}$.

Here, we have defined
\begin{multline}
\tilde{\mathcal{H}}_\epsilon^1(\omega_2(\tau,\epsilon)):=\sum_{j=0}^{Q}\tilde{B}_{j}(im)\epsilon^{n_j-\alpha b_j-\chi_2(b_j-k_{0,2}+k_{0,3}-\gamma_2)}\frac{\tau^{b_j-k_{0,2}+k_{0,3}-\gamma_2}}{\tilde{P}_{2,m}(\tau)\Gamma\left(\frac{b_j-k_{0,2}+k_{0,3}-\gamma_2}{\kappa_2}\right)}\\
-\tilde{Q}(im)\left\{\sum_{\ell=0}^{s_1}\frac{a_{\ell,1}\epsilon^{-\chi_2(k_{\ell,1}-2k_{0,2}+k_{0,3})}}{\tilde{P}_{2,m}(\tau)}\left[\frac{\tau^{k_{\ell,1}-2k_{0,2}+k_{0,3}}}{\Gamma\left(\frac{k_{\ell,1}-2k_{0,2}+k_{0,3}}{\kappa_2}\right)}\star_{\kappa_2}\omega_2\right]\right.\\
\hfill+\sum_{\ell=s_1+1}^{M_1}\frac{a_{\ell,1}\epsilon^{m_{\ell,1}+\beta-\alpha k_{\ell,1}-\chi_2(k_{\ell,1}-2k_{0,2}+k_{0,3})}}{\tilde{P}_{2,m}(\tau)}\left[\frac{\tau^{k_{\ell,1}-2k_{0,2}+k_{0,3}}}{\Gamma\left(\frac{k_{\ell,1}-2k_{0,2}+k_{0,3}}{\kappa_2}\right)}\star_{\kappa_2}\omega_2\right]\\
+\sum_{\ell=1}^{s_2}\frac{-2a_{0,2}a_{\ell,2}}{a_{0,3}}\frac{\epsilon^{-\chi_2(k_{\ell,2}-k_{0,2})}}{\tilde{P}_{2,m}(\tau)}\left[\frac{\tau^{k_{\ell,2}-k_{0,2}}}{\Gamma\left(\frac{k_{\ell,2}-k_{0,2}}{\kappa_2}\right)}\star_{\kappa_2}\omega_2\right]\hfill\\
\hfill+\sum_{\ell=s_2+1}^{M_2}\frac{-2a_{0,2}a_{\ell,2}}{a_{0,3}}\frac{\epsilon^{m_{\ell,2}+2\beta-\alpha k_{\ell,2}-\chi_2(k_{\ell,2}-k_{0,2})}}{\tilde{P}_{2,m}(\tau)} \left[\frac{\tau^{k_{\ell,2}-k_{0,2}}}{\Gamma\left(\frac{k_{\ell,2}-k_{0,2}}{\kappa_2}\right)}\star_{\kappa_2}\omega_2\right]\\
+\sum_{\ell=0}^{s_2}\frac{-2a_{0,2}a_{\ell,2}}{a_{0,3}}\frac{\epsilon^{-\chi_2(k_{\ell,2}-k_{0,2})}}{\tilde{P}_{2,m}(\tau)}\left[\frac{\tau^{k_{\ell,2}-k_{0,2}}}{\Gamma\left(\frac{k_{\ell,2}-k_{0,2}}{\kappa_2}\right)}\star_{\kappa_2}\mathcal{B}_{\kappa_2}J_2(\tau,\epsilon)\star_{\kappa_2}\omega_2\right]\hfill\\
\hfill+\sum_{\ell=s_2+1}^{M_2}\frac{-2a_{0,2}a_{\ell,2}}{a_{0,3}}\frac{\epsilon^{m_{\ell,2}+2\beta-\alpha k_{\ell,2}-\chi_2(k_{\ell,2}-k_{0,2})}}{\tilde{P}_{2,m}(\tau)}\left[\frac{\tau^{k_{\ell,2}-k_{0,2}}}{\Gamma\left(\frac{k_{\ell,2}-k_{0,2}}{\kappa_2}\right)}\star_{\kappa_2}\mathcal{B}_{\kappa_2}J_2(\tau,\epsilon)\star_{\kappa_2}\omega_2\right]\\
+\sum_{\ell=1}^{s_3}\frac{3a_{0,2}^2a_{\ell,3}}{a_{0,3}^2}\frac{\epsilon^{-\chi_2(k_{\ell,3}-k_{0,3})}}{\tilde{P}_{2,m}(\tau)}\left[\frac{\tau^{k_{\ell,3}-k_{0,3}}}{\Gamma\left(\frac{k_{\ell,3}-k_{0,3}}{\kappa_2}\right)}\star_{\kappa_2}\omega_2\right]\hfill\\
\hfill+\sum_{\ell=s_3+1}^{M_3}\frac{3a_{0,2}^2a_{\ell,3}}{a_{0,3}^2}\frac{\epsilon^{m_{\ell,3}+3\beta-\alpha k_{\ell,3}-\chi_2(k_{\ell,3}-k_{0,3})}}{\tilde{P}_{2,m}(\tau)}\left[\frac{\tau^{k_{\ell,3}-k_{0,3}}}{\Gamma\left(\frac{k_{\ell,3}-k_{0,3}}{\kappa_2}\right)}\star_{\kappa_2}\omega_2\right]\\
+\sum_{\ell=0}^{s_3}\frac{3a_{0,2}^2a_{\ell,3}}{a_{0,3}^2}\frac{\epsilon^{-\chi_2(k_{\ell,3}-k_{0,3})}}{\tilde{P}_{2,m}(\tau)}\left[\frac{\tau^{k_{\ell,3}-k_{0,3}}}{\Gamma\left(\frac{k_{\ell,3}-k_{0,3}}{\kappa_2}\right)}\star_{\kappa_2}2\mathcal{B}_{\kappa_2}J_2(\tau,\epsilon)\star_{\kappa_2}\omega_2\right]\hfill\\ 
\hfill+\sum_{\ell=s_3+1}^{M_3}\frac{3a_{0,2}^2a_{\ell,3}}{a_{0,3}^2}\frac{\epsilon^{-\chi_2(k_{\ell,3}-k_{0,3})}}{\tilde{P}_{2,m}(\tau)}\left[\frac{\tau^{k_{\ell,3}-k_{0,3}}}{\Gamma\left(\frac{k_{\ell,3}-k_{0,3}}{\kappa_2}\right)}\star_{\kappa_2}2\mathcal{B}_{\kappa_2}J_2(\tau,\epsilon)\star_{\kappa_2}\omega_2\right]\\
+\sum_{\ell=0}^{s_3}\frac{3a_{0,2}^2a_{\ell,3}}{a_{0,3}^2}\frac{\epsilon^{-\chi_2(k_{\ell,3}-k_{0,3})}}{\tilde{P}_{2,m}(\tau)}\left[\frac{\tau^{k_{\ell,3}-k_{0,3}}}{\Gamma\left(\frac{k_{\ell,3}-k_{0,3}}{\kappa_2}\right)}\star_{\kappa_2}\mathcal{B}_{\kappa_2}J_2(\tau,\epsilon)\star_{\kappa_2}\mathcal{B}_{\kappa_2}J_2(\tau,\epsilon)\star_{\kappa_2}\omega_2\right]\hfill\\    
\hfill\left.+\sum_{\ell=s_3+1}^{M_3}\frac{3a_{0,2}^2a_{\ell,3}}{a_{0,3}^2}\frac{\epsilon^{-\chi_2(k_{\ell,3}-k_{0,3})}}{\tilde{P}_{2,m}(\tau)}\left[\frac{\tau^{k_{\ell,3}-k_{0,3}}}{\Gamma\left(\frac{k_{\ell,3}-k_{0,3}}{\kappa_2}\right)}\star_{\kappa_2}\mathcal{B}_{\kappa_2}J_2(\tau,\epsilon)\star_{\kappa_2}\mathcal{B}_{\kappa_2}J_2(\tau,\epsilon)\star_{\kappa_2}\omega_2\right]\right\},
\end{multline}

\begin{multline}
\tilde{\mathcal{H}}_\epsilon^2(\omega_2(\tau,\epsilon)):=\sum_{\ell=1}^{D-1}\epsilon^{\Delta_{\ell}+\alpha(\delta_{\ell}-d_{\ell})+\beta}\sum_{q_1+q_2=\delta_\ell}\frac{\delta_\ell!}{q_1!q_2!}\prod_{d=0}^{q_1-1}(\gamma_2-d)\epsilon^{-\chi_2(d_{\ell}-q_1-k_{0,2}+k_{0,3}-q_2)}\frac{\tilde{R}_{\ell}(im)}{\tilde{P}_{2,m}(\tau)}\\
\times\left[\frac{\tau^{\tilde{d}_{\ell,q_1,q_2}}}{\Gamma\left(\frac{\tilde{d}_{\ell,q_1,q_2}}{\kappa_2}\right)}\star_{\kappa_2}\left((\kappa_2\tau^{\kappa_2})^{q_2}\omega_2\right)+\sum_{1\le p\le q_2-1}\frac{\tilde{A}_{q_2,p}}{\tilde{P}_{2,m}(\tau)}\left(\frac{\tau^{\tilde{d}_{\ell,q_1,q_2}+\kappa_2(q_2-p)}}{\Gamma\left(\frac{\tilde{d}_{\ell,q_1,q_2}+\kappa_2(q_2-p)}{\kappa_2}\right)}\star_{\kappa_2}\left((\kappa_2\tau^{\kappa_2})^{p}\omega_2\right)\right)\right]\\
-\tilde{Q}(im)\left\{\sum_{\ell=0}^{s_2}a_{\ell,2}\frac{\epsilon^{-\chi_2(k_{\ell,2}-2k_{0,2}+k_{0,3}+\gamma_2)}}{\tilde{P}_{2,m}(\tau)}\left[\frac{\tau^{k_{\ell,2}-2k_{0,2}+k_{0,3}+\gamma_2}}{\Gamma\left(\frac{k_{\ell,2}-2k_{0,2}+k_{0,3}+\gamma_2}{\kappa_2}\right)}\star_{\kappa_2}\omega_2\star_{\kappa_2}^{E}\omega_2\right]\right. \hfill\\
\hfill+\sum_{\ell=s_2+1}^{M_2}a_{\ell,2}\frac{\epsilon^{m_{\ell,2}+2\beta-\alpha k_{\ell,2}-\chi_2(k_{\ell,2}-2k_{0,2}+k_{0,3}+\gamma_2)}}{\tilde{P}_{2,m}(\tau)}\left[\frac{\tau^{k_{\ell,2}-2k_{0,2}+k_{0,3}+\gamma_2}}{\Gamma\left(\frac{k_{\ell,2}-2k_{0,2}+k_{0,3}+\gamma_2
}{\kappa_2}\right)}\star_{\kappa_2}\omega_2\star_{\kappa_2}^{E}\omega_2\right]\\
+\sum_{\ell=0}^{s_3}a_{\ell,3}\frac{\epsilon^{-\chi_2k_{\ell,3}}}{\tilde{P}_{2,m}(\tau)}\left[
\frac{\tau^{k_{\ell,3}}}{\Gamma\left(\frac{k_{\ell,3}}{\kappa_2}\right)}\star_{\kappa_2}\omega_2\star_{\kappa_2}^{E}\omega_2\right]\hfill\\
\hfill+\sum_{\ell=s_3+1}^{M_3}a_{\ell,3}\frac{\epsilon^{m_{\ell,3}+3\beta-\alpha k_{\ell,3}-\chi_2k_{\ell,3}}}{\tilde{P}_{2,m}(\tau)}\left[
\frac{\tau^{k_{\ell,3}}}{\Gamma\left(\frac{k_{\ell,3}}{\kappa_2}\right)}\star_{\kappa_2}\omega_2\star_{\kappa_2}^{E}\omega_2\right]\hfill\\
+\sum_{\ell=0}^{s_3}a_{\ell,3}\frac{\epsilon^{-\chi_2k_{\ell,3}}}{\tilde{P}_{2,m}(\tau)}\left[
\frac{\tau^{k_{\ell,3}}}{\Gamma\left(\frac{k_{\ell,3}}{\kappa_2}\right)}\star_{\kappa_2}\mathcal{B}_{\kappa_2}J_2(\tau,\epsilon)\star_{\kappa_2}\omega_2\star_{\kappa_2}^{E}\omega_2\right]\hfill\\
\hfill\left.+\sum_{\ell=s_3+1}^{M_3}a_{\ell,3}\frac{\epsilon^{m_{\ell,3}+3\beta-\alpha k_{\ell,3}-\chi_2k_{\ell,3}}}{\tilde{P}_{2,m}(\tau)}\left[
\frac{\tau^{k_{\ell,3}}}{\Gamma\left(\frac{k_{\ell,3}}{\kappa_2}\right)}\star_{\kappa_2}\mathcal{B}_{\kappa_2}J_2(\tau,\epsilon)\star_{\kappa_2}\omega_2\star_{\kappa_2}^{E}\omega_2\right]\right\},\hfill
\end{multline}

\begin{multline}
\tilde{\mathcal{H}}_\epsilon^3(\omega_2(\tau,\epsilon)):=\frac{1}{\tilde{P}_{2,m}(\tau)}\sum_{q_1+q_2=\delta_D,q_1\ge1}\frac{\delta_D!}{q_1!q_2!}\prod_{d=0}^{q_1-1}(\gamma_2-d)\tilde{R}_{D}(im)\\
\times\left[\frac{\tau^{\tilde{d}_{D,q_1,q_2}}}{\Gamma\left(\frac{\tilde{d}_{D,q_1,q_2}}{\kappa_2}\right)}\star_{\kappa_2}\left((\kappa_2\tau^{\kappa_2})^{q_2}\omega_2\right)+\sum_{1\le p\le q_2-1}\tilde{A}_{q_2,p}\frac{\tau^{\tilde{d}_{D,q_1,q_2}+\kappa_2(q_2-p)}}{\Gamma\left(\frac{\tilde{d}_{D,q_1,q_2}+\kappa_2(q_2-p)}{\kappa_2}\right)}\star_{\kappa_2}\left((\kappa_2\tau^{\kappa_2})^{p}\omega_2\right)\right]\\
\hfill+\frac{\tilde{R}_{D}(im)}{\tilde{P}_{2,m}(\tau)}\left[(\kappa_2\tau^{\kappa_2})^{\delta_{D}}\omega_2+\sum_{1\le p\le \delta_D-1}\tilde{A}_{\delta_D,p}\frac{\tau^{\kappa_2(\delta_D-p)}}{\Gamma\left(\frac{\kappa_2(\delta_D-p)}{\kappa_2}\right)}\star_{\kappa_2}\left((\kappa_2\tau^{\kappa_2})^{p}\omega_2\right)\right],
\end{multline}

and

\begin{multline}
\tilde{\mathcal{H}}_\epsilon^4(\omega_2(\tau,\epsilon)):=-\tilde{Q}(im)\left\{\sum_{\ell=0}^{s_3}a_{\ell,3}\frac{\epsilon^{-\chi_2(k_{\ell,3}+2\gamma_2-2k_{0,2}+k_{0,3})}}{\tilde{P}_{2,m}(\tau)}\left[
\frac{\tau^{k_{\ell,3}+2\gamma_2-2k_{0,2}+k_{0,3}}}{\Gamma\left(\frac{k_{\ell,3}+2\gamma_2-2k_{0,2}+k_{0,3}}{\kappa_2}\right)}\star_{\kappa_2}\omega_2\star_{\kappa_2}^{E}\omega_2\star_{\kappa_2}^{E}\omega_2\right]\right.\hfill\\
\hfill\left.+\sum_{\ell=s_3+1}^{M_3}a_{\ell,3}\frac{\epsilon^{m_{\ell,3}+3\beta-\alpha k_{\ell,3}-\chi_2(k_{\ell,3}+2\gamma_2-2k_{0,2}+k_{0,3})}}{\tilde{P}_{2,m}(\tau)}\left[
\frac{\tau^{k_{\ell,3}+2\gamma_2-2k_{0,2}+k_{0,3}}}{\Gamma\left(\frac{k_{\ell,3}+2\gamma_2-2k_{0,2}+k_{0,3}}{\kappa_2}\right)}\star_{\kappa_2}\omega_2\star_{\kappa_2}^{E}\omega_2\star_{\kappa_2}^{E}\omega_2\right]\right\}.
\end{multline}
\end{lemma}

\begin{prop}\label{prop1518b}
Under the assumptions (\ref{e887b}), (\ref{e896b}), there exists $r_{\tilde{Q},\tilde{R}_D}>0$, $\epsilon_0>0$ and $\varpi>0$ such that the problem (\ref{e315b}) admits a unique solution $\omega_{\kappa_2}^{\tilde{d}}(\tau,m,\epsilon)$ belonging to the Banach space $F^{\tilde{d}}_{(\nu,\beta,\mu,\chi_2,\alpha,\kappa_2,\epsilon)}$, with 
$$  \left\|\omega_{\kappa_2}^{\tilde{d}}(\tau,m,\epsilon)\right\|_{(\nu,\beta,\mu,\chi_2,\alpha,\kappa_2,\epsilon)}\le\varpi,$$
for every $\epsilon\in D(0,\epsilon_0)\setminus\{0\}$, where $\tilde{d}\in\R$ is such that (\ref{e377b}) and (\ref{e378b}) are satisfied.
\end{prop}
\begin{proof}
It is analogous to the proof of Proposition~\ref{prop1518}, so we omit it.
\end{proof}

\section{Singular analytic solutions of the main problem}\label{seccion5}

This section describes the analytic solutions of the problem in two good coverings in $\C^\star$, and construct them by analyzing the procedure followed in the two problems considered in the previous sections. A Ramis-Sibuya type theorem applied to each problem will lead to the formal solution of the main problem under study.

\begin{defin}\label{defingoodcovering} Let $j\in\{1,2\}$, and let $\varsigma_j \geq 2$ be integer numbers. For all $0 \leq p \leq \varsigma_1-1$ (resp. $0 \leq p \leq \varsigma_2-1$), we consider open sectors
$\mathcal{E}_{p}$ (resp. $\tilde{\mathcal{E}}_{p}$)  centered at $0$, with radius $\epsilon_{0}>0$ and opening larger than
$\frac{\pi}{(\chi_1 + \alpha)\kappa_1}$ (resp. $\frac{\pi}{(\chi_2 + \alpha)\kappa_2}$) such that
$\mathcal{E}_{p} \cap \mathcal{E}_{p+1} \neq \emptyset$ for all
$0 \leq p \leq \varsigma_1-1$ (resp. $\tilde{\mathcal{E}}_{p} \cap \tilde{\mathcal{E}}_{p+1} \neq \emptyset$ for all $0 \leq p \leq \varsigma_2-1$). Moreover, we assume that the intersection of any three different elements in $\{ \mathcal{E}_{p} \}_{0 \leq p \leq \varsigma_1-1}$ (resp. $\{ \tilde{\mathcal{E}}_{p} \}_{0 \leq p \leq \varsigma_2-1}$) is empty
and that $\cup_{p=0}^{\varsigma_1 - 1} \mathcal{E}_{p} = \mathcal{U} \setminus \{ 0 \}=\cup_{p=0}^{\varsigma_2 - 1} \tilde{\mathcal{E}}_{p} = \mathcal{U} \setminus \{ 0 \}$,
where $\mathcal{U}$ is some neighborhood of 0 in $\mathbb{C}$. Each set of sectors
$\{ \mathcal{E}_{p} \}_{0 \leq p \leq \varsigma_1 - 1}$ and $\{ \tilde{\mathcal{E}}_{p} \}_{0 \leq p \leq \varsigma_2 - 1}$ is called a good covering in $\mathbb{C}^{\ast}$. In order to distinguish both good coverings, we will refer each of them as the good covering related to the Gevrey order $(\chi_1+\alpha)\kappa_1$ (resp. $(\chi_2+\alpha)\kappa_2$).
\end{defin}

\begin{defin}\label{defi2414} Let $\{ \mathcal{E}_{p} \}_{0 \leq p \leq \varsigma_1 - 1}$, and $\{ \tilde{\mathcal{E}}_{p} \}_{0 \leq p \leq \varsigma_2 - 1}$ be two good coverings in $\mathbb{C}^{\ast}$, related to Gevrey orders $(\chi_1+\alpha)\kappa_1$ and $(\chi_2+\alpha)\kappa_2$, respectively. For $j\in\{1,2\}$, let $\mathcal{T}_j$ be an open bounded sector centered at 0 with radius $r_{\mathcal{T}}$ and consider a family of open sectors
$$ S_{\mathfrak{d}_{p},\theta_1,\epsilon_{0}r_{\mathcal{T}}} =
\{ T \in \mathbb{C}^{\ast} / |T| < \epsilon_{0}r_{\mathcal{T}} \ \ , \ \ |\mathfrak{d}_{p} - \mathrm{arg}(T)| < \theta/2 \} $$
(resp. 
$$ S_{\tilde{\mathfrak{d}}_{p},\theta_2,\epsilon_{0}r_{\mathcal{T}}} =
\{ T \in \mathbb{C}^{\ast} / |T| < \epsilon_{0}r_{\mathcal{T}} \ \ , \ \ |\tilde{\mathfrak{d}}_{p} - \mathrm{arg}(T)| < \theta/2 \} $$)
with aperture $\theta_1 > \pi/\kappa_1$ (resp. $\theta_2 > \pi/\kappa_2$) and where $\mathfrak{d}_{p} \in \mathbb{R}$, for all $0 \leq p \leq \varsigma_1-1$ (resp. $\tilde{\mathfrak{d}}_{p}\in\R$ for all $0\le p\le \varsigma_2-1$), are directions
which satisfy the following constraints: Let $q_{\ell}(m)$ (resp. $\tilde{q}_{\ell}(m)$) be the roots described in (\ref{e371}) (resp. (\ref{e371b})) of the polynomials $\tilde{P}_{m}(\tau)$ (resp. $\tilde{P}_{2,m}(\tau)$), and
$S_{\mathfrak{d}_p}$, $0 \leq p \leq \varsigma_1 -1$ (resp. $S_{\tilde{\mathfrak{d}}_p}$, $0 \leq p \leq \varsigma_2 -1$) be unbounded sectors centered at 0 with directions $\mathfrak{d}_{p}$ (resp. $\tilde{\mathfrak{d}}_{p}$ and with small aperture. We assume that\\
1) There exists a constant $M_{1}>0$ such that
\begin{equation}
|\tau - q_{\ell}(m)| \geq M_{1}(1 + |\tau|) \label{root_cond_1_in_defin}
\end{equation}
for all $0 \leq \ell \leq \delta_{D}\kappa_1-1$, all $m \in \mathbb{R}$, all
$\tau \in S_{\mathfrak{d}_p} \cup \bar{D}(0,\rho)$, for all
$0 \leq p \leq \varsigma_1-1$, and also 
\begin{equation}
|\tau - \tilde{q}_{\ell}(m)| \geq M_{1}(1 + |\tau|) \label{root_cond_1_in_defin2}
\end{equation}
for all $0 \leq \ell \leq \delta_{D}\kappa_2-1$, all $m \in \mathbb{R}$, all
$\tau \in S_{\tilde{\mathfrak{d}}_p} \cup \bar{D}(0,\rho)$, for all
$0 \leq p \leq \varsigma_2-1$.\\
2) There exists a constant $M_{2}>0$ such that
\begin{equation}
|\tau - q_{\ell_0}(m)| \geq M_{2}|q_{\ell_0}(m)| \label{root_cond_2_in_defin}
\end{equation}
for some $\ell_{0} \in \{0,\ldots,\delta_{D}\kappa_1-1 \}$, all $m \in \mathbb{R}$, all $\tau \in S_{\mathfrak{d}_p} \cup \bar{D}(0,\rho)$, for
all $0 \leq p \leq \varsigma_1 - 1$, and also
\begin{equation}
|\tau - \tilde{q}_{\ell_1}(m)| \geq M_{2}|\tilde{q}_{\ell_1}(m)| \label{root_cond_2_in_defin2}
\end{equation}
for some $\ell_{1} \in \{0,\ldots,\delta_{D}\kappa_2-1 \}$, all $m \in \mathbb{R}$, all $\tau \in S_{\tilde{\mathfrak{d}}_p} \cup \bar{D}(0,\rho)$, for all $0 \leq p \leq \varsigma_2 - 1$.\\

3) For all $0 \leq p \leq \varsigma_1 - 1$, for all $t \in \mathcal{T}_1$, all $\epsilon \in \mathcal{E}_{p}$,
we have that
$\epsilon^{\alpha + \chi_1} t \in S_{\tilde{\mathfrak{d}}_{p},\theta_1,\epsilon_{0}^{\alpha + \chi_1}r_{\mathcal{T}}}$, and for all $0 \leq p \leq \varsigma_2 - 1$, for all $t \in \mathcal{T}_2$, all $\epsilon \in \tilde{\mathcal{E}}_{p}$,
we have that $\epsilon^{\alpha + \chi_2} t \in S_{\tilde{\mathfrak{d}}_{p},\theta_2,\epsilon_{0}^{\alpha + \chi_2}r_{\mathcal{T}}}$.\medskip

\noindent Then, we say that both families,
$\{ (S_{\mathfrak{d}_{p},\theta_1,\epsilon_{0}r_{\mathcal{T}}})_{0 \leq p \leq \varsigma_1-1},\mathcal{T}_1 \}$ and $\{ (S_{\tilde{\mathfrak{d}}_{p},\theta_2,\epsilon_{0}r_{\mathcal{T}}})_{0 \leq p \leq \varsigma_2-1},\mathcal{T}_2 \}$ 
are associated to the good covering $\{ \mathcal{E}_{p} \}_{0 \leq p \leq \varsigma_1 - 1}$, and $\{ \tilde{\mathcal{E}}_{p} \}_{0 \leq p \leq \varsigma_2 - 1}$, respectively.
\end{defin}

We construct two families of holomorphic solutions of the main problem under study (\ref{e85}), with a pole at $(\epsilon,t)=(0,0)$, defined in the sectors $\mathcal{E}_p$ and $\tilde{\mathcal{E}}_q$, for $0\le p\le \varsigma_1-1$ and $0\le q\le \varsigma_2-1$, with respect to $\epsilon$. We also determine the exponential rate of decrement of the difference of two solutions in the intersection of two consecutive sectors of the same family of good coverings. Moreover, this rate depends on the good covering under consideration.

\begin{theo}\label{teo2461}
Let us consider the parameters described at the beginning of Section~\ref{seccion4}, which satisfy (\ref{e92}), (\ref{e93}), (\ref{e94}) and (\ref{e94b}). We consider the nonlinear singularly perturbed PDE (\ref{e85}) with elements determined as in (\ref{e91}), (\ref{e540}) and (\ref{e95}). We choose $\alpha,\beta\in\Q$, $1 \le\kappa_1<\kappa_2$ and $\gamma_1,\gamma_2\in\Q$ such that (\ref{e134}), (\ref{e138}), (\ref{e218}), (\ref{e222}), (\ref{e218b}), (\ref{e222b}) hold, and assume that (\ref{e287}) and (\ref{e287b}) hold. We finally assume (\ref{e366}), (\ref{e887}), (\ref{e896}), (\ref{e366b}), (\ref{e887b}), (\ref{e896b}). 

Let $(\mathcal{E}_{p})_{0\le p\le \varsigma_1-1}$ and $(\tilde{\mathcal{E}}_{p})_{0\le p\le \varsigma_2-1}$ be two good coverings associated to the Gevrey orders $(\chi_1+\alpha)\kappa_1$ and $(\chi_2+\alpha)\kappa_2$, respectively, for which families of open sectors $\{ (S_{\mathfrak{d}_{p},\theta_1,\epsilon_{0}r_{\mathcal{T}}})_{0 \leq p \leq \varsigma_1-1},\mathcal{T}_1 \}$ and $\{ (S_{\tilde{\mathfrak{d}}_{p},\theta_2,\epsilon_{0}r_{\mathcal{T}}})_{0 \leq p \leq \varsigma_2-1},\mathcal{T}_2 \}$ are associated
to each corresponding good covering.

Then, there exist a radius $r_{\tilde{Q},\tilde{R}_{D}}>0$ large enough, $\epsilon_{0}>0$ small enough, for which two families $\{ u_1^{\mathfrak{d}_{p}}(t,z,\epsilon) \}_{0 \leq p \leq \varsigma_1 - 1}$ and $\{ u_2^{\tilde{\mathfrak{d}}_{p}}(t,z,\epsilon) \}_{0 \leq p \leq \varsigma_2 - 1}$ of actual solutions of (\ref{e85}) can be constructed. More precisely, the functions
$u_1^{\mathfrak{d}_{p}}(t,z,\epsilon)$ and $u_2^{\tilde{\mathfrak{d}}_{p}}(t,z,\epsilon)$  solve two singularly perturbed PDEs
\begin{multline}
\tilde{Q}(\partial_z)\left(p_1(t,\epsilon)u(t,z,\epsilon)+p_2(t,\epsilon)u^2(t,z,\epsilon)+p_3(t,\epsilon)u^3(t,z,\epsilon)\right)\\
=\sum_{j=0}^{Q}\tilde{b}_{j}(z)\epsilon^{n_j}t^{b_j}+\sum_{\ell=1}^{D}\epsilon^{\Delta_{\ell}}t^{d_{\ell}}\partial_{t}^{\delta_{\ell}}\tilde{R}_{\ell}(\partial_{z})u(t,z,\epsilon)+F_1(\epsilon^{\alpha}t,\epsilon)
\label{e85b}
\end{multline}
and
\begin{multline}
\tilde{Q}(\partial_z)\left(p_1(t,\epsilon)u(t,z,\epsilon)+p_2(t,\epsilon)u^2(t,z,\epsilon)+p_3(t,\epsilon)u^3(t,z,\epsilon)\right)\\
=\sum_{j=0}^{Q}\tilde{b}_{j}(z)\epsilon^{n_j}t^{b_j}+\sum_{\ell=1}^{D}\epsilon^{\Delta_{\ell}}t^{d_{\ell}}\partial_{t}^{\delta_{\ell}}\tilde{R}_{\ell}(\partial_{z})u(t,z,\epsilon)+F_2(\epsilon^{\alpha}t,\epsilon)
\label{e85c}
\end{multline}
respectively, where the forcing terms $F_1(T,\epsilon)$ and $F_2(T,\epsilon)$ are described in (\ref{e2629}) and (\ref{e235bd}), respectively, and define holomorphic bounded function provided that the additional constraints (\ref{e2712}) and (\ref{e2712b}) are fulfilled.

Each function $u_1^{\mathfrak{d}_{p}}(t,z,\epsilon)$ can be decomposed as
\begin{equation}
u_1^{\mathfrak{d}_{p}}(t,z,\epsilon) =
\epsilon^{\beta} \left( -\frac{a_{0,1}}{a_{0,2}}(\epsilon^{\alpha}t)^{k_{0,1} - k_{0,2}}-\frac{a_{0,1}}{a_{0,2}}(\epsilon^{\alpha}t)^{k_{0,1} - k_{0,2}}\mathcal{J}_1(\epsilon^{\alpha}t) +
(\epsilon^{\alpha}t)^{\gamma_1}v_1^{\mathfrak{d}_{p}}(t,z,\epsilon) \right) \label{decomp_singular_holom_udp1}
\end{equation}
where $\mathcal{J}_1(T)$ is holomorphic on some disc $D(0,d_{\mathcal{J}_1})$, $d_{\mathcal{J}_1}>0$ and
$v_1^{\mathfrak{d}_{p}}(t,z,\epsilon)$ defines a bounded holomorphic function on $\mathcal{T}_1 \times H_{\beta'} \times \mathcal{E}_{p}$ for any given $0 < \beta' < \beta$, with $v_1^{\mathfrak{d}_{p}}(0,z,\epsilon) \equiv 0$ on $H_{\beta'} \times \mathcal{E}_{p}$. Furthermore, there exist
constants $K_{p},M_{p}>0$ and $\sigma>0$ (independent of $\epsilon$) such that
\begin{equation}
\sup_{t \in \mathcal{T}_1 \cap D(0,\sigma), z \in H_{\beta'}}
|v_1^{\mathfrak{d}_{p+1}}(t,z,\epsilon) - v_1^{\mathfrak{d}_{p}}(t,z,\epsilon)| \leq K_{p}
\exp( -\frac{M_p}{|\epsilon|^{(\chi_1 + \alpha)\kappa_1}} ) \label{exp_small_difference_v_dp1}
\end{equation}
for all $\epsilon \in \mathcal{E}_{p+1} \cap \mathcal{E}_{p}$, for all $0 \leq p \leq \varsigma_1-1$.

Each function $u_2^{\tilde{\mathfrak{d}}_{p}}(t,z,\epsilon)$ can be decomposed as
\begin{equation}
u_2^{\tilde{\mathfrak{d}}_{p}}(t,z,\epsilon) =
\epsilon^{\beta} \left( -\frac{a_{0,2}}{a_{0,3}}(\epsilon^{\alpha}t)^{k_{0,2} - k_{0,3}}-\frac{a_{0,2}}{a_{0,3}}(\epsilon^{\alpha}t)^{k_{0,2} - k_{0,3}}\mathcal{J}_2(\epsilon^{\alpha}t) +
(\epsilon^{\alpha}t)^{\gamma_2}v_2^{\tilde{\mathfrak{d}}_{p}}(t,z,\epsilon) \right) \label{decomp_singular_holom_udp2}
\end{equation}
where $\mathcal{J}_2(T)$ is holomorphic on some disc $D(0,d_{\mathcal{J}_2})$, $d_{\mathcal{J}_2}>0$ and
$v_2^{\tilde{\mathfrak{d}}_{p}}(t,z,\epsilon)$ defines a bounded holomorphic function on $\mathcal{T}_2 \times H_{\beta'} \times \tilde{\mathcal{E}}_{p}$ for any given $0 < \beta' < \beta$, with $v_2^{\mathfrak{d}_{p}}(0,z,\epsilon) \equiv 0$ on $H_{\beta'} \times \tilde{\mathcal{E}}_{p}$. Furthermore, there exist
constants $K_{p},M_{p}>0$ and $\sigma>0$ (independent of $\epsilon$) such that
\begin{equation}
\sup_{t \in \mathcal{T}_2 \cap D(0,\sigma), z \in H_{\beta'}}
|v_2^{\tilde{\mathfrak{d}}_{p+1}}(t,z,\epsilon) - v_2^{\tilde{\mathfrak{d}}_{p}}(t,z,\epsilon)| \leq K_{p}
\exp( -\frac{M_p}{|\epsilon|^{(\chi_2 + \alpha)\kappa_2}} ) \label{exp_small_difference_v_dp2}
\end{equation}
for all $\epsilon \in \tilde{\mathcal{E}}_{p+1} \cap \tilde{\mathcal{E}}_{p}$, for all $0 \leq p \leq \varsigma_2-1$.
\end{theo}

\section{Doubly parametric Gevrey asymptotic expansions of the solutions}

In this section, we first recall some classical facts on $k-$Borel summability of formal series with coefficients in a Banach space as
introduced in \cite{ba}, and a cohomological criterion for $k$-summability of formal power series with coefficients in Banach spaces (see
\cite{ba2}, p. 121 or \cite{hssi}, Lemma XI-2-6) which is known as the Ramis-Sibuya theorem in the literature.

Afterwards, we provide the second main result in the present work, in which we obtain the existence of two formal power series, written in power series of the perturbation parameter, and with coefficients in some Banach space, which turn out to be the common Gevrey asymptotic expansions of the functions $v_1^{\mathfrak{d}_p}(t,z,\epsilon)$ for all $0\le p\le \varsigma_1-1$, and $v_2^{\tilde{\mathfrak{d}}_p}(t,z,\epsilon)$ for all $0\le q\le \varsigma_2-1$, in (\ref{decomp_singular_holom_udp1}) and (\ref{decomp_singular_holom_udp2}), of Gevrey orders $((\chi_1+\alpha)\kappa_1)^{-1}$ and $((\chi_2+\alpha)\kappa_2)^{-1}$ respectively. We recall that both functions determine solutions of the main problem (\ref{e85}), as determined in Theorem~\ref{teo2461}.

\subsection{$k-$Summable formal series and Ramis-Sibuya Theorem}

\begin{defin} Let $k \geq 1$ be an integer. A formal series $\hat{X}(\epsilon) = \sum_{j=0}^{\infty}  a_{j} \epsilon^{j} \in \mathbb{F}[[\epsilon]],$ with coefficients in a Banach space $( \mathbb{F}, ||.||_{\mathbb{F}} )$ is said to be $k-$summable
with respect to $\epsilon$ in the direction $d \in \mathbb{R}$ if \medskip

{\bf i)} there exists $\rho \in \mathbb{R}_{+}$ such that the following formal series, called formal
Borel transform of $\hat{X}$ of order $k$ 
$$ \mathcal{B}_{k}(\hat{X})(\tau) = \sum_{j=0}^{\infty} \frac{ a_{j} \tau^{j}  }{ \Gamma(1 + \frac{j}{k}) } \in \mathbb{F}[[\tau]],$$
is absolutely convergent for $|\tau| < \rho$, \medskip

{\bf ii)} A positive number $\delta $ exists such that the series $\mathcal{B}_{k}(\hat{X})(\tau)$ can be analytically continued with respect to $\tau$ in a sector $S_{d,\delta} = \{ \tau \in \mathbb{C}^{\ast} : |d - \mathrm{arg}(\tau) | < \delta \} $. Moreover, there exist $C >0$, and $K >0$
such that $ ||\mathcal{B}(\hat{X})(\tau)||_{\mathbb{F}}
\leq C e^{ K|\tau|^{k}} $ for all $\tau \in S_{d, \delta}$.
\end{defin}
If the definition above is fulfilled, the vector valued Laplace transform of order $k$ of
$\mathcal{B}_{k}(\hat{X})(\tau)$ in the direction $d$ is set as
$$ \mathcal{L}^{d}_{k}(\mathcal{B}_{k}(\hat{X}))(\epsilon) = \epsilon^{-k} \int_{L_{\gamma}}
\mathcal{B}_{k}(\hat{X})(u) e^{ - ( u/\epsilon )^{k} } ku^{k-1}du,$$
along a half-line $L_{\gamma} = \mathbb{R}_{+}e^{i\gamma} \subset S_{d,\delta} \cup \{ 0 \}$, where $\gamma$ depends on
$\epsilon$ and is chosen in such a way that
$\cos(k(\gamma - \mathrm{arg}(\epsilon))) \geq \delta_{1} > 0$, for some fixed $\delta_{1}$, for all
$\epsilon$ in a sector
$$ S_{d,\theta,R^{1/k}} = \{ \epsilon \in \mathbb{C}^{\ast} : |\epsilon| < R^{1/k} \ \ , \ \ |d - \mathrm{arg}(\epsilon) |
< \theta/2 \},$$
where $\frac{\pi}{k} < \theta < \frac{\pi}{k} + 2\delta$ and $0 < R < \delta_{1}/K$. The
function $\mathcal{L}^{d}_{k}(\mathcal{B}_{k}(\hat{X}))(\epsilon)$
is called the $k-$sum of the formal series $\hat{X}(t)$ in the direction $d$. It is bounded and holomorphic on the sector
$S_{d,\theta,R^{1/k}}$ and has the formal series $\hat{X}(\epsilon)$ as Gevrey asymptotic expansion of order $1/k$ with respect to $\epsilon$ on $S_{d,\theta,R^{1/k}}$. In addition to that, it is unique under such property. More precisely, one has that for all $\frac{\pi}{k} < \theta_{1} < \theta$, there exist $C,M > 0$ such that
$$ ||\mathcal{L}^{d}_{k}(\mathcal{B}_{k}(\hat{X}))(\epsilon) - \sum_{p=0}^{n-1}
a_p \epsilon^{p}||_{\mathbb{F}} \leq CM^{n}\Gamma(1+ \frac{n}{k})|\epsilon|^{n} $$
for all $n \geq 1$, all $\epsilon \in S_{d,\theta_{1},R^{1/k}}$.\medskip

\noindent {\bf Theorem (RS)} {\it Let $(\mathbb{F},||.||_{\mathbb{F}})$ be a Banach space over $\mathbb{C}$ and
$\{ \mathcal{E}_{p} \}_{0 \leq p \leq \varsigma-1}$ be a good covering in $\mathbb{C}^{\ast}$. For all
$0 \leq p \leq \varsigma - 1$, let $G_{p}$ be a holomorphic function from $\mathcal{E}_{p}$ into
the Banach space $(\mathbb{F},||.||_{\mathbb{F}})$ and let the cocycle $\Theta_{p}(\epsilon) = G_{p+1}(\epsilon) - G_{p}(\epsilon)$
be a holomorphic function from the sector $Z_{p} = \mathcal{E}_{p+1} \cap \mathcal{E}_{p}$ into $\mathbb{E}$
(with the convention that $\mathcal{E}_{\varsigma} = \mathcal{E}_{0}$ and $G_{\varsigma} = G_{0}$).
We make the following assumptions.\medskip

\noindent {\bf 1)} The functions $G_{p}(\epsilon)$ are bounded as $\epsilon \in \mathcal{E}_{p}$ tends to the origin
in $\mathbb{C}$, for all $0 \leq p \leq \varsigma - 1$.\medskip

\noindent {\bf 2)} The functions $\Theta_{p}(\epsilon)$ are exponentially flat of order $k$ on $Z_{p}$, for all
$0 \leq p \leq \varsigma-1$. This means that there exist constants $C_{p},A_{p}>0$ such that
$$ ||\Theta_{p}(\epsilon)||_{\mathbb{F}} \leq C_{p}e^{-A_{p}/|\epsilon|^{k}} $$
for all $\epsilon \in Z_{p}$, all $0 \leq p \leq \varsigma-1$.\medskip

Then, for all $0 \leq p \leq \varsigma - 1$, the functions $G_{p}(\epsilon)$ are the $k-$sums on $\mathcal{E}_{p}$ of a
common $k-$summable formal series $\hat{G}(\epsilon) \in \mathbb{F}[[\epsilon]]$.}

\subsection{Parametric double Gevrey asymptotic expansions of the solutions and construction of their associated sum}

In this subsection, we denote $\mathbb{F}_j$ the Banach space of holomorphic and bounded functions on $(\mathcal{T}_j \cap D(0,\sigma)) \times H_{\beta'}$ equipped with supremum norm, where $\sigma>0$ is 
defined in Theorem~\ref{teo2461}, and $0<\beta'<\beta$ is a fixed real number. We preserve the choice of $\mathcal{T}_j$, for $j=1,2$ in Definition~\ref{defi2414}.

The second main result in this work is the following.

\begin{theo}\label{teo3045}
Under the hypotheses of Theorem~\ref{teo2461}, there exist two formal power series 
$$ \hat{v}_1(t,z,\epsilon) = \sum_{m \geq 0} v_{m,1}(t,z) \epsilon^{m}\in\mathbb{F}_1[[\epsilon]],\quad  \hat{v}_2(t,z,\epsilon) = \sum_{m \geq 0} v_{m,2}(t,z) \epsilon^{m}\in\mathbb{F}_2[[\epsilon]],$$
such that the functions $v_1^{\mathfrak{d}_{p}}(t,z,\epsilon)$ (resp. $v_2^{\tilde{\mathfrak{d}}_{p}}(t,z,\epsilon)$) in the decomposition
(\ref{decomp_singular_holom_udp1}) (resp. (\ref{decomp_singular_holom_udp2})) are its $(\chi_1+\alpha)\kappa_1-$sums on the sectors $\mathcal{E}_{p}$, for
all $0 \leq p \leq \varsigma_1-1$, viewed as holomorphic functions from $\mathcal{E}_{p}$ into $\mathbb{F}_1$ (resp. its $(\chi_2+\alpha)\kappa_2-$sums on the sectors $\mathcal{E}_{p}$, for
all $0 \leq p \leq \varsigma_2-1$, viewed as holomorphic functions from $\mathcal{E}_{p}$ into $\mathbb{F}_2$). In other words, there exist two constants $C,M>0$ such that
\begin{equation}\label{e3057}
\sup_{t \in \mathcal{T}_1 \cap D(0,\sigma), z \in H_{\beta'}}
|v_1^{\mathfrak{d}_{p}}(t,z,\epsilon) - \sum_{m=0}^{n-1} v_{m,1}(t,z) \epsilon^m| \leq
CM^{n}\Gamma(1+\frac{n}{(\chi_1 + \alpha)\kappa_1})|\epsilon|^{n}
\end{equation}
for all $n \geq 1$, all $0\le p\le \varsigma_1-1$, and all $\epsilon \in \mathcal{E}_{p}$, and 
\begin{equation}\label{e3058}
\sup_{t \in \mathcal{T}_2 \cap D(0,\sigma), z \in H_{\beta'}}
|v_2^{\tilde{\mathfrak{d}}_{p}}(t,z,\epsilon) - \sum_{m=0}^{n-1} v_{m,2}(t,z) \epsilon^m| \leq
CM^{n}\Gamma(1+\frac{n}{(\chi_2 + \alpha)\kappa_2})|\epsilon|^{n}
\end{equation}
for all $n \geq 1$, all $0\le p\le \varsigma_2-1$, and all $\epsilon \in \tilde{\mathcal{E}}_{p}$.
\end{theo}
\begin{proof} We give the proof for the first family of functions, whereas the proof is analogous for the second family. Let $v_1^{\mathfrak{d}_{p}}(t,z,\epsilon)$, $0 \leq p \leq \varsigma_1-1$ be the functions constructed in Theorem~\ref{teo2461}. For all $0 \leq p \leq \varsigma_1-1$, we define $G_{p}(\epsilon) := (t,z) \mapsto
v_1^{\mathfrak{d}_{p}}(t,z,\epsilon)$, which is by construction a
holomorphic and bounded function from $\mathcal{E}_{p}$ into $\mathbb{F}_1$. 
In view of the estimates (\ref{exp_small_difference_v_dp1}), the cocycle
$\Theta_{p}(\epsilon) = G_{p+1}(\epsilon) - G_{p}(\epsilon)$ is exponentially flat of order
$(\chi_1+\alpha)\kappa_1$ on
$Z_{p} = \mathcal{E}_{p} \cap \mathcal{E}_{p+1}$, for any $0 \leq p \leq \varsigma_1-1$. Therefore, Theorem (RS) guarantees the existence of a formal power series
$$ \hat{G}_1(\epsilon) = \sum_{m \geq 0} v_{m,1}(t,z) \epsilon^{m} =: \hat{v}_1(t,z,\epsilon)
\in \mathbb{F}_1[[\epsilon]]$$
such that the functions $G_{p}(\epsilon)$ are the $(\chi_1 + \alpha)\kappa_1-$sums on $\mathcal{E}_{p}$ of $\hat{G}_1(\epsilon)$ as
$\mathbb{F}_1-$valued functions, for all
$0 \leq p \leq \varsigma_1-1$, in $\mathcal{E}_p$. 
\end{proof}

\textbf{Remark:} It is worth mentioning that the formal power series in $\epsilon$
$$\hat{v}_1(T,z,\epsilon):=(\epsilon^{\alpha}T)^{-\gamma_1}\hat{u}_1(T,z,\epsilon)+\frac{a_{0,1}}{a_{0,2}}(\epsilon^{\alpha}T)^{k_{0,1}-k_{0,2}-\gamma_1}-\frac{a_{0,1}}{a_{0,2}}(\epsilon^{\alpha}T)^{k_{0,1}-k_{0,2}-\gamma_1}\mathcal{J}_1(\epsilon^{\alpha}T)$$
and
$$\hat{v}_2(T,z,\epsilon):=(\epsilon^{\alpha}T)^{-\gamma_1}\hat{u}_1(T,z,\epsilon)+\frac{a_{0,2}}{a_{0,3}}(\epsilon^{\alpha}T)^{k_{0,2}-k_{0,3}-\gamma_2}-\frac{a_{0,2}}{a_{0,3}}(\epsilon^{\alpha}T)^{k_{0,2}-k_{0,3}-\gamma_2}\mathcal{J}_2(\epsilon^{\alpha}T)$$
are formal solutions of (\ref{e2627}) and (\ref{e235bc}), respectively. This stament follows from (\ref{e3057}) and (\ref{e3058}). Indeed, one has
$$\frac{1}{j!}\partial_\epsilon^j(v_1^{\mathfrak{d}_p}(t,z,\epsilon))|_{\epsilon=0}=v_{j,1}(t,z),$$
for all $j\ge 0$ and $0\le p\le \varsigma_1-1$, and also
$$\frac{1}{j!}\partial_\epsilon^j(v_2^{\tilde{\mathfrak{d}}_p}(t,z,\epsilon))|_{\epsilon=0}=v_{j,2}(t,z),$$
for all $j\ge 0$ and $0\le p\le \varsigma_2-1$.

\textbf{Remark:} Concerning the Gevrey orders appearing in the asymptotics, we observe that 
$$(\chi_1+\alpha)\kappa_1\le (\chi_2+\alpha)\kappa_2.$$ 
Indeed, from the definition of $\chi_1$ and $\chi_2$ in (\ref{e256}) and (\ref{e256}) respectively, one has
$$\kappa_1(\chi_1+\alpha)=\kappa_1\frac{\Delta_D+\beta-\alpha k_{0,1}}{d_{D}-\delta_D-k_{0,1}}=\frac{\Delta_D+\beta-\alpha k_{0,1}}{\delta_D},$$
and 
$$\kappa_2(\chi_2+\alpha)=\frac{\Delta_D+\beta-\alpha (2k_{0,2}-k_{0,3})}{\delta_D}.$$ The inequality follows from (\ref{e92}).

We conclude the work with an example. 

\textbf{Example:} \label{example}
We consider the equation
\begin{multline}
Q(\partial_z)((a_{0,1}\epsilon^5t^2+a_{0,2}\epsilon^6t^3)u(t,z,\epsilon)+\epsilon^{14}t^6u^2(t,z,\epsilon)+\epsilon^{11}t^{14}u^3(t,z,\epsilon))\\
=b_{0}(z)\epsilon^{3}t+\epsilon^{10}t^5\partial_tR_1(\partial_z)u(t,z,\epsilon)+ \epsilon^{6}t^6\partial_tR_2(\partial_z)u(t,z,\epsilon).
\end{multline}
Here, $k_{0,1}=2$, $k_{0,2}=6$, $k_{0,3}=14$, $\kappa_1=1$, $\kappa_2=3$, $m_{0,1}=5$, $m_{1,1}=6$, $m_{0,2}=14$, $m_{0,3}=11$, $\alpha=2$, $\beta=-1$, $\Delta_{1}=10$, $\Delta_2=12$, $d_{1}=5$, $\delta_1=1$, $d_{2}=6$, $\delta_2=2$, $b_0=1$, $n_0=3$, $\gamma_1=-2$, $\gamma_2=1$.

The constraints (\ref{e92}), (\ref{e94}), (\ref{e94b}), (\ref{e134}), (\ref{e138}), (\ref{e218}), (\ref{e222}), (\ref{e887}), (\ref{e903}), (\ref{e218b}), (\ref{e222b}), (\ref{e887b}), (\ref{e896b}) are satisfied in the example.

Observe that one can divide every term appearing in the previous equation by $\epsilon^3t$, but still one observes the presence of an irregular singularity at $t=0$ and the appearance of singular operators which are treated in the manner we describe in the present work.

The next sections are included for the sake of completeness, and describe in detail the proofs of the results provided throughout the work. We decided to leave it at the end for a more comprehensive reading.

\section{Proof of Proposition~\ref{prop1} and Proposition~\ref{prop3}}

\textit{Proof of Proposition~\ref{prop1}:}
\begin{proof}
Let us denote
$$A:=\left\|\epsilon^{-\gamma_0}a_{\gamma_1,\kappa}(\tau,m)\tilde{R}(im)\tau^{\kappa}\int_{0}^{\tau^\kappa}(\tau^\kappa-s)^{\gamma_2}s^{\gamma_3}f(s^{1/\kappa},m)ds\right\|_{(\nu,\beta,\mu,\chi,\alpha,\kappa,\epsilon)}.$$
It holds that
\begin{multline}
A=\sup_{\tau\bar{D}(0,\rho)\cup S_d,m\in\R}(1+|m|)^{\mu}\exp(\beta|m|)\frac{1+\left|\frac{\tau}{\epsilon^{\chi+\alpha}}\right|^{2\kappa}}{\left|\frac{\tau}{\epsilon^{\chi+\alpha}}\right|}\exp\left(-\nu\left|\frac{\tau}{\epsilon^{\chi+\alpha}}\right|^{\kappa}\right)|\epsilon|^{-\gamma_0}\frac{1}{(1+|\tau|^{\kappa})^{\gamma_1}}\frac{|\tilde{R}(im)|}{|\tilde{R}_D(im)|}\\
\times\left|\tau^{\kappa}\int_{0}^{\tau^\kappa}\left((1+|m|)^{\mu}\exp(\beta|m|)\frac{1+\frac{|s|^2}{|\epsilon|^{(\chi+\alpha)2\kappa}}}{\left|\frac{\tau}{\epsilon^{\chi+\alpha}}\right|}\exp\left(-\nu\frac{|s|}{|\epsilon|^{(\chi+\alpha)\kappa}}\right)f(s^{1/\kappa},m)\right)\mathcal{A}(\tau,s,m,\epsilon)ds\right|,
\end{multline}
where
$$\mathcal{A}(\tau,s,m,\epsilon)=\frac{1}{(1+|m|)^{\mu}\exp(\beta|m|)}\frac{\exp\left(\nu\frac{|s|}{|\epsilon|^{(\chi+\alpha)\kappa}}\right)}{1+\frac{|s|^2}{|\epsilon|^{(\chi+\alpha)2\kappa}}}\frac{|s|^{1/\kappa}}{|\epsilon|^{\chi+\alpha}}(\tau^{\kappa}-s)^{\gamma_2}s^{\gamma_3}.$$

This last expression yields
$$A\le C_{3.1}(\epsilon)\sup_{m\in\R}\left|\frac{\tilde{R}(im)}{\tilde{R}_{D}(im)}\right|\left\|f(\tau,m)\right\|_{(\nu,\beta,\mu,\chi,\alpha,\kappa,\epsilon)},$$
where
\begin{multline}
C_{3.1}(\epsilon)=\sup_{\tau\bar{D}(0,\rho)\cup S_d,m\in\R}\frac{1+\left|\frac{\tau}{\epsilon^{\chi+\alpha}}\right|^{2\kappa}}{\left|\frac{\tau}{\epsilon^{\chi+\alpha}}\right|}\exp\left(-\nu\left|\frac{\tau}{\epsilon^{\chi+\alpha}}\right|^{\kappa}\right)|\epsilon|^{-\gamma_0}\frac{1}{(1+|\tau|^{\kappa})^{\gamma_1}}\\
\times|\tau|^{\kappa}\int_0^{|\tau|^\kappa}\frac{\exp\left(\nu\frac{h}{|\epsilon|^{(\chi+\alpha)\kappa}}\right)}{1+\frac{h^2}{|\epsilon|^{(\chi+\alpha)2\kappa}}}\frac{h^{1/\kappa}}{|\epsilon|^{\chi+\alpha}}(|\tau|^{\kappa}-h)^{\gamma_2}h^{\gamma_3}dh.
\end{multline}
After the change of variable $h=|\epsilon|^{(\chi+\alpha)\kappa}h'$ in the integral of $C_{3.1}(\epsilon)$ and usual estimates one arrives at
\begin{multline}
C_{3.1}(\epsilon)=\sup_{\tau\bar{D}(0,\rho)\cup S_d,m\in\R}\frac{1+\left|\frac{\tau}{\epsilon^{\chi+\alpha}}\right|^{2\kappa}}{\left|\frac{\tau}{\epsilon^{\chi+\alpha}}\right|}\exp\left(-\nu\left|\frac{\tau}{\epsilon^{\chi+\alpha}}\right|^{\kappa}\right)|\epsilon|^{-\gamma_0}\frac{1}{(1+|\tau|^{\kappa})^{\gamma_1}}\\
\times|\tau|^{\kappa}\int_0^{\frac{|\tau|^\kappa}{|\epsilon|^{(\chi+\alpha)\kappa}}}\frac{\exp\left(\nu h'\right)}{1+(h')^2}(h')^{1/\kappa}\left(\frac{|\tau|^{\kappa}}{|\epsilon|^{(\chi+\alpha)\kappa}}-h'\right)^{\gamma_2}(h')^{\gamma_3}dh'|\epsilon|^{(\chi+\alpha)\kappa(\gamma_2+\gamma_3+1)}\\
\le |\epsilon|^{(\chi+\alpha)\kappa(\gamma_2+\gamma_3+1)-\gamma_0+(\chi+\alpha)\kappa}\sup_{x\ge0}\frac{1+x^2}{x^{1/\kappa}}e^{-\nu x}\frac{x}{(1+|\epsilon|^{(\chi+\alpha)\kappa}x)^{\gamma_1}}(G_1(x)+G_2(x)),
\end{multline}
with 
$$G_1(x)=\int_0^{x/2}\frac{e^{\nu h'}}{1+(h')^2}(h')^{\frac{1}{\kappa}+\gamma_3}(x-h')^{\gamma_2}dh',\quad G_2(x)=\int_{x/2}^{x}\frac{e^{\nu h'}}{1+(h')^2}(h')^{\frac{1}{\kappa}+\gamma_3}(x-h')^{\gamma_2}dh'.$$

The steps on stating upper bounds for $G_1$ and $G_2$ are described in Proposition 1,\cite{lama2}, in detail. For the sake of completeness, we give the detailed proof.

Estimates for $G_1(x)$.

We first consider the case in which $1<\gamma_2<0$. Then, it holds that $(x-h')^{\gamma_2}\le(x/2)^{\gamma_2}$ for $0\le h'\le x/2$, and every $x>0$. The first condition in (\ref{e143}) yields
$$G_1(x)\le \left(\frac{x}{2}\right)^{\gamma_2}e^{\nu x/2}\int_{0}^{x/2}(h')^{\frac{1}{\kappa}+\gamma_3}dh'=\left(\frac{x}{2}\right)^{\gamma_2}e^{\nu x/2}\frac{(x/2)^{\frac{1}{\kappa}+\gamma_3+1}}{\frac{1}{\kappa}+\gamma_3+1},\quad x\ge0.$$
Then, one has
\begin{equation}\label{e198}
\sup_{x\ge0}\frac{1+x^2}{x^{1/\kappa}}e^{-\nu x}\frac{x}{(1 + |\epsilon|^{(\chi + \alpha)\kappa}x)^{\gamma_1}} G_{1}(x)
\leq \sup_{x \geq 0} \frac{1 + x^{2}}{x^{1/\kappa}} e^{-\nu x} x G_{1}(x),
\end{equation}
which is finite.

In the case that $\gamma_2>0$, then we have that $(x-h')^{\gamma_2} \leq x^{\gamma_2}$ for all $0 \leq h' \leq x/2$, for $x>0$. Therefore we get
$$ G_{1}(x) \leq x^{\gamma_2} e^{\nu x/2} \int_{0}^{x/2} (h')^{\frac{1}{\kappa} + \gamma_{3}}
dh' = x^{\gamma_2} e^{\nu x/2}
\frac{ (x/2)^{\frac{1}{\kappa} + \gamma_{3} + 1}}{\frac{1}{\kappa} + \gamma_{3} + 1}  $$
for all $x \geq 0$. We conclude
$$\sup_{x \geq 0} \frac{1 + x^{2}}{x^{1/\kappa}} e^{-\nu x}
\frac{x}{(1 + |\epsilon|^{(\chi + \alpha)\kappa}x)^{\gamma_1}} G_{1}(x)
\leq \sup_{x \geq 0} \frac{1 + x^{2}}{x^{1/\kappa}} e^{-\nu x} x G_{1}(x),$$ 
which is finite.

We study $G_2(x)$.

One has $1 + (h')^2 \geq 1 + (x/2)^{2}$ for all $x/2 \leq h' \leq x$. Hence,
\begin{equation}
G_{2}(x) \leq \frac{1}{1 + (\frac{x}{2})^2}
\int_{x/2}^{x} e^{\nu h'} (h')^{\frac{1}{\kappa} + \gamma_{3}}(x-h')^{\gamma_2} dh' \leq
\frac{1}{1 + (\frac{x}{2})^2}G_{2.1}(x) \label{G2<=G21}
\end{equation}
where
$$ G_{2.1}(x) = \int_{0}^{x} e^{\nu h'} (h')^{\frac{1}{\kappa} + \gamma_{3}}(x-h')^{\gamma_2} dh' $$
for all $x \geq 0$.

The same estimates as in (18) in~\cite{lama2} on Mittag-Leffler function lead to
\begin{equation}\label{e228}
G_{2.1}(x)\le K_{2.1}x^{\frac{1}{\kappa}+\gamma_3}e^{\nu x},\quad x\ge 1.
\end{equation}
We now distinguish two cases: $\gamma_3\le -1$ and $\gamma_3>-1$.

In the first situation, we get
$$\sup_{x \geq 1} \frac{1 + x^{2}}{x^{1/\kappa}} e^{-\nu x}
\frac{x}{(1 + |\epsilon|^{(\chi + \alpha)\kappa}x)^{\gamma_1}} G_{2}(x) \leq
\sup_{x \geq 1} \frac{1 + x^{2}}{1 + (x/2)^2}K_{2.1}x^{1 + \gamma_{3}},$$
and by the change of variable $h'=xu'$ we deduce 
\begin{multline}\label{e238}
\sup_{0 \leq x <1} \frac{1 + x^{2}}{x^{1/\kappa}} e^{-\nu x}
\frac{x}{(1 + |\epsilon|^{(\chi + \alpha)\kappa}x)^{\gamma_1}} G_{2}(x)
\leq \sup_{0 \leq x < 1} \frac{1 + x^2}{1 + (x/2)^2} e^{-\nu x}
\frac{x}{x^{1/\kappa}} x^{\frac{1}{\kappa} + \gamma_{3}+\gamma_{2}+1}\\
\times 
\int_{0}^{1} e^{\nu x u'} (u')^{\frac{1}{\kappa} + \gamma_{3}} (1 - u')^{\gamma_2} du',
\end{multline}
which is finite.

The second situation, i.e. $\gamma_3>-1$ and $\gamma_1\ge \gamma_3+1$ is considered by using that  $1 + |\epsilon|^{(\chi + \alpha)\kappa}x \geq
|\epsilon|^{(\chi + \alpha)\kappa}x$ for all $x \geq 1$, and (\ref{e228}) to check the case in which $x\ge1$. For those $0\le x\le 1$ we proceed as in (\ref{e238}) to conclude the result.

\end{proof}

\textit{Proof of Proposition~\ref{prop3}:}
\begin{proof}
The proof is analogous to that of Proposition 3 in~\cite{ma} and follows the guidelines of Proposition 3 in~\cite{lama1}. For the sake of completeness, we reproduce the proof.

We write
\begin{multline}
B = || \tau^{\kappa - 1} \int_{0}^{\tau^{\kappa}} \int_{-\infty}^{+\infty}
f( (\tau^{\kappa} - s')^{1/\kappa},m-m_{1}) g( (s')^{1/\kappa},m_{1})
\frac{1}{(\tau^{\kappa}-s')s'} ds' dm_{1} ||_{(\nu,\beta,\mu,\chi,\alpha,\kappa,\epsilon)}\\
= \sup_{\tau \in \bar{D}(0,\rho) \cup S_{d},m \in \mathbb{R}}
(1 + |m|)^{\mu} \exp(\beta |m|) \frac{1 + |\frac{\tau}{\epsilon^{\chi + \alpha}}|^{2\kappa}}{
|\frac{\tau}{\epsilon^{\chi + \alpha}}|}
\exp( -\nu |\frac{\tau}{\epsilon^{\chi + \alpha}}|^{\kappa} )\\
\times |\tau^{\kappa-1}\int_{0}^{\tau^{\kappa}}
\int_{-\infty}^{+\infty} \{ (1 + |m-m_{1}|)^{\mu} \exp( \beta |m-m_{1}|)\\
\times
\frac{1 + \frac{|\tau^{\kappa} - s'|^{2}}{|\epsilon|^{(\chi + \alpha)2\kappa}}}{
\frac{|\tau^{\kappa} - s'|^{1/\kappa}}{|\epsilon|^{\chi + \alpha}}}
\exp( -\nu \frac{|\tau^{\kappa} - s'|}{|\epsilon|^{(\chi + \alpha)\kappa}} )
f((\tau^{\kappa} - s')^{1/\kappa},m-m_{1}) \}\\
\times
\{ (1 + |m_{1}|)^{\mu} \exp( \beta |m_{1}| ) \frac{ 1 + \frac{|s'|^{2}}{|\epsilon|^{(\chi+\alpha)2\kappa}} }{
\frac{|s'|^{1/\kappa}}{|\epsilon|^{\chi + \alpha}} }
\exp(-\nu \frac{|s'|}{|\epsilon|^{(\chi + \alpha)\kappa}} ) g((s')^{1/\kappa},m_{1}) \}
\times \mathcal{B}(\tau,s,m,m_{1}) ds' dm_{1} | \label{defin_B}
\end{multline}
where
\begin{multline*}
 \mathcal{B}(\tau,s,m,m_{1}) = \frac{\exp( -\beta |m-m_{1}| ) \exp( -\beta |m_{1}| )}{(1 + |m-m_{1}|)^{\mu}(1 + |m_{1}|)^{\mu}}
 \frac{ \frac{|s'|^{1/\kappa} |\tau^{\kappa} - s'|^{1/\kappa}}{|\epsilon|^{2(\chi + \alpha)}} }{
 (1 + \frac{|\tau^{\kappa} - s'|^{2}}{|\epsilon|^{(\chi + \alpha)2\kappa}})
 (1 + \frac{|s'|^{2}}{|\epsilon|^{(\chi + \alpha)2\kappa}})}\\
 \times \exp( \nu \frac{|\tau^{\kappa} - s'|}{|\epsilon|^{(\chi + \alpha)\kappa}} )
 \exp( \nu \frac{|s'|}{|\epsilon|^{(\chi + \alpha)\kappa}} )
 \frac{1}{(\tau^{\kappa} - s')s'}.
\end{multline*}
We also have
$|m| \leq |m-m_{1}| + |m_{1}|$ for all $m,m_{1} \in \mathbb{R}$, from which we get
\begin{equation}
B \leq C_{4}(\epsilon) ||f(\tau,m)||_{(\nu,\beta,\mu,\chi,\alpha,\kappa,\epsilon)}
||g(\tau,m)||_{(\nu,\beta,\mu,\chi,\alpha,\kappa,\epsilon)} \label{B<=norm_f_times_norm_g} 
\end{equation}
where
\begin{multline*}
C_{4}(\epsilon) = \sup_{\tau \in \bar{D}(0,\rho) \cup S_{d},m \in \mathbb{R}}
(1 + |m|)^{\mu} \frac{1 + |\frac{\tau}{\epsilon^{\chi + \alpha}}|^{2\kappa}}{
|\frac{\tau}{\epsilon^{\chi + \alpha}}|}
\exp( -\nu |\frac{\tau}{\epsilon^{\chi + \alpha}}|^{\kappa} )
|\tau|^{\kappa - 1} \\
\times \int_{0}^{|\tau|^{\kappa}}
\int_{-\infty}^{+\infty}
\frac{1}{(1 + |m-m_{1}|)^{\mu}(1 + |m_{1}|)^{\mu}}
\frac{(h')^{1/\kappa} (|\tau|^{\kappa} - h')^{1/\kappa}}{|\epsilon|^{2(\chi + \alpha)}}
\frac{1}{(1 + \frac{(|\tau|^{\kappa} - h')^{2}}{|\epsilon|^{(\chi + \alpha)2\kappa}})
(1 + \frac{(h')^{2}}{|\epsilon|^{(\chi + \alpha)2\kappa}})}\\
\times 
\exp( \nu \frac{|\tau|^{\kappa} - h'}{|\epsilon|^{(\chi + \alpha)\kappa}} )
\exp( \nu \frac{h'}{|\epsilon|^{(\chi + \alpha)\kappa}} )
\frac{1}{(|\tau|^{\kappa} - h')h'} dh' dm_{1}.
\end{multline*}
We provide upper bounds that can be split in two parts,
\begin{equation}
C_{4}(\epsilon) \leq C_{4.1}C_{4.2}(\epsilon) \label{C3<=C31_C32} 
\end{equation}
where
\begin{equation}
C_{4.1} = \sup_{m \in \mathbb{R}} (1 + |m|)^{\mu}
\int_{-\infty}^{+\infty} \frac{1}{(1 + |m-m_{1}|)^{\mu}(1 + |m_{1}|)^{\mu}} dm_{1} \label{C31_defin}
\end{equation}
is finite under the condition that $\mu > 1$ according to Lemma 4 of \cite{ma2}, and
\begin{multline*}
C_{4.2}(\epsilon) = \sup_{\tau \in \bar{D}(0,\rho) \cup S_{d}}
\frac{1 + |\frac{\tau}{\epsilon^{\chi + \alpha}}|^{2\kappa}}{
|\frac{\tau}{\epsilon^{\chi + \alpha}}|}|\tau|^{\kappa - 1}\\
\times
\int_{0}^{|\tau|^{\kappa}} \frac{
\frac{(h')^{1/\kappa} (|\tau|^{\kappa} - h')^{1/\kappa}}{|\epsilon|^{2(\chi + \alpha)}} }{
(1 + \frac{(|\tau|^{\kappa} - h')^{2}}{|\epsilon|^{(\chi + \alpha)2\kappa}})
(1 + \frac{(h')^{2}}{|\epsilon|^{(\chi + \alpha)2\kappa}})}
\frac{1}{(|\tau|^{\kappa} - h')h'} dh'
\end{multline*}
The change of variable $h' = |\epsilon|^{(\chi + \alpha)\kappa}h$ yields 
\begin{multline}
C_{4.2}(\epsilon) = \sup_{\tau \in \bar{D}(0,\rho) \cup S_{d}}
\frac{1 + |\frac{\tau}{\epsilon^{\chi + \alpha}}|^{2\kappa}}{
|\frac{\tau}{\epsilon^{\chi + \alpha}}|}|\tau|^{\kappa - 1}\\
\times
\int_{0}^{\frac{|\tau|^{\kappa}}{|\epsilon|^{(\chi + \alpha)\kappa}}}
\frac{ h^{1/\kappa} (\frac{|\tau|^{\kappa}}{|\epsilon|^{(\chi + \alpha)\kappa}} - h)^{1/\kappa} }
{ (1 + (\frac{|\tau|^{\kappa}}{|\epsilon|^{(\chi + \alpha)\kappa}} - h)^{2}) (1 + h^2) }
\frac{1}{(\frac{|\tau|^{\kappa}}{|\epsilon|^{(\chi + \alpha)\kappa}} - h)h}
\frac{1}{|\epsilon|^{(\chi + \alpha)\kappa}} dh \leq \frac{1}{|\epsilon|^{\chi + \alpha}}
\sup_{x \geq 0} B(x) \label{C32_bounds}
\end{multline}
where
$$ B(x) = \frac{1+x^2}{x^{2/\kappa}}x \int_{0}^{x}
\frac{h^{1/\kappa}(x-h)^{1/\kappa}}{(1 + (x-h)^2)(1 + h^2)}\frac{1}{(x-h)h}dh.$$
A change of variable $h=xu$ in this last expression followed by a partial fraction decomposition
allow us to write
\begin{multline}
B(x) = (1+x^2) \int_{0}^{1} \frac{1}{(1 + x^{2}(1-u)^{2})(1 + x^{2}u^{2})}
\frac{1}{(1-u)^{1-\frac{1}{\kappa}}u^{1 - \frac{1}{\kappa}}}
du\\
= \frac{1 + x^2}{x^{2}+4} \int_{0}^{1} \frac{3 - 2u}{1 + x^{2}(1-u)^{2}}
\frac{1}{(1-u)^{1 - \frac{1}{\kappa}}u^{1 - \frac{1}{\kappa}}} du
+ \frac{1 + x^2}{x^{2}+4} \int_{0}^{1} \frac{2u+1}{1 + x^{2}u^{2}}
\frac{1}{(1-u)^{1 - \frac{1}{\kappa}}u^{1 - \frac{1}{\kappa}}} du \label{Bx_bounded}
\end{multline}
which acquaints us that $B(x)$ is finite provided that $\kappa \geq 1$ and bounded on
$\mathbb{R}_{+}$ w.r.t $x$.

The estimates in (\ref{defin_B}), (\ref{B<=norm_f_times_norm_g}), (\ref{C3<=C31_C32}),
(\ref{C31_defin}), (\ref{C32_bounds}) and
(\ref{Bx_bounded}) allow us to conclude the result.

\end{proof}

\section{Proof of Lemma~\ref{lema568}}\label{seclem4}

\begin{proof}
Let $\epsilon\in D(0,\epsilon_0)\setminus\{0\}$. We first prove that $\mathcal{H}_{\epsilon}(\bar{B}(0,\varpi))\subseteq\bar{B}(0,\varpi)$, in $F^{d}_{(\nu,\beta,\mu,\chi_1,\alpha,\kappa_1,\epsilon)}$.

We first consider the terms in $\mathcal{H}_\epsilon^1$ in order to give upper estimates.

By Lemma~\ref{lema1}, we have

\begin{equation}\label{e597}
\left\|\tilde{B}_{j}(m)\epsilon^{-\chi_1(b_j-k_{0,1}-\gamma_1)}\frac{\tau^{b_j-k_{0,1}-\gamma_1}}{\tilde{P}_m(\tau)}\right\|_{(\nu,\beta,\mu,\chi_1,\alpha,\kappa_1,\epsilon)}\le\frac{C_2}{C_{\tilde{P}}(r_{\tilde{Q},\tilde{R}_D})^{\frac{1}{\delta_D\kappa_1}}}\frac{\left\|\tilde{B}_j(m)\right\|_{(\beta,\mu)}}{\inf_{m\in\R}|\tilde{R}_{D}(im)|}|\epsilon|^{(b_j-k_{0,1}-\gamma_1)\alpha},
\end{equation}
for some $C_2>0$, depending on $\kappa_1,\gamma_1,k_{0,1}$ and $b_j$, $0\le j\le Q$.

We also make use of the first part of Proposition~\ref{prop1}. There exists a constant $C_3>0$ depending on $\nu,\kappa_1, k_{\ell,1}$ for $0\le \ell\le M_1$ , $k_{\ell,2}$ for $0\le \ell\le M_2$, $k_{\ell,3}$ for $0\le \ell\le M_3$, $\tilde{Q}(X)$ and $\tilde{R}_{D}(X)$ such that
\begin{multline}\label{e626}
\left\|\epsilon^{-\chi_1(k_{\ell,1}-k_{0,1})}\frac{\tilde{Q}(im)}{\tilde{P}_m(\tau)}(\tau^{k_{\ell,1}-k_{0,1}}\star_{\kappa_1}\omega_1)\right\|_{(\nu,\beta,\mu,\chi_1,\alpha,\kappa_1,\epsilon)}\\
\le \frac{C_3}{C_{\tilde{P}}(r_{\tilde{Q},\tilde{R}_D})^{\frac{1}{\delta_D \kappa_1}}}|\epsilon|^{\alpha(k_{\ell,1}-k_{0,1})}\left\|\omega_1\right\|_{(\nu,\beta,\mu,\chi_1,\alpha,\kappa_1,\epsilon)},
\end{multline} 

\begin{multline}\label{e627}
\left\|\epsilon^{-\chi_1(k_{\ell,2}-k_{0,2})}\frac{\tilde{Q}(im)}{\tilde{P}_m(\tau)}(\tau^{k_{\ell,2}-k_{0,2}}\star_{\kappa_1}\omega_1)\right\|_{(\nu,\beta,\mu,\chi_1,\alpha,\kappa_1,\epsilon)}\\
\le \frac{C_3}{C_{\tilde{P}}(r_{\tilde{Q},\tilde{R}_D})^{\frac{1}{\delta_D \kappa_1}}}|\epsilon|^{\alpha(k_{\ell,2}-k_{0,2})}\left\|\omega_1\right\|_{(\nu,\beta,\mu,\chi_1,\alpha,\kappa_1,\epsilon)},
\end{multline}

\begin{multline}\label{e628}
\left\|\epsilon^{-\chi_1(k_{\ell,3}+k_{0,1}-2k_{0,2})}\frac{\tilde{Q}(im)}{\tilde{P}_m(\tau)}(\tau^{k_{\ell,3}+k_{0,1}-2k_{0,2}}\star_{\kappa_1}\omega_1)\right\|_{(\nu,\beta,\mu,\chi_1,\alpha,\kappa_1,\epsilon)}\\
\le \frac{C_3}{C_{\tilde{P}}(r_{\tilde{Q},\tilde{R}_D})^{\frac{1}{\delta_D \kappa_1}}}|\epsilon|^{\alpha(k_{\ell,3}+k_{0,1}-2k_{0,2})}\left\|\omega_1\right\|_{(\nu,\beta,\mu,\chi_1,\alpha,\kappa_1,\epsilon)}.
\end{multline}

Let $j\ge 1$. A constant $C_3(j)>0$, depending on $\nu,\kappa_1$, $k_{\ell,2}$ for $0\le \ell\le M_2$, $\tilde{Q}(X)$, $\tilde{R}_{D}(X)$, exists such that

\begin{multline}\label{e629}
\left\|\epsilon^{-\chi_1(k_{\ell,2}-k_{0,2}+j)}\frac{\tilde{Q}(im)}{\tilde{P}_m(\tau)}(\tau^{k_{\ell,2}-k_{0,2}+j}\star_{\kappa_1}\omega_1)\right\|_{(\nu,\beta,\mu,\chi_1,\alpha,\kappa_1,\epsilon)}\\
\le \frac{C_{3.1}(j)}{C_{\tilde{P}}(r_{\tilde{Q},\tilde{R}_D})^{\frac{1}{\delta_D \kappa_1}}}|\epsilon|^{\alpha(k_{\ell,2}-k_{0,2}+j)}\left\|\omega_1\right\|_{(\nu,\beta,\mu,\chi_1,\alpha,\kappa_1,\epsilon)},
\end{multline}

and

\begin{multline}\label{e630}
\left\|\epsilon^{-\chi_1(k_{\ell,3}+k_{0,1}-2k_{0,2}+j)}\frac{\tilde{Q}(im)}{\tilde{P}_m(\tau)}(\tau^{k_{\ell,3}+k_{0,1}-2k_{0,2}+j}\star_{\kappa_1}\omega_1)\right\|_{(\nu,\beta,\mu,\chi_1,\alpha,\kappa_1,\epsilon)}\\
\le \frac{C_{3.2}(j)}{C_{\tilde{P}}(r_{\tilde{Q},\tilde{R}_D})^{\frac{1}{\delta_D \kappa_1}}}|\epsilon|^{\alpha(k_{\ell,3}+k_{0,1}-2k_{0,2}+j)}\left\|\omega_1\right\|_{(\nu,\beta,\mu,\chi_1,\alpha,\kappa_1,\epsilon)}.
\end{multline}

We now determine the dependence on $j$ of the constants $C_{3.1}(j),C_{3.2}(j)>0$ obtained by the application of Proposition~\ref{prop1}. More precisely, one has
\begin{equation}\label{e633}
C_{3.1}(j)\le\hat{C}_3A_3^j\Gamma\left(\frac{k_{\ell,2}-k_{0,2}+j}{\kappa_1}\right),\qquad j\ge 1,
\end{equation}
and
\begin{equation}\label{e634}
C_{3.2}(j)\le\hat{C}_3A_3^j\Gamma\left(\frac{k_{\ell,3}+k_{0,1}-2k_{0,2}+j}{\kappa_1}\right),\qquad j\ge 1,
\end{equation}
for some $\hat{C}_3,A_3>0$ which do not depend on $j$. The proof of (\ref{e633}) is based on a deeper look at the estimates in the proof of the first part of Proposition~\ref{prop1}. We follow the notations in such proof. We only give detail on the proof of (\ref{e633}), because the proof of (\ref{e634}) is analogous.

\textit{proof of(\ref{e633}):} Let $j\ge1$ such that $k_{\ell,2}-k_{0,2}+j>\kappa_1$. The classical estimates
\begin{equation} 
\sup_{x \geq 0} x^{m_1}e^{-m_{2}x} = (\frac{m_1}{m_2})^{m_1} e^{-m_{1}} \label{e1017}
\end{equation}
for any real numbers $m_{1} \geq 0$, $m_{2} > 0$ yield the following:

\begin{multline*}
\sup_{x \geq 0} \frac{1 + x^2}{x^{1/\kappa_1}} e^{-\nu x} xG_{1}(x) \leq \sup_{x \geq 0} (1 + x^{2})
x^{\frac{k_{\ell,2}-k_{0,2}+j}{\kappa_1}} e^{-\frac{\nu}{2} x}\left(\frac{1}{2}\right)^{1/\kappa_1}\kappa_1 \\
\leq 
\left( (\frac{k_{\ell,2}-k_{0,2}+j}{\kappa_1 \nu/2})^{\frac{k_{\ell,2}-k_{0,2}+j}{\kappa_1}}
\exp( -\frac{k_{\ell,2}-k_{0,2}+j}{\kappa_1} ) \right. \\
\left. + (\frac{\frac{k_{\ell,2}-k_{0,2}+j}{\kappa_1} + 2}{\nu/2})^{\frac{k_{\ell,2}-k_{0,2}+j}{\kappa_1}+2}
\exp( -(\frac{k_{\ell,2}-k_{0,2}+j}{\kappa_1} + 2) ) \right) \left(\frac{1}{2}\right)^{1/\kappa_1}\kappa_1
\end{multline*}
Furthermore, according to the Stirling formula $\Gamma(x) \sim \sqrt{2\pi} x^{x-\frac{1}{2}} e^{-x}$ as $x \rightarrow +\infty$ and bearing in mind the functional relation $\Gamma(x+1) = x\Gamma(x)$ for all
$x > 0$, we get two constants $\check{C}_{2}>0$ and $A_{3}>0$ independent of $j$ such that
\begin{multline}
\sup_{x \geq 0} \frac{1 + x^2}{x^{1/\kappa_1}} e^{-\nu x} xG_{1}(x) \leq \check{C}_{2}
A_{3}^{j}( \Gamma( \frac{k_{\ell,2}-k_{0,2}+j}{\kappa_1} ) + \Gamma( \frac{ k_{\ell,2}-k_{0,2}+j}{\kappa_1} + 2 ) )\\
\leq \check{C}_{2}
A_{3}^{j}\left( \Gamma( \frac{k_{\ell,2}-k_{0,2}+j}{\kappa_1} ) +
(\frac{k_{\ell,2}-k_{0,2}+j}{\kappa_1} + 1)(\frac{k_{\ell,2}-k_{0,2}+j}{\kappa_1})
\Gamma( \frac{k_{\ell,2}-k_{0,2}+j}{\kappa_1} ) \right) \label{supG1<Gammaj}
\end{multline}
On the other hand, by direct inspection, we observe that there exists a constant $\check{C}_{2.1}>0$ (independent of
$j$ and $\epsilon$) such that
\begin{equation}
\sup_{0 \leq x < 1} \frac{1 + x^2}{x^{1/\kappa_1}}
e^{-\nu x} \frac{x}{(1 + |\epsilon|^{(\chi_1 + \alpha)\kappa_1}x)^{\gamma_1}} G_{2}(x) \leq 
\check{C}_{2.1} \label{sup01G2<cst}
\end{equation}
Furthermore, there exists a constant $K_{2.1}(j)$ depending on $\nu,\kappa_1$, $k_{\ell,2}$ for $0 \leq \ell \leq M_2$ and $j$,
such that
\begin{equation}
\sup_{x \geq 1} \frac{1 + x^2}{x^{1/\kappa_1}}
e^{-\nu x} \frac{x}{(1 + |\epsilon|^{(\chi_1 + \alpha)\kappa_1}x)^{\gamma_1}} G_{2}(x) \leq 
\sup_{x \geq 1} \frac{1 + x^2}{1 + (\frac{x}{2})^{2}}K_{2.1}(j) \label{supG2<K21j}
\end{equation}
Regarding the proof of Proposition 1 in~\cite{lama2}, we guarantee the existence of a constant $\check{K}_{2.1}>0$ independent of $j$ such that
\begin{equation}
K_{2.1}(j) \leq \check{K}_{2.1}\Gamma( \frac{k_{\ell,2}-k_{0,2}+j}{\kappa_1} ) \label{K21j<Gammaj}
\end{equation}
for all $j \geq 1$. Finally, gathering (\ref{supG1<Gammaj}), (\ref{sup01G2<cst}), (\ref{supG2<K21j}) and (\ref{K21j<Gammaj}), we conclude (\ref{e633}).

\textit{end of proof of (\ref{e633}).}

In view of Lemma~\ref{lema533} and again by Proposition~\ref{prop1}, we have

\begin{multline}\label{e631}
\left\|\epsilon^{-\chi_1(k_{\ell,3}+k_{0,1}-2k_{0,2})}\frac{\tilde{Q}(im)}{\tilde{P}_{m}(\tau)}\left[\tau^{k_{\ell,3}+k_{0,1}-2k_{0,2}+j}\star_{\kappa_1}\mathcal{B}_{\kappa_1}J_1(\tau,\epsilon)\star_{\kappa_1}\mathcal{B}_{\kappa_1}J_1(\tau,\epsilon)\star_{\kappa_1}\omega_1\right]\right\|_{(\nu,\beta,\mu,\chi_1,\alpha,\kappa_1,\epsilon)}\hfill\\
=\left\|\epsilon^{-\chi_1(k_{\ell,3}+k_{0,1}-2k_{0,2})}\frac{\tilde{Q}(im)}{\tilde{P}_{m}(\tau)}\left[\tau^{k_{\ell,3}+k_{0,1}-2k_{0,2}+j}\star_{\kappa_1}\mathcal{B}_{\kappa_1}\tilde{J}_1(\tau,\epsilon)\star_{\kappa_1}\omega_1\right]\right\|_{(\nu,\beta,\mu,\chi_1,\alpha,\kappa_1,\epsilon)}\hfill\\
\le \frac{C_{3.3}(j)}{C_{\tilde{P}}(r_{\tilde{Q},\tilde{R}_D})^{\frac{1}{\delta_D \kappa_1}}}|\epsilon|^{\alpha(k_{\ell,3}+k_{0,1}-2k_{0,2}+j)}\left\|\omega_1\right\|_{(\nu,\beta,\mu,\chi_1,\alpha,\kappa_1,\epsilon)},\hfill
\end{multline}
where
\begin{equation}\label{e677}
C_{3.3}(j)\le\hat{C}_3A_3^j\Gamma\left(\frac{k_{\ell,3}+k_{0,1}-2k_{0,2}+j}{\kappa_1}\right),\qquad j\ge 1.
\end{equation}
This last estimates for $C_{3.3}$ are obtained in the same manner as those in (\ref{e633}).

We choose large enough $r_{\tilde{Q},\tilde{R}_D}>0$ and $\varpi>0$ such that

\begin{multline}\label{e685}
\sum_{j=0}^{Q}|\epsilon_0|^{n_j-\alpha b_j+\alpha(b_j-k_{0,1}-\gamma_1)}\frac{C_2}{C_{\tilde{P}}(r_{\tilde{Q},\tilde{R}_D})^{\frac{1}{\delta_D\kappa_1}}}\frac{\left\|\tilde{B}_j(m)\right\|_{(\beta,\mu)}}{\inf_{m\in\R}|\tilde{R}_{D}(im)|}\\
+\sum_{\ell=1}^{s_1}\frac{|a_{\ell,1}|}{\Gamma\left(\frac{k_{\ell,1}-k_{0,1}}{\kappa_1}\right)}|\epsilon_0|^{\alpha(k_{\ell,1}-k_{0,1})}\frac{C_3}{C_{\tilde{P}}(r_{\tilde{Q},\tilde{R}_D})^{\frac{1}{\delta_D\kappa_1}}}\varpi\hfill\\
\hfill+\sum_{\ell=s_1+1}^{M_1}\frac{|a_{\ell,1}|}{\Gamma\left(\frac{k_{\ell,1}-k_{0,1}}{\kappa_1}\right)}|\epsilon_0|^{m_{\ell,1}+\beta-\alpha k_{\ell,1}+\alpha(k_{\ell,1}-k_{0,1})}\frac{C_3}{C_{\tilde{P}}(r_{\tilde{Q},\tilde{R}_D})^{\frac{1}{\delta_D\kappa_1}}}\varpi\\
+\sum_{\ell=1}^{s_2}\frac{2|a_{\ell,2}a_{0,1}|}{|a_{0,2}|}\frac{|\epsilon_0|^{\alpha(k_{\ell,2}-k_{0,2})}}{\Gamma\left(\frac{k_{\ell,2}-k_{0,2}}{\kappa_1}\right)}\frac{C_3}{C_{\tilde{P}}(r_{\tilde{Q},\tilde{R}_D})^{\frac{1}{\delta_D\kappa_1}}}\varpi\hfill\\
\hfill+\sum_{\ell=s_2+1}^{M_2}\frac{2|a_{\ell,2}a_{0,1}|}{|a_{0,2}|}\frac{|\epsilon_0|^{m_{\ell,2}+2\beta-\alpha k_{\ell,2}+\alpha(k_{\ell,2}-k_{0,2})}}{\Gamma\left(\frac{k_{\ell,2}-k_{0,2}}{\kappa_1}\right)}\frac{C_3}{C_{\tilde{P}}(r_{\tilde{Q},\tilde{R}_D})^{\frac{1}{\delta_D\kappa_1}}}\varpi\\
+\sum_{\ell=0}^{s_2}\frac{2|a_{\ell,2}a_{0,1}|}{|a_{0,2}|}\sum_{j\ge1}|J_j|\frac{|\epsilon_0|^{\alpha(k_{\ell,2}-k_{0,2}+j)}}{\Gamma\left(\frac{k_{\ell,2}-k_{0,2}}{\kappa_1}\right)}\frac{\hat{C}_3A_3^j}{C_{\tilde{P}}(r_{\tilde{Q},\tilde{R}_D})^{\frac{1}{\delta_D\kappa_1}}}\varpi\hfill\\
\hfill+\sum_{\ell=s_2+1}^{M_2}\frac{2|a_{\ell,2}a_{0,1}|}{|a_{0,2}|}\sum_{j\ge1}|J_j|\frac{|\epsilon_0|^{m_{\ell,2}+2\beta-\alpha k_{\ell,2}+\alpha(k_{\ell,2}-k_{0,2}+j)}}{\Gamma\left(\frac{k_{\ell,2}-k_{0,2}}{\kappa_1}\right)}\frac{\hat{C}_3A_3^j}{C_{\tilde{P}}(r_{\tilde{Q},\tilde{R}_D})^{\frac{1}{\delta_D\kappa_1}}}\varpi\\
+\sum_{\ell=0}^{s_3}\frac{3|a_{0,1}|^2|a_{\ell,3}|}{|a_{0,2}|^2}\frac{|\epsilon_0|^{\alpha(k_{\ell,3}+k_{0,1}-2k_{0,2})}}{\Gamma\left(\frac{k_{\ell,3}+k_{0,1}-2k_{0,2}}{\kappa_1}\right)}\frac{C_3}{C_{\tilde{P}}(r_{\tilde{Q},\tilde{R}_D})^{\frac{1}{\delta_D\kappa_1}}}\varpi\hfill\\
\hfill+\sum_{\ell=0}^{s_3}\frac{6|a_{0,1}|^2|a_{\ell,3}|}{|a_{0,2}|^2}\sum_{j\ge1}|J_j|\frac{|\epsilon_0|^{\alpha(k_{\ell,3}+k_{0,1}-2k_{0,2}+j)}}{\Gamma\left(\frac{k_{\ell,3}+k_{0,1}-2k_{0,2}}{\kappa_1}\right)}\frac{\hat{C}_3A_3^j}{C_{\tilde{P}}(r_{\tilde{Q},\tilde{R}_D})^{\frac{1}{\delta_D\kappa_1}}}\varpi\\
+\sum_{\ell=0}^{s_3}\frac{3|a_{0,1}|^2|a_{\ell,3}|}{|a_{0,2}|^2}\sum_{j\ge1}|\tilde{J}_j|\frac{|\epsilon_0|^{\alpha(k_{\ell,3}+k_{0,1}-2k_{0,2}+j)}}{\Gamma\left(\frac{k_{\ell,3}+k_{0,1}-2k_{0,2}}{\kappa_1}\right)}\frac{\hat{C}_3A_3^j}{C_{\tilde{P}}(r_{\tilde{Q},\tilde{R}_D})^{\frac{1}{\delta_D\kappa_1}}}\varpi\hfill\\
\hfill+\sum_{\ell=s_3+1}^{M_3}\frac{3|a_{0,1}|^2|a_{\ell,3}|}{|a_{0,2}|^2}\frac{|\epsilon_0|^{m_{\ell,3}+3\beta-\alpha k_{\ell,3}+\alpha(k_{\ell,3}+k_{0,1}-2k_{0,2})}}{\Gamma\left(\frac{k_{\ell,3}+k_{0,1}-2k_{0,2}}{\kappa_1}\right)}\frac{C_3}{C_{\tilde{P}}(r_{\tilde{Q},\tilde{R}_D})^{\frac{1}{\delta_D\kappa_1}}}\varpi\\
+\sum_{\ell=s_3+1}^{M_3}\frac{6|a_{0,1}|^2|a_{\ell,3}|}{|a_{0,2}|^2}\sum_{j\ge1}|J_j|\frac{|\epsilon_0|^{m_{\ell,3}+3\beta-\alpha k_{\ell,3}+\alpha(k_{\ell,3}+k_{0,1}-2k_{0,2}+j)}}{\Gamma\left(\frac{k_{\ell,3}+k_{0,1}-2k_{0,2}}{\kappa_1}\right)}\frac{\hat{C}_3A_3^j}{C_{\tilde{P}}(r_{\tilde{Q},\tilde{R}_D})^{\frac{1}{\delta_D\kappa_1}}}\varpi\hfill\\
\hfill+\sum_{\ell=s_3+1}^{M_3}\frac{3|a_{0,1}|^2|a_{\ell,3}|}{|a_{0,2}|^2}\sum_{j\ge1}|\tilde{J}_j|\frac{|\epsilon_0|^{m_{\ell,3}+3\beta-\alpha k_{\ell,3}+\alpha(k_{\ell,3}+k_{0,1}-2k_{0,2}+j)}}{\Gamma\left(\frac{k_{\ell,3}+k_{0,1}-2k_{0,2}}{\kappa_1}\right)}\frac{\hat{C}_3A_3^j}{C_{\tilde{P}}(r_{\tilde{Q},\tilde{R}_D})^{\frac{1}{\delta_D\kappa_1}}}\varpi\le \frac{\varpi}{4}.
\end{multline}

Observe that the convergence of the series in $j$ appearing in the previous expression converge provided that $|\epsilon_0|$ is small enough, according to the fact that $J_1$ and $\tilde{J}_1$ are convergent series in a neighborhood of the origin, which yields $|J_j|\le C_J(A_J)^j,$ and $|\tilde{J}_j|\le C_J(A_J)^j$, for some $C_J,A_J>0$.

After the choice in (\ref{e685}), one can apply (\ref{e597}), (\ref{e626}), (\ref{e627}), (\ref{e628}), (\ref{e629}), (\ref{e630}) and (\ref{e631}), together with (\ref{e633}), (\ref{e634}) and (\ref{e677}) to deduce that
\begin{equation}\label{e704} 
\left\|\mathcal{H}^1_\epsilon(\omega_1)\right\|_{(\nu,\beta,\mu,\chi_1,\alpha,\kappa_1,\epsilon)}\le \frac{\varpi}{4}.
\end{equation}

We now give upper bounds associated to $\mathcal{H}_\epsilon^2$.

We define 
\begin{equation}\label{e714}
h_1(\tau,m)=\tau^{\kappa_1-1}\int_{0}^{\tau^{\kappa_1}}\int_{-\infty}^{\infty}\omega_1((\tau^{\kappa_1}-s')^{1/\kappa_1},m-m_1)\omega_1((s')^{1/\kappa_1},m_1)\frac{1}{(\tau^{\kappa_1}-s')s'}ds'dm_1.
\end{equation}

Observe that 
$$\omega_1(\tau,m)\star_{\kappa_1}^E\omega_1(\tau,m)=\tau h_1(\tau,m).$$

In view of Proposition~\ref{prop3}, there exists $C_4>0$, depending on $\mu$ and $\kappa_1$, such that 
\begin{equation}\label{e721}
\left\|h_1(\tau,\epsilon)\right\|_{(\nu,\beta,\mu,\chi_1,\alpha,\kappa_1,\epsilon)}\le\frac{C_4}{|\epsilon|^{\chi_1+\alpha}}\left\|\omega_1\right\|^2_{(\nu,\beta,\mu,\chi_1,\alpha,\kappa_1,\epsilon)}.
\end{equation}

We apply Proposition~\ref{prop1}. 2) to get the existence of $C'_2>0$, depending on $\nu,\kappa_1,d_{\ell},\delta_\ell,k_{0,1},\delta_{D}$, $\tilde{R}_\ell(X)$, for $1\le \ell\le D$, such that
\begin{multline}\label{e726}
\left\|\epsilon^{-\chi_1(d_\ell-k_{0,1}-\delta_\ell)}\frac{\tilde{R}_{\ell}(im)}{\tilde{P}_{m}(\tau)}\left[\tau^{d_{\ell,q_1,q_2}}\star_{\kappa_1}(\tau^{\kappa_1 q_2}\omega_1)\right]\right\|_{(\nu,\beta,\mu,\chi_1,\alpha,\kappa_1,\epsilon)}\\
\le \frac{C'_3}{C_{\tilde{P}}(r_{\tilde{Q},\tilde{R}_D})^{\frac{1}{\delta_D\kappa_1}}}|\epsilon|^{(\chi_1+\alpha)\kappa_1\left(\frac{d_{\ell,q_1,q_2}}{\kappa_1}+q_2-\delta_D+\frac{1}{\kappa_1}\right)-\chi_1(d_{\ell}-k_{0,1}-\delta_\ell)}\left\|\omega_1\right\|_{(\nu,\beta,\mu,\chi_1,\alpha,\kappa_1,\epsilon)},
\end{multline}
for every $q_1\ge0$, $q_2\ge 1$ such that $q_1+q_2=\delta_\ell$, and also
\begin{multline}\label{e727}
\left\|\epsilon^{-\chi_1(d_\ell-k_{0,1}-\delta_\ell)}\frac{\tilde{R}_{\ell}(im)}{\tilde{P}_{m}(\tau)}\left[\tau^{d_{\ell,q_1,q_2}+\kappa_1(q_2-p)}\star_{\kappa_1}(\tau^{\kappa_1 p}\omega_1)\right]\right\|_{(\nu,\beta,\mu,\chi_1,\alpha,\kappa_1,\epsilon)}\\
\le \frac{C'_3}{C_{\tilde{P}}(r_{\tilde{Q},\tilde{R}_D})^{\frac{1}{\delta_D\kappa_1}}}|\epsilon|^{(\chi_1+\alpha)\kappa_1\left(\frac{d_{\ell,q_1,q_2}}{\kappa_1}+q_2-\delta_D+\frac{1}{\kappa_1}\right)-\chi_1(d_{\ell}-k_{0,1}-\delta_\ell)}\left\|\omega_1\right\|_{(\nu,\beta,\mu,\chi_1,\alpha,\kappa_1,\epsilon)},
\end{multline}
for every $1\le p\le q_2-1$.

We apply (\ref{e721}), and Proposition~\ref{prop1}.2) to guarantee the existence of $C'_3>0$, depending on $\nu,\kappa_1,\gamma_1,\delta_{D},k_{0,1}$, and $k_{\ell,2}$ for $0\le \ell\le M_2$ such that

\begin{multline}\label{e728}
\left\|\epsilon^{-\chi_1(k_{\ell,2}+\gamma_1-k_{0,1})}\frac{\tilde{Q}(im)}{\tilde{P}_m(\tau)}\left[\tau^{k_{\ell,2}+\gamma_1-k_{0,1}}\star_{\kappa_1}(\tau h_1(\tau,m))\right]\right\|_{(\nu,\beta,\mu,\chi_1,\alpha,\kappa_1,\epsilon)}\\
=\left\|\epsilon^{-\chi_1(k_{\ell,2}+\gamma_1-k_{0,1})}\frac{\tilde{Q}(im)}{\tilde{P}_m(\tau)}\tau^{\kappa_1}\int_{0}^{\tau^{\kappa_1}}(\tau^{\kappa_1}-s)^{\frac{k_{\ell,2}+\gamma_1-k_{0,1}}{\kappa_1}-1}s^{\frac{1}{\kappa_1}-1} h(s^{1/\kappa_1},m)ds\right\|_{(\nu,\beta,\mu,\chi_1,\alpha,\kappa_1,\epsilon)}\\
\le \frac{C'_3}{C_{\tilde{P}}(r_{\tilde{Q},\tilde{R}_D})^{\frac{1}{\delta_D\kappa_1}}}|\epsilon|^{(\chi_1+\alpha)\kappa_1\left(\frac{k_{\ell,2}+\gamma_1-k_{0,1}}{\kappa_1}+\frac{1}{\kappa_1}\right)-\chi_1(k_{\ell,2}+\gamma_1-k_{0,1})-(\chi_1+\alpha)\kappa_1(\delta_D-\frac{1}{\kappa_1})}\left\|h_1(\tau,\epsilon)\right\|_{(\nu,\beta,\mu,\chi_1,\alpha,\kappa_1,\epsilon)}\\
\le \frac{C_4C'_3}{C_{\tilde{P}}(r_{\tilde{Q},\tilde{R}_D})^{\frac{1}{\delta_D\kappa_1}}}|\epsilon|^{(\chi_1+\alpha)(k_{\ell,2}+\gamma_1-k_{0,1}-\kappa_1\delta_D+1)-\chi_1(k_{\ell,2}+\gamma_1-k_{0,1})}\left\|\omega_1(\tau,\epsilon)\right\|^2_{(\nu,\beta,\mu,\chi_1,\alpha,\kappa_1,\epsilon)}.
\end{multline} 

The same arguments as above follow to get 

\begin{multline}\label{e729}
\left\|\epsilon^{-\chi_1(k_{\ell,3}+\gamma_1-k_{0,2})}\frac{\tilde{Q}(im)}{\tilde{P}_m(\tau)}\left[\tau^{k_{\ell,3}+\gamma_1-k_{0,2}}\star_{\kappa_1}(\tau h_1(\tau,m))\right]\right\|_{(\nu,\beta,\mu,\chi_1,\alpha,\kappa_1,\epsilon)}\\
\le \frac{C_4C'_3}{C_{\tilde{P}}(r_{\tilde{Q},\tilde{R}_D})^{\frac{1}{\delta_D\kappa_1}}}|\epsilon|^{(\chi_1+\alpha)(k_{\ell,3}+\gamma_1-k_{0,2}-\kappa_1\delta_D+1)-\chi_1(k_{\ell,3}+\gamma_1-k_{0,2})}\left\|\omega_1(\tau,\epsilon)\right\|^2_{(\nu,\beta,\mu,\chi_1,\alpha,\kappa_1,\epsilon)},
\end{multline} 

and

\begin{multline}\label{e730}
\left\|\epsilon^{-\chi_1(k_{\ell,3}+\gamma_1-k_{0,2}+j)}\frac{\tilde{Q}(im)}{\tilde{P}_m(\tau)}\left[\tau^{k_{\ell,3}+\gamma_1-k_{0,2}+j}\star_{\kappa_1}(\tau h_1(\tau,m))\right]\right\|_{(\nu,\beta,\mu,\chi_1,\alpha,\kappa_1,\epsilon)}\\
\le \frac{C_4C'_3(j)}{C_{\tilde{P}}(r_{\tilde{Q},\tilde{R}_D})^{\frac{1}{\delta_D\kappa_1}}}|\epsilon|^{(\chi_1+\alpha)(k_{\ell,3}+\gamma_1-k_{0,2}+j-\kappa_1\delta_D+1)-\chi_1(k_{\ell,3}+\gamma_1-k_{0,2}+j)}\left\|\omega_1(\tau,\epsilon)\right\|^2_{(\nu,\beta,\mu,\chi_1,\alpha,\kappa_1,\epsilon)},
\end{multline} 
for every $j\ge1$, where 
\begin{equation}\label{e731}
C'_3(j)\le \hat{C}_3A_3^j\Gamma\left(\frac{k_{\ell,3}+\gamma_1-k_{0,2}}{\kappa_1}\right) \quad j\ge1.
\end{equation}
The proof of such dependence on $j$ is proved in an analogous way as for that of (\ref{e633}).

One can choose $r_{\tilde{Q},\tilde{R}_D}>0$ and $\varpi>0$ such that the following condition holds:

\begin{multline}\label{e766}
\sum_{\ell=1}^{D-1}|\epsilon_0|^{\Delta_{\ell}+\alpha(\delta_{\ell}-d_{\ell})+\beta}\left[\prod_{d=0}^{\delta_\ell-1}|\gamma_1-d|\frac{C_3}{C_{\tilde{P}}(r_{\tilde{Q},\tilde{R}_D})^{\frac{1}{\delta_D\kappa_1}}\Gamma\left(\frac{d_{\ell,\delta_\ell,0}}{\kappa_1}\right)}|\epsilon_0|^{\alpha(d_\ell-k_{0,1}-\delta_\ell)}\varpi\right.\\
+\sum_{q_1+q_2=\delta_\ell,q_2\ge1}\frac{\delta_\ell!}{q_1!q_2!}\prod_{d=0}^{q_1-1}|\gamma_1-d|\left(\frac{C'_3\kappa_1^{q_2}}{C_{\tilde{P}}(r_{\tilde{Q},\tilde{R}_D})^{\frac{1}{\delta_D\kappa_1}}\Gamma\left(\frac{d_{\ell,q_1,q_2}}{\kappa_1}\right)}\right.\hfill\\
\hfill\times |\epsilon_0|^{(\chi_1+\alpha)\kappa_1\left(\frac{d_{\ell,q_1,q_2}}{\kappa_1}+q_2-\delta_D+\frac{1}{\kappa_1}\right)-\chi_1(d_\ell-k_{0,1}-\delta_\ell)}\varpi\\
+\sum_{1\le p\le q_2-1}|A_{q_2,p}|\frac{C'_3\kappa_1^p}{C_{\tilde{P}}(r_{\tilde{Q},\tilde{R}_D})^{\frac{1}{\delta_D\kappa_1}}\Gamma\left(\frac{d_{\ell,q_1,q_2}}{\kappa_1}+q_2-p\right)}\hfill\\
\hfill\left.\left. \times |\epsilon_0|^{(\chi_1+\alpha)\kappa_1\left(\frac{d_{\ell,q_1,q_2}}{\kappa_1}+q_2-\delta_D+\frac{1}{\kappa_1}\right)-\chi_1(d_\ell-k_{0,1}-\delta_\ell)}\varpi   \right)    \right]\\
+\sum_{\ell=0}^{s_2}\frac{|a_{\ell,2}|}{\Gamma\left(\frac{k_{\ell,2}+\gamma_1-k_{0,1}}{\kappa_1}\right)}\frac{C_4C'_3}{C_{\tilde{P}}(r_{\tilde{Q},\tilde{R}_D})^{\frac{1}{\delta_D\kappa_1}}}|\epsilon_0|^{(\chi_1+\alpha)(k_{\ell,2}+\gamma_1-k_{0,1}-\kappa_1\delta_D+1)-\chi_1(k_{\ell,2}+\gamma_1-k_{0,1})}\varpi^2\hfill\\
\hfill+\sum_{\ell=s_2+1}^{M_2}\frac{|a_{\ell,2}|}{\Gamma\left(\frac{k_{\ell,2}+\gamma_1-k_{0,1}}{\kappa_1}\right)}\frac{C_4C'_3}{C_{\tilde{P}}(r_{\tilde{Q},\tilde{R}_D})^{\frac{1}{\delta_D\kappa_1}}}|\epsilon_0|^{m_{\ell,2}+2\beta-\alpha k_{\ell,2}+(\chi_1+\alpha)(k_{\ell,2}+\gamma_1-k_{0,1}-\kappa_1\delta_D+1)-\chi_1(k_{\ell,2}+\gamma_1-k_{0,1})}\varpi^2\\
+\sum_{\ell=0}^{s_3}\frac{3|a_{\ell,3}a_{0,1}|}{|a_{0,2}|}\frac{C_4C'_3}{\Gamma\left(\frac{k_{\ell,3}+\gamma_1-k_{0,2}}{\kappa_1}\right)}\frac{|\epsilon_0|^{(\chi_1+\alpha)(k_{\ell,3}+\gamma_1-k_{0,2}-\kappa_1\delta_D+1)-\chi_1(k_{\ell,3}+\gamma_1-k_{0,2})}}{C_{\tilde{P}}(r_{\tilde{Q},\tilde{R}_D})^{\frac{1}{\delta_D\kappa_1}}}\varpi^2\hfill\\
\hfill+\sum_{\ell=s_3+1}^{M_3}\frac{3|a_{\ell,3}a_{0,1}|}{|a_{0,2}|}\frac{C_4C'_3}{\Gamma\left(\frac{k_{\ell,3}+\gamma_1-k_{0,2}}{\kappa_1}\right)}\frac{|\epsilon_0|^{m_{\ell,3}+3\beta-\alpha k_{\ell,3}+(\chi_1+\alpha)(k_{\ell,3}+\gamma_1-k_{0,2}-\kappa_1\delta_D+1)-\chi_1(k_{\ell,3}+\gamma_1-k_{0,2})}}{C_{\tilde{P}}(r_{\tilde{Q},\tilde{R}_D})^{\frac{1}{\delta_D\kappa_1}}}\varpi^2\\
+\sum_{\ell=0}^{s_3}\frac{3|a_{\ell,3}a_{0,1}|}{|a_{0,2}|}\sum_{j\ge1}|J_j|\frac{|\epsilon_0|^{(\chi_1+\alpha)(k_{\ell,3}+\gamma_1-k_{0,2}+j-\kappa_1\delta_D+1)-\chi_1(k_{\ell,3}+\gamma_1-k_{0,2}+j)}}{\Gamma\left(\frac{k_{\ell,3}+\gamma_1-k_{0,2}}{\kappa_1}\right)}\frac{C_4\hat{C}_3A_3^j}{C_{\tilde{P}}(r_{\tilde{Q},\tilde{R}_D})^{\frac{1}{\delta_D\kappa_1}}}\varpi^2\\
+\sum_{\ell=0}^{s_3}\frac{3|a_{\ell,3}a_{0,1}|}{|a_{0,2}|}\sum_{j\ge1}|J_j|\frac{|\epsilon_0|^{m_{\ell,3}+3\beta-\alpha k_{\ell,3}+(\chi_1+\alpha)(k_{\ell,3}+\gamma_1-k_{0,2}+j-\kappa_1\delta_D+1)-\chi_1(k_{\ell,3}+\gamma_1-k_{0,2}+j)}}{\Gamma\left(\frac{k_{\ell,3}+\gamma_1-k_{0,2}}{\kappa_1}\right)}\frac{C_4\hat{C}_3A_3^j}{C_{\tilde{P}}(r_{\tilde{Q},\tilde{R}_D})^{\frac{1}{\delta_D\kappa_1}}}\varpi^2\\
\le\frac{\varpi}{4}\hfill
\end{multline}

The choice in (\ref{e766}) allows to guarantee from (\ref{e726}), (\ref{e727}), (\ref{e728}), (\ref{e729}), (\ref{e730}), together with (\ref{e731}), that
\begin{equation}\label{e784}
\left\|\mathcal{H}_\epsilon^2(\omega_1)\right\|_{(\nu,\beta,\mu,\chi_1,\alpha,\kappa_1,\epsilon)}\le \frac{\varpi}{4}.
\end{equation}

We now give upper estimates for $\mathcal{H}_{\epsilon}^3(\omega_1(\tau,m))$.

Proposition~\ref{prop1}.1) yields 
\begin{multline}\label{e793}
\left\|\epsilon^{-\chi_1(d_{D}-k_{0,1}-\delta_D)}\frac{\tilde{R}_{D}(im)}{\tilde{P}_m(\tau)}\left[\tau^{d_{D,\delta_D,0}}\star_{\kappa_1}\omega_1\right]\right\|_{(\nu,\beta,\mu,\chi_1,\alpha,\kappa_1,\epsilon)}\\
\le\frac{C_3}{C_{\tilde{P}}(r_{\tilde{Q},\tilde{R}_D})^{\frac{1}{\delta_D\kappa_1}}}|\epsilon_0|^{(\chi_1+\alpha)d_{d_{D},\delta_D,0}-\chi_1(d_D-k_{0,1}-\delta_D)}\left\|\omega_1\right\|_{(\nu,\beta,\mu,\chi_1,\alpha,\kappa_1,\epsilon)}\\
=\frac{C_3}{C_{\tilde{P}}(r_{\tilde{Q},\tilde{R}_D})^{\frac{1}{\delta_D\kappa_1}}}|\epsilon_0|^{\alpha(d_{D}-k_{0,1}-\delta_D)}\left\|\omega_1\right\|_{(\nu,\beta,\mu,\chi_1,\alpha,\kappa_1,\epsilon)},
\end{multline}
for some $C_3>0$, depending on $\nu,\kappa_1, k_{0,1},\delta_{D},d_{D}$.

We also apply Proposition~\ref{prop1}.2) to guarantee the existence of $C'_3>0$, depending on $\nu,\kappa_1, k_{0,1},\delta_{D},d_{D}$ with
\begin{multline}\label{e794}
\left\|\epsilon^{-\chi_1(d_{D}-k_{0,1}-\delta_D)}\frac{\tilde{R}_{D}(im)}{\tilde{P}_m(\tau)}\left[\tau^{d_{D,q_1,q_2}}\star_{\kappa_1}(\tau^{\kappa_1 q_2}\omega_1)\right]\right\|_{(\nu,\beta,\mu,\chi_1,\alpha,\kappa_1,\epsilon)}\\
\le\frac{C'_3}{C_{\tilde{P}}(r_{\tilde{Q},\tilde{R}_D})^{\frac{1}{\delta_D\kappa_1}}}|\epsilon_0|^{(\chi_1+\alpha)\kappa_1\left(\frac{d_{d_{D},q_1,q_2}}{\kappa_1}+q_2-\delta_D\right)-\chi_1(d_D-k_{0,1}-\delta_D)}\left\|\omega_1\right\|_{(\nu,\beta,\mu,\chi_1,\alpha,\kappa_1,\epsilon)},
\end{multline}

for every $q_1\ge 1$ and $q_2\ge 1$ with $q_1+q_2=\delta_D$. In addition to that, it holds

\begin{equation}\label{e795}
\left\|\frac{\tilde{R}_{D}(im)}{\tilde{P}_m(\tau)}\left[\tau^{\kappa_1(\delta_D-p)}\star_{\kappa_1}(\tau^{\kappa_1 p}\omega_1)\right]\right\|_{(\nu,\beta,\mu,\chi_1,\alpha,\kappa_1,\epsilon)}\\
\le\frac{C'_3}{C_{\tilde{P}}(r_{\tilde{Q},\tilde{R}_D})^{\frac{1}{\delta_D\kappa_1}}}|\epsilon_0|^{\chi_1+\alpha}\left\|\omega_1\right\|_{(\nu,\beta,\mu,\chi_1,\alpha,\kappa_1,\epsilon)},
\end{equation}
for all $1\le p\le \delta_D-1$.

We choose $r_{\tilde{Q},\tilde{R}_D}$ and $\varpi$ such that

\begin{multline}\label{e816}
|\epsilon_0|^{\Delta_D+\alpha(\delta_D-d_D)+\beta}\left[\prod_{d=0}^{\delta_D-1}|\gamma_1-d|\frac{C_3}{C_{\tilde{P}}(r_{\tilde{Q},\tilde{R}_D})^{\frac{1}{\delta_D\kappa_1}}\Gamma\left(\frac{d_{D,\delta_D,0}}{\kappa_1}\right)}|\epsilon_0|^{\alpha(d_{D}-k_{0,1}-\delta_D)}\varpi\right.\\
+\sum_{q_1+q_2=\delta_D,q_1\ge1,q_2\ge1}\frac{\delta_D!}{q_1!q_2!}\prod_{d=0}^{q_1-1}|\gamma_1-d|\left(\frac{C'_3\kappa_1^{q_2}}{C_{\tilde{P}}(r_{\tilde{Q},\tilde{R}_D})^{\frac{1}{\delta_D\kappa_1}}\Gamma\left(\frac{d_{D,q_1,q_2}}{\kappa_1}\right)}\right.\\
\times |\epsilon_0|^{(\chi_1+\alpha)\kappa_1\left(\frac{d_{d_{D},q_1,q_2}}{\kappa_1}+q_2-\delta_D+\frac{1}{\kappa_1}\right)-\chi_1(d_D-k_{0,1}-\delta_D)}\varpi+\sum_{1\le p\le q_2-1}|A_{q_2,p}|\frac{C'_3\kappa_1^p}{C_{\tilde{P}}(r_{\tilde{Q},\tilde{R}_D})^{\frac{1}{\delta_D\kappa_1}}\Gamma\left(\frac{d_{D,q_1,q_2}}{\kappa_1}+q_2-p\right)}\\
\left.\left.\times|\epsilon_0|^{(\chi_1+\alpha)\kappa_1\left(\frac{d_{d_{D},q_1,q_2}}{\kappa_1}+q_2-\delta_D+\frac{1}{\kappa_1}\right)-\chi_1(d_D-k_{0,1}-\delta_D)}\varpi\right)\right]\\
+\sum_{1\le p\le \delta_D-1}|A_{\delta_D,p}|\frac{C'_3\kappa_1^p}{C_{\tilde{P}}(r_{\tilde{Q},\tilde{R}_D})^{\frac{1}{\delta_D\kappa_1}}\Gamma\left(\delta_D-p\right)}|\epsilon_0|^{\chi_1+\alpha}\varpi\le\frac{\varpi}{4}.
\end{multline}

The choice in (\ref{e816}) allows to guarantee from (\ref{e793}), (\ref{e794}) and (\ref{e795}) that
\begin{equation}\label{e823}
\left\|\mathcal{H}_\epsilon^3(\omega_1)\right\|_{(\nu,\beta,\mu,\chi_1,\alpha,\kappa_1,\epsilon)}\le \frac{\varpi}{4}.
\end{equation}

We finally give upper bounds for the elements involved in $\mathcal{H}_\epsilon^4(\omega_1)$.

Let
$$h_2(\tau,m)=\tau^{\kappa_1-1}\int_{0}^{\tau^{\kappa_1}}\int_{-\infty}^{\infty}\omega_1((\tau^{\kappa_1}-s')^{1/\kappa_1},m-m_1)(\omega_1\star_{\kappa_1}^{E}\omega_1)((s')^{1/\kappa_1},m_1)\frac{1}{(\tau^{\kappa_1}-s')s'}ds'dm_1.$$

Regarding Proposition~\ref{prop3}, one has
\begin{equation}\label{e840}
\left\|h_2(\tau,m)\right\|_{(\nu,\beta,\mu,\chi_1,\alpha,\kappa_1,\epsilon)}\le\frac{C_4}{|\epsilon|^{\chi_1+\alpha}}\left\|\omega_1\right\|_{(\nu,\beta,\mu,\chi_1,\alpha,\kappa_1,\epsilon)}\left\| \omega_1\star_{\kappa_1}^{E}\omega_1 \right\|_{(\nu,\beta,\mu,\chi_1,\alpha,\kappa_1,\epsilon)}.
\end{equation}

Moreover, in view of Corollary~\ref{coro3}, we get
\begin{equation}\label{e852}
\left\| \omega_1\star_{\kappa_1}^{E}\omega_1 \right\|_{(\nu,\beta,\mu,\chi_1,\alpha,\kappa_1,\epsilon)}\le C_4\left\| \omega_1 \right\|^2_{(\nu,\beta,\mu,\chi_1,\alpha,\kappa_1,\epsilon)}
\end{equation}

Following the same argument as in (\ref{e728}), in view of (\ref{e840}), (\ref{e852}), and from Proposition~\ref{prop1}.2, we have that

\begin{multline}\label{e860}
\left\|\epsilon^{-\chi_1(k_{\ell,3}+2\gamma_1-k_{0,1})}\frac{\tilde{Q}(im)}{\tilde{P}_m(\tau)}\left[\tau^{k_{\ell,3}+2\gamma_1-k_{0,1}}\star_{\kappa_1}(\tau h_2(\tau,m))\right]\right\|_{(\nu,\beta,\mu,\chi_1,\alpha,\kappa_1,\epsilon)}\\
\le \frac{C'_3}{C_{\tilde{P}}(r_{\tilde{Q},\tilde{R}_D})^{\frac{1}{\delta_D\kappa_1}}}|\epsilon_0|^{(\chi_1+\alpha)\kappa_1\left(\frac{k_{\ell,3}+2\gamma_1-k_{0,1}}{\kappa_1}+\frac{1}{\kappa_1}\right)-\chi_1(k_{\ell,3}+2\gamma_1-k_{0,1})-(\chi_1+\alpha)\kappa_1(\delta_D-\frac{1}{\kappa_1})}\left\|h_2(\tau,m)\right\|_{(\nu,\beta,\mu,\chi_1,\alpha,\kappa_1,\epsilon)}\\
\le \frac{(C_4)^2C'_3}{C_{\tilde{P}}(r_{\tilde{Q},\tilde{R}_D})^{\frac{1}{\delta_D\kappa_1}}}|\epsilon_0|^{(\chi_1+\alpha)(k_{\ell,3}+2\gamma_1-k_{0,1}-\kappa_1\delta_D+1)-\chi_1(k_{\ell,3}+2\gamma_1-k_{0,1})}\left\|\omega_1\right\|^3_{(\nu,\beta,\mu,\chi_1,\alpha,\kappa_1,\epsilon)}
\end{multline}

We recall that Lemma~\ref{lema887} holds, and hypothesis (\ref{e887}). We choose $r_{\tilde{Q},\tilde{R}_D}$ and $\varpi$ such that

\begin{multline}\label{e868}
\sum_{\ell=0}^{s_3}\frac{|a_{\ell,3}|}{\Gamma\left(\frac{k_{\ell,3}+2\gamma_1-k_{0,1}}{\kappa_1}\right)}\frac{(C_4)^2C'_3}{C_{\tilde{P}}(r_{\tilde{Q},\tilde{R}_D})^{\frac{1}{\delta_D\kappa_1}}}|\epsilon_0|^{(\chi_1+\alpha)(k_{\ell,3}+2\gamma_1-k_{0,1}-\kappa_1\delta_D+1)-\chi_1(k_{\ell,3}+2\gamma_1-k_{0,1})}\varpi^3\\
+\sum_{\ell=s_3+1}^{M_3}\frac{|a_{\ell,3}|}{\Gamma\left(\frac{k_{\ell,3}+2\gamma_1-k_{0,1}}{\kappa_1}\right)}\frac{(C_4)^2C'_3}{C_{\tilde{P}}(r_{\tilde{Q},\tilde{R}_D})^{\frac{1}{\delta_D\kappa_1}}}|\epsilon_0|^{m_{\ell,3}+3\beta-\alpha k_{\ell,3}+(\chi_1+\alpha)(k_{\ell,3}+2\gamma_1-k_{0,1}-\kappa_1\delta_D+1)-\chi_1(k_{\ell,3}+2\gamma_1-k_{0,1})}\varpi^3\\
\le\frac{\varpi}{4}
\end{multline}

The choice in (\ref{e868}) allows to guarantee from (\ref{e860}) that
\begin{equation}\label{e874}
\left\|\mathcal{H}_\epsilon^4(\omega_1)\right\|_{(\nu,\beta,\mu,\chi_1,\alpha,\kappa_1,\epsilon)}\le \frac{\varpi}{4}.
\end{equation}

Observe that (\ref{e887}) and (\ref{e218}) imply 
$$(\chi_1 + \alpha)( k_{\ell,3}+2\gamma_1-k_{0,1}-\kappa_1\delta_D+1) - \chi_1(k_{\ell,3}+2\gamma_1-k_{0,1}) \geq 0,$$
for every $\ell\in\{0,\ldots,M_3\}$.

In view of (\ref{e704}), (\ref{e784}), (\ref{e823}) and (\ref{e868}) we conclude the first part of the proof of Lemma~\ref{lema568}, namely, the existence of $\varpi>0$ such that $\mathcal{H}_{\epsilon}$ sends $\bar{B}(0,\varpi)\subseteq F^{d}_{(\nu,\beta,\mu,\chi_1,\alpha,\kappa_1,\epsilon)}$ into itself.

We proceed to give proof for (\ref{e442}). For $\epsilon\in D(0,\epsilon_0)\setminus\{0\}$ fixed above, we take $\omega_1,\omega_2\in \bar{B}(0,\varpi)\subseteq F^{d}_{(\nu,\beta,\mu,\chi_1,\alpha,\kappa_1,\epsilon)}$.

We start with $\mathcal{H}_\epsilon^1$. Analogous arguments as in the first part of the proof leading to the upper bounds for $\mathcal{H}_\epsilon^1$, we get that

\begin{multline}\label{e626b}
\left\|\epsilon^{-\chi_1(k_{\ell,1}-k_{0,1})}\frac{\tilde{Q}(im)}{\tilde{P}_m(\tau)}(\tau^{k_{\ell,1}-k_{0,1}}\star_{\kappa_1}(\omega_1-\omega_2))\right\|_{(\nu,\beta,\mu,\chi_1,\alpha,\kappa_1,\epsilon)}\\
\le \frac{C_3}{C_{\tilde{P}}(r_{\tilde{Q},\tilde{R}_D})^{\frac{1}{\delta_D \kappa_1}}}|\epsilon|^{\alpha(k_{\ell,1}-k_{0,1})}\left\|\omega_1-\omega_2\right\|_{(\nu,\beta,\mu,\chi_1,\alpha,\kappa_1,\epsilon)},
\end{multline} 

\begin{multline}\label{e627b}
\left\|\epsilon^{-\chi_1(k_{\ell,2}-k_{0,2})}\frac{\tilde{Q}(im)}{\tilde{P}_m(\tau)}(\tau^{k_{\ell,2}-k_{0,2}}\star_{\kappa_1}(\omega_1-\omega_2))\right\|_{(\nu,\beta,\mu,\chi_1,\alpha,\kappa_1,\epsilon)}\\
\le \frac{C_3}{C_{\tilde{P}}(r_{\tilde{Q},\tilde{R}_D})^{\frac{1}{\delta_D \kappa_1}}}|\epsilon|^{\alpha(k_{\ell,2}-k_{0,2})}\left\|\omega_1-\omega_2\right\|_{(\nu,\beta,\mu,\chi_1,\alpha,\kappa_1,\epsilon)},
\end{multline}

\begin{multline}\label{e628b}
\left\|\epsilon^{-\chi_1(k_{\ell,3}+k_{0,1}-2k_{0,2})}\frac{\tilde{Q}(im)}{\tilde{P}_m(\tau)}(\tau^{k_{\ell,3}+k_{0,1}-2k_{0,2}}\star_{\kappa_1}(\omega_1-\omega_2))\right\|_{(\nu,\beta,\mu,\chi_1,\alpha,\kappa_1,\epsilon)}\\
\le \frac{C_3}{C_{\tilde{P}}(r_{\tilde{Q},\tilde{R}_D})^{\frac{1}{\delta_D \kappa_1}}}|\epsilon|^{\alpha(k_{\ell,3}+k_{0,1}-2k_{0,2})}\left\|\omega_1-\omega_2\right\|_{(\nu,\beta,\mu,\chi_1,\alpha,\kappa_1,\epsilon)},
\end{multline}

and for every $j\ge 1$,
 
\begin{multline}\label{e629b}
\left\|\epsilon^{-\chi_1(k_{\ell,2}-k_{0,2}+j)}\frac{\tilde{Q}(im)}{\tilde{P}_m(\tau)}(\tau^{k_{\ell,2}-k_{0,2}+j}\star_{\kappa_1}(\omega_1-\omega_2))\right\|_{(\nu,\beta,\mu,\chi_1,\alpha,\kappa_1,\epsilon)}\\
\le \frac{\hat{C}_3A_3^j\Gamma\left(\frac{k_{\ell,2}-k_{0,2}+j}{\kappa_1}\right)}{C_{\tilde{P}}(r_{\tilde{Q},\tilde{R}_D})^{\frac{1}{\delta_D \kappa_1}}}|\epsilon|^{\alpha(k_{\ell,2}-k_{0,2}+j)}\left\|\omega_1-\omega_2\right\|_{(\nu,\beta,\mu,\chi_1,\alpha,\kappa_1,\epsilon)},
\end{multline}

and

\begin{multline}\label{e630b}
\left\|\epsilon^{-\chi_1(k_{\ell,3}+k_{0,1}-2k_{0,2}+j)}\frac{\tilde{Q}(im)}{\tilde{P}_m(\tau)}(\tau^{k_{\ell,3}+k_{0,1}-2k_{0,2}+j}\star_{\kappa_1}(\omega_1-\omega_2))\right\|_{(\nu,\beta,\mu,\chi_1,\alpha,\kappa_1,\epsilon)}\\
\le \frac{\hat{C}_3A_3^j\Gamma\left(\frac{k_{\ell,3}+k_{0,1}-2k_{0,2}+j}{\kappa_1}\right)}{C_{\tilde{P}}(r_{\tilde{Q},\tilde{R}_D})^{\frac{1}{\delta_D \kappa_1}}}|\epsilon|^{\alpha(k_{\ell,3}+k_{0,1}-2k_{0,2}+j)}\left\|\omega_1-\omega_2\right\|_{(\nu,\beta,\mu,\chi_1,\alpha,\kappa_1,\epsilon)}.
\end{multline}

We also have 

\begin{multline}\label{e631b}
\left\|\epsilon^{-\chi_1(k_{\ell,3}+k_{0,1}-2k_{0,2})}\frac{\tilde{Q}(im)}{\tilde{P}_{m}(\tau)}\left[\tau^{k_{\ell,3}+k_{0,1}-2k_{0,2}+j}\star_{\kappa_1}\mathcal{B}_{\kappa_1}J_1(\tau,\epsilon)\star_{\kappa_1}\mathcal{B}_{\kappa_1}J_1(\tau,\epsilon)\star_{\kappa_1}(\omega_1-\omega_2)\right]\right\|_{(\nu,\beta,\mu,\chi_1,\alpha,\kappa_1,\epsilon)}\hfill\\
\le \frac{\hat{C}_3A_3^j\Gamma\left(\frac{k_{\ell,3}+k_{0,1}-2k_{0,2}+j}{\kappa_1}\right)}{C_{\tilde{P}}(r_{\tilde{Q},\tilde{R}_D})^{\frac{1}{\delta_D \kappa_1}}}|\epsilon|^{\alpha(k_{\ell,3}+k_{0,1}-2k_{0,2}+j)}\left\|\omega_1-\omega_2\right\|_{(\nu,\beta,\mu,\chi_1,\alpha,\kappa_1,\epsilon)}.\hfill
\end{multline}

We choose large enough $r_{\tilde{Q},\tilde{R}_D}>0$ and $\varpi>0$ such that

\begin{multline}\label{e685b}
\sum_{\ell=1}^{s_1}\frac{|a_{\ell,1}|}{\Gamma\left(\frac{k_{\ell,1}-k_{0,1}}{\kappa_1}\right)}|\epsilon_0|^{\alpha(k_{\ell,1}-k_{0,1})}\frac{C_3}{C_{\tilde{P}}(r_{\tilde{Q},\tilde{R}_D})^{\frac{1}{\delta_D\kappa_1}}} \hfill\\
\hfill+\sum_{\ell=s_1+1}^{M_1}\frac{|a_{\ell,1}|}{\Gamma\left(\frac{k_{\ell,1}-k_{0,1}}{\kappa_1}\right)}|\epsilon_0|^{m_{\ell,1}+\beta-\alpha k_{\ell,1}+\alpha(k_{\ell,1}-k_{0,1})}\frac{C_3}{C_{\tilde{P}}(r_{\tilde{Q},\tilde{R}_D})^{\frac{1}{\delta_D\kappa_1}}} \\
+\sum_{\ell=1}^{s_2}\frac{2|a_{\ell,2}a_{0,1}|}{|a_{0,2}|}\frac{|\epsilon_0|^{\alpha(k_{\ell,2}-k_{0,2})}}{\Gamma\left(\frac{k_{\ell,2}-k_{0,2}}{\kappa_1}\right)}\frac{C_3}{C_{\tilde{P}}(r_{\tilde{Q},\tilde{R}_D})^{\frac{1}{\delta_D\kappa_1}}} \hfill\\
\hfill+\sum_{\ell=s_2+1}^{M_2}\frac{2|a_{\ell,2}a_{0,1}|}{|a_{0,2}|}\frac{|\epsilon_0|^{m_{\ell,2}+2\beta-\alpha k_{\ell,2}+\alpha(k_{\ell,2}-k_{0,2})}}{\Gamma\left(\frac{k_{\ell,2}-k_{0,2}}{\kappa_1}\right)}\frac{C_3}{C_{\tilde{P}}(r_{\tilde{Q},\tilde{R}_D})^{\frac{1}{\delta_D\kappa_1}}} \\
+\sum_{\ell=0}^{s_2}\frac{2|a_{\ell,2}a_{0,1}|}{|a_{0,2}|}\sum_{j\ge1}|J_j|\frac{|\epsilon_0|^{\alpha(k_{\ell,2}-k_{0,2}+j)}}{\Gamma\left(\frac{k_{\ell,2}-k_{0,2}}{\kappa_1}\right)}\frac{\hat{C}_3A_3^j}{C_{\tilde{P}}(r_{\tilde{Q},\tilde{R}_D})^{\frac{1}{\delta_D\kappa_1}}} \hfill\\
\hfill+\sum_{\ell=s_2+1}^{M_2}\frac{2|a_{\ell,2}a_{0,1}|}{|a_{0,2}|}\sum_{j\ge1}|J_j|\frac{|\epsilon_0|^{m_{\ell,2}+2\beta-\alpha k_{\ell,2}+\alpha(k_{\ell,2}-k_{0,2}+j)}}{\Gamma\left(\frac{k_{\ell,2}-k_{0,2}}{\kappa_1}\right)}\frac{\hat{C}_3A_3^j}{C_{\tilde{P}}(r_{\tilde{Q},\tilde{R}_D})^{\frac{1}{\delta_D\kappa_1}}} \\
+\sum_{\ell=0}^{s_3}\frac{3|a_{0,1}|^2|a_{\ell,3}|}{|a_{0,2}|^2}\frac{|\epsilon_0|^{\alpha(k_{\ell,3}+k_{0,1}-2k_{0,2})}}{\Gamma\left(\frac{k_{\ell,3}+k_{0,1}-2k_{0,2}}{\kappa_1}\right)}\frac{C_3}{C_{\tilde{P}}(r_{\tilde{Q},\tilde{R}_D})^{\frac{1}{\delta_D\kappa_1}}} \hfill\\
\hfill+\sum_{\ell=0}^{s_3}\frac{6|a_{0,1}|^2|a_{\ell,3}|}{|a_{0,2}|^2}\sum_{j\ge1}|J_j|\frac{|\epsilon_0|^{\alpha(k_{\ell,3}+k_{0,1}-2k_{0,2}+j)}}{\Gamma\left(\frac{k_{\ell,3}+k_{0,1}-2k_{0,2}}{\kappa_1}\right)}\frac{\hat{C}_3A_3^j}{C_{\tilde{P}}(r_{\tilde{Q},\tilde{R}_D})^{\frac{1}{\delta_D\kappa_1}}} \\
+\sum_{\ell=0}^{s_3}\frac{3|a_{0,1}|^2|a_{\ell,3}|}{|a_{0,2}|^2}\sum_{j\ge1}|\tilde{J}_j|\frac{|\epsilon_0|^{\alpha(k_{\ell,3}+k_{0,1}-2k_{0,2}+j)}}{\Gamma\left(\frac{k_{\ell,3}+k_{0,1}-2k_{0,2}}{\kappa_1}\right)}\frac{\hat{C}_3A_3^j}{C_{\tilde{P}}(r_{\tilde{Q},\tilde{R}_D})^{\frac{1}{\delta_D\kappa_1}}} \hfill\\
\hfill+\sum_{\ell=s_3+1}^{M_3}\frac{3|a_{0,1}|^2|a_{\ell,3}|}{|a_{0,2}|^2}\frac{|\epsilon_0|^{m_{\ell,3}+3\beta-\alpha k_{\ell,3}+\alpha(k_{\ell,3}+k_{0,1}-2k_{0,2})}}{\Gamma\left(\frac{k_{\ell,3}+k_{0,1}-2k_{0,2}}{\kappa_1}\right)}\frac{C_3}{C_{\tilde{P}}(r_{\tilde{Q},\tilde{R}_D})^{\frac{1}{\delta_D\kappa_1}}} \\
+\sum_{\ell=s_3+1}^{M_3}\frac{6|a_{0,1}|^2|a_{\ell,3}|}{|a_{0,2}|^2}\sum_{j\ge1}|J_j|\frac{|\epsilon_0|^{m_{\ell,3}+3\beta-\alpha k_{\ell,3}+\alpha(k_{\ell,3}+k_{0,1}-2k_{0,2}+j)}}{\Gamma\left(\frac{k_{\ell,3}+k_{0,1}-2k_{0,2}}{\kappa_1}\right)}\frac{\hat{C}_3A_3^j}{C_{\tilde{P}}(r_{\tilde{Q},\tilde{R}_D})^{\frac{1}{\delta_D\kappa_1}}} \hfill\\
\hfill+\sum_{\ell=s_3+1}^{M_3}\frac{3|a_{0,1}|^2|a_{\ell,3}|}{|a_{0,2}|^2}\sum_{j\ge1}|\tilde{J}_j|\frac{|\epsilon_0|^{m_{\ell,3}+3\beta-\alpha k_{\ell,3}+\alpha(k_{\ell,3}+k_{0,1}-2k_{0,2}+j)}}{\Gamma\left(\frac{k_{\ell,3}+k_{0,1}-2k_{0,2}}{\kappa_1}\right)}\frac{\hat{C}_3A_3^j}{C_{\tilde{P}}(r_{\tilde{Q},\tilde{R}_D})^{\frac{1}{\delta_D\kappa_1}}} \le \frac{1}{8}.
\end{multline}

As a result, we can affirm that
\begin{equation}\label{e704b} 
\left\|\mathcal{H}^1_\epsilon(\omega_1)-\mathcal{H}^1_\epsilon(\omega_1)\right\|_{(\nu,\beta,\mu,\chi_1,\alpha,\kappa_1,\epsilon)}\le \frac{1}{8}\left\|\omega_1-\omega_2\right\|_{(\nu,\beta,\mu,\chi_1,\alpha,\kappa_1,\epsilon)}.
\end{equation}

We proceed with the upper bounds associated to $\mathcal{H}_{\epsilon}^2$. For this purpose, we observe that
\begin{multline}\label{e947}
\omega_1((\tau^{\kappa_1}-s')^{1/\kappa_1},m-m_1)\omega_1((s')^{1/\kappa_1},m_1)-\omega_2((\tau^{\kappa_1}-s')^{1/\kappa_1},m-m_1)\omega_2((s')^{1/\kappa_1},m_1)\\
=\left(\omega_1((\tau^{\kappa_1}-s')^{1/\kappa_1},m-m_1)-\omega_2((\tau^{\kappa_1}-s')^{1/\kappa_1},m-m_1)\right)\omega_1((s')^{1/\kappa_1},m_1)\\
+\omega_2((\tau^{\kappa_1}-s')^{1/\kappa_1},m-m_1)\left(\omega_1((s')^{1/\kappa_1},m_1)-\omega_2((s')^{1/\kappa_1},m_1)\right).
\end{multline}

For $j=1,2$, we define
$$h_{1j}(\tau,m)=\tau^{\kappa_1-1}\int_{0}^{\tau^{\kappa_1}}\int_{-\infty}^{\infty}\omega_j((\tau^{\kappa_1}-s')^{1/\kappa_1},m-m_1)\omega_j((s')^{1/\kappa_1},m_1)\frac{1}{(\tau^{\kappa_1}-s')s'}ds'dm_1.$$

The expression in (\ref{e947}) and Proposition~\ref{prop3} yield
\begin{multline}\label{e957}
\left\|h_{11}(\tau,m)-h_{12}(\tau,m)\right\|_{(\nu,\beta,\mu,\chi_1,\alpha,\kappa_1,\epsilon)}\le \frac{C_4}{|\epsilon|^{\chi_1+\alpha}}(\left\|\omega_1\right\|_{(\nu,\beta,\mu,\chi_1,\alpha,\kappa_1,\epsilon)}+\left\|\omega_2\right\|_{(\nu,\beta,\mu,\chi_1,\alpha,\kappa_1,\epsilon)})\\
\times\left\|\omega_1-\omega_2\right\|_{(\nu,\beta,\mu,\chi_1,\alpha,\kappa_1,\epsilon)}
\end{multline}

The estimates in (\ref{e726}), (\ref{e727}), (\ref{e728}), (\ref{e729}) and (\ref{e730}) together with (\ref{e957}) provide the following bounds.

\begin{multline}\label{e726b}
\left\|\epsilon^{-\chi_1(d_\ell-k_{0,1}-\delta_\ell)}\frac{\tilde{R}_{\ell}(im)}{\tilde{P}_{m}(\tau)}\left[\tau^{d_{\ell,q_1,q_2}}\star_{\kappa_1}(\tau^{\kappa_1 q_2}(\omega_1-\omega_2))\right]\right\|_{(\nu,\beta,\mu,\chi_1,\alpha,\kappa_1,\epsilon)}\\
\le \frac{C'_3}{C_{\tilde{P}}(r_{\tilde{Q},\tilde{R}_D})^{\frac{1}{\delta_D\kappa_1}}}|\epsilon|^{(\chi_1+\alpha)\kappa_1\left(\frac{d_{\ell,q_1,q_2}}{\kappa_1}+q_2-\delta_D+\frac{1}{\kappa_1}\right)-\chi_1(d_{\ell}-k_{0,1}-\delta_\ell)}\left\|\omega_1-\omega_2\right\|_{(\nu,\beta,\mu,\chi_1,\alpha,\kappa_1,\epsilon)},
\end{multline}
for every $q_1\ge0$, $q_2\ge 1$ such that $q_1+q_2=\delta_\ell$, and also
\begin{multline}\label{e727b}
\left\|\epsilon^{-\chi_1(d_\ell-k_{0,1}-\delta_\ell)}\frac{\tilde{R}_{\ell}(im)}{\tilde{P}_{m}(\tau)}\left[\tau^{d_{\ell,q_1,q_2}+\kappa_1(q_2-p)}\star_{\kappa_1}(\tau^{\kappa_1 p}(\omega_1-\omega_2))\right]\right\|_{(\nu,\beta,\mu,\chi_1,\alpha,\kappa_1,\epsilon)}\\
\le \frac{C'_3}{C_{\tilde{P}}(r_{\tilde{Q},\tilde{R}_D})^{\frac{1}{\delta_D\kappa_1}}}|\epsilon|^{(\chi_1+\alpha)\kappa_1\left(\frac{d_{\ell,q_1,q_2}}{\kappa_1}+q_2-\delta_D+\frac{1}{\kappa_1}\right)-\chi_1(d_{\ell}-k_{0,1}-\delta_\ell)}\left\|\omega_1-\omega_2\right\|_{(\nu,\beta,\mu,\chi_1,\alpha,\kappa_1,\epsilon)},
\end{multline}
for every $1\le p\le q_2-1$, and

\begin{multline}\label{e728b}
\left\|\epsilon^{-\chi_1(k_{\ell,2}+\gamma_1-k_{0,1})}\frac{\tilde{Q}(im)}{\tilde{P}_m(\tau)}\left[\tau^{k_{\ell,2}+\gamma_1-k_{0,1}}\star_{\kappa_1}(\tau[ h_{11}(\tau,m)-h_{12}(\tau,m)])\right]\right\|_{(\nu,\beta,\mu,\chi_1,\alpha,\kappa_1,\epsilon)}\\
\le \frac{C'_3}{C_{\tilde{P}}(r_{\tilde{Q},\tilde{R}_D})^{\frac{1}{\delta_D\kappa_1}}}|\epsilon|^{(\chi_1+\alpha)\kappa_1\left(\frac{k_{\ell,2}+\gamma_1-k_{0,1}}{\kappa_1}+\frac{1}{\kappa_1}\right)-\chi_1(k_{\ell,2}+\gamma_1-k_{0,1})-(\chi_1+\alpha)\kappa_1(\delta_D-\frac{1}{\kappa_1})}\left\|h_{11}-h_{12}\right\|_{(\nu,\beta,\mu,\chi_1,\alpha,\kappa_1,\epsilon)}\\
\le \frac{C_4C'_3}{C_{\tilde{P}}(r_{\tilde{Q},\tilde{R}_D})^{\frac{1}{\delta_D\kappa_1}}}|\epsilon|^{(\chi_1+\alpha)(k_{\ell,2}+\gamma_1-k_{0,1}-\kappa_1\delta_D+1)-\chi_1(k_{\ell,2}+\gamma_1-k_{0,1})}\\
\times(\left\|\omega_1\right\|_{(\nu,\beta,\mu,\chi_1,\alpha,\kappa_1,\epsilon)}+\left\|\omega_2\right\|_{(\nu,\beta,\mu,\chi_1,\alpha,\kappa_1,\epsilon)})\left\|\omega_1-\omega_2\right\|_{(\nu,\beta,\mu,\chi_1,\alpha,\kappa_1,\epsilon)},
\end{multline} 

\begin{multline}\label{e729b}
\left\|\epsilon^{-\chi_1(k_{\ell,3}+\gamma_1-k_{0,2})}\frac{\tilde{Q}(im)}{\tilde{P}_m(\tau)}\left[\tau^{k_{\ell,3}+\gamma_1-k_{0,2}}\star_{\kappa_1}(\tau[ h_{11}(\tau,m)-h_{12}(\tau,m)])\right]\right\|_{(\nu,\beta,\mu,\chi_1,\alpha,\kappa_1,\epsilon)}\\
\le \frac{C_4C'_3}{C_{\tilde{P}}(r_{\tilde{Q},\tilde{R}_D})^{\frac{1}{\delta_D\kappa_1}}}|\epsilon|^{(\chi_1+\alpha)(k_{\ell,3}+\gamma_1-k_{0,2}-\kappa_1\delta_D+1)-\chi_1(k_{\ell,3}+\gamma_1-k_{0,2})}\\
\times(\left\|\omega_1\right\|_{(\nu,\beta,\mu,\chi_1,\alpha,\kappa_1,\epsilon)}+\left\|\omega_2\right\|_{(\nu,\beta,\mu,\chi_1,\alpha,\kappa_1,\epsilon)})\left\|\omega_1-\omega_2\right\|_{(\nu,\beta,\mu,\chi_1,\alpha,\kappa_1,\epsilon)},
\end{multline} 

, and

\begin{multline}\label{e730b}
\left\|\epsilon^{-\chi_1(k_{\ell,3}+\gamma_1-k_{0,2}+j)}\frac{\tilde{Q}(im)}{\tilde{P}_m(\tau)}\left[\tau^{k_{\ell,3}+\gamma_1-k_{0,2}+j}\star_{\kappa_1}(\tau[ h_{11}(\tau,m)-h_{12}(\tau,m)])\right]\right\|_{(\nu,\beta,\mu,\chi_1,\alpha,\kappa_1,\epsilon)}\\
\le \frac{C_4\hat{C}_3A_3^j\Gamma\left(\frac{k_{\ell,3}+\gamma_1-k_{0,2}}{\kappa_1}\right)}{C_{\tilde{P}}(r_{\tilde{Q},\tilde{R}_D})^{\frac{1}{\delta_D\kappa_1}}}|\epsilon|^{(\chi_1+\alpha)(k_{\ell,3}+\gamma_1-k_{0,2}+j-\kappa_1\delta_D+1)-\chi_1(k_{\ell,3}+\gamma_1-k_{0,2}+j)}\\
\times(\left\|\omega_1\right\|_{(\nu,\beta,\mu,\chi_1,\alpha,\kappa_1,\epsilon)}+\left\|\omega_2\right\|_{(\nu,\beta,\mu,\chi_1,\alpha,\kappa_1,\epsilon)})\left\|\omega_1-\omega_2\right\|_{(\nu,\beta,\mu,\chi_1,\alpha,\kappa_1,\epsilon)},
\end{multline} 
for every $j\ge1$.

One can choose $r_{\tilde{Q},\tilde{R}_D}>0$ and $\varpi>0$ such that the following condition holds:

\begin{multline}\label{e766b}
\sum_{\ell=1}^{D-1}|\epsilon_0|^{\Delta_{\ell}+\alpha(\delta_{\ell}-d_{\ell})+\beta}\left[\prod_{d=0}^{\delta_\ell-1}|\gamma_1-d|\frac{C_3}{C_{\tilde{P}}(r_{\tilde{Q},\tilde{R}_D})^{\frac{1}{\delta_D\kappa_1}}\Gamma\left(\frac{d_{\ell,\delta_\ell,0}}{\kappa_1}\right)}|\epsilon_0|^{\alpha(d_\ell-k_{0,1}-\delta_\ell)}\right.\\
+\sum_{q_1+q_2=\delta_\ell,q_2\ge1}\frac{\delta_\ell!}{q_1!q_2!}\prod_{d=0}^{q_1-1}|\gamma_1-d|\left(\frac{C'_3\kappa_1^{q_2}}{C_{\tilde{P}}(r_{\tilde{Q},\tilde{R}_D})^{\frac{1}{\delta_D\kappa_1}}\Gamma\left(\frac{d_{\ell,q_1,q_2}}{\kappa_1}\right)}\right.\hfill\\
\hfill\times |\epsilon_0|^{(\chi_1+\alpha)\kappa_1\left(\frac{d_{\ell,q_1,q_2}}{\kappa_1}+q_2-\delta_D+\frac{1}{\kappa_1}\right)-\chi_1(d_\ell-k_{0,1}-\delta_\ell)}\\
+\sum_{1\le p\le q_2-1}|A_{q_2,p}|\frac{C'_3\kappa_1^p}{C_{\tilde{P}}(r_{\tilde{Q},\tilde{R}_D})^{\frac{1}{\delta_D\kappa_1}}\Gamma\left(\frac{d_{\ell,q_1,q_2}}{\kappa_1}+q_2-p\right)}\hfill\\
\hfill\left.\left. \times |\epsilon_0|^{(\chi_1+\alpha)\kappa_1\left(\frac{d_{\ell,q_1,q_2}}{\kappa_1}+q_2-\delta_D+\frac{1}{\kappa_1}\right)-\chi_1(d_\ell-k_{0,1}-\delta_\ell)}   \right)    \right]\\
+\sum_{\ell=0}^{s_2}\frac{|a_{\ell,2}|}{\Gamma\left(\frac{k_{\ell,2}+\gamma_1-k_{0,1}}{\kappa_1}\right)}\frac{C_4C'_3}{C_{\tilde{P}}(r_{\tilde{Q},\tilde{R}_D})^{\frac{1}{\delta_D\kappa_1}}}|\epsilon_0|^{(\chi_1+\alpha)(k_{\ell,2}+\gamma_1-k_{0,1}-\kappa_1\delta_D+1)-\chi_1(k_{\ell,2}+\gamma_1-k_{0,1})}2\varpi\hfill\\
\hfill+\sum_{\ell=s_2+1}^{M_2}\frac{|a_{\ell,2}|}{\Gamma\left(\frac{k_{\ell,2}+\gamma_1-k_{0,1}}{\kappa_1}\right)}\frac{C_4C'_3}{C_{\tilde{P}}(r_{\tilde{Q},\tilde{R}_D})^{\frac{1}{\delta_D\kappa_1}}}|\epsilon_0|^{m_{\ell,2}+2\beta-\alpha k_{\ell,2}+(\chi_1+\alpha)(k_{\ell,2}+\gamma_1-k_{0,1}-\kappa_1\delta_D+1)-\chi_1(k_{\ell,2}+\gamma_1-k_{0,1})}2\varpi\\
+\sum_{\ell=0}^{s_3}\frac{3|a_{\ell,3}a_{0,1}|}{|a_{0,2}|}\frac{C_4C'_3}{\Gamma\left(\frac{k_{\ell,3}+\gamma_1-k_{0,2}}{\kappa_1}\right)}\frac{|\epsilon_0|^{(\chi_1+\alpha)(k_{\ell,3}+\gamma_1-k_{0,2}-\kappa_1\delta_D+1)-\chi_1(k_{\ell,3}+\gamma_1-k_{0,2})}}{C_{\tilde{P}}(r_{\tilde{Q},\tilde{R}_D})^{\frac{1}{\delta_D\kappa_1}}}2\varpi\hfill\\
\hfill+\sum_{\ell=s_3+1}^{M_3}\frac{3|a_{\ell,3}a_{0,1}|}{|a_{0,2}|}\frac{C_4C'_3}{\Gamma\left(\frac{k_{\ell,3}+\gamma_1-k_{0,2}}{\kappa_1}\right)}\frac{|\epsilon_0|^{m_{\ell,3}+3\beta-\alpha k_{\ell,3}+(\chi_1+\alpha)(k_{\ell,3}+\gamma_1-k_{0,2}-\kappa_1\delta_D+1)-\chi_1(k_{\ell,3}+\gamma_1-k_{0,2})}}{C_{\tilde{P}}(r_{\tilde{Q},\tilde{R}_D})^{\frac{1}{\delta_D\kappa_1}}}2\varpi\\
+\sum_{\ell=0}^{s_3}\frac{3|a_{\ell,3}a_{0,1}|}{|a_{0,2}|}\sum_{j\ge1}|J_j|\frac{|\epsilon_0|^{(\chi_1+\alpha)(k_{\ell,3}+\gamma_1-k_{0,2}+j-\kappa_1\delta_D+1)-\chi_1(k_{\ell,3}+\gamma_1-k_{0,2}+j)}}{\Gamma\left(\frac{k_{\ell,3}+\gamma_1-k_{0,2}}{\kappa_1}\right)}\frac{C_4\hat{C}_3A_3^j}{C_{\tilde{P}}(r_{\tilde{Q},\tilde{R}_D})^{\frac{1}{\delta_D\kappa_1}}}2\varpi\\
+\sum_{\ell=0}^{s_3}\frac{3|a_{\ell,3}a_{0,1}|}{|a_{0,2}|}\sum_{j\ge1}|J_j|\frac{|\epsilon_0|^{m_{\ell,3}+3\beta-\alpha k_{\ell,3}+(\chi_1+\alpha)(k_{\ell,3}+\gamma_1-k_{0,2}+j-\kappa_1\delta_D+1)-\chi_1(k_{\ell,3}+\gamma_1-k_{0,2}+j)}}{\Gamma\left(\frac{k_{\ell,3}+\gamma_1-k_{0,2}}{\kappa_1}\right)}\frac{C_4\hat{C}_3A_3^j}{C_{\tilde{P}}(r_{\tilde{Q},\tilde{R}_D})^{\frac{1}{\delta_D\kappa_1}}}2\varpi\\
\le\frac{1}{8}.\hfill.
\end{multline}

We obtain
\begin{equation}\label{e705b} 
\left\|\mathcal{H}^2_\epsilon(\omega_1)-\mathcal{H}^2_\epsilon(\omega_1)\right\|_{(\nu,\beta,\mu,\chi_1,\alpha,\kappa_1,\epsilon)}\le \frac{1}{8}\left\|\omega_1-\omega_2\right\|_{(\nu,\beta,\mu,\chi_1,\alpha,\kappa_1,\epsilon)}.
\end{equation}

We now study the term $\mathcal{H}^3$. Analogous estimates as those stated in the first statement of this part of the proof lead to

\begin{multline}\label{e793b}
\left\|\epsilon^{-\chi_1(d_{D}-k_{0,1}-\delta_D)}\frac{\tilde{R}_{D}(im)}{\tilde{P}_m(\tau)}\left[\tau^{d_{D,\delta_D,0}}\star_{\kappa_1}(\omega_1-\omega_2)\right]\right\|_{(\nu,\beta,\mu,\chi_1,\alpha,\kappa_1,\epsilon)}\\
\le\frac{C_3}{C_{\tilde{P}}(r_{\tilde{Q},\tilde{R}_D})^{\frac{1}{\delta_D\kappa_1}}}|\epsilon_0|^{(\chi_1+\alpha)d_{d_{D},\delta_D,0}-\chi_1(d_D-k_{0,1}-\delta_D)}\left\|\omega_1-\omega_2\right\|_{(\nu,\beta,\mu,\chi_1,\alpha,\kappa_1,\epsilon)}\\
=\frac{C_3}{C_{\tilde{P}}(r_{\tilde{Q},\tilde{R}_D})^{\frac{1}{\delta_D\kappa_1}}}|\epsilon_0|^{\alpha(d_{D}-k_{0,1}-\delta_D)}\left\|(\omega_1-\omega_2)\right\|_{(\nu,\beta,\mu,\chi_1,\alpha,\kappa_1,\epsilon)},
\end{multline}

\begin{multline}\label{e794b}
\left\|\epsilon^{-\chi_1(d_{D}-k_{0,1}-\delta_D)}\frac{\tilde{R}_{D}(im)}{\tilde{P}_m(\tau)}\left[\tau^{d_{D,q_1,q_2}}\star_{\kappa_1}[\tau^{\kappa_1 q_2}(\omega_1-\omega_2)]\right]\right\|_{(\nu,\beta,\mu,\chi_1,\alpha,\kappa_1,\epsilon)}\\
\le\frac{C'_3}{C_{\tilde{P}}(r_{\tilde{Q},\tilde{R}_D})^{\frac{1}{\delta_D\kappa_1}}}|\epsilon_0|^{(\chi_1+\alpha)\kappa_1\left(\frac{d_{d_{D},q_1,q_2}}{\kappa_1}+q_2-\delta_D\right)-\chi_1(d_D-k_{0,1}-\delta_D)}\left\|\omega_1-\omega_2\right\|_{(\nu,\beta,\mu,\chi_1,\alpha,\kappa_1,\epsilon)},
\end{multline}

for every $q_1\ge 1$ and $q_2\ge 1$ with $q_1+q_2=\delta_D$, and

\begin{multline}\label{e795b}
\left\|\frac{\tilde{R}_{D}(im)}{\tilde{P}_m(\tau)}\left[\tau^{\kappa_1(\delta_D-p)}\star_{\kappa_1}(\tau^{\kappa_1 p}(\omega_1-\omega_2))\right]\right\|_{(\nu,\beta,\mu,\chi_1,\alpha,\kappa_1,\epsilon)}\\
\le\frac{C'_3}{C_{\tilde{P}}(r_{\tilde{Q},\tilde{R}_D})^{\frac{1}{\delta_D\kappa_1}}}|\epsilon_0|^{\chi_1+\alpha}\left\|\omega_1-\omega_2\right\|_{(\nu,\beta,\mu,\chi_1,\alpha,\kappa_1,\epsilon)},
\end{multline}
for all $1\le p\le \delta_D-1$.

We choose $r_{\tilde{Q},\tilde{R}_D}$ and $\varpi$ such that

\begin{multline}\label{e816b}
|\epsilon_0|^{\Delta_D+\alpha(\delta_D-d_D)+\beta}\left[\prod_{d=0}^{\delta_D-1}|\gamma_1-d|\frac{C_3}{C_{\tilde{P}}(r_{\tilde{Q},\tilde{R}_D})^{\frac{1}{\delta_D\kappa_1}}\Gamma\left(\frac{d_{D,\delta_D,0}}{\kappa_1}\right)}|\epsilon_0|^{\alpha(d_{D}-k_{0,1}-\delta_D)}\right.\\
+\sum_{q_1+q_2=\delta_D,q_1\ge1,q_2\ge1}\frac{\delta_D!}{q_1!q_2!}\prod_{d=0}^{q_1-1}|\gamma_1-d|\left(\frac{C'_3\kappa_1^{q_2}}{C_{\tilde{P}}(r_{\tilde{Q},\tilde{R}_D})^{\frac{1}{\delta_D\kappa_1}}\Gamma\left(\frac{d_{D,q_1,q_2}}{\kappa_1}\right)}\right.\\
\times |\epsilon_0|^{(\chi_1+\alpha)\kappa_1\left(\frac{d_{d_{D},q_1,q_2}}{\kappa_1}+q_2-\delta_D+\frac{1}{\kappa_1}\right)-\chi_1(d_D-k_{0,1}-\delta_D)}+\sum_{1\le p\le q_2-1}|A_{q_2,p}|\frac{C'_3\kappa_1^p}{C_{\tilde{P}}(r_{\tilde{Q},\tilde{R}_D})^{\frac{1}{\delta_D\kappa_1}}\Gamma\left(\frac{d_{D,q_1,q_2}}{\kappa_1}+q_2-p\right)}\\
\left.\left.\times|\epsilon_0|^{(\chi_1+\alpha)\kappa_1\left(\frac{d_{d_{D},q_1,q_2}}{\kappa_1}+q_2-\delta_D+\frac{1}{\kappa_1}\right)-\chi_1(d_D-k_{0,1}-\delta_D)}\right)\right]\\
+\sum_{1\le p\le \delta_D-1}|A_{\delta_D,p}|\frac{C'_3\kappa_1^p}{C_{\tilde{P}}(r_{\tilde{Q},\tilde{R}_D})^{\frac{1}{\delta_D\kappa_1}}\Gamma\left(\delta_D-p\right)}|\epsilon_0|^{\chi_1+\alpha}\le\frac{1}{8}.
\end{multline}

Then, we get 
\begin{equation}\label{e706b} 
\left\|\mathcal{H}^3_\epsilon(\omega_1)-\mathcal{H}^3_\epsilon(\omega_1)\right\|_{(\nu,\beta,\mu,\chi_1,\alpha,\kappa_1,\epsilon)}\le \frac{1}{8}\left\|\omega_1-\omega_2\right\|_{(\nu,\beta,\mu,\chi_1,\alpha,\kappa_1,\epsilon)}.
\end{equation}

We conclude with the estimation associated to the last term, $\mathcal{H}^4_\epsilon$.

We have

\begin{multline}\label{e947b}
\omega_1((\tau^{\kappa_1}-s')^{1/\kappa_1},m-m_1)[\omega_1\star_{\kappa_1}^E\omega_1]((s')^{1/\kappa_1},m_1)\\-\omega_2((\tau^{\kappa_1}-s')^{1/\kappa_1},m-m_1)[\omega_2\star_{\kappa_1}^E\omega_2]((s')^{1/\kappa_1},m_1)\\
=\left(\omega_1((\tau^{\kappa_1}-s')^{1/\kappa_1},m-m_1)-\omega_2((\tau^{\kappa_1}-s')^{1/\kappa_1},m-m_1)\right)[\omega_1\star_{\kappa_1}^E\omega_1]((s')^{1/\kappa_1},m_1)\\
+\omega_2((\tau^{\kappa_1}-s')^{1/\kappa_1},m-m_1)\left([\omega_1\star_{\kappa_1}^E\omega_1]((s')^{1/\kappa_1},m_1)-[\omega_2\star_{\kappa_1}^E\omega_2]((s')^{1/\kappa_1},m_1)\right).
\end{multline}

For $j=1,2$, we define
$$h_{2j}(\tau,m)=\tau^{\kappa_1-1}\int_{0}^{\tau^{\kappa_1}}\int_{-\infty}^{\infty}\omega_j((\tau^{\kappa_1}-s')^{1/\kappa_1},m-m_1)[\omega_j\star_{\kappa_1}^{E}\omega_j]((s')^{1/\kappa_1},m_1)\frac{1}{(\tau^{\kappa_1}-s')s'}ds'dm_1.$$

In view of (\ref{e947}) and Corollary~\ref{coro3}, it is straight to check that
\begin{multline}\label{e1070}
\left\| [\omega_1\star^{E}_{\kappa_1}\omega_1]-[\omega_2\star^{E}_{\kappa_1}\omega_2]\right\|_{(\nu,\beta,\mu,\chi_1,\alpha,\kappa_1,\epsilon)}\\
\le C_4 \left\|\omega_1-\omega_2\right\|_{(\nu,\beta,\mu,\chi_1,\alpha,\kappa_1,\epsilon)}\left(\left\|\omega_1\right\|_{(\nu,\beta,\mu,\chi_1,\alpha,\kappa_1,\epsilon)}+\left\|\omega_2\right\|_{(\nu,\beta,\mu,\chi_1,\alpha,\kappa_1,\epsilon)}\right)
\end{multline}

Regarding Proposition~\ref{prop3}, (\ref{e852}), (\ref{e1070}) and by (\ref{e947b}) we have
\begin{multline}\label{e957b}
\left\|h_{21}(\tau,m)-h_{22}(\tau,m)\right\|_{(\nu,\beta,\mu,\chi_1,\alpha,\kappa_1,\epsilon)}\le \frac{C_4}{|\epsilon|^{\chi_1+\alpha}}\\
\times\left[ \left\|\omega_1-\omega_2\right\|_{(\nu,\beta,\mu,\chi_1,\alpha,\kappa_1,\epsilon)}\left\|\omega_1\star^{E}_{\kappa_1}\omega_1\right\|_{(\nu,\beta,\mu,\chi_1,\alpha,\kappa_1,\epsilon)}\right.\\
\hfill\left.+\left\|\omega_2\right\|_{(\nu,\beta,\mu,\chi_1,\alpha,\kappa_1,\epsilon)}
\left(\left\| [\omega_1\star^{E}_{\kappa_1}\omega_1]-[\omega_2\star^{E}_{\kappa_1}\omega_2]\right\|_{(\nu,\beta,\mu,\chi_1,\alpha,\kappa_1,\epsilon)}\right)\right]\\
\le\frac{C_4}{|\epsilon|^{\chi_1+\alpha}}\left[ \left\|\omega_1-\omega_2\right\|_{(\nu,\beta,\mu,\chi_1,\alpha,\kappa_1,\epsilon)}C_4\left\|\omega_1\right\|^2_{(\nu,\beta,\mu,\chi_1,\alpha,\kappa_1,\epsilon)}\hfill   \right.\\
\left.+\left\|\omega_2\right\|_{(\nu,\beta,\mu,\chi_1,\alpha,\kappa_1,\epsilon)}
C_4 \left\|\omega_1-\omega_2\right\|_{(\nu,\beta,\mu,\chi_1,\alpha,\kappa_1,\epsilon)}\left(\left\|\omega_1\right\|_{(\nu,\beta,\mu,\chi_1,\alpha,\kappa_1,\epsilon)}+\left\|\omega_2\right\|_{(\nu,\beta,\mu,\chi_1,\alpha,\kappa_1,\epsilon)}\right)\right]\\
\hfill\le\frac{3C_4^2\varpi^2}{|\epsilon|^{\chi_1+\alpha}} \left\|\omega_1-\omega_2\right\|_{(\nu,\beta,\mu,\chi_1,\alpha,\kappa_1,\epsilon)}
\end{multline}

In view of (\ref{e957b}), and analogous arguments as for the corresponding part of the first statement we have

\begin{multline}\label{e860b}
\left\|\epsilon^{-\chi_1(k_{\ell,3}+2\gamma_1-k_{0,1})}\frac{\tilde{Q}(im)}{\tilde{P}_m(\tau)}\left[\tau^{k_{\ell,3}+2\gamma_1-k_{0,1}}\star_{\kappa_1}[\tau (h_{21}(\tau,m)-h_{22}(\tau,m)]\right]\right\|_{(\nu,\beta,\mu,\chi_1,\alpha,\kappa_1,\epsilon)}\\
\le \frac{C'_3}{C_{\tilde{P}}(r_{\tilde{Q},\tilde{R}_D})^{\frac{1}{\delta_D\kappa_1}}}|\epsilon_0|^{(\chi_1+\alpha)\kappa_1\left(\frac{k_{\ell,3}+2\gamma_1-k_{0,1}}{\kappa_1}+\frac{1}{\kappa_1}\right)-\chi_1(k_{\ell,3}+2\gamma_1-k_{0,1})-(\chi_1+\alpha)\kappa_1(\delta_D-\frac{1}{\kappa_1})}\left\|h_{21}-h_{22}\right\|_{(\nu,\beta,\mu,\chi_1,\alpha,\kappa_1,\epsilon)}\\
\le\frac{3(C_4)^2C'_3\varpi^2}{C_{\tilde{P}}(r_{\tilde{Q},\tilde{R}_D})^{\frac{1}{\delta_D\kappa_1}}}|\epsilon_0|^{(\chi_1+\alpha)(k_{\ell,3}+2\gamma_1-k_{0,1}-\kappa_1\delta_D+1)-\chi_1(k_{\ell,3}+2\gamma_1-k_{0,1})} \left\|\omega_1-\omega_2\right\|_{(\nu,\beta,\mu,\chi_1,\alpha,\kappa_1,\epsilon)}
\end{multline}

We choose $r_{\tilde{Q},\tilde{R}_D}$ and $\varpi$ such that

\begin{multline}\label{e868b}
\sum_{\ell=0}^{s_3}\frac{|a_{\ell,3}|}{\Gamma\left(\frac{k_{\ell,3}+2\gamma_1-k_{0,1}}{\kappa_1}\right)}\frac{(C_4)^2C'_3\varpi^2}{C_{\tilde{P}}(r_{\tilde{Q},\tilde{R}_D})^{\frac{1}{\delta_D\kappa_1}}}|\epsilon_0|^{(\chi_1+\alpha)(k_{\ell,3}+2\gamma_1-k_{0,1}-\kappa_1\delta_D+1)-\chi_1(k_{\ell,3}+2\gamma_1-k_{0,1})}\\
+\sum_{\ell=s_3+1}^{M_3}\frac{|a_{\ell,3}|}{\Gamma\left(\frac{k_{\ell,3}+2\gamma_1-k_{0,1}}{\kappa_1}\right)}\frac{(C_4)^2C'_3\varpi^2}{C_{\tilde{P}}(r_{\tilde{Q},\tilde{R}_D})^{\frac{1}{\delta_D\kappa_1}}}|\epsilon_0|^{m_{\ell,3}+3\beta-\alpha k_{\ell,3}+(\chi_1+\alpha)(k_{\ell,3}+2\gamma_1-k_{0,1}-\kappa_1\delta_D+1)-\chi_1(k_{\ell,3}+2\gamma_1-k_{0,1})}\\
\le\frac{1}{8}
\end{multline}

The choice in (\ref{e868b}) allows to guarantee from (\ref{e860b}) that
\begin{equation}\label{e874b}
\left\|\mathcal{H}_\epsilon^4(\omega_1)-\mathcal{H}_\epsilon^4(\omega_2)\right\|_{(\nu,\beta,\mu,\chi_1,\alpha,\kappa_1,\epsilon)}\le \frac{1}{8}\left\|\omega_1-\omega_2\right\|_{(\nu,\beta,\mu,\chi_1,\alpha,\kappa_1,\epsilon)}.
\end{equation}

Finally, from (\ref{e704b}), (\ref{e705b}), (\ref{e706b}) and (\ref{e874b}), we conclude that

\begin{equation}\label{e1114}
\left\|\mathcal{H}_\epsilon(\omega_1)-\mathcal{H}_\epsilon(\omega_2)\right\|_{(\nu,\beta,\mu,\chi_1,\alpha,\kappa_1,\epsilon)}\le \frac{1}{2}\left\|\omega_1-\omega_2\right\|_{(\nu,\beta,\mu,\chi_1,\alpha,\kappa_1,\epsilon)}.
\end{equation}

\end{proof}

\section{Proof of Lemma~\ref{lema568b}}\label{seccionaux2}
\begin{proof}
The proof is analogous to that of Lemma~\ref{lema568}, adapted to the elements involved within the second auxiliary equation.

We first study the terms in $\tilde{\mathcal{H}}^1_{\epsilon}$.

By Lemma~\ref{lema1}, we have 

\begin{multline}\label{e597b}
\left\|\tilde{B}_{j}(m)\epsilon^{-\chi_2(b_j-k_{0,2}+k_{0,3}-\gamma_2)}\frac{\tau^{b_j-k_{0,2}+k_{0,3}-\gamma_2}}{\tilde{P}_{2,m}(\tau)}\right\|_{(\nu,\beta,\mu,\chi_2,\alpha,\kappa_2,\epsilon)}\hfill\\
\hfill\le\frac{C_2}{C_{\tilde{P}}(r_{\tilde{Q},\tilde{R}_D})^{\frac{1}{\delta_D\kappa_2}}}\frac{\left\|\tilde{B}_j(m)\right\|_{(\beta,\mu)}}{\inf_{m\in\R}|\tilde{R}_{D}(im)|}|\epsilon|^{(b_j-k_{0,2}+k_{0,3}-\gamma_2)\alpha},
\end{multline}
for some $C_2>0$, depending on $\kappa_2,\gamma_2,k_{0,2},k_{0,3}$ and $b_j$, $0\le j\le Q$.

We also make use of the first part of Proposition~\ref{prop1}. There exists a constant $C_3>0$ depending on $\nu,\kappa_2, k_{\ell,1}$ for $0\le \ell\le M_1$ , $k_{\ell,2}$ for $0\le \ell\le M_2$, $k_{\ell,3}$ for $0\le \ell\le M_3$, $\tilde{Q}(X)$ and $\tilde{R}_{D}(X)$ such that
\begin{multline}\label{e626c}
\left\|\epsilon^{-\chi_2(k_{\ell,1}-2k_{0,2}+k_{0,3})}\frac{\tilde{Q}(im)}{\tilde{P}_{2,m}(\tau)}(\tau^{k_{\ell,1}-2k_{0,2}+k_{0,3}}\star_{\kappa_2}\omega_2)\right\|_{(\nu,\beta,\mu,\chi_2,\alpha,\kappa_2,\epsilon)}\\
\le \frac{C_3}{C_{\tilde{P}}(r_{\tilde{Q},\tilde{R}_D})^{\frac{1}{\delta_D \kappa_2}}}|\epsilon|^{\alpha(k_{\ell,1}-2k_{0,2}+k_{0,3})}\left\|\omega_2\right\|_{(\nu,\beta,\mu,\chi_2,\alpha,\kappa_2,\epsilon)},
\end{multline} 

\begin{multline}\label{e627c}
\left\|\epsilon^{-\chi_2(k_{\ell,2}-k_{0,2})}\frac{\tilde{Q}(im)}{\tilde{P}_{2,m}(\tau)}(\tau^{k_{\ell,2}-k_{0,2}}\star_{\kappa_2}\omega_2)\right\|_{(\nu,\beta,\mu,\chi_2,\alpha,\kappa_2,\epsilon)}\\
\le \frac{C_3}{C_{\tilde{P}}(r_{\tilde{Q},\tilde{R}_D})^{\frac{1}{\delta_D \kappa_2}}}|\epsilon|^{\alpha(k_{\ell,2}-k_{0,2})}\left\|\omega_2\right\|_{(\nu,\beta,\mu,\chi_2,\alpha,\kappa_2,\epsilon)},
\end{multline}

\begin{multline}\label{e628c}
\left\|\epsilon^{-\chi_2(k_{\ell,3}-k_{0,3})}\frac{\tilde{Q}(im)}{\tilde{P}_{2,m}(\tau)}(\tau^{k_{\ell,3}-k_{0,3}}\star_{\kappa_2}\omega_2)\right\|_{(\nu,\beta,\mu,\chi_2,\alpha,\kappa_2,\epsilon)}\\
\le \frac{C_3}{C_{\tilde{P}}(r_{\tilde{Q},\tilde{R}_D})^{\frac{1}{\delta_D \kappa_2}}}|\epsilon|^{\alpha(k_{\ell,3}-k_{0,3})}\left\|\omega_2\right\|_{(\nu,\beta,\mu,\chi_2,\alpha,\kappa_2,\epsilon)}.
\end{multline}

Let $j\ge 1$. A constant $C_3(j)>0$, depending on $\nu,\kappa_2$, $k_{\ell,2}$ for $0\le \ell\le M_2$, $\tilde{Q}(X)$, $\tilde{R}_{D}(X)$, exists such that

\begin{multline}\label{e629c}
\left\|\epsilon^{-\chi_2(k_{\ell,2}-k_{0,2}+j)}\frac{\tilde{Q}(im)}{\tilde{P}_{2,m}(\tau)}(\tau^{k_{\ell,2}-k_{0,2}+j}\star_{\kappa_2}\omega_2)\right\|_{(\nu,\beta,\mu,\chi_2,\alpha,\kappa_2,\epsilon)}\\
\le \frac{C_{3.1}(j)}{C_{\tilde{P}}(r_{\tilde{Q},\tilde{R}_D})^{\frac{1}{\delta_D \kappa_2}}}|\epsilon|^{\alpha(k_{\ell,2}-k_{0,2}+j)}\left\|\omega_2\right\|_{(\nu,\beta,\mu,\chi_2,\alpha,\kappa_2,\epsilon)},
\end{multline}

and

\begin{multline}\label{e630c}
\left\|\epsilon^{-\chi_2(k_{\ell,3}-k_{0,3}+j)}\frac{\tilde{Q}(im)}{\tilde{P}_{2,m}(\tau)}(\tau^{k_{\ell,3}-k_{0,3}+j}\star_{\kappa_2}\omega_2)\right\|_{(\nu,\beta,\mu,\chi_2,\alpha,\kappa_2,\epsilon)}\\
\le \frac{C_{3.2}(j)}{C_{\tilde{P}}(r_{\tilde{Q},\tilde{R}_D})^{\frac{1}{\delta_D \kappa_2}}}|\epsilon|^{\alpha(k_{\ell,3}-k_{0,3}+j)}\left\|\omega_2\right\|_{(\nu,\beta,\mu,\chi_2,\alpha,\kappa_2,\epsilon)}.
\end{multline}

The constants $C_{3.1}(j),C_{3.2}(j)>0$ are such that

\begin{equation}\label{e633c}
C_{3.1}(j)\le\hat{C}_3A_3^j\Gamma\left(\frac{k_{\ell,2}-k_{0,2}+j}{\kappa_2}\right),\qquad j\ge 1,
\end{equation}

\begin{equation}\label{e634c}
C_{3.2}(j)\le\hat{C}_3A_3^j\Gamma\left(\frac{k_{\ell,3}-k_{0,3}+j}{\kappa_2}\right),\qquad j\ge 1.
\end{equation}
The proof of both estimates is analogous to that of (\ref{e633}), so we omit them.

In view of Lemma~\ref{lema533} and again Proposition~\ref{prop1}, we have

\begin{multline}\label{e631c}
\left\|\epsilon^{-\chi_2(k_{\ell,3}-k_{0,3})}\frac{\tilde{Q}(im)}{\tilde{P}_{2,m}(\tau)}\left[\tau^{k_{\ell,3}-k_{0,3}+j}\star_{\kappa_2}\mathcal{B}_{\kappa_2}J_2(\tau,\epsilon)\star_{\kappa_2}\mathcal{B}_{\kappa_2}J_2(\tau,\epsilon)\star_{\kappa_2}\omega_2\right]\right\|_{(\nu,\beta,\mu,\chi_2,\alpha,\kappa_2,\epsilon)}\hfill\\
=\left\|\epsilon^{-\chi_2(k_{\ell,3}-k_{0,3})}\frac{\tilde{Q}(im)}{\tilde{P}_{2,m}(\tau)}\left[\tau^{k_{\ell,3}-k_{0,3}+j}\star_{\kappa_2}\mathcal{B}_{\kappa_2}\tilde{J}_2(\tau,\epsilon)\star_{\kappa_2}\omega_2\right]\right\|_{(\nu,\beta,\mu,\chi_2,\alpha,\kappa_2,\epsilon)}\hfill\\
\le \frac{C_{3.3}(j)}{C_{\tilde{P}}(r_{\tilde{Q},\tilde{R}_D})^{\frac{1}{\delta_D \kappa_2}}}|\epsilon|^{\alpha(k_{\ell,3}-k_{0,3}+j)}\left\|\omega_2\right\|_{(\nu,\beta,\mu,\chi_2,\alpha,\kappa_2,\epsilon)},\hfill
\end{multline}
where
\begin{equation}\label{e677b}
C_{3.3}(j)\le\hat{C}_3A_3^j\Gamma\left(\frac{k_{\ell,3}-k_{0,3}+j}{\kappa_2}\right),\qquad j\ge 1.
\end{equation}
This last estimates for $C_{3.3}$ are obtained in the same manner as those in (\ref{e633}).

We choose large enough $r_{\tilde{Q},\tilde{R}_D}>0$ and $\varpi>0$ such that

\begin{multline}
\sum_{j=0}^{Q}|\epsilon_0|^{n_j-\alpha b_j+\alpha(b_j-k_{0,2}+k_{0,3}-\gamma_2)}\frac{C_2}{C_{\tilde{P}}(r_{\tilde{Q},\tilde{R}_D})^{\frac{1}{\delta_D\kappa_2}}}\frac{\left\|\tilde{B}_j(m)\right\|_{(\beta,\mu)}}{\inf_{m\in\R}|\tilde{R}_{D}(im)|}\\
+\sum_{\ell=0}^{s_1}\frac{|a_{\ell,1}|}{\Gamma\left(\frac{k_{\ell,1}-2k_{0,2}+k_{0,3}}{\kappa_2}\right)}|\epsilon_0|^{\alpha(k_{\ell,1}-2k_{0,2}+k_{0,3})}\frac{C_3}{C_{\tilde{P}}(r_{\tilde{Q},\tilde{R}_D})^{\frac{1}{\delta_D\kappa_2}}}\varpi\hfill\\
\hfill+\sum_{\ell=s_1+1}^{M_1}\frac{|a_{\ell,1}|}{\Gamma\left(\frac{k_{\ell,1}-k_{0,2}+k_{0,3}}{\kappa_2}\right)}|\epsilon_0|^{m_{\ell,1}+\beta-\alpha k_{\ell,1}+\alpha(k_{\ell,1}-2k_{0,1}+k_{0,3})}\frac{C_3}{C_{\tilde{P}}(r_{\tilde{Q},\tilde{R}_D})^{\frac{1}{\delta_D\kappa_2}}}\varpi\\
+\sum_{\ell=1}^{s_2}\frac{2|a_{\ell,2}a_{0,2}|}{|a_{0,3}|}\frac{|\epsilon_0|^{\alpha(k_{\ell,2}-k_{0,2})}}{\Gamma\left(\frac{k_{\ell,2}-k_{0,2}}{\kappa_2}\right)}\frac{C_3}{C_{\tilde{P}}(r_{\tilde{Q},\tilde{R}_D})^{\frac{1}{\delta_D\kappa_2}}}\varpi\hfill\\
\hfill+\sum_{\ell=s_2+1}^{M_2}\frac{2|a_{\ell,2}a_{0,2}|}{|a_{0,3}|}\frac{|\epsilon_0|^{m_{\ell,2}+2\beta-\alpha k_{\ell,2}+\alpha(k_{\ell,2}-k_{0,2})}}{\Gamma\left(\frac{k_{\ell,2}-k_{0,2}}{\kappa_2}\right)}\frac{C_3}{C_{\tilde{P}}(r_{\tilde{Q},\tilde{R}_D})^{\frac{1}{\delta_D\kappa_2}}}\varpi\\
+\sum_{\ell=0}^{s_2}\frac{2|a_{\ell,2}a_{0,2}|}{|a_{0,3}|}\sum_{j\ge1}|J_{2,j}|\frac{|\epsilon_0|^{\alpha(k_{\ell,2}-k_{0,2}+j)}}{\Gamma\left(\frac{k_{\ell,2}-k_{0,2}}{\kappa_2}\right)}\frac{\hat{C}_3A_3^j}{C_{\tilde{P}}(r_{\tilde{Q},\tilde{R}_D})^{\frac{1}{\delta_D\kappa_2}}}\varpi\hfill\\
\hfill+\sum_{\ell=s_2+1}^{M_2}\frac{2|a_{\ell,2}a_{0,2}|}{|a_{0,3}|}\sum_{j\ge1}|J_{2,j}|\frac{|\epsilon_0|^{m_{\ell,2}+2\beta-\alpha k_{\ell,2}+\alpha(k_{\ell,2}-k_{0,2}+j)}}{\Gamma\left(\frac{k_{\ell,2}-k_{0,2}}{\kappa_2}\right)}\frac{\hat{C}_3A_3^j}{C_{\tilde{P}}(r_{\tilde{Q},\tilde{R}_D})^{\frac{1}{\delta_D\kappa_2}}}\varpi\\
+\sum_{\ell=1}^{s_3}\frac{3|a_{0,2}|^2|a_{\ell,3}|}{|a_{0,3}|^2}\frac{|\epsilon_0|^{\alpha(k_{\ell,3}-k_{0,3})}}{\Gamma\left(\frac{k_{\ell,3}-k_{0,3}}{\kappa_2}\right)}\frac{C_3}{C_{\tilde{P}}(r_{\tilde{Q},\tilde{R}_D})^{\frac{1}{\delta_D\kappa_2}}}\varpi\hfill\\
\hfill+\sum_{\ell=0}^{s_3}\frac{6|a_{0,2}|^2|a_{\ell,3}|}{|a_{0,3}|^2}\sum_{j\ge1}|J_{2,j}|\frac{|\epsilon_0|^{\alpha(k_{\ell,3}-k_{0,3}+j)}}{\Gamma\left(\frac{k_{\ell,3}-k_{0,3}}{\kappa_2}\right)}\frac{\hat{C}_3A_3^j}{C_{\tilde{P}}(r_{\tilde{Q},\tilde{R}_D})^{\frac{1}{\delta_D\kappa_2}}}\varpi\\
+\sum_{\ell=0}^{s_3}\frac{3|a_{0,2}|^2|a_{\ell,3}|}{|a_{0,3}|^2}\sum_{j\ge1}|\tilde{J}_{2,j}|\frac{|\epsilon_0|^{\alpha(k_{\ell,3}-k_{0,3}+j)}}{\Gamma\left(\frac{k_{\ell,3}-k_{0,3}}{\kappa_2}\right)}\frac{\hat{C}_3A_3^j}{C_{\tilde{P}}(r_{\tilde{Q},\tilde{R}_D})^{\frac{1}{\delta_D\kappa_2}}}\varpi\hfill\\
\hfill+\sum_{\ell=s_3+1}^{M_3}\frac{3|a_{0,2}|^2|a_{\ell,3}|}{|a_{0,3}|^2}\frac{|\epsilon_0|^{m_{\ell,3}+3\beta-\alpha k_{\ell,3}+\alpha(k_{\ell,3}-k_{0,3})}}{\Gamma\left(\frac{k_{\ell,3}-k_{0,3}}{\kappa_2}\right)}\frac{C_3}{C_{\tilde{P}}(r_{\tilde{Q},\tilde{R}_D})^{\frac{1}{\delta_D\kappa_2}}}\varpi\\
+\sum_{\ell=s_3+1}^{M_3}\frac{6|a_{0,2}|^2|a_{\ell,3}|}{|a_{0,3}|^2}\sum_{j\ge1}|J_{2,j}|\frac{|\epsilon_0|^{\alpha(k_{\ell,3}-k_{0,3}+j)}}{\Gamma\left(\frac{k_{\ell,3}-k_{0,3}}{\kappa_2}\right)}\frac{\hat{C}_3A_3^j}{C_{\tilde{P}}(r_{\tilde{Q},\tilde{R}_D})^{\frac{1}{\delta_D\kappa_2}}}\varpi\hfill\\
\hfill+\sum_{\ell=s_3+1}^{M_3}\frac{3|a_{0,2}|^2|a_{\ell,3}|}{|a_{0,3}|^2}\sum_{j\ge1}|\tilde{J}_{2,j}|\frac{|\epsilon_0|^{\alpha(k_{\ell,3}-k_{0,3}+j)}}{\Gamma\left(\frac{k_{\ell,3}-k_{0,3}}{\kappa_2}\right)}\frac{\hat{C}_3A_3^j}{C_{\tilde{P}}(r_{\tilde{Q},\tilde{R}_D})^{\frac{1}{\delta_D\kappa_2}}}\varpi\le \frac{\varpi}{4}.
\end{multline}

This choice allows us to deduce that

\begin{equation}\label{e704c} 
\left\|\tilde{\mathcal{H}}^1_\epsilon(\omega_1)\right\|_{(\nu,\beta,\mu,\chi_2,\alpha,\kappa_2,\epsilon)}\le \frac{\varpi}{4}.
\end{equation}

We now give upper bounds associated to $\tilde{\mathcal{H}}_\epsilon^2$.

We define 
\begin{equation}\label{e714b}
\tilde{h}_1(\tau,m)=\tau^{\kappa_2-1}\int_{0}^{\tau^{\kappa_2}}\int_{-\infty}^{\infty}\omega_2((\tau^{\kappa_1}-s')^{1/\kappa_2},m-m_1)\omega_2((s')^{1/\kappa_2},m_1)\frac{1}{(\tau^{\kappa_2}-s')s'}ds'dm_1.
\end{equation}

Observe that 
$$\omega_2(\tau,m)\star_{\kappa_2}^E\omega_2(\tau,m)=\tau \tilde{h}_1(\tau,m).$$

In view of Proposition~\ref{prop3}, there exists $C_4>0$, depending on $\mu$ and $\kappa_2$, such that 
\begin{equation}\label{e721b}
\left\|\tilde{h}_1(\tau,\epsilon)\right\|_{(\nu,\beta,\mu,\chi_2,\alpha,\kappa_2,\epsilon)}\le\frac{C_4}{|\epsilon|^{\chi_2+\alpha}}\left\|\omega_2\right\|^2_{(\nu,\beta,\mu,\chi_2,\alpha,\kappa_2,\epsilon)}.
\end{equation}

We proceed as in the previous proof on the upper bounds for $\mathcal{H}^{2}_{\epsilon}$ to get that

\begin{multline}\label{e726c}
\left\|\epsilon^{-\chi_2(d_\ell-k_{0,2}+k_{0,3}-\delta_\ell)}\frac{\tilde{R}_{\ell}(im)}{\tilde{P}_{2,m}(\tau)}\left[\tau^{\tilde{d}_{\ell,q_1,q_2}}\star_{\kappa_2}(\tau^{\kappa_2 q_2}\omega_2)\right]\right\|_{(\nu,\beta,\mu,\chi_2,\alpha,\kappa_2,\epsilon)}\\
\le \frac{C'_3}{C_{\tilde{P}}(r_{\tilde{Q},\tilde{R}_D})^{\frac{1}{\delta_D\kappa_2}}}|\epsilon|^{(\chi_2+\alpha)\kappa_2\left(\frac{\tilde{d}_{\ell,q_1,q_2}}{\kappa_2}+q_2-\delta_D+\frac{1}{\kappa_2}\right)-\chi_2(\tilde{d}_{\ell}-k_{0,2}+k_{0,3}-\delta_\ell)}\left\|\omega_2\right\|_{(\nu,\beta,\mu,\chi_2,\alpha,\kappa_2,\epsilon)},
\end{multline}
for every $q_1\ge0$, $q_2\ge 1$ such that $q_1+q_2=\delta_\ell$, and also

\begin{multline}\label{e727c}
\left\|\epsilon^{-\chi_1(d_\ell-k_{0,2}+k_{0,3}-\delta_\ell)}\frac{\tilde{R}_{\ell}(im)}{\tilde{P}_{2,m}(\tau)}\left[\tau^{\tilde{d}_{\ell,q_1,q_2}+\kappa_2(q_2-p)}\star_{\kappa_2}(\tau^{\kappa_2 p}\omega_2)\right]\right\|_{(\nu,\beta,\mu,\chi_2,\alpha,\kappa_2,\epsilon)}\\
\le \frac{C'_3}{C_{\tilde{P}}(r_{\tilde{Q},\tilde{R}_D})^{\frac{1}{\delta_D\kappa_2}}}|\epsilon|^{(\chi_2+\alpha)\kappa_2\left(\frac{\tilde{d}_{\ell,q_1,q_2}}{\kappa_2}+q_2-\delta_D+\frac{1}{\kappa_2}\right)-\chi_2(d_{\ell}-k_{0,2}+k_{0,3}-\delta_\ell)}\left\|\omega_2\right\|_{(\nu,\beta,\mu,\chi_2,\alpha,\kappa_2,\epsilon)},
\end{multline}
for every $1\le p\le q_2-1$. Also,

\begin{multline}\label{e728c}
\left\|\epsilon^{-\chi_2(k_{\ell,2}+\gamma_2-2k_{0,2}+k_{0,3})}\frac{\tilde{Q}(im)}{\tilde{P}_{2,m}(\tau)}\left[\tau^{k_{\ell,2}+\gamma_2-2k_{0,2}+k_{0,3}}\star_{\kappa_2}(\tau \tilde{h}_1(\tau,m))\right]\right\|_{(\nu,\beta,\mu,\chi_2,\alpha,\kappa_2,\epsilon)}\\
\le \frac{C_4C'_3}{C_{\tilde{P}}(r_{\tilde{Q},\tilde{R}_D})^{\frac{1}{\delta_D\kappa_2}}}|\epsilon|^{(\chi_2+\alpha)(k_{\ell,2}+\gamma_2-2k_{0,2}+k_{0,3}-\kappa_2\delta_D+1)-\chi_2(k_{\ell,2}+\gamma_2-2k_{0,2}+k_{0,3})}\left\|\omega_2(\tau,\epsilon)\right\|^2_{(\nu,\beta,\mu,\chi_2,\alpha,\kappa_2,\epsilon)}.
\end{multline} 

The same arguments as above follow to get 

\begin{multline}\label{e729c}
\left\|\epsilon^{-\chi_2k_{\ell,3}}\frac{\tilde{Q}(im)}{\tilde{P}_{2,m}(\tau)}\left[\tau^{k_{\ell,3}}\star_{\kappa_2}(\tau \tilde{h}_1(\tau,m))\right]\right\|_{(\nu,\beta,\mu,\chi_2,\alpha,\kappa_2,\epsilon)}\\
\le \frac{C_4C'_3}{C_{\tilde{P}}(r_{\tilde{Q},\tilde{R}_D})^{\frac{1}{\delta_D\kappa_2}}}|\epsilon|^{(\chi_2+\alpha)(k_{\ell,3}-\kappa_2\delta_D+1)-\chi_2k_{\ell,3}}\left\|\omega_2(\tau,\epsilon)\right\|^2_{(\nu,\beta,\mu,\chi_2,\alpha,\kappa_2,\epsilon)},
\end{multline} 

and

\begin{multline}\label{e730c}
\left\|\epsilon^{-\chi_2(k_{\ell,3}+j)}\frac{\tilde{Q}(im)}{\tilde{P}_{2,m}(\tau)}\left[\tau^{k_{\ell,3}+j}\star_{\kappa_2}(\tau \tilde{h}_1(\tau,m))\right]\right\|_{(\nu,\beta,\mu,\chi_2,\alpha,\kappa_2,\epsilon)}\\
\le \frac{C_4C'_3(j)}{C_{\tilde{P}}(r_{\tilde{Q},\tilde{R}_D})^{\frac{1}{\delta_D\kappa_2}}}|\epsilon|^{(\chi_2+\alpha)(k_{\ell,3}+j-\kappa_2\delta_D+1)-\chi_2(k_{\ell,3}+j)}\left\|\omega_2(\tau,\epsilon)\right\|^2_{(\nu,\beta,\mu,\chi_2,\alpha,\kappa_2,\epsilon)},
\end{multline} 
for every $j\ge1$, where 
\begin{equation}\label{e731b}
C'_3(j)\le \hat{C}_3A_3^j\Gamma\left(\frac{k_{\ell,3}}{\kappa_2}\right) \quad j\ge1.
\end{equation}
The proof of such dependence on $j$ is proved in an analogous way as for that of (\ref{e633}).

One can choose $r_{\tilde{Q},\tilde{R}_D}>0$ and $\varpi>0$ such that the following condition holds:

\begin{multline}\label{e766c}
\sum_{\ell=1}^{D-1}|\epsilon_0|^{\Delta_{\ell}+\alpha(\delta_{\ell}-d_{\ell})+\beta}\left[\prod_{d=0}^{\delta_\ell-1}|\gamma_2-d|\frac{C_3}{C_{\tilde{P}}(r_{\tilde{Q},\tilde{R}_D})^{\frac{1}{\delta_D\kappa_2}}\Gamma\left(\frac{\tilde{d}_{\ell,\delta_\ell,0}}{\kappa_2}\right)}|\epsilon_0|^{\alpha(d_\ell-k_{0,2}+k_{0,3}-\delta_\ell)}\varpi\right.\\
+\sum_{q_1+q_2=\delta_\ell,q_2\ge1}\frac{\delta_\ell!}{q_1!q_2!}\prod_{d=0}^{q_1-1}|\gamma_2-d|\left(\frac{C'_3\kappa_2^{q_2}}{C_{\tilde{P}}(r_{\tilde{Q},\tilde{R}_D})^{\frac{1}{\delta_D\kappa_2}}\Gamma\left(\frac{\tilde{d}_{\ell,q_1,q_2}}{\kappa_2}\right)}\right.\hfill\\
\hfill\times |\epsilon_0|^{(\chi_2+\alpha)\kappa_2\left(\frac{\tilde{d}_{\ell,q_1,q_2}}{\kappa_2}+q_2-\delta_D+\frac{1}{\kappa_2}\right)-\chi_2(d_\ell-k_{0,2}+k_{0,3}-\delta_\ell)}\varpi\\
+\sum_{1\le p\le q_2-1}|\tilde{A}_{q_2,p}|\frac{C'_3\kappa_2^p}{C_{\tilde{P}}(r_{\tilde{Q},\tilde{R}_D})^{\frac{1}{\delta_D\kappa_2}}\Gamma\left(\frac{\tilde{d}_{\ell,q_1,q_2}}{\kappa_2}+q_2-p\right)}\hfill\\
\hfill\left.\left. \times |\epsilon_0|^{(\chi_2+\alpha)\kappa_2\left(\frac{\tilde{d}_{\ell,q_1,q_2}}{\kappa_2}+q_2-\delta_D+\frac{1}{\kappa_2}\right)-\chi_2(d_\ell-k_{0,2}+k_{0,3}-\delta_\ell)}\varpi   \right)    \right]\\
+\sum_{\ell=0}^{s_2}\frac{|a_{\ell,2}|}{\Gamma\left(\frac{k_{\ell,2}+\gamma_2-2k_{0,2}+k_{0,3}}{\kappa_2}\right)}\frac{C_4C'_3}{C_{\tilde{P}}(r_{\tilde{Q},\tilde{R}_D})^{\frac{1}{\delta_D\kappa_2}}}|\epsilon_0|^{(\chi_2+\alpha)(k_{\ell,2}+\gamma_2-2k_{0,2}+k_{0,3}-\kappa_2\delta_D+1)-\chi_2(k_{\ell,2}+\gamma_2-2k_{0,2}+k_{0,3})}\varpi^2\hfill\\
\hfill+\sum_{\ell=s_2+1}^{M_2}\frac{|a_{\ell,2}|C_4C'_3}{\Gamma\left(\frac{k_{\ell,2}+\gamma_2-2k_{0,2}+k_{0,3}}{\kappa_2}\right)}\frac{|\epsilon_0|^{m_{\ell,2}+2\beta-\alpha k_{\ell,2}+(\chi_2+\alpha)(k_{\ell,2}+\gamma_2-2k_{0,2}+k_{0,3}-\kappa_2\delta_D+1)-\chi_2(k_{\ell,2}+\gamma_2-2k_{0,2}+k_{0,3})}}{C_{\tilde{P}}(r_{\tilde{Q},\tilde{R}_D})^{\frac{1}{\delta_D\kappa_2}}}\varpi^2\\
+\sum_{\ell=0}^{s_3}|a_{\ell,3}|\frac{C_4C'_3}{\Gamma\left(\frac{k_{\ell,3}}{\kappa_2}\right)}\frac{|\epsilon_0|^{(\chi_2+\alpha)(k_{\ell,3}-\kappa_2\delta_D+1)-\chi_2k_{\ell,3}}}{C_{\tilde{P}}(r_{\tilde{Q},\tilde{R}_D})^{\frac{1}{\delta_D\kappa_2}}}\varpi^2\hfill\\
\hfill+\sum_{\ell=s_3+1}^{M_3}|a_{\ell,3}|\frac{C_4C'_3}{\Gamma\left(\frac{k_{\ell,3}}{\kappa_2}\right)}\frac{|\epsilon_0|^{m_{\ell,3}+3\beta-\alpha k_{\ell,3}+(\chi_2+\alpha)(k_{\ell,3}-\kappa_2\delta_D+1)-\chi_2k_{\ell,3}}}{C_{\tilde{P}}(r_{\tilde{Q},\tilde{R}_D})^{\frac{1}{\delta_D\kappa_2}}}\varpi^2\\
+\sum_{\ell=0}^{s_3}|a_{\ell,3}|\sum_{j\ge1}|J_{2,j}|\frac{|\epsilon_0|^{(\chi_2+\alpha)(k_{\ell,3}+j-\kappa_2\delta_D+1)-\chi_2(k_{\ell,3}+j)}}{\Gamma\left(\frac{k_{\ell,3}}{\kappa_2}\right)}\frac{C_4\hat{C}_3A_3^j}{C_{\tilde{P}}(r_{\tilde{Q},\tilde{R}_D})^{\frac{1}{\delta_D\kappa_2}}}\varpi^2\\
+\sum_{\ell=s_3+1}^{M_3}|a_{\ell,3}|\sum_{j\ge1}|J_{2,j}|\frac{|\epsilon_0|^{m_{\ell,3}+3\beta-\alpha k_{\ell,3}+(\chi_2+\alpha)(k_{\ell,3}+j-\kappa_2\delta_D+1)-\chi_2(k_{\ell,3}+j)}}{\Gamma\left(\frac{k_{\ell,3}}{\kappa_2}\right)}\frac{C_4\hat{C}_3A_3^j}{C_{\tilde{P}}(r_{\tilde{Q},\tilde{R}_D})^{\frac{1}{\delta_D\kappa_2}}}\varpi^2\\
\le\frac{\varpi}{4}\hfill
\end{multline}

We have that

\begin{equation}\label{e784c}
\left\|\tilde{\mathcal{H}}_{\epsilon}^{3}(\omega_2)\right\|_{(\nu,\beta,\mu,\chi_2,\alpha,\kappa_2,\epsilon)}\le\frac{\varpi}{4}.
\end{equation}

We now give upper estimates for $\tilde{\mathcal{H}}_{\epsilon}^3(\omega_2(\tau,m))$.

It holds that
 
\begin{multline}\label{e793c}
\left\|\epsilon^{-\chi_2(d_{D}-2k_{0,2}+k_{0,3}-\delta_D)}\frac{\tilde{R}_{D}(im)}{\tilde{P}_{2,m}(\tau)}\left[\tau^{\tilde{d}_{D,\delta_D,0}}\star_{\kappa_2}\omega_2\right]\right\|_{(\nu,\beta,\mu,\chi_2,\alpha,\kappa_2,\epsilon)}\\
\le\frac{C_3}{C_{\tilde{P}}(r_{\tilde{Q},\tilde{R}_D})^{\frac{1}{\delta_D\kappa_2}}}|\epsilon_0|^{(\chi_2+\alpha)\tilde{d}_{d_{D},\delta_D,0}-\chi_2(d_D-2k_{0,2}+k_{0,3}-\delta_D)}\left\|\omega_2\right\|_{(\nu,\beta,\mu,\chi_2,\alpha,\kappa_2,\epsilon)}\\
=\frac{C_3}{C_{\tilde{P}}(r_{\tilde{Q},\tilde{R}_D})^{\frac{1}{\delta_D\kappa_2}}}|\epsilon_0|^{\alpha(d_{D}-2k_{0,2}+k_{0,3}-\delta_D)}\left\|\omega_2\right\|_{(\nu,\beta,\mu,\chi_2,\alpha,\kappa_2,\epsilon)},
\end{multline}

\begin{multline}\label{e794c}
\left\|\epsilon^{-\chi_2(d_{D}-2k_{0,2}+k_{0,3}-\delta_D)}\frac{\tilde{R}_{D}(im)}{\tilde{P}_{2,m}(\tau)}\left[\tau^{\tilde{d}_{D,q_1,q_2}}\star_{\kappa_2}(\tau^{\kappa_2 q_2}\omega_2)\right]\right\|_{(\nu,\beta,\mu,\chi_2,\alpha,\kappa_2,\epsilon)}\\
\le\frac{C'_3}{C_{\tilde{P}}(r_{\tilde{Q},\tilde{R}_D})^{\frac{1}{\delta_D\kappa_2}}}|\epsilon_0|^{(\chi_2+\alpha)\kappa_2\left(\frac{\tilde{d}_{d_{D},q_1,q_2}}{\kappa_2}+q_2-\delta_D\right)-\chi_2(d_D-2k_{0,2}+k_{0,3}-\delta_D)}\left\|\omega_2\right\|_{(\nu,\beta,\mu,\chi_2,\alpha,\kappa_2,\epsilon)},
\end{multline}

for every $q_1\ge 1$ and $q_2\ge 1$ with $q_1+q_2=\delta_D$. In addition to that, it holds

\begin{equation}\label{e795c}
\left\|\frac{\tilde{R}_{D}(im)}{\tilde{P}_{2,m}(\tau)}\left[\tau^{\kappa_2(\delta_D-p)}\star_{\kappa_2}(\tau^{\kappa_2 p}\omega_2)\right]\right\|_{(\nu,\beta,\mu,\chi_2,\alpha,\kappa_2,\epsilon)}\\
\le\frac{C'_3}{C_{\tilde{P}}(r_{\tilde{Q},\tilde{R}_D})^{\frac{1}{\delta_D\kappa_2}}}|\epsilon_0|^{\chi_2+\alpha}\left\|\omega_2\right\|_{(\nu,\beta,\mu,\chi_2,\alpha,\kappa_2,\epsilon)},
\end{equation}
for all $1\le p\le \delta_D-1$.

We choose $r_{\tilde{Q},\tilde{R}_D}$ and $\varpi$ such that

\begin{multline}\label{e816c}
|\epsilon_0|^{\Delta_D+\alpha(\delta_D-d_D)+\beta}\left[\prod_{d=0}^{\delta_D-1}|\gamma_2-d|\frac{C_3}{C_{\tilde{P}}(r_{\tilde{Q},\tilde{R}_D})^{\frac{1}{\delta_D\kappa_2}}\Gamma\left(\frac{\tilde{d}_{D,\delta_D,0}}{\kappa_2}\right)}|\epsilon_0|^{\alpha(d_{D}-2k_{0,2}+k_{0,3}-\delta_D)}\varpi\right.\\
+\sum_{q_1+q_2=\delta_D,q_1\ge1,q_2\ge1}\frac{\delta_D!}{q_1!q_2!}\prod_{d=0}^{q_1-1}|\gamma_2-d|\left(\frac{C'_3\kappa_2^{q_2}}{C_{\tilde{P}}(r_{\tilde{Q},\tilde{R}_D})^{\frac{1}{\delta_D\kappa_2}}\Gamma\left(\frac{\tilde{d}_{D,q_1,q_2}}{\kappa_2}\right)}\right.\\
\hfill\times |\epsilon_0|^{(\chi_2+\alpha)\kappa_2\left(\frac{\tilde{d}_{d_{D},q_1,q_2}}{\kappa_2}+q_2-\delta_D+\frac{1}{\kappa_2}\right)-\chi_2(d_D-2k_{0,2}+k_{0,3}-\delta_D)}\varpi\\
+\sum_{1\le p\le q_2-1}|\tilde{A}_{q_2,p}|\frac{C'_3\kappa_2^p}{C_{\tilde{P}}(r_{\tilde{Q},\tilde{R}_D})^{\frac{1}{\delta_D\kappa_2}}\Gamma\left(\frac{\tilde{d}_{D,q_1,q_2}}{\kappa_2}+q_2-p\right)}\\
\left.\left.\times|\epsilon_0|^{(\chi_2+\alpha)\kappa_2\left(\frac{\tilde{d}_{d_{D},q_1,q_2}}{\kappa_2}+q_2-\delta_D+\frac{1}{\kappa_2}\right)-\chi_2(d_D-2k_{0,2}+k_{0,3}-\delta_D)}\varpi\right)\right]\\
+\sum_{1\le p\le \delta_D-1}|\tilde{A}_{\delta_D,p}|\frac{C'_3\kappa_2^p}{C_{\tilde{P}}(r_{\tilde{Q},\tilde{R}_D})^{\frac{1}{\delta_D\kappa_2}}\Gamma\left(\delta_D-p\right)}|\epsilon_0|^{\chi_2+\alpha}\varpi\le\frac{\varpi}{4}.
\end{multline}

This choice guarantees that
\begin{equation}\label{e823c}
\left\|\tilde{\mathcal{H}}_\epsilon^3(\omega_2)\right\|_{(\nu,\beta,\mu,\chi_2,\alpha,\kappa_2,\epsilon)}\le \frac{\varpi}{4}.
\end{equation}

We finally give upper bounds for the elements involved in $\tilde{\mathcal{H}}_\epsilon^4(\omega_2)$.

Let $\tilde{h}_1$ be defined in (\ref{e714b}), and let
$$\tilde{h}_2(\tau,m)=\tau^{\kappa_2-1}\int_{0}^{\tau^{\kappa_2}}\int_{-\infty}^{\infty}\omega_2((\tau^{\kappa_2}-s')^{1/\kappa_2},m-m_1)(\omega_2\star_{\kappa_2}^{E}\omega_2)((s')^{1/\kappa_2},m_1)\frac{1}{(\tau^{\kappa_2}-s')s'}ds'dm_1.$$

Regarding Proposition~\ref{prop3}, one has
\begin{equation}\label{e840b}
\left\|\tilde{h}_2(\tau,m)\right\|_{(\nu,\beta,\mu,\chi_2,\alpha,\kappa_2,\epsilon)}\le\frac{C_4}{|\epsilon|^{\chi_2+\alpha}}\left\|\omega_2\right\|_{(\nu,\beta,\mu,\chi_2,\alpha,\kappa_2,\epsilon)}\left\| \omega_2\star_{\kappa_2}^{E}\omega_2 \right\|_{(\nu,\beta,\mu,\chi_2,\alpha,\kappa_2,\epsilon)}.
\end{equation}

Moreover, in view of Corollary~\ref{coro3}, we get
\begin{equation}\label{e852b}
\left\| \omega_2\star_{\kappa_2}^{E}\omega_2 \right\|_{(\nu,\beta,\mu,\chi_2,\alpha,\kappa_2,\epsilon)}\le C_4\left\| \omega_2 \right\|^2_{(\nu,\beta,\mu,\chi_2,\alpha,\kappa_2,\epsilon)}
\end{equation}

Analogous arguments as in the proof of upper bounds for $\mathcal{H}_\epsilon^4$ we have that

\begin{multline}\label{e860c}
\left\|\epsilon^{-\chi_2(k_{\ell,3}+2\gamma_2-2k_{0,2}+k_{0,3})}\frac{\tilde{Q}(im)}{\tilde{P}_{2,m}(\tau)}\left[\tau^{k_{\ell,3}+2\gamma_2-2k_{0,2}+k_{0,3}}\star_{\kappa_2}(\tau \tilde{h}_2(\tau,m))\right]\right\|_{(\nu,\beta,\mu,\chi_2,\alpha,\kappa_2,\epsilon)}\\
\le \frac{C'_3}{C_{\tilde{P}}(r_{\tilde{Q},\tilde{R}_D})^{\frac{1}{\delta_D\kappa_2}}}|\epsilon_0|^{(\chi_2+\alpha)\kappa_2\left(\frac{k_{\ell,3}+2\gamma_2-2k_{0,2}+k_{0,3}}{\kappa_2}+\frac{1}{\kappa_2}\right)-\chi_2(k_{\ell,3}+2\gamma_2-2k_{0,2}+k_{0,3})-(\chi_2+\alpha)\kappa_2(\delta_D-\frac{1}{\kappa_2})}\left\|\tilde{h}_2(\tau,m)\right\|_{(\nu,\beta,\mu,\chi_2,\alpha,\kappa_2,\epsilon)}\\
\le \frac{(C_4)^2C'_3}{C_{\tilde{P}}(r_{\tilde{Q},\tilde{R}_D})^{\frac{1}{\delta_D\kappa_2}}}|\epsilon_0|^{(\chi_2+\alpha)(k_{\ell,3}+2\gamma_2-2k_{0,2}+k_{0,3}-\kappa_2\delta_D+1)-\chi_2(k_{\ell,3}+2\gamma_2-2k_{0,2}+k_{0,3})}\left\|\omega_2\right\|^3_{(\nu,\beta,\mu,\chi_2,\alpha,\kappa_2,\epsilon)}
\end{multline}

We choose $r_{\tilde{Q},\tilde{R}_D}$ and $\varpi$ such that

\begin{multline}\label{e868c}
\sum_{\ell=0}^{s_3}\frac{|a_{\ell,3}|}{\Gamma\left(\frac{k_{\ell,3}+2\gamma_2-2k_{0,2}+k_{0,3}}{\kappa_2}\right)}\frac{(C_4)^2C'_3}{C_{\tilde{P}}(r_{\tilde{Q},\tilde{R}_D})^{\frac{1}{\delta_D\kappa_2}}}|\epsilon_0|^{(\chi_2+\alpha)(k_{\ell,3}+2\gamma_2-2k_{0,2}+k_{0,2}-\kappa_2\delta_D+1)-\chi_2(k_{\ell,3}+2\gamma_2-2k_{0,2}+k_{0,3})}\varpi^3\\
+\sum_{\ell=s_3+1}^{M_3}\frac{|a_{\ell,3}|}{\Gamma\left(\frac{k_{\ell,3}+2\gamma_2-2k_{0,2}+k_{0,3}}{\kappa_2}\right)}\frac{(C_4)^2C'_3}{C_{\tilde{P}}(r_{\tilde{Q},\tilde{R}_D})^{\frac{1}{\delta_D\kappa_2}}}\\
\times|\epsilon_0|^{m_{\ell,3}+3\beta-\alpha k_{\ell,3}+(\chi_2+\alpha)(k_{\ell,3}+2\gamma_2-2k_{0,2}+k_{0,3}-\kappa_2\delta_D+1)-\chi_2(k_{\ell,3}+2\gamma_2-2k_{0,2}+k_{0,3})}\varpi^3\\
\le\frac{\varpi}{4}
\end{multline}

This choice allows to guarantee that
\begin{equation}\label{e874c}
\left\|\tilde{\mathcal{H}}_\epsilon^4(\omega_2)\right\|_{(\nu,\beta,\mu,\chi_2,\alpha,\kappa_2,\epsilon)}\le \frac{\varpi}{4}.
\end{equation}

In view of (\ref{e704c}), (\ref{e784c}), (\ref{e823c}) and (\ref{e868c}) we conclude the first part of the proof of Lemma~\ref{lema568b}, namely, the existence of $\varpi>0$ such that $\tilde{\mathcal{H}}_{\epsilon}$ sends $\bar{B}(0,\varpi)\subseteq F^{d}_{(\nu,\beta,\mu,\chi_2,\alpha,\kappa_2,\epsilon)}$ into itself.

In order to prove the second part of Lemma~\ref{lema568b}, we split the proof into four parts, which correspond to the terms associated to $\tilde{\mathcal{H}}_{\epsilon}$. Let $\omega_1,\omega_2\in\bar{B}(0,\varpi)\subseteq F^{d}_{(\nu,\beta,\mu,\chi_2,\alpha,\kappa_2,\epsilon)}$.

We state upper bounds concerning the term $\tilde{\mathcal{H}}^1_{\epsilon}$. We have

\begin{multline}\label{e626d}
\left\|\epsilon^{-\chi_2(k_{\ell,1}-2k_{0,2}+k_{0,3})}\frac{\tilde{Q}(im)}{\tilde{P}_{2,m}(\tau)}(\tau^{k_{\ell,1}-2k_{0,2}+k_{0,3}}\star_{\kappa_2}(\omega_1-\omega_2))\right\|_{(\nu,\beta,\mu,\chi_2,\alpha,\kappa_2,\epsilon)}\\
\le \frac{C_3}{C_{\tilde{P}}(r_{\tilde{Q},\tilde{R}_D})^{\frac{1}{\delta_D \kappa_2}}}|\epsilon|^{\alpha(k_{\ell,1}-2k_{0,2}+k_{0,3})}\left\|\omega_1-\omega_2\right\|_{(\nu,\beta,\mu,\chi_2,\alpha,\kappa_2,\epsilon)},
\end{multline} 

\begin{multline}\label{e627d}
\left\|\epsilon^{-\chi_2(k_{\ell,2}-k_{0,2})}\frac{\tilde{Q}(im)}{\tilde{P}_{2,m}(\tau)}(\tau^{k_{\ell,2}-k_{0,2}}\star_{\kappa_2}(\omega_1-\omega_2))\right\|_{(\nu,\beta,\mu,\chi_2,\alpha,\kappa_2,\epsilon)}\\
\le \frac{C_3}{C_{\tilde{P}}(r_{\tilde{Q},\tilde{R}_D})^{\frac{1}{\delta_D \kappa_2}}}|\epsilon|^{\alpha(k_{\ell,2}-k_{0,2})}\left\|\omega_1-\omega_2\right\|_{(\nu,\beta,\mu,\chi_2,\alpha,\kappa_2,\epsilon)},
\end{multline}

\begin{multline}\label{e628d}
\left\|\epsilon^{-\chi_2(k_{\ell,3}-k_{0,3})}\frac{\tilde{Q}(im)}{\tilde{P}_{2,m}(\tau)}(\tau^{k_{\ell,3}-k_{0,3}}\star_{\kappa_2}(\omega_1-\omega_2))\right\|_{(\nu,\beta,\mu,\chi_2,\alpha,\kappa_2,\epsilon)}\\
\le \frac{C_3}{C_{\tilde{P}}(r_{\tilde{Q},\tilde{R}_D})^{\frac{1}{\delta_D \kappa_2}}}|\epsilon|^{\alpha(k_{\ell,3}-k_{0,3})}\left\|\omega_1-\omega_2\right\|_{(\nu,\beta,\mu,\chi_2,\alpha,\kappa_2,\epsilon)}.
\end{multline}

Let $j\ge 1$. A constant $C_3(j)>0$, depending on $\nu,\kappa_2$, $k_{\ell,2}$ for $0\le \ell\le M_2$, $\tilde{Q}(X)$, $\tilde{R}_{D}(X)$, exists such that

\begin{multline}\label{e629d}
\left\|\epsilon^{-\chi_2(k_{\ell,2}-k_{0,2}+j)}\frac{\tilde{Q}(im)}{\tilde{P}_{2,m}(\tau)}(\tau^{k_{\ell,2}-k_{0,2}+j}\star_{\kappa_2}(\omega_1-\omega_2
))\right\|_{(\nu,\beta,\mu,\chi_2,\alpha,\kappa_2,\epsilon)}\\
\le \frac{C_{3.1}(j)}{C_{\tilde{P}}(r_{\tilde{Q},\tilde{R}_D})^{\frac{1}{\delta_D \kappa_2}}}|\epsilon|^{\alpha(k_{\ell,2}-k_{0,2}+j)}\left\|\omega_1-\omega_2
\right\|_{(\nu,\beta,\mu,\chi_2,\alpha,\kappa_2,\epsilon)},
\end{multline}

and

\begin{multline}\label{e630d}
\left\|\epsilon^{-\chi_2(k_{\ell,3}-k_{0,3}+j)}\frac{\tilde{Q}(im)}{\tilde{P}_{2,m}(\tau)}(\tau^{k_{\ell,3}-k_{0,3}+j}\star_{\kappa_2}(\omega_1-\omega_2
))\right\|_{(\nu,\beta,\mu,\chi_2,\alpha,\kappa_2,\epsilon)}\\
\le \frac{C_{3.2}(j)}{C_{\tilde{P}}(r_{\tilde{Q},\tilde{R}_D})^{\frac{1}{\delta_D \kappa_2}}}|\epsilon|^{\alpha(k_{\ell,3}-k_{0,3}+j)}\left\|\omega_1-\omega_2
\right\|_{(\nu,\beta,\mu,\chi_2,\alpha,\kappa_2,\epsilon)}.
\end{multline}

The constants $C_{3.1}(j),C_{3.2}(j)>0$ are as in (\ref{e633c}) and (\ref{e634c}), respectively.

We also have

\begin{multline}\label{e631d}
\left\|\epsilon^{-\chi_2(k_{\ell,3}-k_{0,3})}\frac{\tilde{Q}(im)}{\tilde{P}_{2,m}(\tau)}\left[\tau^{k_{\ell,3}-k_{0,3}+j}\star_{\kappa_2}\mathcal{B}_{\kappa_2}J_2(\tau,\epsilon)\star_{\kappa_2}\mathcal{B}_{\kappa_2}J_2(\tau,\epsilon)\star_{\kappa_2}\omega_1-\omega_2\right]\right\|_{(\nu,\beta,\mu,\chi_2,\alpha,\kappa_2,\epsilon)}\hfill\\
\le \frac{C_{3.3}(j)}{C_{\tilde{P}}(r_{\tilde{Q},\tilde{R}_D})^{\frac{1}{\delta_D \kappa_2}}}|\epsilon|^{\alpha(k_{\ell,3}-k_{0,3}+j)}\left\|\omega_1-\omega_2\right\|_{(\nu,\beta,\mu,\chi_2,\alpha,\kappa_2,\epsilon)},\hfill
\end{multline}
where $C_{3.3}(j)$ is as in (\ref{e677b}).

We choose large enough $r_{\tilde{Q},\tilde{R}_D}>0$ and $\varpi>0$ such that

\begin{multline}
\sum_{\ell=0}^{s_1}\frac{|a_{\ell,1}|}{\Gamma\left(\frac{k_{\ell,1}-2k_{0,2}+k_{0,3}}{\kappa_2}\right)}|\epsilon_0|^{\alpha(k_{\ell,1}-2k_{0,2}+k_{0,3})}\frac{C_3}{C_{\tilde{P}}(r_{\tilde{Q},\tilde{R}_D})^{\frac{1}{\delta_D\kappa_2}}}\hfill\\
\hfill+\sum_{\ell=s_1+1}^{M_1}\frac{|a_{\ell,1}|}{\Gamma\left(\frac{k_{\ell,1}-k_{0,2}+k_{0,3}}{\kappa_2}\right)}|\epsilon_0|^{m_{\ell,1}+\beta-\alpha k_{\ell,1}+\alpha(k_{\ell,1}-2k_{0,1}+k_{0,3})}\frac{C_3}{C_{\tilde{P}}(r_{\tilde{Q},\tilde{R}_D})^{\frac{1}{\delta_D\kappa_2}}}\\
+\sum_{\ell=1}^{s_2}\frac{2|a_{\ell,2}a_{0,2}|}{|a_{0,3}|}\frac{|\epsilon_0|^{\alpha(k_{\ell,2}-k_{0,2})}}{\Gamma\left(\frac{k_{\ell,2}-k_{0,2}}{\kappa_2}\right)}\frac{C_3}{C_{\tilde{P}}(r_{\tilde{Q},\tilde{R}_D})^{\frac{1}{\delta_D\kappa_2}}}\hfill\\
\hfill+\sum_{\ell=s_2+1}^{M_2}\frac{2|a_{\ell,2}a_{0,2}|}{|a_{0,3}|}\frac{|\epsilon_0|^{m_{\ell,2}+2\beta-\alpha k_{\ell,2}+\alpha(k_{\ell,2}-k_{0,2})}}{\Gamma\left(\frac{k_{\ell,2}-k_{0,2}}{\kappa_2}\right)}\frac{C_3}{C_{\tilde{P}}(r_{\tilde{Q},\tilde{R}_D})^{\frac{1}{\delta_D\kappa_2}}}\\
+\sum_{\ell=0}^{s_2}\frac{2|a_{\ell,2}a_{0,2}|}{|a_{0,3}|}\sum_{j\ge1}|J_{2,j}|\frac{|\epsilon_0|^{\alpha(k_{\ell,2}-k_{0,2}+j)}}{\Gamma\left(\frac{k_{\ell,2}-k_{0,2}}{\kappa_2}\right)}\frac{\hat{C}_3A_3^j}{C_{\tilde{P}}(r_{\tilde{Q},\tilde{R}_D})^{\frac{1}{\delta_D\kappa_2}}}\hfill\\
\hfill+\sum_{\ell=s_2+1}^{M_2}\frac{2|a_{\ell,2}a_{0,2}|}{|a_{0,3}|}\sum_{j\ge1}|J_{2,j}|\frac{|\epsilon_0|^{m_{\ell,2}+2\beta-\alpha k_{\ell,2}+\alpha(k_{\ell,2}-k_{0,2}+j)}}{\Gamma\left(\frac{k_{\ell,2}-k_{0,2}}{\kappa_2}\right)}\frac{\hat{C}_3A_3^j}{C_{\tilde{P}}(r_{\tilde{Q},\tilde{R}_D})^{\frac{1}{\delta_D\kappa_2}}}\\
+\sum_{\ell=1}^{s_3}\frac{3|a_{0,2}|^2|a_{\ell,3}|}{|a_{0,3}|^2}\frac{|\epsilon_0|^{\alpha(k_{\ell,3}-k_{0,3})}}{\Gamma\left(\frac{k_{\ell,3}-k_{0,3}}{\kappa_2}\right)}\frac{C_3}{C_{\tilde{P}}(r_{\tilde{Q},\tilde{R}_D})^{\frac{1}{\delta_D\kappa_2}}}\hfill\\
\hfill+\sum_{\ell=0}^{s_3}\frac{6|a_{0,2}|^2|a_{\ell,3}|}{|a_{0,3}|^2}\sum_{j\ge1}|J_{2,j}|\frac{|\epsilon_0|^{\alpha(k_{\ell,3}-k_{0,3}+j)}}{\Gamma\left(\frac{k_{\ell,3}-k_{0,3}}{\kappa_2}\right)}\frac{\hat{C}_3A_3^j}{C_{\tilde{P}}(r_{\tilde{Q},\tilde{R}_D})^{\frac{1}{\delta_D\kappa_2}}}\\
+\sum_{\ell=0}^{s_3}\frac{3|a_{0,2}|^2|a_{\ell,3}|}{|a_{0,3}|^2}\sum_{j\ge1}|\tilde{J}_{2,j}|\frac{|\epsilon_0|^{\alpha(k_{\ell,3}-k_{0,3}+j)}}{\Gamma\left(\frac{k_{\ell,3}-k_{0,3}}{\kappa_2}\right)}\frac{\hat{C}_3A_3^j}{C_{\tilde{P}}(r_{\tilde{Q},\tilde{R}_D})^{\frac{1}{\delta_D\kappa_2}}}\hfill\\
\hfill+\sum_{\ell=s_3+1}^{M_3}\frac{3|a_{0,2}|^2|a_{\ell,3}|}{|a_{0,3}|^2}\frac{|\epsilon_0|^{m_{\ell,3}+3\beta-\alpha k_{\ell,3}+\alpha(k_{\ell,3}-k_{0,3})}}{\Gamma\left(\frac{k_{\ell,3}-k_{0,3}}{\kappa_2}\right)}\frac{C_3}{C_{\tilde{P}}(r_{\tilde{Q},\tilde{R}_D})^{\frac{1}{\delta_D\kappa_2}}}\\
+\sum_{\ell=s_3+1}^{M_3}\frac{6|a_{0,2}|^2|a_{\ell,3}|}{|a_{0,3}|^2}\sum_{j\ge1}|J_{2,j}|\frac{|\epsilon_0|^{\alpha(k_{\ell,3}-k_{0,3}+j)}}{\Gamma\left(\frac{k_{\ell,3}-k_{0,3}}{\kappa_2}\right)}\frac{\hat{C}_3A_3^j}{C_{\tilde{P}}(r_{\tilde{Q},\tilde{R}_D})^{\frac{1}{\delta_D\kappa_2}}}\hfill\\
\hfill+\sum_{\ell=s_3+1}^{M_3}\frac{3|a_{0,2}|^2|a_{\ell,3}|}{|a_{0,3}|^2}\sum_{j\ge1}|\tilde{J}_{2,j}|\frac{|\epsilon_0|^{\alpha(k_{\ell,3}-k_{0,3}+j)}}{\Gamma\left(\frac{k_{\ell,3}-k_{0,3}}{\kappa_2}\right)}\frac{\hat{C}_3A_3^j}{C_{\tilde{P}}(r_{\tilde{Q},\tilde{R}_D})^{\frac{1}{\delta_D\kappa_2}}}\le \frac{1}{8}.
\end{multline}

This choice allows us to deduce that

\begin{equation}\label{e704d} 
\left\|\tilde{\mathcal{H}}^1_\epsilon(\omega_1)-\tilde{\mathcal{H}}^1_\epsilon(\omega_2)\right\|_{(\nu,\beta,\mu,\chi_2,\alpha,\kappa_2,\epsilon)}\le \frac{1}{8}\left\|\omega_1-\omega_2\right\|_{(\nu,\beta,\mu,\chi_2,\alpha,\kappa_2,\epsilon)}.
\end{equation}

We now give upper bounds associated to $\tilde{\mathcal{H}}_\epsilon^2$.

We have
\begin{multline}\label{e947c}
\omega_1((\tau^{\kappa_2}-s')^{1/\kappa_2},m-m_1)\omega_1((s')^{1/\kappa_2},m_1)-\omega_2((\tau^{\kappa_2}-s')^{1/\kappa_2},m-m_1)\omega_2((s')^{1/\kappa_2},m_1)\\
=\left(\omega_1((\tau^{\kappa_2}-s')^{1/\kappa_2},m-m_1)-\omega_2((\tau^{\kappa_2}-s')^{1/\kappa_2},m-m_1)\right)\omega_1((s')^{1/\kappa_2},m_1)\\
+\omega_2((\tau^{\kappa_2}-s')^{1/\kappa_2},m-m_1)\left(\omega_1((s')^{1/\kappa_2},m_1)-\omega_2((s')^{1/\kappa_2},m_1)\right).
\end{multline}

For $j=1,2$, we define
$$\tilde{h}_{1j}(\tau,m)=\tau^{\kappa_2-1}\int_{0}^{\tau^{\kappa_2}}\int_{-\infty}^{\infty}\omega_j((\tau^{\kappa_2}-s')^{1/\kappa_2},m-m_1)\omega_j((s')^{1/\kappa_2},m_1)\frac{1}{(\tau^{\kappa_2}-s')s'}ds'dm_1.$$

We get
\begin{multline}\label{e957d}
\left\|\tilde{h}_{11}(\tau,m)-\tilde{h}_{12}(\tau,m)\right\|_{(\nu,\beta,\mu,\chi_2,\alpha,\kappa_2,\epsilon)}\le \frac{C_4}{|\epsilon|^{\chi_2+\alpha}}(\left\|\omega_1\right\|_{(\nu,\beta,\mu,\chi_2,\alpha,\kappa_2,\epsilon)}+\left\|\omega_2\right\|_{(\nu,\beta,\mu,\chi_2,\alpha,\kappa_2,\epsilon)})\\
\times\left\|\omega_1-\omega_2\right\|_{(\nu,\beta,\mu,\chi_2,\alpha,\kappa_2,\epsilon)},
\end{multline}

which leads to

\begin{multline}\label{e726d}
\left\|\epsilon^{-\chi_2(d_\ell-k_{0,2}+k_{0,3}-\delta_\ell)}\frac{\tilde{R}_{\ell}(im)}{\tilde{P}_{2,m}(\tau)}\left[\tau^{\tilde{d}_{\ell,q_1,q_2}}\star_{\kappa_2}(\tau^{\kappa_2 q_2}(\omega_1-\omega_2))\right]\right\|_{(\nu,\beta,\mu,\chi_2,\alpha,\kappa_2,\epsilon)}\\
\le \frac{C'_3}{C_{\tilde{P}}(r_{\tilde{Q},\tilde{R}_D})^{\frac{1}{\delta_D\kappa_2}}}|\epsilon|^{(\chi_2+\alpha)\kappa_2\left(\frac{\tilde{d}_{\ell,q_1,q_2}}{\kappa_2}+q_2-\delta_D+\frac{1}{\kappa_2}\right)-\chi_2(\tilde{d}_{\ell}-k_{0,2}+k_{0,3}-\delta_\ell)}\left\|\omega_1-\omega_2\right\|_{(\nu,\beta,\mu,\chi_2,\alpha,\kappa_2,\epsilon)},
\end{multline}
for every $q_1\ge0$, $q_2\ge 1$ such that $q_1+q_2=\delta_\ell$, and also

\begin{multline}\label{e727d}
\left\|\epsilon^{-\chi_1(d_\ell-k_{0,2}+k_{0,3}-\delta_\ell)}\frac{\tilde{R}_{\ell}(im)}{\tilde{P}_{2,m}(\tau)}\left[\tau^{\tilde{d}_{\ell,q_1,q_2}+\kappa_2(q_2-p)}\star_{\kappa_2}(\tau^{\kappa_2 p}\omega_1-\omega_2)\right]\right\|_{(\nu,\beta,\mu,\chi_2,\alpha,\kappa_2,\epsilon)}\\
\le \frac{C'_3}{C_{\tilde{P}}(r_{\tilde{Q},\tilde{R}_D})^{\frac{1}{\delta_D\kappa_2}}}|\epsilon|^{(\chi_2+\alpha)\kappa_2\left(\frac{\tilde{d}_{\ell,q_1,q_2}}{\kappa_2}+q_2-\delta_D+\frac{1}{\kappa_2}\right)-\chi_2(d_{\ell}-k_{0,2}+k_{0,3}-\delta_\ell)}\left\|\omega_1-\omega_2\right\|_{(\nu,\beta,\mu,\chi_2,\alpha,\kappa_2,\epsilon)},
\end{multline}
for every $1\le p\le q_2-1$. Also,

\begin{multline}\label{e728d}
\left\|\epsilon^{-\chi_2(k_{\ell,2}+\gamma_2-2k_{0,2}+k_{0,3})}\frac{\tilde{Q}(im)}{\tilde{P}_{2,m}(\tau)}\left[\tau^{k_{\ell,2}+\gamma_2-2k_{0,2}+k_{0,3}}\star_{\kappa_2}(\tau [\tilde{h}_{11}(\tau,m)-\tilde{h}_{12}(\tau,m)])\right]\right\|_{(\nu,\beta,\mu,\chi_2,\alpha,\kappa_2,\epsilon)}\\
\le \frac{C_4C'_3}{C_{\tilde{P}}(r_{\tilde{Q},\tilde{R}_D})^{\frac{1}{\delta_D\kappa_2}}}|\epsilon|^{(\chi_2+\alpha)(k_{\ell,2}+\gamma_2-2k_{0,2}+k_{0,3}-\kappa_2\delta_D+1)-\chi_2(k_{\ell,2}+\gamma_2-2k_{0,2}+k_{0,3})}\\
\times(\left\|\omega_1\right\|_{(\nu,\beta,\mu,\chi_2,\alpha,\kappa_2,\epsilon)}+\left\|\omega_2\right\|_{(\nu,\beta,\mu,\chi_2,\alpha,\kappa_2,\epsilon)})\left\|\omega_1-\omega_2\right\|_{(\nu,\beta,\mu,\chi_2,\alpha,\kappa_2,\epsilon)}.
\end{multline} 

The same arguments as above follow to get 

\begin{multline}\label{e729d}
\left\|\epsilon^{-\chi_2k_{\ell,3}}\frac{\tilde{Q}(im)}{\tilde{P}_{2,m}(\tau)}\left[\tau^{k_{\ell,3}}\star_{\kappa_2}(\tau [\tilde{h}_{11}(\tau,m)-\tilde{h}_{12}(\tau,m)])\right]\right\|_{(\nu,\beta,\mu,\chi_2,\alpha,\kappa_2,\epsilon)}\\
\le \frac{C_4C'_3}{C_{\tilde{P}}(r_{\tilde{Q},\tilde{R}_D})^{\frac{1}{\delta_D\kappa_2}}}|\epsilon|^{(\chi_2+\alpha)(k_{\ell,3}-\kappa_2\delta_D+1)-\chi_2k_{\ell,3}}\\
\times(\left\|\omega_1\right\|_{(\nu,\beta,\mu,\chi_2,\alpha,\kappa_2,\epsilon)}+\left\|\omega_2\right\|_{(\nu,\beta,\mu,\chi_2,\alpha,\kappa_2,\epsilon)})\left\|\omega_1-\omega_2\right\|_{(\nu,\beta,\mu,\chi_2,\alpha,\kappa_2,\epsilon)},
\end{multline} 

and

\begin{multline}\label{e730d}
\left\|\epsilon^{-\chi_2(k_{\ell,3}+j)}\frac{\tilde{Q}(im)}{\tilde{P}_{2,m}(\tau)}\left[\tau^{k_{\ell,3}+j}\star_{\kappa_2}(\tau [\tilde{h}_{11}(\tau,m)-\tilde{h}_{12}(\tau,m)])\right]\right\|_{(\nu,\beta,\mu,\chi_2,\alpha,\kappa_2,\epsilon)}\\
\le \frac{C_4C'_3(j)}{C_{\tilde{P}}(r_{\tilde{Q},\tilde{R}_D})^{\frac{1}{\delta_D\kappa_2}}}|\epsilon|^{(\chi_2+\alpha)(k_{\ell,3}+j-\kappa_2\delta_D+1)-\chi_2(k_{\ell,3}+j)}\\
\times(\left\|\omega_1\right\|_{(\nu,\beta,\mu,\chi_2,\alpha,\kappa_2,\epsilon)}+\left\|\omega_2\right\|_{(\nu,\beta,\mu,\chi_2,\alpha,\kappa_2,\epsilon)})\left\|\omega_1-\omega_2\right\|_{(\nu,\beta,\mu,\chi_2,\alpha,\kappa_2,\epsilon)},
\end{multline} 
for every $j\ge1$, where $C'_3(j)$ is determined in (\ref{e731}).

One can choose $r_{\tilde{Q},\tilde{R}_D}>0$ and $\varpi>0$ such that the following condition holds:

\begin{multline}\label{e766d}
\sum_{\ell=1}^{D-1}|\epsilon_0|^{\Delta_{\ell}+\alpha(\delta_{\ell}-d_{\ell})+\beta}\left[\prod_{d=0}^{\delta_\ell-1}|\gamma_2-d|\frac{C_3}{C_{\tilde{P}}(r_{\tilde{Q},\tilde{R}_D})^{\frac{1}{\delta_D\kappa_2}}\Gamma\left(\frac{\tilde{d}_{\ell,\delta_\ell,0}}{\kappa_2}\right)}|\epsilon_0|^{\alpha(d_\ell-k_{0,2}+k_{0,3}-\delta_\ell)}\right.\\
+\sum_{q_1+q_2=\delta_\ell,q_2\ge1}\frac{\delta_\ell!}{q_1!q_2!}\prod_{d=0}^{q_1-1}|\gamma_2-d|\left(\frac{C'_3\kappa_2^{q_2}}{C_{\tilde{P}}(r_{\tilde{Q},\tilde{R}_D})^{\frac{1}{\delta_D\kappa_2}}\Gamma\left(\frac{\tilde{d}_{\ell,q_1,q_2}}{\kappa_2}\right)}\right.\hfill\\
\hfill\times |\epsilon_0|^{(\chi_2+\alpha)\kappa_2\left(\frac{\tilde{d}_{\ell,q_1,q_2}}{\kappa_2}+q_2-\delta_D+\frac{1}{\kappa_2}\right)-\chi_2(d_\ell-k_{0,2}+k_{0,3}-\delta_\ell)}\\
+\sum_{1\le p\le q_2-1}|\tilde{A}_{q_2,p}|\frac{C'_3\kappa_2^p}{C_{\tilde{P}}(r_{\tilde{Q},\tilde{R}_D})^{\frac{1}{\delta_D\kappa_2}}\Gamma\left(\frac{\tilde{d}_{\ell,q_1,q_2}}{\kappa_2}+q_2-p\right)}\hfill\\
\hfill\left.\left. \times |\epsilon_0|^{(\chi_2+\alpha)\kappa_2\left(\frac{\tilde{d}_{\ell,q_1,q_2}}{\kappa_2}+q_2-\delta_D+\frac{1}{\kappa_2}\right)-\chi_2(d_\ell-k_{0,2}+k_{0,3}-\delta_\ell)}   \right)    \right]\\
+\sum_{\ell=0}^{s_2}\frac{|a_{\ell,2}|}{\Gamma\left(\frac{k_{\ell,2}+\gamma_2-2k_{0,2}+k_{0,3}}{\kappa_2}\right)}\frac{C_4C'_3}{C_{\tilde{P}}(r_{\tilde{Q},\tilde{R}_D})^{\frac{1}{\delta_D\kappa_2}}}|\epsilon_0|^{(\chi_2+\alpha)(k_{\ell,2}+\gamma_2-2k_{0,2}+k_{0,3}-\kappa_2\delta_D+1)-\chi_2(k_{\ell,2}+\gamma_2-2k_{0,2}+k_{0,3})}2\varpi \hfill\\
\hfill+\sum_{\ell=s_2+1}^{M_2}\frac{|a_{\ell,2}|C_4C'_3}{\Gamma\left(\frac{k_{\ell,2}+\gamma_2-2k_{0,2}+k_{0,3}}{\kappa_2}\right)}\frac{|\epsilon_0|^{m_{\ell,2}+2\beta-\alpha k_{\ell,2}+(\chi_2+\alpha)(k_{\ell,2}+\gamma_2-2k_{0,2}+k_{0,3}-\kappa_2\delta_D+1)-\chi_2(k_{\ell,2}+\gamma_2-2k_{0,2}+k_{0,3})}}{C_{\tilde{P}}(r_{\tilde{Q},\tilde{R}_D})^{\frac{1}{\delta_D\kappa_2}}}2\varpi\\
+\sum_{\ell=0}^{s_3}|a_{\ell,3}|\frac{C_4C'_3}{\Gamma\left(\frac{k_{\ell,3}}{\kappa_2}\right)}\frac{|\epsilon_0|^{(\chi_2+\alpha)(k_{\ell,3}-\kappa_2\delta_D+1)-\chi_2k_{\ell,3}}}{C_{\tilde{P}}(r_{\tilde{Q},\tilde{R}_D})^{\frac{1}{\delta_D\kappa_2}}}2\varpi\hfill\\
\hfill+\sum_{\ell=s_3+1}^{M_3}|a_{\ell,3}|\frac{C_4C'_3}{\Gamma\left(\frac{k_{\ell,3}}{\kappa_2}\right)}\frac{|\epsilon_0|^{m_{\ell,3}+3\beta-\alpha k_{\ell,3}+(\chi_2+\alpha)(k_{\ell,3}-\kappa_2\delta_D+1)-\chi_2k_{\ell,3}}}{C_{\tilde{P}}(r_{\tilde{Q},\tilde{R}_D})^{\frac{1}{\delta_D\kappa_2}}}2\varpi\\
+\sum_{\ell=0}^{s_3}|a_{\ell,3}|\sum_{j\ge1}|J_{2,j}|\frac{|\epsilon_0|^{(\chi_2+\alpha)(k_{\ell,3}+j-\kappa_2\delta_D+1)-\chi_2(k_{\ell,3}+j)}}{\Gamma\left(\frac{k_{\ell,3}}{\kappa_2}\right)}\frac{C_4\hat{C}_3A_3^j}{C_{\tilde{P}}(r_{\tilde{Q},\tilde{R}_D})^{\frac{1}{\delta_D\kappa_2}}}2\varpi\\
+\sum_{\ell=s_3+1}^{M_3}|a_{\ell,3}|\sum_{j\ge1}|J_{2,j}|\frac{|\epsilon_0|^{m_{\ell,3}+3\beta-\alpha k_{\ell,3}+(\chi_2+\alpha)(k_{\ell,3}+j-\kappa_2\delta_D+1)-\chi_2(k_{\ell,3}+j)}}{\Gamma\left(\frac{k_{\ell,3}}{\kappa_2}\right)}\frac{C_4\hat{C}_3A_3^j}{C_{\tilde{P}}(r_{\tilde{Q},\tilde{R}_D})^{\frac{1}{\delta_D\kappa_2}}}2\varpi\\
\le\frac{1}{8}.\hfill
\end{multline}

We get
 
\begin{equation}\label{e705d} 
\left\|\tilde{\mathcal{H}}^2_\epsilon(\omega_1)-\tilde{\mathcal{H}}^2_\epsilon(\omega_1)\right\|_{(\nu,\beta,\mu,\chi_2,\alpha,\kappa_2,\epsilon)}\le \frac{1}{8}\left\|\omega_1-\omega_2\right\|_{(\nu,\beta,\mu,\chi_2,\alpha,\kappa_2,\epsilon)}.
\end{equation}

We now give upper estimates for the difference associated to $\tilde{\mathcal{H}}_{\epsilon}^3$.

It holds that
 
\begin{multline}\label{e793d}
\left\|\epsilon^{-\chi_2(d_{D}-2k_{0,2}+k_{0,3}-\delta_D)}\frac{\tilde{R}_{D}(im)}{\tilde{P}_{2,m}(\tau)}\left[\tau^{\tilde{d}_{D,\delta_D,0}}\star_{\kappa_2}(\omega_1-\omega_2)\right]\right\|_{(\nu,\beta,\mu,\chi_2,\alpha,\kappa_2,\epsilon)}\\
\le\frac{C_3}{C_{\tilde{P}}(r_{\tilde{Q},\tilde{R}_D})^{\frac{1}{\delta_D\kappa_2}}}|\epsilon_0|^{\alpha(d_{D}-2k_{0,2}+k_{0,3}-\delta_D)}\left\|\omega_1-\omega_2\right\|_{(\nu,\beta,\mu,\chi_2,\alpha,\kappa_2,\epsilon)},
\end{multline}

\begin{multline}\label{e794d}
\left\|\epsilon^{-\chi_2(d_{D}-2k_{0,2}+k_{0,3}-\delta_D)}\frac{\tilde{R}_{D}(im)}{\tilde{P}_{2,m}(\tau)}\left[\tau^{\tilde{d}_{D,q_1,q_2}}\star_{\kappa_2}(\tau^{\kappa_2 q_2}(\omega_1-\omega_2))\right]\right\|_{(\nu,\beta,\mu,\chi_2,\alpha,\kappa_2,\epsilon)}\\
\le\frac{C'_3}{C_{\tilde{P}}(r_{\tilde{Q},\tilde{R}_D})^{\frac{1}{\delta_D\kappa_2}}}|\epsilon_0|^{(\chi_2+\alpha)\kappa_2\left(\frac{\tilde{d}_{d_{D},q_1,q_2}}{\kappa_2}+q_2-\delta_D\right)-\chi_2(d_D-2k_{0,2}+k_{0,3}-\delta_D)}\left\|\omega_1-\omega_2\right\|_{(\nu,\beta,\mu,\chi_2,\alpha,\kappa_2,\epsilon)},
\end{multline}

for every $q_1\ge 1$ and $q_2\ge 1$ with $q_1+q_2=\delta_D$. In addition to that, it holds

\begin{multline}\label{e795d}
\left\|\frac{\tilde{R}_{D}(im)}{\tilde{P}_{2,m}(\tau)}\left[\tau^{\kappa_2(\delta_D-p)}\star_{\kappa_2}(\tau^{\kappa_2 p}(\omega_1-\omega_2))\right]\right\|_{(\nu,\beta,\mu,\chi_2,\alpha,\kappa_2,\epsilon)}\hfill\\
\hfill\le\frac{C'_3}{C_{\tilde{P}}(r_{\tilde{Q},\tilde{R}_D})^{\frac{1}{\delta_D\kappa_2}}}|\epsilon_0|^{\chi_2+\alpha}\left\|\omega_1-\omega_2\right\|_{(\nu,\beta,\mu,\chi_2,\alpha,\kappa_2,\epsilon)},
\end{multline}
for all $1\le p\le \delta_D-1$.

We choose $r_{\tilde{Q},\tilde{R}_D}$ and $\varpi$ such that

\begin{multline}\label{e816d}
|\epsilon_0|^{\Delta_D+\alpha(\delta_D-d_D)+\beta}\left[\prod_{d=0}^{\delta_D-1}|\gamma_2-d|\frac{C_3}{C_{\tilde{P}}(r_{\tilde{Q},\tilde{R}_D})^{\frac{1}{\delta_D\kappa_2}}\Gamma\left(\frac{\tilde{d}_{D,\delta_D,0}}{\kappa_2}\right)}|\epsilon_0|^{\alpha(d_{D}-2k_{0,2}+k_{0,3}-\delta_D)}\right.\\
+\sum_{q_1+q_2=\delta_D,q_1\ge1,q_2\ge1}\frac{\delta_D!}{q_1!q_2!}\prod_{d=0}^{q_1-1}|\gamma_2-d|\left(\frac{C'_3\kappa_2^{q_2}}{C_{\tilde{P}}(r_{\tilde{Q},\tilde{R}_D})^{\frac{1}{\delta_D\kappa_2}}\Gamma\left(\frac{\tilde{d}_{D,q_1,q_2}}{\kappa_2}\right)}\right.\\
\hfill\times |\epsilon_0|^{(\chi_2+\alpha)\kappa_2\left(\frac{\tilde{d}_{d_{D},q_1,q_2}}{\kappa_2}+q_2-\delta_D+\frac{1}{\kappa_2}\right)-\chi_2(d_D-2k_{0,2}+k_{0,3}-\delta_D)}\\
+\sum_{1\le p\le q_2-1}|\tilde{A}_{q_2,p}|\frac{C'_3\kappa_2^p}{C_{\tilde{P}}(r_{\tilde{Q},\tilde{R}_D})^{\frac{1}{\delta_D\kappa_2}}\Gamma\left(\frac{\tilde{d}_{D,q_1,q_2}}{\kappa_2}+q_2-p\right)}\\
\left.\left.\times|\epsilon_0|^{(\chi_2+\alpha)\kappa_2\left(\frac{\tilde{d}_{d_{D},q_1,q_2}}{\kappa_2}+q_2-\delta_D+\frac{1}{\kappa_2}\right)-\chi_2(d_D-2k_{0,2}+k_{0,3}-\delta_D)}\right)\right]\\
+\sum_{1\le p\le \delta_D-1}|\tilde{A}_{\delta_D,p}|\frac{C'_3\kappa_2^p}{C_{\tilde{P}}(r_{\tilde{Q},\tilde{R}_D})^{\frac{1}{\delta_D\kappa_2}}\Gamma\left(\delta_D-p\right)}|\epsilon_0|^{\chi_2+\alpha}\le\frac{1}{8}.
\end{multline}

This choice guarantees that
\begin{equation}\label{e823d}
\left\|\tilde{\mathcal{H}}_\epsilon^3(\omega_1)-\tilde{\mathcal{H}}_\epsilon^3(\omega_2)\right\|_{(\nu,\beta,\mu,\chi_2,\alpha,\kappa_2,\epsilon)}\le \frac{1}{8}\left\|\omega_1-\omega_2\right\|_{(\nu,\beta,\mu,\chi_2,\alpha,\kappa_2,\epsilon)}.
\end{equation}

We finally give upper bounds for the elements involved in $\tilde{\mathcal{H}}_\epsilon^4$.

We have

\begin{multline}\label{e947d}
\omega_1((\tau^{\kappa_2}-s')^{1/\kappa_2},m-m_1)[\omega_1\star_{\kappa_2}^E\omega_1]((s')^{1/\kappa_2},m_1)\\-\omega_2((\tau^{\kappa_2}-s')^{1/\kappa_2},m-m_1)[\omega_2\star_{\kappa_2}^E\omega_2]((s')^{1/\kappa_2},m_1)\\
=\left(\omega_1((\tau^{\kappa_2}-s')^{1/\kappa_2},m-m_1)-\omega_2((\tau^{\kappa_2}-s')^{1/\kappa_2},m-m_1)\right)[\omega_1\star_{\kappa_2}^E\omega_1]((s')^{1/\kappa_2},m_1)\\
+\omega_2((\tau^{\kappa_2}-s')^{1/\kappa_2},m-m_1)\left([\omega_1\star_{\kappa_2}^E\omega_1]((s')^{1/\kappa_2},m_1)-[\omega_2\star_{\kappa_2}^E\omega_2]((s')^{1/\kappa_2},m_1)\right).
\end{multline}

For $j=1,2$, we define
$$\tilde{h}_{2j}(\tau,m)=\tau^{\kappa_2-1}\int_{0}^{\tau^{\kappa_2}}\int_{-\infty}^{\infty}\omega_j((\tau^{\kappa_2}-s')^{1/\kappa_2},m-m_1)[\omega_j\star_{\kappa_2}^{E}\omega_j]((s')^{1/\kappa_2},m_1)\frac{1}{(\tau^{\kappa_2}-s')s'}ds'dm_1.$$

It is straight to check that
\begin{multline}\label{e1070d}
\left\| [\omega_1\star^{E}_{\kappa_2}\omega_1]-[\omega_2\star^{E}_{\kappa_2}\omega_2]\right\|_{(\nu,\beta,\mu,\chi_2,\alpha,\kappa_2,\epsilon)}\\
\le C_4 \left\|\omega_1-\omega_2\right\|_{(\nu,\beta,\mu,\chi_2,\alpha,\kappa_2,\epsilon)}\left(\left\|\omega_1\right\|_{(\nu,\beta,\mu,\chi_2,\alpha,\kappa_2,\epsilon)}+\left\|\omega_2\right\|_{(\nu,\beta,\mu,\chi_2,\alpha,\kappa_2,\epsilon)}\right)
\end{multline}

Followig the same steps as in the proof of (\ref{e957b}), we derive

$$\left\|\tilde{h}_{21}(\tau,m)-\tilde{h}_{22}(\tau,m)\right\|_{(\nu,\beta,\mu,\chi_2,\alpha,\kappa_2,\epsilon)}\le\frac{3C_4^2\varpi^2}{|\epsilon|^{\chi_2+\alpha}} \left\|\omega_1-\omega_2\right\|_{(\nu,\beta,\mu,\chi_2,\alpha,\kappa_2,\epsilon)}.$$

It also holds that

$$\left\|\tilde{h}_2(\tau,m)\right\|_{(\nu,\beta,\mu,\chi_2,\alpha,\kappa_2,\epsilon)}\le\frac{C_4}{|\epsilon|^{\chi_2+\alpha}}\left\|\omega_2\right\|_{(\nu,\beta,\mu,\chi_2,\alpha,\kappa_2,\epsilon)}\left\| \omega_2\star_{\kappa_2}^{E}\omega_2 \right\|_{(\nu,\beta,\mu,\chi_2,\alpha,\kappa_2,\epsilon)}.$$

Both, together with (\ref{e852b}), yield

\begin{multline}\label{e860d}
\left\|\epsilon^{-\chi_2(k_{\ell,3}+2\gamma_2-2k_{0,2}+k_{0,3})}\frac{\tilde{Q}(im)}{\tilde{P}_{2,m}(\tau)}\left[\tau^{k_{\ell,3}+2\gamma_2-2k_{0,2}+k_{0,3}}\star_{\kappa_2}[\tau (\tilde{h}_{21}(\tau,m)-\tilde{h}_{22}(\tau,m))]\right]\right\|_{(\nu,\beta,\mu,\chi_2,\alpha,\kappa_2,\epsilon)}\\
\le \frac{3\varpi^2(C_4)^2C'_3}{C_{\tilde{P}}(r_{\tilde{Q},\tilde{R}_D})^{\frac{1}{\delta_D\kappa_2}}}|\epsilon_0|^{(\chi_2+\alpha)(k_{\ell,3}+2\gamma_2-2k_{0,2}+k_{0,3}-\kappa_2\delta_D+1)-\chi_2(k_{\ell,3}+2\gamma_2-2k_{0,2}+k_{0,3})}\left\|\omega_1-\omega_2\right\|_{(\nu,\beta,\mu,\chi_2,\alpha,\kappa_2,\epsilon)}.
\end{multline}

We choose $r_{\tilde{Q},\tilde{R}_D}$ and $\varpi$ such that

\begin{multline}\label{e868d}
\sum_{\ell=0}^{s_3}\frac{|a_{\ell,3}|}{\Gamma\left(\frac{k_{\ell,3}+2\gamma_2-2k_{0,2}+k_{0,3}}{\kappa_2}\right)}\frac{(C_4)^2C'_3\varpi^2}{C_{\tilde{P}}(r_{\tilde{Q},\tilde{R}_D})^{\frac{1}{\delta_D\kappa_2}}}|\epsilon_0|^{(\chi_2+\alpha)(k_{\ell,3}+2\gamma_2-2k_{0,2}+k_{0,2}-\kappa_2\delta_D+1)-\chi_2(k_{\ell,3}+2\gamma_2-2k_{0,2}+k_{0,3})}\\
+\sum_{\ell=s_3+1}^{M_3}\frac{|a_{\ell,3}|}{\Gamma\left(\frac{k_{\ell,3}+2\gamma_2-2k_{0,2}+k_{0,3}}{\kappa_2}\right)}\frac{(C_4)^2C'_3\varpi^2}{C_{\tilde{P}}(r_{\tilde{Q},\tilde{R}_D})^{\frac{1}{\delta_D\kappa_2}}}\\
\times |\epsilon_0|^{m_{\ell,3}+3\beta-\alpha k_{\ell,3}+(\chi_2+\alpha)(k_{\ell,3}+2\gamma_2-2k_{0,2}+k_{0,3}-\kappa_2\delta_D+1)-\chi_2(k_{\ell,3}+2\gamma_2-2k_{0,2}+k_{0,3})}\\
\le\frac{1}{8}.
\end{multline}

This choice allows to guarantee that
\begin{equation}\label{e874d}
\left\|\tilde{\mathcal{H}}_\epsilon^4(\omega_1)-\tilde{\mathcal{H}}_\epsilon^4(\omega_2)\right\|_{(\nu,\beta,\mu,\chi_2,\alpha,\kappa_2,\epsilon)}\le \frac{1}{8}\left\|\omega_1-\omega_2\right\|_{(\nu,\beta,\mu,\chi_2,\alpha,\kappa_2,\epsilon)}.
\end{equation}

In view of (\ref{e704d}), (\ref{e705d}), (\ref{e823d}) and (\ref{e874d}) we conclude that
\begin{equation}\label{e1889}
\left\|\tilde{\mathcal{H}}_\epsilon(\omega_1)-\tilde{\mathcal{H}}_\epsilon(\omega_2)\right\|_{(\nu,\beta,\mu,\chi_2,\alpha,\kappa_2,\epsilon)}\le \frac{1}{2}\left\|\omega_1-\omega_2\right\|_{(\nu,\beta,\mu,\chi_2,\alpha,\kappa_2,\epsilon)},
\end{equation}

which completes the proof.
\end{proof}

\section{Proof of Theorem~\ref{teo2461}}

\begin{proof}
We proceed with the construction of the two families of actual solutions of the main problem through the steps taken in Section~\ref{seccion4}. We start with the first family of solutions.

Let $( \mathcal{E}_{p} )_{0 \leq p \leq \varsigma_1 - 1}$ be a good covering in $\mathbb{C}^{\ast}$ associated to the Gevrey order $(\chi_1+\alpha)\kappa_1$, and let
$\{ (S_{\mathfrak{d}_{p},\theta_1,\epsilon_{0}r_{\mathcal{T}}})_{0 \leq p \leq \varsigma_1-1},\mathcal{T}_1 \}$ be a
family of sectors associated to this good covering. From Proposition~\ref{prop1518}, we see that for each direction $\mathfrak{d}_{p}$, one can get a solution $\omega_{\kappa_1}^{\mathfrak{d}_{p}}(\tau,m,\epsilon)$
of the convolution equation (\ref{e315}) that belongs to the space
$F_{(\nu,\beta,\mu,\chi_1,\alpha,\kappa_1,\epsilon)}^{\mathfrak{d}_{p}}$ and thus satisfies the next bounds
\begin{equation}
|\omega_{\kappa_1}^{\mathfrak{d}_{p}}(\tau,m,\epsilon)| \leq \varpi (1 + |m|)^{-\mu} e^{-\beta |m|}
\frac{|\frac{\tau}{\epsilon^{\chi_1 + \alpha}}|}{1 + |\frac{\tau}{\epsilon^{\chi_1 + \alpha}}|^{2 \kappa_1}}
\exp( \nu |\frac{\tau}{\epsilon^{\chi_1 + \alpha}}|^{\kappa_1} ) \label{omega_kappa_dp_in_F1}
\end{equation}
for all $\tau \in \bar{D}(0,\rho) \cup S_{\mathfrak{d}_{p}}$, all $m \in \mathbb{R}$, all
$\epsilon \in D(0,\epsilon_{0}) \setminus \{ 0 \}$, for some well chosen $\varpi>0$. Besides, these functions $\omega_{\kappa_1}^{\mathfrak{d}_{p}}(\tau,m,\epsilon)$ are analytic continuations w.r.t $\tau$ of a common convergent series 
$$ \omega_{\kappa_1}(\tau,m,\epsilon) = \sum_{n \geq 1}
\frac{\omega_{n,1}(m,\epsilon)}{\Gamma(\frac{n}{\kappa_1})} \tau^{n} $$
with coefficients in the Banach space $E_{(\beta,\mu)}$
solution of (\ref{e315}) for all $\tau \in D(0,\rho)$. In particular, we see that the formal power
series
$$ \Omega_{\kappa_1}(\mathbb{T},m,\epsilon) = \sum_{n \geq 1} \omega_{n,1}(m,\epsilon) \mathbb{T}^{n} $$
is $m_{\kappa_1}-$summable in direction $\mathfrak{d}_{p}$
as a series with coefficients in the Banach space $E_{(\beta,\mu)}$ for all $\epsilon \in
D(0,\epsilon_{0}) \setminus \{ 0 \}$, in the sense of Definition~\ref{defi398}. We denote
\begin{equation}
\Omega_{\kappa_1}^{\mathfrak{d}_{p}}(\mathbb{T},m,\epsilon) = \kappa_1 \int_{L_{\gamma}}
\omega_{\kappa_1}^{\mathfrak{d}_{p}}(u,m,\epsilon)
\exp( -(\frac{u}{\mathbb{T}})^{\kappa_1} ) \frac{du}{u}
\end{equation}
its $m_{\kappa_1}-$sum along direction $\mathfrak{d}_{p}$, where
$L_{\gamma} = \mathbb{R}_{+}e^{i\gamma} \subset S_{\mathfrak{d}_{p}}$, which defines an 
$E_{(\beta,\mu)}-$valued analytic function
with respect to $\mathbb{T}$ on a sector
$$ S_{\mathfrak{d}_{p},\theta_1,h'|\epsilon|^{\chi_1 + \alpha}} =
\{ \mathbb{T} \in \mathbb{C}^{\ast} : |\mathbb{T}| < h'|\epsilon|^{\chi_1+\alpha} \ \ , \ \
|\mathfrak{d}_{p} - \mathrm{arg}(\mathbb{T})| < \theta/2 \} $$
for $\frac{\pi}{\kappa_1} < \theta_1 < \frac{\pi}{\kappa_1} + \mathrm{Ap}(S_{\mathfrak{d}_{p}})$ (where
$\mathrm{Ap}(S_{\mathfrak{d}_{p}})$ denotes the aperture of the sector $S_{\mathfrak{d}_{p}}$) and some
$h'>0$ (independent of $\epsilon$), for all
$\epsilon \in D(0,\epsilon_{0}) \setminus \{ 0 \}$.

Bearing in mind the identities of Proposition~\ref{prop462} and using
the properties for the $m_{\kappa_1}-$sum with respect to derivatives and products (within the Banach algebra $\mathbb{E}=E_{(\beta,\mu)}$ equipped with the convolution product $\star$ as described in Proposition~\ref{prop10}), we check that the functions $\Omega_{\kappa_1}^{\mathfrak{d}_{p}}(\mathbb{T},m,\epsilon)$ solves the problem
\begin{equation}
\mathcal{L}_{\mathbb{T},m,\epsilon}^1(\Omega_{\kappa_1}^{\mathfrak{d}_{p}}(\mathbb{T},m,\epsilon)) =
\mathcal{R}^1_{\mathbb{T},m,\epsilon}(\Omega_{\kappa_1}^{\mathfrak{d}_{p}}(\mathbb{T},m,\epsilon)),
\label{PDE_Omega_kappa_dp1}
\end{equation}

where
\begin{multline}
\tilde{Q}(im)\Omega_{\kappa_1}^{\mathfrak{d}_{p}}(\mathbb{T},m,\epsilon)\left[-a_{0,1}+\sum_{\ell=1}^{s_1}a_{\ell,1}\epsilon^{-\chi_1(k_{\ell,1}-k_{0,1})}\mathbb{T}^{k_{\ell,1}-k_{0,1}}+\sum_{\ell=s_1+1}^{M_1}a_{\ell,1}\epsilon^{m_{\ell,1}+\beta-\alpha k_{\ell,1}-\chi_1(k_{\ell,1}-k_{0,1})}\mathbb{T}^{k_{\ell,1}-k_{0,1}}\right.\\
\qquad+\left(\sum_{\ell=1}^{s_2}a_{\ell,2}\epsilon^{-\chi_1(k_{\ell,2}-k_{0,2})}\mathbb{T}^{k_{\ell,2}-k_{0,2}}+\sum_{\ell=s_2+1}^{M_2}a_{\ell,2}\epsilon^{m_{\ell,2}+2\beta-\alpha k_{\ell,2}-\chi_1(k_{\ell,2}-k_{0,2})}\mathbb{T}^{k_{\ell,2}-k_{0,2}}\right)\left(\frac{-2a_{0,1}}{a_{0,2}}\right)\\
+\left(\sum_{\ell=0}^{s_2}a_{\ell,2}\epsilon^{-\chi_1(k_{\ell,2}-k_{0,2})}\mathbb{T}^{k_{\ell,2}-k_{0,2}}+\sum_{\ell=s_2+1}^{M_2}a_{\ell,2}\epsilon^{m_{\ell,2}+2\beta-\alpha k_{\ell,2}-\chi_1(k_{\ell,2}-k_{0,2})}\mathbb{T}^{k_{\ell,2}-k_{0,2}}\right)\left(\frac{-2a_{0,1}}{a_{0,2}}\mathcal{J}_1(\epsilon^{-\chi_1}\mathbb{T})\right)\\
+\left(\sum_{\ell=0}^{s_3}a_{\ell,3}\epsilon^{-\chi_1(k_{\ell,3}-k_{0,1})}\mathbb{T}^{k_{\ell,3}-k_{0,1}}+\sum_{\ell=s_3+1}^{M_3}a_{\ell,3}\epsilon^{m_{\ell,3}+3\beta-\alpha k_{\ell,3}-\chi_1(k_{\ell,3}-k_{0,1})}\mathbb{T}^{k_{\ell,3}-k_{0,1}}\right)\\
\left.\hfill\times3\left(\frac{a_{01}}{a_{02}}\epsilon^{-\chi_1(k_{0,1}-k_{0,2})}\mathbb{T}^{k_{0,1}-k_{0,2}}(1+\mathcal{J}_1(\epsilon^{-\chi_1}\mathbb{T}))\right)^2\right]\\
+\tilde{Q}(im)\left(\int_{-\infty}^{+\infty}\Omega_{\kappa_1}^{\mathfrak{d}_{p}}(\mathbb{T},m-m_1,\epsilon)\Omega_{\kappa_1}^{\mathfrak{d}_{p}}(\mathbb{T},m_1,\epsilon)dm_1\right)\\
\times\left[\left(\sum_{\ell=0}^{s_2}a_{\ell,2}\epsilon^{-\chi_1k_{\ell,2}}\mathbb{T}^{k_{\ell,2}}+\sum_{\ell=s_2+1}^{M_2}a_{\ell,2}\epsilon^{m_{\ell,2}+2\beta-\alpha k_{\ell,2}-\chi_1k_{\ell,2}}\mathbb{T}^{k_{\ell,2}}\right)\epsilon^{-\chi_1(\gamma_1-k_{0,1})}\mathbb{T}^{\gamma_1-k_{0,1}}\right.\\
+\left(\sum_{\ell=0}^{s_3}a_{\ell,3}\epsilon^{-\chi_1k_{\ell,3}}\mathbb{T}^{k_{\ell,3}}+\sum_{\ell=s_3+1}^{M_3}a_{\ell,3}\epsilon^{m_{\ell,3}+3\beta-\alpha k_{\ell,3}-\chi_1k_{\ell,3}}\mathbb{T}^{k_{\ell,3}}\right)\hfill\\
\left.\hfill\times\left(\frac{-3a_{01}}{a_{02}}\epsilon^{-\chi_1(k_{0,1}-k_{0,2})}\mathbb{T}^{k_{0,1}-k_{0,2}}(1+\mathcal{J}_1(\epsilon^{-\chi_1}\mathbb{T}))\right)\epsilon^{-\chi_1(\gamma_1-k_{0,1})}\mathbb{T}^{\gamma_1-k_{0,1}}\right]\\
+\tilde{Q}(im)\left(\int_{-\infty}^{+\infty}\int_{-\infty}^{+\infty}\Omega_{\kappa_1}^{\mathfrak{d}_{p}}(\mathbb{T},m-m_1,\epsilon)\Omega_{\kappa_1}^{\mathfrak{d}_{p}}(\mathbb{T},m_1-m_2,\epsilon)\Omega_{\kappa_1}^{\mathfrak{d}_{p}}(\mathbb{T},m_2,\epsilon)dm_2dm_1\right)\\
\times\left[\left(\sum_{\ell=0}^{s_3}a_{\ell,3}\epsilon^{-\chi_1k_{\ell,3}}\mathbb{T}^{k_{\ell,3}}+\sum_{\ell=s_3+1}^{M_3}a_{\ell,3}\epsilon^{m_{\ell,3}+3\beta-\alpha k_{\ell,3}-\chi_1k_{\ell,3}}\mathbb{T}^{k_{\ell,3}}\right)\epsilon^{-\chi_1(2\gamma_1-k_{0,1})}\mathbb{T}^{2\gamma_1-k_{0,1}}\right]\label{e236cc}
\end{multline}
and
\begin{multline}\label{e236cd}
\mathcal{R}_{\mathbb{T},m,\epsilon}(\Omega_{\kappa_1}^{\mathfrak{d}_{p}}(\mathbb{T},m,\epsilon))=
\sum_{j=0}^{Q} \tilde{B}_{j}(m) \epsilon^{n_{j} - \alpha b_{j} - \chi_1( b_{j} - k_{0,1} - \gamma_1)}
\mathbb{T}^{b_{j} - k_{0,1} - \gamma_1}\\
+ \sum_{\ell=1}^{D-1} \epsilon^{\Delta_{\ell} + \alpha(\delta_{\ell} - d_{\ell}) + \beta}
\times \sum_{q_{1}+q_{2} = \delta_{\ell}} \frac{\delta_{\ell}!}{q_{1}!q_{2}!}
\Pi_{d=0}^{q_{1}-1}(\gamma_1 - d)\epsilon^{-\chi_1(d_{\ell} - k_{0,1} - q_{1}-q_{2})}
\tilde{R}_{\ell}(im)\\
\times \mathbb{T}^{d_{\ell,q_{1},q_{2}}} \{ (\mathbb{T}^{\kappa_1 + 1}\partial_{\mathbb{T}})^{q_2}
+ \sum_{1 \leq p \leq q_{2}-1} A_{q_{2},p}
\mathbb{T}^{\kappa_1(q_{2}-p)}(\mathbb{T}^{\kappa_1 + 1}\partial_{\mathbb{T}})^{p} \}
\Omega_{\kappa_1}^{\mathfrak{d}_{p}}(\mathbb{T},m,\epsilon)\\
+ \epsilon^{\Delta_{D} + \alpha(\delta_{D} - d_{D}) + \beta}
\times \sum_{q_{1}+q_{2} = \delta_{D},q_{1} \geq 1} \frac{\delta_{D}!}{q_{1}!q_{2}!}
\Pi_{d=0}^{q_{1}-1}(\gamma_1 - d)\epsilon^{-\chi_1(d_{D} - k_{0,1} - q_{1} - q_{2})}
\tilde{R}_{D}(im)\\
\times \mathbb{T}^{d_{D,q_{1},q_{2}}} \{ (\mathbb{T}^{\kappa_1 + 1}\partial_{\mathbb{T}})^{q_2}
+ \sum_{1 \leq p \leq q_{2}-1} A_{q_{2},p}
\mathbb{T}^{\kappa_1(q_{2}-p)}(\mathbb{T}^{\kappa_1 + 1}\partial_{\mathbb{T}})^{p} \}
\Omega_{\kappa_1}^{\mathfrak{d}_{p}}(\mathbb{T},m,\epsilon)\\
+ \tilde{R}_{D}(im) \{ (\mathbb{T}^{\kappa_1 + 1}\partial_{\mathbb{T}})^{\delta_D} +
\sum_{1 \leq p \leq \delta_{D}-1} A_{\delta_{D},p} \mathbb{T}^{\kappa_1(\delta_{D}-p)}
(\mathbb{T}^{\kappa_1+1}\partial_{\mathbb{T}})^{p} \}\Omega_{\kappa_1}^{\mathfrak{d}_{p}}(\mathbb{T},m,\epsilon).
\end{multline}

We consider the function
\begin{equation}
\mathbb{V}_1^{\mathfrak{d}_{p}}(\mathbb{T},z,\epsilon) =
\mathcal{F}^{-1}( m \mapsto  \Omega_{\kappa_1}^{\mathfrak{d}_{p}}(\mathbb{T},m,\epsilon) )(z),
\end{equation}
which defines a bounded holomorphic function w.r.t $\mathbb{T}$ on
$S_{\mathfrak{d}_{p},\theta_1,h'|\epsilon|^{\chi_1 + \alpha}}$, w.r.t $z$ on $H_{\beta'}$ for any
$0 < \beta' < \beta$ and for all $\epsilon$ on $D(0,\epsilon_{0}) \setminus \{ 0 \}$. Using the
properties of the Fourier inverse transform described in Proposition~\ref{prop479} and the expression in (\ref{e301}), we derive from (\ref{PDE_Omega_kappa_dp1}) the next equation satisfied by
$\mathbb{V}_1^{\mathfrak{d}_{p}}(\mathbb{T},z,\epsilon)$:
\begin{multline}
\tilde{Q}(\partial_z)\mathbb{V}_1^{\mathfrak{d}_{p}}(\mathbb{T},z,\epsilon)\left[-a_{0,1}+\sum_{\ell=1}^{s_1}a_{\ell,1}\epsilon^{-\chi_1(k_{\ell,1}-k_{0,1})}\mathbb{T}^{k_{\ell,1}-k_{0,1}}+\sum_{\ell=s_1+1}^{M_1}a_{\ell,1}\epsilon^{m_{\ell,1}+\beta-\alpha k_{\ell,1}-\chi_1(k_{\ell,1}-k_{0,1})}\mathbb{T}^{k_{\ell,1}-k_{0,1}}\right.\\
\qquad+\left(\sum_{\ell=1}^{s_2}a_{\ell,2}\epsilon^{-\chi_1(k_{\ell,2}-k_{0,2})}\mathbb{T}^{k_{\ell,2}-k_{0,2}}+\sum_{\ell=s_2+1}^{M_2}a_{\ell,2}\epsilon^{m_{\ell,2}+2\beta-\alpha k_{\ell,2}-\chi_1(k_{\ell,2}-k_{0,2})}\mathbb{T}^{k_{\ell,2}-k_{0,2}}\right)\left(\frac{-2a_{0,1}}{a_{0,2}}\right)\\
+\left(\sum_{\ell=0}^{s_2}a_{\ell,2}\epsilon^{-\chi_1(k_{\ell,2}-k_{0,2})}\mathbb{T}^{k_{\ell,2}-k_{0,2}}+\sum_{\ell=s_2+1}^{M_2}a_{\ell,2}\epsilon^{m_{\ell,2}+2\beta-\alpha k_{\ell,2}-\chi_1(k_{\ell,2}-k_{0,2})}\mathbb{T}^{k_{\ell,2}-k_{0,2}}\right)\left(\frac{-2a_{0,1}}{a_{0,2}}\mathcal{J}_1(\epsilon^{-\chi_1}\mathbb{T})\right)\\
+\left(\sum_{\ell=0}^{s_3}a_{\ell,3}\epsilon^{-\chi_1(k_{\ell,3}-k_{0,1})}\mathbb{T}^{k_{\ell,3}-k_{0,1}}+\sum_{\ell=s_3+1}^{M_3}a_{\ell,3}\epsilon^{m_{\ell,3}+3\beta-\alpha k_{\ell,3}-\chi_1(k_{\ell,3}-k_{0,1})}\mathbb{T}^{k_{\ell,3}-k_{0,1}}\right)\\
\left.\hfill\times3\left(\frac{a_{01}}{a_{02}}\epsilon^{-\chi_1(k_{0,1}-k_{0,2})}\mathbb{T}^{k_{0,1}-k_{0,2}}(1+\mathcal{J}_1(\epsilon^{-\chi_1}\mathbb{T}))\right)^2\right]\\
+\tilde{Q}(\partial_z)(\mathbb{V}_1^{\mathfrak{d}_{p}}(\mathbb{T},z,\epsilon))^2
\times\left[\left(\sum_{\ell=0}^{s_2}a_{\ell,2}\epsilon^{-\chi_1k_{\ell,2}}\mathbb{T}^{k_{\ell,2}}+\sum_{\ell=s_2+1}^{M_2}a_{\ell,2}\epsilon^{m_{\ell,2}+2\beta-\alpha k_{\ell,2}-\chi_1k_{\ell,2}}\mathbb{T}^{k_{\ell,2}}\right)\epsilon^{-\chi_1(\gamma_1-k_{0,1})}\mathbb{T}^{\gamma_1-k_{0,1}}\right.\\
+\left(\sum_{\ell=0}^{s_3}a_{\ell,3}\epsilon^{-\chi_1k_{\ell,3}}\mathbb{T}^{k_{\ell,3}}+\sum_{\ell=s_3+1}^{M_3}a_{\ell,3}\epsilon^{m_{\ell,3}+3\beta-\alpha k_{\ell,3}-\chi_1k_{\ell,3}}\mathbb{T}^{k_{\ell,3}}\right)\hfill\\
\left.\hfill\times\left(\frac{-3a_{01}}{a_{02}}\epsilon^{-\chi_1(k_{0,1}-k_{0,2})}\mathbb{T}^{k_{0,1}-k_{0,2}}(1+\mathcal{J}_1(\epsilon^{-\chi_1}\mathbb{T}))\right)\epsilon^{-\chi_1(\gamma_1-k_{0,1})}\mathbb{T}^{\gamma_1-k_{0,1}}\right]\\
+\tilde{Q}(\partial_z)(\mathbb{V}_1^{\mathfrak{d}_{p}}(\mathbb{T},z,\epsilon))^3\left[\left(\sum_{\ell=0}^{s_3}a_{\ell,3}\epsilon^{-\chi_1k_{\ell,3}}\mathbb{T}^{k_{\ell,3}}+\sum_{\ell=s_3+1}^{M_3}a_{\ell,3}\epsilon^{m_{\ell,3}+3\beta-\alpha k_{\ell,3}-\chi_1k_{\ell,3}}\mathbb{T}^{k_{\ell,3}}\right)\epsilon^{-\chi_1(2\gamma_1-k_{0,1})}\mathbb{T}^{2\gamma_1-k_{0,1}}\right]\\
= \sum_{j=0}^{Q} \tilde{b}_{j}(z) \epsilon^{n_{j} - \alpha b_{j} - \chi_1( b_{j} - k_{0,1} - \gamma_1)}
\mathbb{T}^{b_{j} - k_{0,1} - \gamma_1}+ \sum_{\ell=1}^{D-1} \epsilon^{\Delta_{\ell} + \alpha(\delta_{\ell} - d_{\ell}) + \beta}\\
\times \sum_{q_{1}+q_{2} = \delta_{\ell}} \frac{\delta_{\ell}!}{q_{1}!q_{2}!}
\Pi_{d=0}^{q_{1}-1}(\gamma_1 - d)\epsilon^{-\chi_1(d_{\ell} - k_{0,1} - q_{1})}
\mathbb{T}^{d_{\ell} - k_{0,1} - q_{1}}\tilde{R}_{\ell}(\partial_{z}) \epsilon^{\chi_1 q_{2}}
\partial_{\mathbb{T}}^{q_2}\mathbb{V}_1^{\mathfrak{d}_{p}}(\mathbb{T},z,\epsilon)\\
+ \epsilon^{\Delta_{D} + \alpha(\delta_{D} - d_{D}) + \beta}
\sum_{q_{1}+q_{2} = \delta_{D},q_{1} \geq 1} \frac{\delta_{D}!}{q_{1}!q_{2}!}
\Pi_{d=0}^{q_{1}-1}(\gamma_1 - d)\epsilon^{-\chi_1(d_{D} - k_{0,1} - q_{1})}
\mathbb{T}^{d_{D} - k_{0,1} - q_{1}}\\
\times \tilde{R}_{D}(\partial_{z}) \epsilon^{\chi_1 q_{2}}
\partial_{\mathbb{T}}^{q_2}\mathbb{V}_1^{\mathfrak{d}_{p}}(\mathbb{T},z,\epsilon)
+ \mathbb{T}^{d_{D}-k_{0,1}}\tilde{R}_{D}(\partial_{z})\partial_{\mathbb{T}}^{\delta_{D}}
\mathbb{V}_1^{\mathfrak{d}_{p}}(\mathbb{T},z,\epsilon).\label{e2626}
\end{multline}
The function $V_1^{\mathfrak{d}_{p}}(T,z,\epsilon) = \mathbb{V}_1^{\mathfrak{d}_{p}}(\epsilon^{\chi_1}T,z,\epsilon)$ defines a bounded holomorphic function w.r.t $T$ such that
$T \in \epsilon^{-\chi_1}S_{\mathfrak{d}_{p},\theta_1,h'|\epsilon|^{\chi_1 + \alpha}}$ and
w.r.t $z$ on $H_{\beta'}$ for any $0 < \beta' < \beta$, for all $\epsilon \in D(0,\epsilon_{0})
\setminus \{ 0 \}$. It holds that $V_1^{\mathfrak{d}_{p}}(T,z,\epsilon)$ solves the equation
\begin{multline}
\tilde{Q}(\partial_z)V_1^{\mathfrak{d}_{p}}(T,z,\epsilon)\left[-a_{0,1}T^{k_{0,1}+\gamma_1}+\sum_{\ell=1}^{s_1}a_{\ell,1}T^{k_{\ell,1}+\gamma_1}+\sum_{\ell=s_1+1}^{M_1}a_{\ell,1}\epsilon^{m_{\ell,1}+\beta-\alpha k_{\ell,1}}T^{k_{\ell,1}+\gamma_1}\right.\\
\qquad+\left(\sum_{\ell=1}^{s_2}a_{\ell,2}T^{k_{\ell,2}-k_{0,2}+k_{0,1}+\gamma_1}+\sum_{\ell=s_2+1}^{M_2}a_{\ell,2}\epsilon^{m_{\ell,2}+2\beta-\alpha k_{\ell,2}}T^{k_{\ell,2}-k_{0,2}+k_{0,1}+\gamma_1}\right)\left(\frac{-2a_{0,1}}{a_{0,2}}\right)\\
+\left(\sum_{\ell=0}^{s_2}a_{\ell,2}T^{k_{\ell,2}-k_{0,2}+k_{0,1}+\gamma}+\sum_{\ell=s_2+1}^{M_2}a_{\ell,2}\epsilon^{m_{\ell,2}+2\beta-\alpha k_{\ell,2}}T^{k_{\ell,2}-k_{0,2}}+k_{0,1}+\gamma_1\right)\left(\frac{-2a_{0,1}}{a_{0,2}}\mathcal{J}_1(T)\right)\\
+\left(\sum_{\ell=0}^{s_3}a_{\ell,3}T^{k_{\ell,3}+\gamma_1}+\sum_{\ell=s_3+1}^{M_3}a_{\ell,3}\epsilon^{m_{\ell,3}+3\beta-\alpha k_{\ell,3}}T^{k_{\ell,3}+\gamma_1}\right)\\
\left.\hfill\times3\left(\frac{a_{01}}{a_{02}}T^{2k_{0,1}-k_{0,2}+\gamma_1}(1+\mathcal{J}_1(T))\right)^2\right]\\
+\tilde{Q}(\partial_z)(V_1^{\mathfrak{d}_{p}}(T,z,\epsilon))^2
\times\left[\left(\sum_{\ell=0}^{s_2}a_{\ell,2}T^{k_{\ell,2}+k_{0,1}+\gamma_1}+\sum_{\ell=s_2+1}^{M_2}a_{\ell,2}\epsilon^{m_{\ell,2}+2\beta-\alpha k_{\ell,2}}T^{k_{\ell,2}+k_{0,1}+\gamma_1}\right)T^{2\gamma_1}\right.\\
+\left(\sum_{\ell=0}^{s_3}a_{\ell,3}T^{k_{\ell,3}+k_{0,1}+\gamma_1}+\sum_{\ell=s_3+1}^{M_3}a_{\ell,3}\epsilon^{m_{\ell,3}+3\beta-\alpha k_{\ell,3}}T^{k_{\ell,3}+k_{0,1}+\gamma_1}\right)\hfill\\
\left.\hfill\times\left(\frac{-3a_{01}}{a_{02}}T^{2k_{0,1}-k_{0,2}+\gamma_1}(1+\mathcal{J}_1(T))\right)T^{2\gamma_1}\right]\\
+\tilde{Q}(\partial_z)(V_1^{\mathfrak{d}_{p}}(T,z,\epsilon))^3\left[\left(\sum_{\ell=0}^{s_3}a_{\ell,3}T^{k_{\ell,3}+k_{0,1}+\gamma_1}+\sum_{\ell=s_3+1}^{M_3}a_{\ell,3}\epsilon^{m_{\ell,3}+3\beta-\alpha k_{\ell,3}}T^{k_{\ell,3}+k_{0,1}+\gamma_1}\right)T^{3\gamma_1}\right]\\
= \sum_{j=0}^{Q} \tilde{b}_{j}(z) \epsilon^{n_{j} - \alpha b_{j}}
T^{b_{j}}+ \sum_{\ell=1}^{D} \epsilon^{\Delta_{\ell} + \alpha(\delta_{\ell} - d_{\ell}) + \beta}\\
\times \sum_{q_{1}+q_{2} = \delta_{\ell}} \frac{\delta_{\ell}!}{q_{1}!q_{2}!}
\Pi_{d=0}^{q_{1}-1}(\gamma_1 - d)
T^{d_{\ell} +\gamma_1 - q_{1}}\tilde{R}_{\ell}(\partial_{z})
\partial_{T}^{q_2}V_1^{\mathfrak{d}_{p}}(T,z,\epsilon).\label{e2627}
\end{multline}

We now consider the function
\begin{equation}\label{e2665}
U_1^{\mathfrak{d}_p}(T,z,\epsilon)=-\frac{a_{0,1}}{a_{0,2}}T^{k_{0,1}-k_{0,2}}-\frac{a_{0,1}}{a_{0,2}}T^{k_{0,1}-k_{0,2}}\mathcal{J}_1(T)+T^{\gamma_1}V_1^{\mathfrak{d}_{p}}(T,z,\epsilon).
\end{equation}

which defines a bounded holomorphic function w.r.t $\mathbb{T}$ on
$\epsilon^{-\chi_1}S_{\mathfrak{d}_{p},\theta_1,h'|\epsilon|^{\chi_1 + \alpha}}$, w.r.t $z$ on $H_{\beta'}$ for any
$0 < \beta' < \beta$ and for all $\epsilon$ on $D(0,\epsilon_{0}) \setminus \{ 0 \}$. Notice that this function might exhibit a pole at $T=0$ in the case that $k_{0,1}<k_{0,2}$. Regarding (\ref{e2627}), we derive that $U^{\mathfrak{d}_{p}}(T,z,\epsilon)$ is a solution of the PDE
\begin{multline}
\tilde{Q}( \partial_z)\left(\left(\sum_{\ell=0}^{M_1}a_{\ell,1}\epsilon^{m_{\ell,1}+\beta-\alpha k_{\ell,1}}T^{k_{\ell,1}}\right)U_1^{\mathfrak{d}_p}(T,z,\epsilon)+\left(\sum_{\ell=0}^{M_2}a_{\ell,2}\epsilon^{m_{\ell,2}+2\beta-\alpha k_{\ell,2}}T^{k_{\ell,2}}\right)(U_1^{\mathfrak{d}_p}(T,z,\epsilon))^2\right.\\
\left.+\left(\sum_{\ell=0}^{M_3}a_{\ell,3}\epsilon^{m_{\ell,3}+3\beta-\alpha k_{\ell,3}}T^{k_{\ell,3}}\right)(U_1^{\mathfrak{d}_p}(T,z,\epsilon))^3\right)\\
=\sum_{j=0}^{Q}\tilde{b}_{j}(z)\epsilon^{n_j-\alpha b_j}T^{b_j}+F_1(T,\epsilon)+\sum_{\ell=1}^{D}\epsilon^{\Delta_{\ell}+\alpha(\delta_\ell-d_\ell)+\beta}T^{d_{\ell}}\tilde{R}_{\ell}(\partial_{z})\partial_{T}^{\delta_{\ell}}U_1^{\mathfrak{d}_p}(T,z,\epsilon),\label{e2628}
\end{multline}
where $F_1(T,\epsilon)$ contributes to the forcing term and is given by

\begin{multline}
F_1(T,\epsilon)=-\tilde{Q}(0)\left(\frac{a_{0,1}}{a_{0,2}}T^{k_{0,1}-k_{0,2}}+\frac{a_{0,1}}{a_{0,2}}T^{k_{0,1}-k_{0,2}}\mathcal{J}_1(T)\right)\\
\times\left[-a_{0,1}T^{k_{0,1}}+\sum_{\ell=1}^{s_1}a_{\ell,1}T^{k_{\ell,1}}+\sum_{\ell=s_1+1}^{M_1}a_{\ell,1}\epsilon^{m_{\ell,1}+\beta-\alpha k_{\ell,1}}T^{k_{\ell,1}}\right.\\
\qquad+\left(\sum_{\ell=1}^{s_2}a_{\ell,2}T^{k_{\ell,2}-k_{0,2}+k_{0,1}}+\sum_{\ell=s_2+1}^{M_2}a_{\ell,2}\epsilon^{m_{\ell,2}+2\beta-\alpha k_{\ell,2}}T^{k_{\ell,2}-k_{0,2}+k_{0,1}}\right)\left(\frac{-2a_{0,1}}{a_{0,2}}\right)\\
+\left(\sum_{\ell=0}^{s_2}a_{\ell,2}T^{k_{\ell,2}-k_{0,2}+k_{0,1}}+\sum_{\ell=s_2+1}^{M_2}a_{\ell,2}\epsilon^{m_{\ell,2}+2\beta-\alpha k_{\ell,2}}T^{k_{\ell,2}-k_{0,2}}+k_{0,1}\right)\left(\frac{-2a_{0,1}}{a_{0,2}}\mathcal{J}_1(T)\right)\\
\left.+3\left(\sum_{\ell=0}^{s_3}a_{\ell,3}T^{k_{\ell,3}}+\sum_{\ell=s_3+1}^{M_3}a_{\ell,3}\epsilon^{m_{\ell,3}+3\beta-\alpha k_{\ell,3}}T^{k_{\ell,3}}\right)s\left(\frac{a_{01}}{a_{02}}T^{2k_{0,1}-k_{0,2}}(1+\mathcal{J}_1(T))\right)^2\right]\\
-\tilde{Q}(0)\left(\frac{a_{0,1}}{a_{0,2}}T^{k_{0,1}-k_{0,2}}+\frac{a_{0,1}}{a_{0,2}}T^{k_{0,1}-k_{0,2}}\mathcal{J}_1(T)\right)^2\left[\left(\sum_{\ell=0}^{s_2}a_{\ell,2}T^{k_{\ell,2}+k_{0,1}}+\sum_{\ell=s_2+1}^{M_2}a_{\ell,2}\epsilon^{m_{\ell,2}+2\beta-\alpha k_{\ell,2}}T^{k_{\ell,2}+k_{0,1}}\right)\right.\\
\left.+\left(\sum_{\ell=0}^{s_3}a_{\ell,3}T^{k_{\ell,3}+k_{0,1}}+\sum_{\ell=s_3+1}^{M_3}a_{\ell,3}\epsilon^{m_{\ell,3}+3\beta-\alpha k_{\ell,3}}T^{k_{\ell,3}+k_{0,1}}\right)\left(\frac{-3a_{01}}{a_{02}}T^{2k_{0,1}-k_{0,2}}(1+\mathcal{J}_1(T))\right)\right]\\
-\tilde{Q}(0)\left(\frac{a_{0,1}}{a_{0,2}}T^{k_{0,1}-k_{0,2}}+\frac{a_{0,1}}{a_{0,2}}T^{k_{0,1}-k_{0,2}}\mathcal{J}_1(T)\right)^3\left[\left(\sum_{\ell=0}^{s_3}a_{\ell,3}T^{k_{\ell,3}+k_{0,1}}+\sum_{\ell=s_3+1}^{M_3}a_{\ell,3}\epsilon^{m_{\ell,3}+3\beta-\alpha k_{\ell,3}}T^{k_{\ell,3}+k_{0,1}}\right)\right]\\
+ \sum_{\ell=1}^{D} \epsilon^{\Delta_{\ell} + \alpha(\delta_{\ell} - d_{\ell}) + \beta}T^{d_{\ell}}\tilde{R}_{\ell}(0)
\partial_{T}^{\delta_\ell}\left(\frac{a_{0,1}}{a_{0,2}}T^{k_{0,1}-k_{0,2}}+\frac{a_{0,1}}{a_{0,2}}T^{k_{0,1}-k_{0,2}}\mathcal{J}_1(T)\right).\label{e2629}
\end{multline}

Taking into account that $U_{01}(T)$ is a solution of (\ref{e162}), we derive the following expression of $F_1(T,\epsilon)$:

\begin{multline}
F_1(T,\epsilon)=\tilde{Q}(0)U_{01}(T)\left[\sum_{\ell=s_1+1}^{M_1}a_{\ell,1}\epsilon^{m_{\ell,1}+\beta-\alpha k_{\ell,1}}T^{k_{\ell,1}}+\left(\sum_{\ell=s_2+1}^{M_2}\frac{-2a_{0,1}a_{\ell,2}}{a_{0,2}}\epsilon^{m_{\ell,2}+2\beta-\alpha k_{\ell,2}}T^{k_{\ell,2}-k_{0,2}+k_{0,1}}\right)\right.\\
+\left(\sum_{\ell=s_2+1}^{M_2}a_{\ell,2}\epsilon^{m_{\ell,2}+2\beta-\alpha k_{\ell,2}}T^{k_{\ell,2}-k_{0,2}}+k_{0,1}\right)\left(\frac{-2a_{0,1}}{a_{0,2}}\mathcal{J}_1(T)\right)\\
\left.+3\left(\sum_{\ell=s_3+1}^{M_3}a_{\ell,3}\epsilon^{m_{\ell,3}+3\beta-\alpha k_{\ell,3}}T^{k_{\ell,3}}\right)\left(\frac{a_{01}}{a_{02}}T^{2k_{0,1}-k_{0,2}}(1+\mathcal{J}_1(T))\right)^2\right]\\
-\tilde{Q}(0)U_{01}(T)^2\left[\left(\sum_{\ell=0}^{s_2}a_{\ell,2}T^{k_{\ell,2}+k_{0,1}}+\sum_{\ell=s_2+1}^{M_2}a_{\ell,2}\epsilon^{m_{\ell,2}+2\beta-\alpha k_{\ell,2}}T^{k_{\ell,2}+k_{0,1}}\right)\right.\\
\left.+\left(\sum_{\ell=0}^{s_3}a_{\ell,3}T^{k_{\ell,3}+k_{0,1}}+\sum_{\ell=s_3+1}^{M_3}a_{\ell,3}\epsilon^{m_{\ell,3}+3\beta-\alpha k_{\ell,3}}T^{k_{\ell,3}+k_{0,1}}\right)\left(\frac{-3a_{01}}{a_{02}}T^{2k_{0,1}-k_{0,2}}(1+\mathcal{J}_1(T))\right)\right]\\
-\tilde{Q}(0)U_{01}(T)^3\left[\left(\sum_{\ell=0}^{s_3}a_{\ell,3}T^{k_{\ell,3}+k_{0,1}}+\sum_{\ell=s_3+1}^{M_3}a_{\ell,3}\epsilon^{m_{\ell,3}+3\beta-\alpha k_{\ell,3}}T^{k_{\ell,3}+k_{0,1}}\right)\right]\\
+ \sum_{\ell=1}^{D} \epsilon^{\Delta_{\ell} + \alpha(\delta_{\ell} - d_{\ell}) + \beta}T^{d_{\ell}}\tilde{R}_{\ell}(0)
\partial_{T}^{\delta_\ell}\left(\frac{a_{0,1}}{a_{0,2}}T^{k_{0,1}-k_{0,2}}+\frac{a_{0,1}}{a_{0,2}}T^{k_{0,1}-k_{0,2}}\mathcal{J}_1(T)\right).\label{e26290}
\end{multline}

The function $F_1(T,\epsilon)$ turns out to be bounded holomorphic with respect to $\epsilon$ and it is analytic with respect to $T$ in some neighborhood of the origin in the case that
\begin{multline}\label{e2706}
k_{\ell_1,1}+k_{0,1}-k_{0,2}\ge 0,\quad k_{\ell_2,2}+2(k_{0,1}-2k_{0,2})\ge0\\
k_{\ell_3,3}+5k_{0,1}-3k_{0,2}\ge0,\quad 2( k_{0,1}-k_{0,2})+k_{\ell_2,2}+k_{0,1}\ge0,\quad k_{\ell_3,3}+4k_{0,1}-3k_{0,2}\ge0,\\
d_\ell+k_{0,1}-k_{0,2}-\delta_{\ell}\ge0
\end{multline}
holds for every $s_1+1\le\ell_1\le M_1$, $s_2+1\le\ell_2\le M_2$, $s_3+1\le\ell_2\le M_3$ and $1\le \ell\le D$. In view of the assumptions made on these parameters, and (\ref{e92}), the next conditions are sufficient so that (\ref{e2706})  hold:
\begin{equation}\label{e2712}
2k_{0,1}-k_{0,2}\ge0\qquad d_\ell+k_{0,1}-k_{0,2}-\delta_{\ell}\ge0.
\end{equation}
 
We conclude by writing 
\begin{equation}\label{e2713}
u_1^{\mathfrak{d}_{p}}(t,z,\epsilon) = \epsilon^{\beta}U_1^{\mathfrak{d}_{p}}(\epsilon^{\alpha}t,z,\epsilon)
= \epsilon^{\beta} \left( U_{01}(\epsilon^{\alpha}t) + (\epsilon^{\alpha}t)^{\gamma_1}
\mathbb{V}_1^{\mathfrak{d}_{p}}(\epsilon^{\chi_1 + \alpha}t,z,\epsilon) \right),
\end{equation}

which defines a holomorphic function w.r.t $t$ on $\mathcal{T}_1$, w.r.t $z \in H_{\beta'}$ for any
$0 < \beta' < \beta$, w.r.t $\epsilon \in \mathcal{E}_{p}$, where $\mathcal{T}_1$ and $\mathcal{E}_{p}$ are
sectors described in Definition~\ref{defi2414}. As a result, $u_1^{\mathfrak{d}_{p}}(t,z,\epsilon)$ admits the decomposition
(\ref{decomp_singular_holom_udp1}) with $v_1^{\mathfrak{d}_{p}}(t,z,\epsilon) =
\mathbb{V}_1^{\mathfrak{d}_{p}}(\epsilon^{\chi_1 + \alpha}t,z,\epsilon)$ which determines a bounded holomorphic
function on $\mathcal{T}_1 \times H_{\beta'} \times \mathcal{E}_{p}$ for any given $0 < \beta' < \beta$ with the
property $v_1^{\mathfrak{d}_{p}}(0,z,\epsilon) \equiv 0$ for all $(z,\epsilon) \in H_{\beta'} \times \mathcal{E}_{p}$.
Again, the function $u_1^{\mathfrak{d}_{p}}(t,z,\epsilon)$ may be meromorphic in both $t$ and $\epsilon$ in the vicinity of the origin. From (\ref{e2628}) and (\ref{e2629}) we deduce that $u_1^{\mathfrak{d}_{p}}(t,z,\epsilon)$ solves the next main problem
\begin{multline*}
\tilde{Q}( \partial_z)\left(\left(\sum_{\ell=0}^{M_1}a_{\ell,1}\epsilon^{m_{\ell,1}}t^{k_{\ell,1}}\right)u_1^{\mathfrak{d}_p}(t,z,\epsilon)+\left(\sum_{\ell=0}^{M_2}a_{\ell,2}\epsilon^{m_{\ell,2}}t^{k_{\ell,2}}\right)(u_1^{\mathfrak{d}_p}(t,z,\epsilon))^2\right.\\
\left.+\left(\sum_{\ell=0}^{M_3}a_{\ell,3}\epsilon^{m_{\ell,3}}t^{k_{\ell,3}}\right)(u_1^{\mathfrak{d}_p}(t,z,\epsilon))^3\right)\\
=\sum_{j=0}^{Q}\tilde{b}_{j}(z)\epsilon^{n_j}t^{b_j}+F_1(\epsilon^\alpha t,\epsilon)+\sum_{\ell=1}^{D}\epsilon^{\Delta_{\ell}}t^{d_{\ell}}\tilde{R}_{\ell}(\partial_{z})u_1^{\mathfrak{d}_p}(t,z,\epsilon),
\end{multline*}
with additional forcing term $F_1(\epsilon^{\alpha}t,\epsilon)$. We apply the operator $\partial_{z}^{\upsilon}$ on the left and right handside of this last equation, to get that $u_1^{\mathfrak{d}_{p}}(t,z,\epsilon)$ is an actual solution of the main problem (\ref{e85}).

We proceed with the proof of (\ref{exp_small_difference_v_dp1}). The steps followed are analogous to those taken in the proof of Theorem 1 in \cite{lama1}. We give the details for the sake of completeness. Let $p \in \{ 0,\ldots,\varsigma_1 - 1 \}$. The function
$v_1^{\mathfrak{d}_{p}}(t,z,\epsilon)$ can be written as a $m_{\kappa_1}-$Laplace and Fourier transform
\begin{equation}
 v_1^{\mathfrak{d}_{p}}(t,z,\epsilon) = \frac{\kappa_1}{(2\pi)^{1/2}}
 \int_{-\infty}^{+\infty} \int_{L_{\gamma_p}} \omega_{\kappa_1}^{\mathfrak{d}_p}(u,m,\epsilon)
 \exp( -(\frac{u}{\epsilon^{\chi_1 + \alpha}t})^{\kappa_1} ) e^{izm} \frac{du}{u} dm 
\end{equation}
where $L_{\gamma_p} = \mathbb{R}_{+}e^{i \gamma_{p}} \subset S_{\mathfrak{d}_p}$. Using the fact that the function $u \mapsto \omega_{\kappa_1}(u,m,\epsilon) \exp( -(\frac{u}{\epsilon^{\chi_1 + \alpha} t})^{\kappa_1} )/u$
is holomorphic on $D(0,\rho)$ for all
$(m,\epsilon) \in \mathbb{R} \times (D(0,\epsilon_{0}) \setminus \{ 0 \})$, its integral along the union of a segment starting from
0 to $(\rho/2)e^{i\gamma_{p+1}}$, an arc of circle with radius $\rho/2$ which connects
$(\rho/2)e^{i\gamma_{p+1}}$ and $(\rho/2)e^{i\gamma_{p}}$ and a segment starting from
$(\rho/2)e^{i\gamma_{p}}$ to 0, is vanishing. Therefore, we can write the difference
$v_1^{\mathfrak{d}_{p+1}} - v_1^{\mathfrak{d}_{p}}$ as a sum of three integrals,
\begin{multline}
v_1^{\mathfrak{d}_{p+1}}(t,z,\epsilon) - v_1^{\mathfrak{d}_{p}}(t,z,\epsilon) =
\frac{\kappa_1}{(2\pi)^{1/2}}\int_{-\infty}^{+\infty}
\int_{L_{\rho/2,\gamma_{p+1}}}
\omega_{\kappa_1}^{\mathfrak{d}_{p+1}}(u,m,\epsilon) e^{-(\frac{u}{\epsilon^{\chi_1 + \alpha} t})^{\kappa_1}}
e^{izm} \frac{du}{u} dm\\ -
\frac{\kappa_1}{(2\pi)^{1/2}}\int_{-\infty}^{+\infty}
\int_{L_{\rho/2,\gamma_{p}}}
\omega_{\kappa_1}^{\mathfrak{d}_p}(u,m,\epsilon) e^{-(\frac{u}{\epsilon^{\chi_1 + \alpha} t})^{\kappa_1}}
e^{izm} \frac{du}{u} dm\\
+ \frac{\kappa_1}{(2\pi)^{1/2}}\int_{-\infty}^{+\infty}
\int_{C_{\rho/2,\gamma_{p},\gamma_{p+1}}}
\omega_{\kappa_1}(u,m,\epsilon) e^{-(\frac{u}{\epsilon^{\chi_1 + \alpha} t})^{\kappa_1}}
e^{izm} \frac{du}{u} dm \label{difference_v_dp_decomposition}
\end{multline}
where $L_{\rho/2,\gamma_{p+1}} = [\rho/2,+\infty)e^{i\gamma_{p+1}}$,
$L_{\rho/2,\gamma_{p}} = [\rho/2,+\infty)e^{i\gamma_{p}}$ and
$C_{\rho/2,\gamma_{p},\gamma_{p+1}}$ is an arc of circle with radius connecting
$(\rho/2)e^{i\gamma_{p}}$ and $(\rho/2)e^{i\gamma_{p+1}}$ with a well chosen orientation.\medskip

We give estimates for the quantity
$$ I_{1} = \left| \frac{\kappa_1}{(2\pi)^{1/2}}\int_{-\infty}^{+\infty}
\int_{L_{\rho/2,\gamma_{p+1}}}
\omega_{\kappa_1}^{\mathfrak{d}_{p+1}}(u,m,\epsilon) e^{-(\frac{u}{\epsilon^{\chi_1 + \alpha} t})^{\kappa_1}}
e^{izm} \frac{du}{u} dm \right|.
$$
By construction, the direction $\gamma_{p+1}$ (which depends on $\epsilon^{\chi_1 + \alpha} t$) is chosen in such
a way that
$\cos( \kappa_1( \gamma_{p+1} - \mathrm{arg}(\epsilon^{\chi_1 + \alpha} t) )) \geq \delta_{1}$, for all
$\epsilon \in \mathcal{E}_{p} \cap \mathcal{E}_{p+1}$, all $t \in \mathcal{T}_1$, for some fixed $\delta_{1} > 0$.
From the estimates (\ref{omega_kappa_dp_in_F1}), we get that
\begin{multline}
I_{1} \leq \frac{\kappa_1}{(2\pi)^{1/2}} \int_{-\infty}^{+\infty} \int_{\rho/2}^{+\infty}
\varpi (1+|m|)^{-\mu} e^{-\beta|m|}
\frac{ \frac{r}{|\epsilon|^{\chi_1 + \alpha}}}{1 + (\frac{r}{|\epsilon|^{\chi_1 + \alpha}})^{2\kappa_1} } \\
\times \exp( \nu (\frac{r}{|\epsilon|^{\chi_1 + \alpha}})^{\kappa_1} )
\exp(-\frac{\cos(\kappa_1(\gamma_{p+1} -
\mathrm{arg}(\epsilon^{\chi_1 + \alpha} t)))}{|\epsilon^{\chi_1 + \alpha} t|^{\kappa_1}}
r^{\kappa_1}) e^{-m\mathrm{Im}(z)} \frac{dr}{r} dm\\
\leq \frac{\kappa_1 \varpi}{(2\pi)^{1/2}} \int_{-\infty}^{+\infty} e^{-(\beta - \beta')|m|} dm
\int_{\rho/2}^{+\infty}\frac{1}{|\epsilon|^{\chi_1 + \alpha}}
\exp( -(\frac{\delta_{1}}{|t|^{\kappa_1}} - \nu)(\frac{r}{|\epsilon|^{\chi_1 + \alpha}})^{\kappa_1} ) dr\\
\leq  \frac{2\kappa_1 \varpi}{(2\pi)^{1/2}} \int_{0}^{+\infty} e^{-(\beta - \beta')m} dm
\int_{\rho/2}^{+\infty} \frac{|\epsilon|^{(\chi_1 + \alpha)(\kappa_1-1)}}{(\frac{\delta_1}{|t|^{\kappa_1}} - \nu)\kappa_1
(\frac{\rho}{2})^{\kappa_1-1}}
\times \frac{ (\frac{\delta_1}{|t|^{\kappa_1}} - \nu)\kappa_1 r^{\kappa_1-1} }{|\epsilon|^{(\chi_1 + \alpha)\kappa_1}}
\exp( -(\frac{\delta_{1}}{|t|^{\kappa_1}} - \nu)(\frac{r}{|\epsilon|^{\chi_1 + \alpha}})^{\kappa_1} ) dr\\
\leq
\frac{2 \kappa_1 \varpi}{(2\pi)^{1/2}} \frac{|\epsilon|^{(\chi_1+\alpha)(\kappa_1-1)}}{(\beta - \beta')
(\frac{\delta_{1}}{|t|^{\kappa_1}} - \nu) \kappa_1 (\frac{\rho}{2})^{\kappa_1-1}}
\exp( -(\frac{\delta_1}{|t|^{\kappa_1}} - \nu) \frac{(\rho/2)^{\kappa_1}}{|\epsilon|^{(\chi_1+\alpha)\kappa_1}} )\\
\leq \frac{2 \kappa_1 \varpi}{(2\pi)^{1/2}} \frac{|\epsilon|^{(\chi_1+\alpha)(\kappa_1-1)}}{(\beta - \beta')
\delta_{2}\kappa_1(\frac{\rho}{2})^{\kappa_1-1}} \exp( -\delta_{2}
\frac{(\rho/2)^{\kappa_1}}{|\epsilon|^{(\chi_1+\alpha)\kappa_1}} ) \label{I_1_exp_small_order_chi_alpha_kappa}
\end{multline}
for all $t \in \mathcal{T}_1$ and $|\mathrm{Im}(z)| \leq \beta'$ with
$|t| < (\frac{\delta_{1}}{\delta_{2} + \nu})^{1/\kappa_1}$, for some $\delta_{2}>0$, for all
$\epsilon \in \mathcal{E}_{p} \cap \mathcal{E}_{p+1}$.\medskip

In the same way, we also give estimates for the integral
$$ I_{2} = \left| \frac{\kappa_1}{(2\pi)^{1/2}}\int_{-\infty}^{+\infty}
\int_{L_{\rho/2,\gamma_{p}}}
\omega_{\kappa_1}^{\mathfrak{d}_{p}}(u,m,\epsilon) e^{-(\frac{u}{\epsilon^{\chi_1 + \alpha} t})^{\kappa_1}}
e^{izm} \frac{du}{u} dm \right|.
$$
Namely, the direction $\gamma_{p}$ (which depends on $\epsilon^{\chi_1 + \alpha} t$) is chosen in such a way that
$\cos( \kappa_1( \gamma_{p} - \mathrm{arg}(\epsilon^{\chi_1 + \alpha} t) )) \geq \delta_{1}$, for all
$\epsilon \in \mathcal{E}_{p} \cap \mathcal{E}_{p+1}$, all $t \in \mathcal{T}_1$, for some fixed $\delta_{1} > 0$.
Again from the estimates (\ref{omega_kappa_dp_in_F1}) and following the same steps as in
(\ref{I_1_exp_small_order_chi_alpha_kappa}), we deduce that
\begin{equation}
I_{2} \leq \frac{2 \kappa_1 \varpi}{(2\pi)^{1/2}} \frac{|\epsilon|^{(\chi_1+\alpha)(\kappa_1-1)}}{(\beta - \beta')
\delta_{2}\kappa_1(\frac{\rho}{2})^{\kappa_1-1}}
\exp( -\delta_{2} \frac{(\rho/2)^{\kappa_1}}{|\epsilon|^{(\chi_1+\alpha)\kappa_1}} )
\label{I_2_exp_small_order_chi_alpha_kappa}
\end{equation}
for all $t \in \mathcal{T}$ and $|\mathrm{Im}(z)| \leq \beta'$ with
$|t| < (\frac{\delta_{1}}{\delta_{2} + \nu})^{1/\kappa_1}$, for some $\delta_{2}>0$, for all
$\epsilon \in \mathcal{E}_{p} \cap \mathcal{E}_{p+1}$.\medskip

Finally, we give upper bound estimates for the integral
$$
I_{3} = \left| \frac{\kappa_1}{(2\pi)^{1/2}}\int_{-\infty}^{+\infty}
\int_{C_{\rho/2,\gamma_{p},\gamma_{p+1}}}
\omega_{\kappa_1}(u,m,\epsilon) e^{-(\frac{u}{\epsilon^{\chi_1 + \alpha} t})^{\kappa_1}} e^{izm} \frac{du}{u} dm \right|.
$$
By construction, the arc of circle $C_{\rho/2,\gamma_{p},\gamma_{p+1}}$ is chosen in such a way that
$\cos(\kappa_1(\theta - \mathrm{arg}(\epsilon^{\chi_1 + \alpha} t))) \geq \delta_{1}$, for all
$\theta \in [\gamma_{p},\gamma_{p+1}]$ (if
$\gamma_{p} < \gamma_{p+1}$), $\theta \in [\gamma_{p+1},\gamma_{p}]$ (if
$\gamma_{p+1} < \gamma_{p}$), for all $t \in \mathcal{T}_1$, all $\epsilon \in \mathcal{E}_{p} \cap \mathcal{E}_{p+1}$, for some
fixed $\delta_{1}>0$. Bearing in mind (\ref{omega_kappa_dp_in_F1}) and (\ref{e1017}), we get that
\begin{multline}
I_{3} \leq \frac{\kappa_1}{(2\pi)^{1/2}} \int_{-\infty}^{+\infty}  \left| \int_{\gamma_{p}}^{\gamma_{p+1}} \right.
\varpi (1+|m|)^{-\mu} e^{-\beta|m|}
\frac{ \frac{\rho/2}{|\epsilon|^{\chi_1 + \alpha}}}{1 + (\frac{\rho/2}{|\epsilon|^{\chi_1 + \alpha}})^{2\kappa_1} } \\
\times \exp( \nu (\frac{\rho/2}{|\epsilon|^{\chi_1 + \alpha}})^{\kappa_1} )
\exp( -\frac{\cos(\kappa_1(\theta - \mathrm{arg}(\epsilon^{\chi_1+\alpha} t)))}{|\epsilon^{\chi_1 + \alpha} t|^{\kappa_1}}
(\frac{\rho}{2})^{\kappa_1})
\left. e^{-m\mathrm{Im}(z)} d\theta \right| dm\\
\leq \frac{\kappa_1 \varpi }{(2\pi)^{1/2}} \int_{-\infty}^{+\infty}
e^{-(\beta - \beta')|m|} dm \times
|\gamma_{p} - \gamma_{p+1}| \frac{\rho/2}{|\epsilon|^{\chi_1 + \alpha}}
\exp( -\frac{( \frac{\delta_1}{|t|^{\kappa_1}} - \nu)}{2} (\frac{\rho/2}{|\epsilon|^{\chi_1 + \alpha}})^{\kappa_1}) \\
 \times \exp( -\frac{( \frac{\delta_1}{|t|^{\kappa_1}} - \nu)}{2}
 (\frac{\rho/2}{|\epsilon|^{\chi_1+\alpha}})^{\kappa_1})\\
\leq \frac{ 2 \kappa_1 \varpi |\gamma_{p} - \gamma_{p+1}|}{(2\pi)^{1/2}(\beta - \beta')}
\sup_{x \geq 0} x^{1/\kappa_1}e^{-(\frac{\delta_1}{|t|^{\kappa_1}} - \nu)x} \times
\exp( -\frac{( \frac{\delta_1}{|t|^{\kappa_1}} - \nu)}{2} (\frac{\rho/2}{|\epsilon|^{\chi_1 + \alpha}})^{\kappa_1})\\
\leq \frac{2 \kappa_1
\varpi |\gamma_{p} - \gamma_{p+1}|}{(2\pi)^{1/2}(\beta - \beta')} (\frac{1/\kappa_1}{\delta_2})^{1/\kappa_1}
e^{-1/\kappa_1} \exp( -\frac{\delta_{2}}{2} (\frac{\rho/2}{|\epsilon|^{\chi_1 + \alpha}})^{\kappa_1})
\label{I_3_exp_small_order_chi_alpha_kappa}
\end{multline}
for all $t \in \mathcal{T}_1$ and $|\mathrm{Im}(z)| \leq \beta'$ with
$|t| < (\frac{\delta_{1}}{\delta_{2} + \nu})^{1/\kappa_1}$, for some $\delta_{2}>0$, for all
$\epsilon \in \mathcal{E}_{p} \cap \mathcal{E}_{p+1}$.\medskip

Finally, gathering the three above inequalities (\ref{I_1_exp_small_order_chi_alpha_kappa}),
(\ref{I_2_exp_small_order_chi_alpha_kappa}) and (\ref{I_3_exp_small_order_chi_alpha_kappa}), we deduce
from the decomposition (\ref{difference_v_dp_decomposition}) that
\begin{multline*}
|v_1^{\mathfrak{d}_{p+1}}(t,z,\epsilon) - v_1^{\mathfrak{d}_{p}}(t,z,\epsilon)| \leq
 \frac{4 \kappa_1 \varpi}{(2\pi)^{1/2}} \frac{|\epsilon|^{(\chi_1 + \alpha)(\kappa_1-1)}}{(\beta - \beta')
\delta_{2}\kappa_1(\frac{\rho}{2})^{\kappa_1-1}}
\exp( -\delta_{2} \frac{(\rho/2)^{\kappa_1}}{|\epsilon|^{(\chi_1 + \alpha)\kappa_1}})\\
+  \frac{2 \kappa_1 \varpi |\gamma_{p} - \gamma_{p+1}|}{(2\pi)^{1/2}(\beta - \beta')}
(\frac{1/\kappa_1}{\delta_2})^{1/\kappa_1}
e^{-1/\kappa_1} \exp( -\frac{\delta_{2}}{2} (\frac{\rho/2}{|\epsilon|^{\chi_1 + \alpha}})^{\kappa_1})
\end{multline*}
for all $t \in \mathcal{T}_1$ and $|\mathrm{Im}(z)| \leq \beta'$ with
$|t| < (\frac{\delta_{1}}{\delta_{2} + \nu})^{1/k}$, for some $\delta_{2}>0$, for all
$\epsilon \in \mathcal{E}_{p} \cap \mathcal{E}_{p+1}$. Therefore, the inequality
(\ref{exp_small_difference_v_dp1}) holds.

For the proof of (\ref{decomp_singular_holom_udp2}) and (\ref{exp_small_difference_v_dp2}) we can follow the same arguments as in the first part of the proof. We only give some details on the procedure which differ from the previous ones.

The construction leads us to $V_2^{\tilde{\mathfrak{d}}_p}(T,z,\epsilon)=\mathbb{V}_2^{\tilde{\mathfrak{d}}_p}(\epsilon^{\chi_2}T,z,\epsilon)$, defining a bounded holomorphic function with respect to $T$ on those $T$ such that such that $T \in \epsilon^{-\chi_2}S_{\tilde{\mathfrak{d}}_{p},\theta_2,h'|\epsilon|^{\chi_2 + \alpha}}$ and
w.r.t $z$ on $H_{\beta'}$ for any $0 < \beta' < \beta$, for all $\epsilon \in D(0,\epsilon_{0})
\setminus \{ 0 \}$. It holds that $V_2^{\tilde{\mathfrak{d}}_{p}}(T,z,\epsilon)$ solves the equation
\begin{multline}
\tilde{Q}(\partial_z)V_2^{\tilde{\mathfrak{d}}_{p}}(T,z,\epsilon)\left[\frac{a_{0,2}^2}{a_{0,3}}T^{2k_{0,2}-k_{0,3}+\gamma_2}+\sum_{\ell=0}^{s_1}a_{\ell,1}T^{k_{\ell,1}+\gamma_2}+\sum_{\ell=s_1+1}^{M_1}a_{\ell,1}\epsilon^{m_{\ell,1}+\beta-\alpha k_{\ell,1}}T^{k_{\ell,1}+\gamma_2}\right.\\
+\sum_{\ell=1}^{s_2}\frac{-2a_{0,2}a_{\ell,2}}{a_{0,3}}T^{k_{\ell,2}+k_{0,2}-k_{0,3}+\gamma_2}+\sum_{\ell=s_2+1}^{M_2}\frac{-2a_{0,2}a_{\ell,2}}{a_{0,3}}\epsilon^{m_{\ell,2}+2\beta-\alpha k_{\ell,2}}T^{k_{\ell,2}+k_{0,2}-k_{0,3}+\gamma_2}\hfill\\
+\sum_{\ell=0}^{s_2}\frac{-2a_{0,2}a_{\ell,2}}{a_{0,3}}T^{k_{\ell,2}+k_{0,2}-k_{0,3}+\gamma_1}\mathcal{J}_2(T)+\sum_{\ell=s_2+1}^{M_2}\frac{-2a_{0,2}a_{\ell,2}}{a_{0,3}}\epsilon^{m_{\ell,2}+2\beta-\alpha k_{\ell,2}}T^{k_{\ell,2}+k_{0,2}-k_{0,3}+\gamma_1}\mathcal{J}_2(T)\\
\hfill+\sum_{\ell=1}^{s_3}\frac{3a_{0,2}^2a_{\ell,3}}{a_{0,3}^2}T^{k_{\ell,3}+2k_{0,2}-2k_{0,3}+\gamma_2}+\sum_{\ell=s_3+1}^{M_3}\frac{3a_{0,2}^2a_{\ell,3}}{a_{0,3}^2}\epsilon^{m_{\ell,3}+3\beta-\alpha k_{\ell,3}}T^{k_{\ell,3}+2k_{0,2}-2k_{0,3}+\gamma_2}\\
\left.+\left(\sum_{\ell=0}^{s_3}\frac{3a_{0,2}^2a_{\ell,3}}{a_{0,3}^2}T^{k_{\ell,3}+2k_{0,2}-2k_{0,3}+\gamma_2}+\sum_{\ell=s_3+1}^{M_3}\frac{3a_{0,2}^2a_{\ell,3}}{a_{0,3}^2}\epsilon^{m_{\ell,3}+3\beta-\alpha k_{\ell,3}}T^{k_{\ell,3}+2k_{0,2}-2k_{0,3}+\gamma_2}\right)\left(2\mathcal{J}_2(T)+\mathcal{J}_2^2(T)\right)\right]\\
+\tilde{Q}(\partial_{z})(V_2^{\tilde{\mathfrak{d}}_{p}}(T,z,\epsilon))^2\left[\left(\sum_{\ell=0}^{s_2}a_{\ell,2}T^{k_{\ell,2}+k_{0,2}+\gamma_2}+\sum_{\ell=s_2+1}^{M_2}a_{\ell,2}\epsilon^{m_{\ell,2}+2\beta-\alpha k_{\ell,2}}T^{k_{\ell,2}+k_{0,2}+\gamma_2}\right)T^{\gamma_2-k_{0,2}}\right.\\
\hfill\left.+\left(\sum_{\ell=0}^{s_3}a_{\ell,3}T^{k_{\ell,3}+2k_{0,2}-k_{0,3}+\gamma_2}+\sum_{\ell=s_3+1}^{M_3}a_{\ell,3}\epsilon^{m_{\ell,3}+3\beta-\alpha k_{\ell,3}}T^{k_{\ell,3}+2k_{0,2}-k_{0,3}+\gamma_2}\right)\left(1+\mathcal{J}_2(T)\right)\right]\\
+\tilde{Q}(\partial_{z})(V_2^{\tilde{\mathfrak{d}}_{p}}(T,z,\epsilon)
)^3\left[\left(\sum_{\ell=0}^{s_3}a_{\ell,3}T^{k_{\ell,3}+3\gamma_2}+\sum_{\ell=s_3+1}^{M_3}a_{\ell,3}\epsilon^{m_{\ell,3}+3\beta-\alpha k_{\ell,3}}T^{k_{\ell,3}+3\gamma_2}\right)\right]\\
=\sum_{j=0}^{Q}\tilde{b}_{j}(z)\epsilon^{n_j-\alpha b_j}T^{b_j}\\
+\sum_{\ell=1}^{D}\epsilon^{\Delta_{\ell}+\alpha(\delta_{\ell}-d_{\ell})+\beta}\left(\sum_{q_1+q_2=\delta_\ell}\frac{\delta_\ell!}{q_1!q_2!}\prod_{d=0}^{q_1-1}(\gamma_2-d)T^{d_{\ell}-q_1+\gamma_2}\tilde{R}_{\ell}(\partial_z)\partial_{T}^{q_2}V_2^{\tilde{\mathfrak{d}}_{p}}(T,z,\epsilon)\right).\label{e235bc}
\end{multline}

We define the function $U_2^{\tilde{\mathfrak{d}}_{p}}(T,z,\epsilon)$ taking into account (\ref{e216b}), which solves the equation

\begin{multline}
\tilde{Q}( \partial_z)\left(\left(\sum_{\ell=0}^{M_1}a_{\ell,1}\epsilon^{m_{\ell,1}+\beta-\alpha k_{\ell,1}}T^{k_{\ell,1}}\right)U_2^{\tilde{\mathfrak{d}}_p}(T,z,\epsilon)+\left(\sum_{\ell=0}^{M_2}a_{\ell,2}\epsilon^{m_{\ell,2}+2\beta-\alpha k_{\ell,2}}T^{k_{\ell,2}}\right)(U_2^{\tilde{\mathfrak{d}}_p}(T,z,\epsilon))^2\right.\\
\left.+\left(\sum_{\ell=0}^{M_3}a_{\ell,3}\epsilon^{m_{\ell,3}+3\beta-\alpha k_{\ell,3}}T^{k_{\ell,3}}\right)(U_2^{\tilde{\mathfrak{d}}_p}(T,z,\epsilon))^3\right)\\
=\sum_{j=0}^{Q}\tilde{b}_{j}(z)\epsilon^{n_j-\alpha b_j}T^{b_j}+F_2(T,\epsilon)+\sum_{\ell=1}^{D}\epsilon^{\Delta_{\ell}+\alpha(\delta_\ell-d_\ell)+\beta}T^{d_{\ell}}\tilde{R}_{\ell}(\partial_{z})\partial_{T}^{\delta_{\ell}}U_1^{\tilde{\mathfrak{d}}_p}(T,z,\epsilon),\label{e2628b}
\end{multline}
with 

\begin{multline}
F_{2}(T,\epsilon)=-\tilde{Q}(0)\left(\frac{a_{0,2}}{a_{0,3}}T^{k_{0,2}-k_{0,3}}+\frac{a_{0,2}}{a_{0,3}}T^{k_{0,2}-k_{0,3}}\mathcal{J}_2(T)\right)\\
\times\left[\frac{a_{0,2}^2}{a_{0,3}}T^{2k_{0,2}-k_{0,3}}+\sum_{\ell=0}^{s_1}a_{\ell,1}T^{k_{\ell,1}}+\sum_{\ell=s_1+1}^{M_1}a_{\ell,1}\epsilon^{m_{\ell,1}+\beta-\alpha k_{\ell,1}}T^{k_{\ell,1}}\right.\\
+\sum_{\ell=1}^{s_2}\frac{-2a_{0,2}a_{\ell,2}}{a_{0,3}}T^{k_{\ell,2}+k_{0,2}-k_{0,3}}+\sum_{\ell=s_2+1}^{M_2}\frac{-2a_{0,2}a_{\ell,2}}{a_{0,3}}\epsilon^{m_{\ell,2}+2\beta-\alpha k_{\ell,2}}T^{k_{\ell,2}+k_{0,2}-k_{0,3}}\hfill\\
+\sum_{\ell=0}^{s_2}\frac{-2a_{0,2}a_{\ell,2}}{a_{0,3}}T^{k_{\ell,2}+k_{0,2}-k_{0,3}}\mathcal{J}_2(T)+\sum_{\ell=s_2+1}^{M_2}\frac{-2a_{0,2}a_{\ell,2}}{a_{0,3}}\epsilon^{m_{\ell,2}+2\beta-\alpha k_{\ell,2}}T^{k_{\ell,2}+k_{0,2}-k_{0,3}}\mathcal{J}_2(T)\\
\hfill+\sum_{\ell=1}^{s_3}\frac{3a_{0,2}^2a_{\ell,3}}{a_{0,3}^2}T^{k_{\ell,3}+2k_{0,2}-2k_{0,3}}+\sum_{\ell=s_3+1}^{M_3}\frac{3a_{0,2}^2a_{\ell,3}}{a_{0,3}^2}\epsilon^{m_{\ell,3}+3\beta-\alpha k_{\ell,3}}T^{k_{\ell,3}+2k_{0,2}-2k_{0,3}}\\
\left.+\left(\sum_{\ell=0}^{s_3}\frac{3a_{0,2}^2a_{\ell,3}}{a_{0,3}^2}T^{k_{\ell,3}+2k_{0,2}-2k_{0,3}}+\sum_{\ell=s_3+1}^{M_3}\frac{3a_{0,2}^2a_{\ell,3}}{a_{0,3}^2}\epsilon^{m_{\ell,3}+3\beta-\alpha k_{\ell,3}}T^{k_{\ell,3}+2k_{0,2}-2k_{0,3}}\right)\left(2\mathcal{J}_2(T)+\mathcal{J}_2^2(T)\right)\right]\\
-\tilde{Q}(0)\left(\frac{a_{0,2}}{a_{0,3}}T^{k_{0,2}-k_{0,3}}+\frac{a_{0,2}}{a_{0,3}}T^{k_{0,2}-k_{0,3}}\mathcal{J}_2(T)\right)^2\\
\left[\left(\sum_{\ell=0}^{s_2}a_{\ell,2}T^{k_{\ell,2}}+\sum_{\ell=s_2+1}^{M_2}a_{\ell,2}\epsilon^{m_{\ell,2}+2\beta-\alpha k_{\ell,2}}T^{k_{\ell,2}}\right)\right.\\
\hfill\left.+\left(\sum_{\ell=0}^{s_3}a_{\ell,3}T^{k_{\ell,3}+2k_{0,2}-k_{0,3}}+\sum_{\ell=s_3+1}^{M_3}a_{\ell,3}\epsilon^{m_{\ell,3}+3\beta-\alpha k_{\ell,3}}T^{k_{\ell,3}+2k_{0,2}-k_{0,3}}\right)\left(1+\mathcal{J}_2(T)\right)\right]\\
-\tilde{Q}(0)\left(\frac{a_{0,2}}{a_{0,3}}T^{k_{0,2}-k_{0,3}}+\frac{a_{0,2}}{a_{0,3}}T^{k_{0,2}-k_{0,3}}\mathcal{J}_2(T)\right)^3\\
\times\left[\left(\sum_{\ell=0}^{s_3}a_{\ell,3}T^{k_{\ell,3}}+\sum_{\ell=s_3+1}^{M_3}a_{\ell,3}\epsilon^{m_{\ell,3}+3\beta-\alpha k_{\ell,3}}T^{k_{\ell,3}}\right)\right]\\
+\sum_{\ell=1}^{D}\epsilon^{\Delta_{\ell}+\alpha(\delta_{\ell}-d_{\ell})+\beta}T^{d_\ell}\tilde{R}_{\ell}(0)\partial_{T}^{\delta_\ell}\left(\frac{a_{0,2}}{a_{0,3}}T^{k_{0,2}-k_{0,3}}+\frac{a_{0,2}}{a_{0,3}}T^{k_{0,2}-k_{0,3}}\mathcal{J}_2(T)\right),\label{e235bd}
\end{multline}

which can be rewritten in the following form:

\begin{multline}
F_{2}(T,\epsilon)=-\tilde{Q}(0)U_{02}(T)\left[\sum_{\ell=s_1+1}^{M_1}a_{\ell,1}\epsilon^{m_{\ell,1}+\beta-\alpha k_{\ell,1}}T^{k_{\ell,1}}+\sum_{\ell=s_2+1}^{M_2}\frac{-2a_{0,2}a_{\ell,2}}{a_{0,3}}\epsilon^{m_{\ell,2}+2\beta-\alpha k_{\ell,2}}T^{k_{\ell,2}+k_{0,2}-k_{0,3}}\right.\hfill\\
+\sum_{\ell=s_2+1}^{M_2}\frac{-2a_{0,2}a_{\ell,2}}{a_{0,3}}\epsilon^{m_{\ell,2}+2\beta-\alpha k_{\ell,2}}T^{k_{\ell,2}+k_{0,2}-k_{0,3}}\mathcal{J}_2(T)+\sum_{\ell=s_3+1}^{M_3}\frac{3a_{0,2}^2a_{\ell,3}}{a_{0,3}^2}\epsilon^{m_{\ell,3}+3\beta-\alpha k_{\ell,3}}T^{k_{\ell,3}+2k_{0,2}-2k_{0,3}}\\
\left.+\left(\sum_{\ell=s_3+1}^{M_3}\frac{3a_{0,2}^2a_{\ell,3}}{a_{0,3}^2}\epsilon^{m_{\ell,3}+3\beta-\alpha k_{\ell,3}}T^{k_{\ell,3}+2k_{0,2}-2k_{0,3}}\right)\left(2\mathcal{J}_2(T)+\mathcal{J}_2^2(T)\right)\right]\\
-\tilde{Q}(0)(U_{02}(T))^2\left[\left(\sum_{\ell=s_2+1}^{M_2}a_{\ell,2}\epsilon^{m_{\ell,2}+2\beta-\alpha k_{\ell,2}}T^{k_{\ell,2}}\right)\right.\hfill\\
\hfill\left.+\sum_{\ell=s_3+1}^{M_3}a_{\ell,3}\epsilon^{m_{\ell,3}+3\beta-\alpha k_{\ell,3}}T^{k_{\ell,3}+2k_{0,2}-k_{0,3}}\left(1+\mathcal{J}_2(T)\right)\right]\\
-\tilde{Q}(0)(U_{02}(T))^3\left[\sum_{\ell=s_3+1}^{M_3}a_{\ell,3}\epsilon^{m_{\ell,3}+3\beta-\alpha k_{\ell,3}}T^{k_{\ell,3}}\right]\hfill\\
+\sum_{\ell=1}^{D}\epsilon^{\Delta_{\ell}+\alpha(\delta_{\ell}-d_{\ell})+\beta}T^{d_\ell}\tilde{R}_{\ell}(0)\partial_{T}^{\delta_\ell}\left(\frac{a_{0,2}}{a_{0,3}}T^{k_{0,2}-k_{0,3}}+\frac{a_{0,2}}{a_{0,3}}T^{k_{0,2}-k_{0,3}}\mathcal{J}_2(T)\right),\label{e235be}
\end{multline}

The function $F_{2}(T,\epsilon)$ is holomorphic with respect to $\epsilon$ and is analytic with respect to $T$ in some neighborhood of the origin if it holds that
\begin{equation}\label{e2712b}
3k_{0,2}-2k_{0,3}\ge 0,\quad d_\ell+k_{0,2}-k_{0,3}-\delta_{\ell}\ge0,
\end{equation}
for every $1\le \ell\le D$. 

In view of the condition (\ref{e2712b}) and (\ref{e92}), we can affirm that (\ref{e2712b}) is more restrictive than (\ref{e2712}).

We put 
\begin{equation}\label{e2713b}
u_2^{\tilde{\mathfrak{d}}_{p}}(t,z,\epsilon) = \epsilon^{\beta}U_2^{\tilde{\mathfrak{d}}_{p}}(\epsilon^{\alpha}t,z,\epsilon)
= \epsilon^{\beta} \left( U_{02}(\epsilon^{\alpha}t) + (\epsilon^{\alpha}t)^{\gamma_2}
\mathbb{V}_2^{\tilde{\mathfrak{d}}_{p}}(\epsilon^{\chi_2 + \alpha}t,z,\epsilon) \right),
\end{equation}

which defines a holomorphic function w.r.t $t$ on $\mathcal{T}_2$, w.r.t $z \in H_{\beta'}$ for any
$0 < \beta' < \beta$, w.r.t $\epsilon \in \tilde{\mathcal{E}}_{p}$, where $\mathcal{T}_2$ and $\tilde{\mathcal{E}}_{p}$ are
sectors described in Definition~\ref{defi2414}. As a result, $u_2^{\tilde{\mathfrak{d}}_{p}}(t,z,\epsilon)$ admits the decomposition
(\ref{decomp_singular_holom_udp2}) with $v_2^{\tilde{\mathfrak{d}}_{p}}(t,z,\epsilon) =
\mathbb{V}_2^{\tilde{\mathfrak{d}}_{p}}(\epsilon^{\chi_2 + \alpha}t,z,\epsilon)$ which determines a bounded holomorphic
function on $\mathcal{T}_2 \times H_{\beta'} \times \tilde{\mathcal{E}}_{p}$ for any given $0 < \beta' < \beta$ with the
property $v_2^{\tilde{\mathfrak{d}}_{p}}(0,z,\epsilon) \equiv 0$ for all $(z,\epsilon) \in H_{\beta'} \times \tilde{\mathcal{E}}_{p}$.
Again, the function $u_2^{\tilde{\mathfrak{d}}_{p}}(t,z,\epsilon)$ may be meromorphic in both $t$ and $\epsilon$ in the vicinity of the origin. From (\ref{e2628b}) and (\ref{e235bd}) we deduce that $u_2^{\tilde{\mathfrak{d}}_{p}}(t,z,\epsilon)$ solves the main problem
\begin{multline*}
\tilde{Q}( \partial_z)\left(\left(\sum_{\ell=0}^{M_1}a_{\ell,1}\epsilon^{m_{\ell,1}}t^{k_{\ell,1}}\right)u_2^{\tilde{\mathfrak{d}}_p}(t,z,\epsilon)+\left(\sum_{\ell=0}^{M_2}a_{\ell,2}\epsilon^{m_{\ell,2}}t^{k_{\ell,2}}\right)(u_2^{\tilde{\mathfrak{d}}_p}(t,z,\epsilon))^2\right.\\
\left.+\left(\sum_{\ell=0}^{M_3}a_{\ell,3}\epsilon^{m_{\ell,3}}t^{k_{\ell,3}}\right)(u_2^{\tilde{\mathfrak{d}}_p}(t,z,\epsilon))^3\right)\\
=\sum_{j=0}^{Q}\tilde{b}_{j}(z)\epsilon^{n_j}t^{b_j}+F_2(\epsilon^\alpha t,\epsilon)+\sum_{\ell=1}^{D}\epsilon^{\Delta_{\ell}}t^{d_{\ell}}\tilde{R}_{\ell}(\partial_{z})u_1^{\mathfrak{d}_p}(t,z,\epsilon),
\end{multline*}
with additional forcing term $F_2(\epsilon^{\alpha}t,\epsilon)$. We apply the operator $\partial_{z}^{\upsilon}$ on the left and right handside of this last equation, to get that $u_2^{\tilde{\mathfrak{d}}_{p}}(t,z,\epsilon)$ is an actual solution of the main problem (\ref{e85}).

The proof of (\ref{exp_small_difference_v_dp2}) coincides with that of (\ref{exp_small_difference_v_dp1}) step by step.

The proof is completed.
\end{proof}

\end{document}